\documentclass[1 [leqno,11pt]{amsart}    
\usepackage{amssymb, amsmath,amsmath,latexsym,amssymb,amsfonts,amsbsy,     amsthm,mathtools,graphicx,CJKutf8,CJKnumb,CJKulem,color}    
\usepackage{hyperref}
\usepackage{scalerel,stackengine}
\stackMath
\newcommand\reallywidehat[1]{%
\savestack{\tmpbox}{\stretchto{%
  \scaleto{%
    \scalerel*[\widthof{\ensuremath{#1}}]{\kern-.6pt\bigwedge\kern-.6pt}%
    {\rule[-\textheight/2]{1ex}{\textheight}}
  }{\textheight}%
}{0.5ex}}%
\stackon[1pt]{#1}{\tmpbox}%
}
\mathtoolsset{showonlyrefs=true}
\usepackage{cancel}
\usepackage[dvipsnames]{xcolor}
\numberwithin{equation}{section}

\allowdisplaybreaks

\let\rhat=\reallywidehat
\let\al=\alpha
\let\b=\beta
\let\g=\gamma
\let\d=\delta

\let\la=\lambda

\let\s=\sigma
\let\f=\frac

\let\om=\omega

\let\na=\nabla
\let\th=\theta
\let\pa=\partial
\let\ep=\epsilon

\def\rmA{\mathrm{A}}
\def\rmB{\mathrm{B}}
\def\rmC{\mathrm{C}}
\def\rmCK{\mathrm{CK}}
\def\rmD{\mathrm{D}}
\def\rmE{\mathrm{E}}
\def\rmF{\mathrm{F}}

\def\rmH{\mathrm{H}}
\def\rmJ{\mathrm{J}}
\def\rmL{\mathrm{L}}
\def\rmM{\mathrm{M}}
\def\rmN{\mathrm{N}}
\def\rmNR{\mathrm{NR}}
\def\rmNL{\mathrm{NL}}
\def\rmR{\mathrm{R}}
\def\rmS{\mathrm{S}}
\def\rmT{\mathrm{T}}
\def\rmK{\mathrm{K}}

\def\rmI{\mathrm{I}}
\def\rmg{\mathrm{g}}

\def\cG{{\mathcal G}}

\def\bfK{\mathbf{K}}

\def\cM{{\mathcal M}}

\def\cR{{\mathcal R}}

\def\bbD{\mathbb{D}}

\def\R{\mathbf R}
\def\Z{\mathbf Z}

\def\no{\noindent}

\def\eqdef{\buildrel\hbox{\footnotesize def}\over =}

\def\bbT{\mathbb{T}}

\usepackage{accents}
\newcommand*{\dt}[1]{%
  \accentset{\mbox{\large\bfseries .}}{#1}}

\newcommand{\beq}{\begin{equation}}
\newcommand{\eeq}{\end{equation}}
\newcommand{\ben}{\begin{eqnarray}}
\newcommand{\een}{\end{eqnarray}}
\newcommand{\beno}{\begin{eqnarray*}}
\newcommand{\eeno}{\end{eqnarray*}}
\definecolor{schrift}{RGB}{0,73,114}

\newtheorem{theorem}{Theorem}[section]

\newtheorem{lemma}[theorem]{Lemma}
\newtheorem{proposition}[theorem]{Proposition}

\newtheorem{remark}[theorem]{Remark}

\begin{document}
\title[2D Boussinesq system without thermal diffusivity]{Stability of Couette flow for 2D Boussinesq system without thermal diffusivity}
\author[N. Masmoudi]{Nader Masmoudi}
\address{Department of Mathematics, New York University in Abu Dhabi, Saadiyat Island, P.O. Box 129188, Abu Dhabi, United Arab Emirates. Courant Institute of Mathematical Sciences, New York University, 251 Mercer Street, New York, NY 10012, USA,}
\email{masmoudi@cims.nyu.edu}
\author[B. Said-Houari]{Belkacem Said-Houari}
\address{Department of Mathematics, College of Sciences, University of
Sharjah, P. O. Box: 27272, Sharjah, United Arab Emirates.}
\email{bhouari@sharjah.ac.ae}
\author[W. Zhao]{Weiren Zhao}
\address{Department of Mathematics, New York University in Abu Dhabi, Saadiyat Island, P.O. Box 129188, Abu Dhabi, United Arab Emirates.}
\email{zjzjzwr@126.com,  wz19@nyu.edu}

\date{\today}  

\maketitle

\begin{abstract}
In this paper, we prove the stability of Couette flow for 2D Navier-Stokes Boussinesq system without thermal diffusivity for the initial perturbation in Gevrey-$\f 1s$, ($1/3<s\leq 1$). The synergism of density mixing, vorticity mixing and velocity diffusion leads to the stability. 
\end{abstract}

\section{Introduction}

The stability of shear flow in a stratified medium is of interests  in many fields of research such as: fluid dynamics, geophysics, astrophysics, mathematics,...
Density stratification can strongly affect the dynamic of fluids like air in
the atmosphere or water in the ocean and the stability question  of stratified flows
dates back to Taylor 1914 \cite{Taylor_1931} and Goldstein 1931 \cite%
{Goldstein_1931} and since then there has been an active search towards the
understanding of the stability of density-stratified flows. 
 The question
that many researchers want to answer is the following: for a given steady state is
it (asymptotically)  stable relative to small disturbances? 

This is the
problem of the hydrodynamic stability which is one of the most classical problems in the study of fluid dynamics and its investigation dates  back to Rayleigh, Orr, Summerfeld, B\'enard among others. See for instance the book of Drazin and Reid \cite{Drazin_2004} and reference therein.

In this paper, 
we consider the 2D Navier-Stokes Boussinesq system without thermal diffusivity in $\bbT\times \mathbb{R}$:
\beq\label{eq:NSB1}
\left\{
\begin{aligned}  
&\pa_tv+v\cdot \na v+\na P=-\bar{\rho} g e_2+\nu\Delta v,\\
&\pa_t\bar{\rho} +v\cdot\na \bar{\rho} =0,\\
&\na\cdot v=0,
\end{aligned}
\right.
\eeq
where $(x,y)\in \mathbb{T}\times \mathbb{R}$, $v=(v^x,v^y)$ is the velocity field, $P$ is the pressure and  $\bar{\rho}$ is the density and $g=1$ being the normalized gravitational constant 
and $e_2=(0,1)$ is the unit vector in the vertical direction and $\nu$ is the kinematic viscosity. The first equation is the Navier-Stokes equation with the buoyancy forcing {term }$-\bar{\rho} g e_2$ in the vertical direction. The second equation is the transport equation of the density and the  third equation is the incompressibility condition which represents the mass conservation. 

The Boussinesq  system \eqref{eq:NSB1} attracted the attention of many mathematicians: first,  due to its wide range of applications, 
 see for example  \cite{Majda_2003, Pedlosky_1987,Vallis_2006,Constantin_Doering_1999} and second, due to the fact that  the 2D Boussinesq model retain some key features of the 3D Euler and Navier-Stokes equations. For instance it has been known that the inviscid 2D Boussinesq equations are identical to the incompressible axi-symmetric swirling 3D Euler equations, as pointed out in \cite{Majda_Bertozzi_2002}.
 During the last decades  many interesting results were obtained in different directions. 
 
 One of the important research direction is to find the minimal dissipation in the Boussinesq system that yields a global existence  under the  lowest possible regularity.  When both the viscosity and diffusivity are present in the    Boussinesq system, then the system is known to be globally well-posed for smooth and arbitrary large initial data. See for instance \cite{Cannon_Dibenedetto_1980, Foias_Manley_Temam_1987} and also  \cite{Temam_1997}. In the absence of the  diffusivity, the global existence was proved in \cite{Hou_Li_2005} (see also \cite{Chae_2006}). An extension of the results in \cite{Hou_Li_2005,Chae_2006} to a rough initial data in some Besov type spaces has been obtained  in \cite{Hmidi_Keraani_2007}. More importantly,  it has been proved in \cite{Danchin_Paicu_2008}, that the $L^2$ regularity of the initial data is enough to prove the global existence  of the solution. See also  \cite{Danchin_Paicu_2009} where a similar result, but under some extra assumptions, has been proved in the presence of the diffusivity  only.  If the diffusion or the viscosity acts on the horizontal direction on one of the equations only, the authors in \cite{Abidi_Hmidi_2007,Adhikari_al_2016,Danchin_Paicu_2011,Doering_Wu_Zhao_Zheng_2018} showed a global existence result for initial data with different regularities. Under the same regularity assumption as in \cite{Danchin_Paicu_2011}, the uniqueness of the solution was shown in \cite{Larios_Lunasin_Titi_2013}. 
Recently, and by considering  only partial dissipation on the vertical direction in both equations, a  global existence result was obtained in \cite{Li_Titi_2016} under very low regularity assumptions. 
  We recall that in the absence of viscosity and diffusion, the   
 global well-posedness of the inviscid  Boussinesq system is still largely open. See \cite{Chae_2006} and \cite{Chae_Nam_1997} for investigation in this direction.   

Despite the large literature on the Boussinesq system, the  asymptotic stability of solutions has not been well studied. For the non-flowing steady states $v_s=0,\, \rho_s=y$, in \cite{Castro_Cordoba_Lear_2019,Tao_al_2020}, the authors studied the stability problem of 2D Boussinesq with different settings. 

For the flowing steady states as in the case of Couette flow:
\begin{equation}\label{Steady_States}
v_s=(y,0), \quad \rho_s=-r_0 y+1, \quad
p_s=\int_0^y \rho_s(y_1)dy_1=y-\f{r_0}{2}y^2 
\end{equation}
the asymptotic stability problem is very challenging.

The goal of this  paper, is to study the stability of the Couette flow described by the steady state \eqref{Steady_States}. Before stating our main results, let us first recall previous works about the stability problem of flowing steady states. 
%

The linear inviscid 2D Boussinesq system
with shear flows has been extensively studied starting from the work of Taylor \cite{Taylor_1931},  Goldstein  \cite{Goldstein_1931} and Synge \cite{Synge}. We also refer to the book of Lin \cite{Lin_Book}. 
It has been proved  that the stability of the solutions of the linearized inviscid 2D Boussinesq system with a shear flow  
  is determined by the competition between the stabilizing forces and the  vertical shear flow $U(y)$. 

In general if $\rho_s$ is the steady state for the density, we define the local Richardson number $\g(y)$ to be such that
\begin{equation}\label{Richa}
\g(y)^2=-\f{g\pa_y\rho_s(y)}{(\pa_yU(y))^2}. 
\end{equation}
This number measures the ratio of the stabilizing effect of the gravity to the destabilizing effect of the shear. We assume that $\pa_y\rho_s\leq 0 $ (stable stratification) so that $\gamma^2\geq 0$.  

The Richardson number  is one of the control parameters of the stability of stratified shear follows.  The Miles--Howard theorem \cite{Howard_1961,Miles_1961} guarantees that any flow in the inviscid non-diffusive limit
is linearly stable if the local Richardson number everywhere
exceeds the value $1/4$, however unstable modes can arise for Richardson number smaller than $1/4$ \cite{Drazin_1958}.   

The introduction of viscosity or diffusivity in  stratified shear flows may seem to lead to stability, however this is not always correct. As shown in Miller and Lindzen \cite{Miller_Lindzen_1988} for some particular  geometry, the addition of viscosity may allow over-reflection and subsequent instability even if the Richardson number is everywhere greater than $1/4$. In fact they  proved  a normal mode instability for a Richardson number  as large as $0.349$. This is in contrast to the Miles--Howard theorem  in the inviscid case where it shows that unstable modes cannot exist for  any flow with a Richardson number greater than $1/4$.  Hence, it seems an interesting problem to investigate the stability of the stratified shear flows when viscosity is added.  

In the physics literature, there have been a lot of work devoted to the stability of the Couette flow  in the linearized stratified inviscid flow. See for example \cite{Hoiland,Case_1,Dikii_1,Kuo,Hartman_1975,Chimonas,Brown_Stewartson,Farrell_Ioannou}. But they are less mathematically rigorous results.  

In \cite{Yang_2016_1} Yang and Lin studied the linear asymptotic stability of the steady state \eqref{Steady_States} for the 2D Euler Boussinesq system ($\nu=0$). 
\beq\label{eq:LinEB}
\left\{
\begin{aligned}
&\pa_t\om+y\pa_x\om=-\g^2\pa_x\th,\\
&\pa_t\th+y\pa_x\th=u^y,\\
&u=(u_1,u_2)=(-\pa_y\psi,\pa_x\psi),\quad \Delta\psi=\om.
\end{aligned}
\right.
\eeq
See also \cite{CotiZelati} for more linear results of general shear flows. 
They showed that the decay rates depend crucially on the Richardson number. More precisely, they obtained the following decay rates for the velocity components (which confirms the decay rate stated in \cite{Farrell_Ioannou}) 
\beq\label{eq:be-sol-1}
\begin{aligned}
&(u_{1\neq},u_{2\neq},\th_{\neq},\om_{\neq})\lesssim \left(\langle t\rangle ^{-\f12}, \langle t\rangle^{-\f32},\langle t\rangle ^{-\f12},\langle t\rangle ^{\f12}\right)\quad \text{ if}\  \g^2>\f14,\\
&(u_{1\neq},u_{2\neq},\th_{\neq},\om_{\neq})\lesssim \ln (e+|t|)\times\left(\langle t\rangle ^{-\f12}, \langle t\rangle^{-\f32},\langle t\rangle ^{-\f12},\langle t\rangle ^{\f12}\right)\quad \text{ if}\  \g^2=\f14,\\
&(u_{1\neq},u_{2\neq},\th_{\neq},\om_{\neq})\lesssim \langle t\rangle ^{\sqrt{\f14-\g^2}}\times\left(\langle t\rangle ^{-\f12}, \langle t\rangle^{-\f32},\langle t\rangle ^{-\f12},\langle t\rangle ^{\f12}\right)\quad \text{ if}\  0<\g^2<\f14,
 \end{aligned}
 \eeq    
where $f_{\neq}=f-\f{1}{2\pi}\int_{\bbT} f(x,y)dx$ denotes the non-zero mode. 
Let us also point out that the linearized Euler Boussinesq ($\nu=0$) system around \eqref{Steady_States} with $r_0=0$
\beq\label{eq:LinEB-2}
\left\{
\begin{aligned}
&\pa_t\om+y\pa_x\om=-\pa_x\th,\\
&\pa_t\th+y\pa_x\th=0,\\
&u=(u_1,u_2)=(-\pa_y\psi,\pa_x\psi),\quad \Delta\psi=\om,
\end{aligned}
\right.
\eeq
is a couple system with two transport equations (transport diffusion equation if $\nu\neq 0$). The behavior of the solutions are easy to obtain. Indeed we have
\begin{align}\label{eq:be-sol-2}
(u_{1\neq},u_{2\neq},\th_{\neq},\om_{\neq})\lesssim \left(1, \langle t\rangle^{-1},1,t\right),
\end{align} 
which is same as the limit behavior of the solutions of \eqref{eq:LinEB} as $\g^2=r_0\to 0$. However, since the limit process $\g^2=r_0\to 0$ is a singular limit, the system \eqref{eq:LinEB} does not converge to \eqref{eq:LinEB-2}. 

Let us also point out here that if $\th=0$ then\eqref{eq:LinEB-2}  is the linearized Euler (Navier-Stokes) equation whose solutions behave as follows:
\begin{align}\label{eq:be-sol-3}
(u_{1\neq},u_{2\neq},\om_{\neq})\lesssim \left(\langle t\rangle^{-1}, \langle t\rangle^{-2},1\right). 
\end{align}
The vorticity does not grow and behaves better. {\bf The buoyancy forcing term $\pa_x\th$ leads to a growth of the vorticity even at the linear level, which destabilizes the system.} 

In the recent paper \cite{Deng_all_2020}, the authors investigated the stability of the Couette flow for the 2D Navier-Stokes Boussinesq system with both dissipation and thermal diffusion.  They also considered the problem with a weaker  stabilization mechanism and studied the partial dissipation case. The mechanism leading to stability is  the so-called inviscid damping and enhanced dissipation which we will introduce later.

In this paper, we study the system without thermal diffusivity which is the natural and the physical setting. Also mathematically it is much more interesting and challenging. The stability problem of systems with various diffusion terms is always a difficult problem. We refer to \cite{ren2014global,wei2017global} for  similar challenges appearing in MHD. 

In order to state our main result, we introduce the perturbation: $v=u+(y,0)$, $P=p+p_s$ and $\bar{\rho} =\varrho+\rho_s$, then $(u,p,\varrho)$ satisfies
\beq\label{eq:NSB2}
\left\{
\begin{aligned}
&\pa_tu+y\pa_xu+\left(\begin{matrix}u^y\\0\end{matrix}\right)+u\cdot \na u+\na p=-\varrho e_2+\Delta u,\\
&\pa_t\varrho+y\pa_x\varrho-r_0u^y+u\cdot\na \varrho=0,\\
&\na\cdot u=0.           
\end{aligned}
\right.
\eeq
For $r_0>0$, we introduce $\th=\f{1}{r_0}\varrho$, the Richardson number $\g=\sqrt{r_0}$ and the vorticity $$\om=\na\times u=\pa_x u^y-\pa_y u^x$$ which satisfies
\beq\label{eq:NSB3}
\left\{
\begin{aligned}
&\pa_t\om+y\pa_x\om+u\cdot \na \om=-\g^2\pa_x\th+\Delta \om,\\
&\pa_t\th+y\pa_x\th+u\cdot\na \th=u^y,\\
&u=\na^{\bot}\psi=(-\pa_y\psi,\pa_x\psi),\quad \Delta\psi=\om.
\end{aligned}
\right.
\eeq

We also study the case $r_0=0$, and let $\th=\varrho$, then the perturbation $(\th, \om)$ satisfies
\begin{equation}\label{eq:NSB4}
\left\{
\begin{aligned}
&\pa_t\om+y\pa_x\om+u\cdot \na \om=-\pa_x\th+\Delta \om,\\
&\pa_t\th+y\pa_x\th+u\cdot\na \th=0,\\
&u=\na^{\bot}\psi=(-\pa_y\psi,\pa_x\psi),\quad \Delta\psi=\om.
\end{aligned}
\right.
\end{equation}

The dissipation term $\Delta\om$ can hopefully stabilize the system. However comparing to the full diffusion case, in this system, the perturbed density $\th$ does not decay, which leads to a linear growth of the vorticity due to the presence of the buoyancy force term $\pa_x\th$. 

Let us point out that due to the dissipation term, the behavior of the solutions changes a lot. 
Our main result reads  as follows: 
\begin{theorem}\label{Thm:main}
Let $(\om,\th)$ solves \eqref{eq:NSB3} with $\g\neq 0$. Let $(u,\psi)$ be the corresponding velocity field and stream function. 
For all $\f13<s\leq 1$ and $\la_0>\la'>0$, there exists an $ \ep_0=\ep_0(\la_0,\la',s,\g)\leq \f12$ such that for all $\ep\leq \ep_0$ if $(\om_{in},\th_{in})$ and $(u_{in},\psi_{in})$ satisfy 
\beno
\int u_{in}dxdy=\int \om_{in}dxdy=\int \th_{in}dxdy=0,
\eeno
$\int |y\om_{in}(x,y)|dxdy+\int|y\th_{in}(x,y)|dxdy<\ep$ and 
\begin{equation}\label{L_1_Assumption}
\left\|\f{1}{2\pi}\int_{\bbT}\psi_{in}(x,\cdot)dx\right\|_{L^1}\leq \ep
\end{equation}
and
\beno
\|\om_{in}\|_{\mathcal{G}^{\la_0}}^2
+\|\th_{in}\|_{\mathcal{G}^{\la_0}}^2=\sum_{k}\int (|\hat{\om}_{in}(k,\eta)|^2+|\hat{\th}_{in}(k,\eta)|^2)e^{2\la_0|k,\eta|^s}d\eta\leq \ep^2,
\eeno
then there exists $\th_{\infty}$ with $\int \th_{\infty}dxdy=0$ and $\|\th_{\infty}\|_{\mathcal{G}^{\la'}}\lesssim \ep$ such that 
\beq\label{eq: scatting}
\|\th(t,x+ty+\Phi(t,y),y)-\th_{\infty}(x,y)\|_{\mathcal{G}^{\la'}}\lesssim \f{\ep^2}{\langle t\rangle}\ln (e+t)+\f{\ep}{\langle t\rangle^{3}}
\eeq
where $\Phi(t,y)$ is given explicitly by 
\beq\label{eq:Phi(t,y)}
\Phi(t,y)=\f{1}{2\pi}\int_0^t\int_{\mathbb{T}}U^x(\tau,x,y)dxd\tau  
\eeq
Moreover, it holds that
\begin{align*}
\left\|\om(t,x,y)-\f{1}{2\pi}\int\om(t,x,\cdot)dx\right\|_{L^2}&\lesssim \f{\ep}{\langle t\rangle^2},\\
\left\|u^x(t,x,y)-\f{1}{2\pi}\int u^x(t,x,\cdot)dx\right\|_{L^2}&\lesssim \f{\ep}{\langle t\rangle^3},\\
\left\|u^y(t,x,y)\right\|_{L^2}&\lesssim \f{\ep}{\langle t\rangle^4}.
\end{align*}
\end{theorem}

Let us remark that the Gevrey regularity of the initial perturbations does not change for different Richardson numbers. The size of perturbations $\ep$ may vary for different Richardson numbers. 
We prove Theorem \ref{Thm:main} in this paper. One can easily follow the same proof and obtain the following stability result for \eqref{eq:NSB4}: 

\begin{theorem}\label{Thm:main-2}
Let $(\om,\th)$ solves \eqref{eq:NSB4}. 
Let $(u,\psi)$ be the corresponding velocity field and stream function. 
For all $\f13<s\leq 1$ and $\la_0>\la'>0$, there exists an $\ep_0=\ep_0(\la_0,\la',s)\leq \f12$ such that for all $\ep\leq \ep_0$ if $(\om_{in},\th_{in})$ and $(u_{in},\psi_{in})$ satisfy 
\beno
\int u_{in}dxdy=\int \om_{in}dxdy=\int \th_{in}dxdy=0,
\eeno
$\int |y\om_{in}(x,y)|dxdy+\int|y\th_{in}(x,y)|dxdy<\ep$ and 
\beno
\left\|\f{1}{2\pi}\int_{\bbT}\psi_{in}(x,\cdot)dx\right\|_{L^1}\leq \ep
\eeno
and
\beno
\|\om_{in}\|_{\mathcal{G}^{\la_0}}^2
+\|\th_{in}\|_{\mathcal{G}^{\la_0}}^2=\sum_{k}\int (|\hat{\om}_{in}(k,\eta)|^2+|\hat{\th}_{in}(k,\eta)|^2)e^{2\la_0|k,\eta|^s}d\eta\leq \ep^2,
\eeno  
then there exists $\th_{\infty}$ with $\int \th_{\infty}dxdy=0$ and $\|\th_{\infty}\|_{\mathcal{G}^{\la'}}\lesssim \ep$ such that 
\beq\label{eq: scatting-2}
\|\th(t,x+ty+\Phi(t,y),y)-\th_{\infty}(x,y)\|_{\mathcal{G}^{\la'}}\lesssim \f{\ep^2}{\langle t\rangle}\ln(e+t)
\eeq
where $\Phi(t,y)$ is given explicitly by 
\beq\label{eq:Phi(t,y)-2}
\Phi(t,y)=\f{1}{2\pi}\int_0^t\int_{\mathbb{T}}U^x(\tau,x,y)dxd\tau
\eeq
Moreover, it holds that
\begin{align*}
\left\|\om(t,x,y)-\f{1}{2\pi}\int\om(t,x,\cdot)dx\right\|_{L^2}&\lesssim \f{\ep}{\langle t\rangle^2},\\
\left\|u^x(t,x,y)-\f{1}{2\pi}\int u^x(t,x,\cdot)dx\right\|_{L^2}&\lesssim \f{\ep}{\langle t\rangle^3},\\
\left\|u^y(t,x,y)\right\|_{L^2}&\lesssim \f{\ep}{\langle t\rangle^4}.
\end{align*}
\end{theorem}
The results are surprising at first glance. Normally the dissipation term may have smoothing effect on the system. Then the infinite regularity assumptions on initial perturbations are not necessary. However, the buoyancy force term $\pa_x\th$ brings the trouble. One may find some evidences about the necessity of Gevrey-$3_-$ regularity assumptions on the initial perturbation in Section \ref{Section_3}. 
As mentioned, the buoyancy force term $\pa_x\th$ leads to a time growth of the vorticity in the inviscid model. The presence of the viscosity term $\Delta \om$ is a physical setting which also stabilizes the equation of vorticity but destabilizes the equation of density. It causes a significant change of the behavior even at the linear level comparing to the inviscid case. Although the vorticity $\om$ decays as $\f{1}{t^2}$, the density $\th$ becomes worse and does not decay any more. To characterize the mixing effects of the buoyancy force term $\pa_x\th$ and dissipation term $\Delta\om$, we introduce an important good unknown $K$ in this paper, see Section \ref{Section_Main_Result} for more details.  

The mechanism leading to stability is the synergism of density mixing, vorticity mixing and velocity diffusion. It is similar to the inviscid damping caused by vorticity mixing. 
In \cite{Orr_1907}, Orr observed an important phenomenon that the velocity will tend to 0 as $t\to \infty$. This phenomenon is so-called inviscid damping. In \cite{Bedrossian_2015}, Bedrossian and Masmoudi proved nonlinear inviscid damping around the Couette flow in Gevrey class $2_-$ (see also \cite{ionescu_jia_2019_couette}). Nonlinear asymptotic stability and inviscid damping are sensitive to the topology of the perturbation. There are also some negative results. In \cite{Lin_2011}, Lin and Zeng constructed periodic solutions near Couette flow. Recently, Deng and Masmoudi \cite{Deng_Masmoudi_2018} proved some instability for initial perturbations in Gevrey class $2_+$. For general shear flow, due to the presence of the nonlocal operator the inviscid damping for general shear flows is a challenge problem even at the  linear level.
For the linear inviscid damping we refer to \cite{Case_1,zillinger_2017_linear,wei_zhang_zhao2018linear,grenier_Nguyen_Rousset2020linear,jia2020linear,jia2020linear_Gevrey} for the results for general monotone flows. 
For non-monotone flows such as the Poiseuille flow and the Kolmogorov flow, another dynamic phenomena should be taken into consideration, which is so-called the vorticity depletion phenomena, predicted by Bouchet and Morita   \cite{bouchet2010large} and later proved by Wei, Zhang and Zhao \cite{wei_Zhang_Zhao2019linear,wei_zhang_zhao2020linear}. Very recently, Ionescu and Jia \cite{ionescu2020nonlinear}, Masmoudi and Zhao \cite{masmoudi_zhao2020nonlinear} proved the nonlinear inviscid damping for stable monotone shear flow independently. The inviscid damping is the analogue in hydrodynamics of the Landau damping found by Landau \cite{Landau_1946} and later proved by Mouhot and Villani \cite{Mouhot_Villani_2011} (see also \cite{bedrossian_Masmoudi_Mouhot_2016, Bedrossian_2021}), which shows the rapid decay of the electric field of the Vlasov equation around homogeneous equilibrium. See \cite{zillinger2017circular,ren_zhao_2017,ionescu_jia_2019_axi,bedrossian_Coti_Vicol2019vortex,wei_Zhang_Zhu2020linear} for similar phenomena in various system. 

It remains a very interesting problem to study the nonlinear asymptotic stability/instability of shear flow for the Euler Boussinesq system. 

We also remark that when $\th=0$, the system \eqref{eq:NSB4} reduces to the 2D Navier Stokes. The stability problem of 2D Couette flow has previously been investigated. One may refer to \cite{Bedrossian_2016, bedrossian_Vicol_Wang2018sobolev,masmoudi_zhao2020enhanced,Masmoudi_Zhao_2019} for infinite channel case, and to \cite{chen_Li_Wei_Zhang2018transition,bedrossian_he2019inviscid} for finite channel case and to \cite{Bedrossian_Germain_Masmoudi_I,Bedrossian_Germain_Masmoudi_II,BGM_2017_1,Wei_Zhang_2020,Chen_wei_zhang_2020_3d} for stability results of 3D Couette flow. We also refer to references \cite{cotizelati_elgindi2020enhanced_poiseuille,li_wei_Zhang2020pseudospectral,ding_lin2020enhanced,lin_Xu2019metastability} for the stability results of other shear flows. 

In a forthcoming paper, the small viscosity case will be studied, where the Richardson number will play an important role. Of course, under the assumption that the initial perturbations are sufficiently small (depending on the viscosity), and by following the proof in this paper, one can prove the stability results. The main problem in the small viscosity case should be the optimality of the size. 
\subsection{Notation and conventions}

See \cite[Appendix A]{Bedrossian_2015}  for the Fourier analysis conventions we are taking. A convention we generally use is to denote the discrete $x$ (or $z$) frequencies as subscripts. By convention we always use Greek letters such as $\eta$ and $\xi$ to denote frequencies in the $y$ or $v$ direction and lowercase Latin characters commonly used as indices such as $k$ and $l$ to denote frequencies in the $x$ or $z$ direction (which are discrete). Another convention we use is to denote $\rmM,\rmN,\rmK$ as dyadic integers. That is $\rmM,\rmN,\rmK\in \mathbb{D}$ where 
\beno
\mathbb{D}=\left\{\f12,1,2,4,8,...,2^j,...\right\}.
\eeno
When a sum is written with indices $\rmK,\rmM,\rmM',\rmN$ or $\rmN'$ it will always be over a subset of $\mathbb{D}$. 
We will mix use same $\rmA$ for $\rmA f=(\rmA(\eta)\hat{f}(\eta))^{\vee}$ or $\rmA\hat{f}=\rmA(\eta)\hat{f}(\eta)$, where $\rmA$ is a Fourier multiplier.

We use the notation $f\lesssim g$ when there exists a constant $C>0$ independent of the parameters of interest such that $f\leq Cg$ (we analogously define $g\gtrsim f$). Similarly, we use the notation $f\approx g$ when there exists $C>0$ such that $C^{-1}g\leq f\leq Cg$. 

We will denote the $l^1$ vector norm $|k,\eta|=|k|+|\eta|$, which by convention is the norm taken in our work. Similarly, given a scalar or vector in $\R^n$ we denote 
\beno
\langle v\rangle = (1+|v|^2)^{\f12}. 
\eeno
We use a similar notation to denote the $x$ or $z$ average of a function: $<f>=\f{1}{2\pi}\int f(x,y)dx=f_0$. We also frequently use the notation $f_{\neq}=P_{\neq}f=f-f_0$. We denote the standard $L^p$ norms by $\|\cdot\|_p$ for $1\leq p\leq \infty$. 

For any $f$ defined on $\R$, we make common use of the Gevery-$\f1s$ norm with Sobolev correction defined by
\beno
\|f\|_{\mathcal{G}^{\la,\s;s}}=\sum_k\int\left|\hat{f}_k(\eta)\right|^2e^{2\la|k,\eta|^s}\langle k,\eta\rangle^{2\s}d\eta.
\eeno

For $\eta\geq 0$, we define $\rmE(\eta)\in \mathbf{Z}$ to be the integer part. We define for $\eta\in \R$ and $1\leq |k|\leq \rmE(|\eta|^{\f13})$ with $\eta k\geq 0$, $t_{k,\eta}^{-}=\big|\f{\eta}{k}\big|-\f{|\eta|}{2|k|^3}$, $t_{k,\eta}^{+}=\big|\f{\eta}{k}\big|+\f{|\eta|}{2|k|^3}$ and the critical intervals
\beno
\mathrm{I}_{k,\eta}=\left\{
\begin{split}
&[t_{k,\eta}^-,t_{k,\eta}^+]\quad &\text{if}\ \eta k\geq 0 \ \text{and } 1\leq |k|\leq \rmE(|\eta|^{\f13}),\\
&\emptyset \quad &\text{otherwise}.
\end{split}\right.
\eeno
We also introduce $\rmI_{k,\eta}\eqdef [t_{k,\eta}^-,t_{k,\eta}^+]\subset [\f{2\eta}{2k+1},\f{2\eta}{2k-1}]\eqdef\bar{\rmI}_{k,\eta}$.

\section{Main difficulties, ideas and sketch of the proof}\label{Section_Main_Result}
We next give the proof of Theorem \ref{Thm:main}, starting the primary steps as propositions which are proved in subsequent sections. 
The stability or instability of the steady state \eqref{Steady_States} for the  2D nonlinear  Euler Boussinesq system (i.e., \eqref{eq:NSB1} with $\nu=0$)   is unknown due to the growth of vorticity, and to authors knowledge, even no partial result is available. The complexity of the problem becomes  clear form the linear behavior of the vorticity of the linearized problem. In fact as shown in \cite{Yang_2016_1}, the vorticity $w(t)$ grows roughly like $\sqrt{t}$. This seems far away from the situation of the Euler equation discussed in \cite{Bedrossian_2015} where the vorticity stays bounded, and even without any time growth,  a Gevrey-$2$ regularity in \cite{Bedrossian_2015} was necessary to close the estimates and prove stability. See also \cite{Deng_Masmoudi_2018} for a negative result if  the  regularity is  below Gevery-$2 $. Hence, it seems that a stability result for the 2D nonlinear  Euler Boussinesq system may not be possible even for analytic regularity, since a  small perturbation of the steady state \eqref{Steady_States}  may amplify  by  a very large factor. Therefore, from  the  stability point of view, the presence of the viscosity term $\Delta v$ in \eqref{eq:NSB1} is completely  justified both physically and mathematically. However the   presence of viscosity will  lead to other complications,  since it acts as a stabilizing factor only for the vorticity equation, but due to the presence of the buoyancy forcing term $-\bar{\rho} g e_2$, the viscosity  has a  destabilizing effect on the equation of density, since in the absence of viscosity, $\rho$ decays roughly like $\f{1}{\sqrt{t}}$, but in the presence of viscosity,  $\rho$ does not decay at all. This leads to a major difficulty in the analysis and  due to this fact and from the growth of the toy model in Section \ref{Section_3}, it seems that a Gevrey-$3 $ regularity is needed. In a forthcoming paper, we will study the optimality of this regularity. 

Hence, in order to use the damping of the vorticity equation for the density equation, we introduce a new unknown $K$ that connect the vorticity and the density (see the definition of $K$ in \eqref{K_Eq_Linear} for the linear problem and the adapted one \eqref{K_Equation}  for the nonlinear problem). The unknown $K$ creates somehow a balance between the buoyancy term and the viscosity term. However, in terms of analysis it leads to some extra terms that we should control carefully. (See the definition of $\mathbf{H}$ in \eqref{eq: H_formula}). 

Another issue in the proof is that even in the presence of viscosity, it seems not possible to use the nonlinear coordinate systems introduced in \cite{Bedrossian_2016}, since this leads to a shear flow term in the equation of density, which cannot be controlled. So, due to this  we rely on an inviscid change of coordinates as in \cite{Bedrossian_2015}. However, due to  the presence of the viscosity, the control of the coordinate systems is   different from  the one in \cite{Bedrossian_2015} and the extra $L^1$-control  \eqref{L_1_Assumption} is needed to get enough decay of the coordinates.

Also, compared to \cite{Bedrossian_2015} the norm introduced here (see \eqref{Multiplier}) contains the two extra components  $\cM_k(t,\eta)$ and $\rmB_k(t,\eta)$. The multiplier $\cM_k(t,\eta)$   is used to control the growth  in appropriate  time regime and together with $\rmB_k(t,\eta)$ they  have  been also used as  ``ghost" weight in phase place to control the growth coming from some linear terms.               

One of the key ideas in the proof of Theorem \eqref{Thm:main} is the construction of time-dependent norm which contains several components, each component is  introduced  to control the growth predicted by the toy model in different time regimes.  (See  Section \ref{Section_3} for more details).  
Armed with such a norm, and by applying energy estimates, we were able to allow the loss of regularity at specific frequency and time and it enables us also to pay regularity for time decay in order to close the energy estimates.  
Another complicated issue in the proof is the absence of any damping in the density equation, for this reason we need to find a nice combination that connects the density  to the velocity (see the definition of $K$ in \eqref{K_Equation}). This combination allows us to transfer damping from the velocity equation to the density equation.
Another important remark, which is well known in this Gevrey-type estimates is that by allowing  $\lambda$ to shrink (see \eqref{lambda_Decay}) we were able to introduce the $\rmC\rmK_\la$  terms that will play a role of an extra damping  term that will help to control many terms in some specific time regime.   

\subsection{Linearized behavior and an important good unknown}
Before beginning the proof of Theorem \ref{Thm:main}, we discuss the linearized behavior in more detail and mention some of the main challenges that must be overcome for a nonlinear result. The linearized equation of \eqref{eq:NSB3} or \eqref{eq:NSB4} can be written as:
\beq\label{eq:LNSB3-4}
\left\{  
\begin{aligned}
&\pa_t\om+y\pa_x\om=-\g^2\pa_x\th+\Delta \om,\\
&\pa_t\th+y\pa_x\th=\g_1u^y,\\
&u=\na^{\bot}\psi=(-\pa_y\psi,\pa_x\psi),\quad \Delta\psi=\om,
\end{aligned}
\right.
\eeq
where the parameters $\g, \g_1$ varies in different cases. Let us consider the simple case: $\g=1$ and $\g_1=0$, which are the parameters of the linearized equation of  \eqref{eq:NSB4}. Now we introduce the linear change of coordinates: 
\beq
z=x-ty,\quad f(t,z,y)=\om(t,x,y),\quad \rho(t,z,y)=\th(t,x,y).
\eeq

From \eqref{eq:LNSB3-4} with $\g=1$ and $\g_1=0$, we have
\beq\label{eq:LNSB5}
\left\{
\begin{aligned}
&\pa_tf=-\pa_z\rho+(\pa_{v}-t\pa_z)^2f+\pa_{zz}f,\\
&\pa_t\rho=0,
\end{aligned}
\right.
\eeq
which gives us that $\hat{\rho}(t,k,\eta)=\hat{\rho}_{in}(k,\eta)$ and for $k\neq 0$
\beno
\hat{f}(t,k,\eta)=e^{-\int_0^t(ks-\eta)^2+k^2ds}\left(\hat{f}_{in}(k,\eta)-ik\hat{\rho}_{in}(k,\eta)\int_0^te^{\int_0^{\tau}(ks-\eta)^2+k^2ds}d\tau\right). 
\eeno
By the fact that  
\begin{align*}
e^{-\int_0^t(ks-\eta)^2+k^2ds}\int_0^te^{\int_0^{\tau}(ks-\eta)^2+k^2ds}d\tau\lesssim \f{1}{(kt-\eta)^2+k^2},
\end{align*}
we get that 
\ben\label{eq:est-lin-1}
|\hat{f}(t,k,\eta)|\lesssim e^{-\int_0^t(ks-\eta)^2+k^2ds}|\hat{f}_{in}(k,\eta)|+\f{k}{(kt-\eta)^2+k^2}|\hat{\rho}_{in}(k,\eta)|. 
\een
However for the nonlinear system and the case $\g_1\neq 0$ even at the linear level, we can not write the precise formula for the solutions. It is not just a technical difficulty. Indeed, to obtain the behavior of the solution, we should balance the effect of the buoyancy force term $\pa_x\th$ and dissipation term $\Delta\om$. 
We define the following ``good" unknown: 
\begin{equation}\label{K_Eq_Linear}
K=-\g^2\pa_z\rho+\Delta_{L}f
\end{equation}
which has good properties. Indeed we have
\begin{align*}
\pa_t\hat{K}+(k^2+(\eta-kt)^2)\hat{K}=-\g^2\f{k^2}{k^2+(\eta-kt)^2}\hat{f}+2k(\eta-kt)\hat{f}
\end{align*}
which gives us that
\begin{align*}
\f12\pa_t|\hat{K}|^2+(k^2+(\eta-kt)^2)|\hat{K}|^2
&\leq \g^2\f{k^2}{k^2+(\eta-kt)^2}|f||\hat{K}|+2k|\eta-kt||f||\hat{K}|\\
&\leq \g^2\f{k^2}{k^2+(\eta-kt)^2}|\hat{f}||\hat{K}|+10k^2|\hat{f}|^2+\f{1}{10}(\eta-kt)^2|\hat{K}|^2. 
\end{align*}
Note that $K$ satisfies a diffusion equation (or transport diffusion equation in the $(t,x,y)$ coordinate) with forcing terms that decay fast enough. 

Therefore we get $|\hat{K}(t,k,\eta)|\lesssim |\hat{K}_{in}(k,\eta)|$ and
\beno
|\hat{f}(t,k,\eta)|\lesssim \f{1}{(kt-\eta)^2+k^2}(|k||\hat{\rho}_{in}|+|\hat{K}_{in}|). 
\eeno

The good unknown $K$ characterizes the mixing effects of the buoyancy force term $\pa_x\th$ and dissipation term $\Delta\om$, which is one of the key structures we found in this system. For the nonlinear system, $K$ will change slightly due to the nonlinear change of coordinates, but we will use, without ambiguity, the same notation, see \eqref{K_Equation} below.

\subsection{Coordinate transform}
In order to tackle the nonlinear problem \eqref{eq:NSB3}, 
we make suitable nonlinear change of variables. The basic idea of this change of variables is to get a rid of the zero mode, since the zero mode does not decay. There are two different types of change of coordinates: one is the inviscid one for the Euler equation, see \cite{Bedrossian_2015}, the other one is viscous one for the Navier-Stokes equation, see \cite{Bedrossian_2016}. The coordinates system were chosen in a very natural way in both cases. However, in this paper, the diffusion term only appears in the vorticity equation, the density equation is a transport equation. It is always a challenge problem when the diffusion terms are different in one system, see \cite{wei2017global}. Here we use the inviscid change of coordinates. 
 
Let 
$$u_0^x(t,y)=-\frac{1}{2\pi }\int_{\mathbb{T}} \partial _{y}\psi(t,x,y)dx=\frac{1}{2\pi }\int_{\mathbb{T}} u^{x}(t,x,y)dx,$$ then 
system \eqref{eq:LNSB3-4} can be rewritten
as 
\begin{equation}
\left\{ 
\begin{array}{ll}
\partial _{t}\omega +(y+u_0^x(t,y))\partial _{x}\omega +(\nabla
^{\perp }\psi_{\neq }\cdot \nabla )\omega =-\gamma^2\partial
_{x}\theta-\Delta \om ,\vspace{0.2cm} &  \\ 
\partial _{t}\theta +(y+u_0^x(t,y))\partial _{x}\theta +(\nabla
^{\perp }\psi_{\neq }\cdot \nabla )\theta =\gamma_1\partial_x\psi.  & 
\end{array}%
\right.  \label{Main_System_Bousq_Vorticity_Change}
\end{equation}
To remove the non-decaying zero mode from the above system, we introduce the change of variables 
$(t,x,y)\to (t,z,v)$ with 
\begin{equation}
\left\{ 
\begin{aligned}
z&=x-tv, \\ 
v&=y+\f{1}{t}\int_0^t\f{1}{2\pi}\int_{\bbT}u^x(\tau,x,y)dxd\tau. 
\end{aligned}
\right.  \label{Change_Variables_2}
\end{equation}
We then have
\begin{equation}
\frac{d}{dt}\big( t(\pa_{y}v(t,y)-1)\big) =-\omega_0(t,y).
\end{equation}
Let 
\begin{align*}
f (t,z,v )&=\omega (t,x,y),   \quad 
 \rho (t,z,v )=\theta (t,x,y),  \quad 
\phi (t,z,v )=\psi (t,x,y),\\
v'(t,v)&=\pa_yv(t,y),\quad 
v''(t,v)=\pa_{yy}v(t,y),\quad
g(t,v)=\pa_tv(t,y),\\
\tilde{u}_0(t,v)&=u_0^x(t,y),\quad
h(t,v)=v'(t,v)-1.
\end{align*}
Then $v''(t,v)=\f{1}{2}\pa_v\big((v')^2-1\big)$
and $\phi $ satisfies the equation 
\begin{equation}
\Delta _{t}\phi \eqdef\left[ \partial _{zz}+(v')^{2}(\partial
_{\upsilon }-t\partial _{z})^{2}+v''(\partial _{\upsilon
}-t\partial _{z })\right] \phi =f,
\end{equation}
where $\Delta_t$ can be regarded as a perturbation of $\Delta_L$. 

Hence, we obtain the following equations  for $f$:
\begin{equation}
\partial _{t}f+\mathbf{u}\cdot
\nabla _{z,\upsilon }f  
=-\gamma ^{2}\partial _{z} \rho + \Delta_{t}f,\qquad \Delta _{t}\phi =f,
\end{equation}
with 
\begin{equation}
\mathbf{u}(t,z,\upsilon)=\Big( 
\begin{matrix}
0 \\ 
g  
\end{matrix}%
\Big)+v'\nabla _{z,\upsilon }^{\perp }P_{\neq }\phi.
\end{equation}
We also obtain the following equation for $\rho$: 
\begin{subequations}\label{Main_System_I}
\begin{equation} \label{Main_New_system_3}
\begin{aligned}
&\partial _{t} \rho +\mathbf{u}\cdot
\nabla _{z,\upsilon } \rho =\gamma_1\partial _{z}\phi ,\vspace{0.05cm} \\ 
&\Delta _{t}\phi =f,  
\end{aligned}
\end{equation}
Following the argument of the linearized system, let us introduce
\begin{equation}\label{K_Equation}
K=-\gamma^2\partial_z \rho+\Delta_t f. 
\end{equation}

Hence, $K$ satisfies the equation: 
 \begin{equation}\label{K_Equation}
\begin{aligned}
&\partial_t K+\mathbf{u}\cdot \nabla_{z,v}K-\Delta_t K
=\mathbf{H}-\gamma^2\gamma_1\partial _{zz}\phi-2(\partial_{v}-t\partial_z) \partial_zf
\end{aligned}  
\end{equation}
\end{subequations}
where
\begin{equation}\label{eq: H_formula}
\begin{aligned}
\mathbf{H}
=&\gamma^2 v^{\prime}\nabla^{\perp}_{z,v}\partial_z P_{\neq }\phi\cdot \nabla_{z,v} \rho
-2h(\partial_{v}-t\partial_z) \partial_zf+2f_0v'\pa_z(\pa_v-t\pa_z)f\\
&-2(v^\prime)^3(\partial_{v}-t\partial_z)\nabla^\perp_L  \phi_{\neq}\cdot  \nabla_L (\partial_{v}-t\partial_z)f
-2v^\prime \nabla^\perp_L \partial_z \phi_{\neq}\cdot \nabla_L \partial_zf
\\
&-\pa_z\big(v^{\prime\prime}v^\prime (\partial_{v}-t\partial_z)  \phi_{\neq}(\pa_v-t\pa_z) f \big).
\end{aligned}
\end{equation}

\begin{subequations}\label{Main_System_2}
We also have
\begin{equation}\label{Equ_Av_f}
\partial_tf_{0}+ g\partial_v f_0-(v^\prime)^2 \partial_{vv}f_0 -
 v^{\prime\prime} \partial_{v}f_{0}+ v^\prime  < \nabla^\perp \phi_{\neq}\cdot
\nabla _{z,\upsilon }f_{\neq}>=0. 
\end{equation}
Therefore
\begin{equation}
\partial _{t}h +g\partial _{v }h =-\frac{1}{t}\left( f_0+h \right) \label{eq: h-equ}
\end{equation}
and
\begin{equation*}
\partial _{t}g+\frac{2}{t}g+g\partial _{\upsilon }g=-\frac{1}{t}v^{\prime
}< \nabla ^{\perp }\phi _{\neq }\cdot \nabla \tilde{u}_{\neq}>+\f1 t v'\pa_vf_0
\end{equation*}
where $\tilde{u}(t,z,v)=u(t,x,y)$.
Let 
\begin{equation}
\bar{h}=-\f{1}{t}(f_0+h)=v'\pa_vg, \label{eq: bar h}
\end{equation}
\end{subequations}
then $\bar{h}$ satisfies 
\beno
\partial_t\overline{h}+\frac{2}{t}\overline{h}+g\partial_v\overline{h}+\frac{1}{t}K_0
=\frac{1}{t}v' < \nabla^\perp_{z,v} \phi_{\neq}\cdot \nabla _{z,v }f>.  
\eeno
In the sequel, we will perform the estimates using system \eqref{Main_System_I} together with \eqref{Main_System_2}. 

\subsubsection{Discussion of the nonlinear change of coordinates}
In this paper, we use the inviscid nonlinear change of coordinates (see \cite{Bedrossian_2016} for the viscous nonlinear change of coordinates). 

By the definition of $g$, an easy calculation shows that
\beno
g(t,v)=\pa_tv(t,y)=\f{1}{t^2}\int_{0}^t s\pa_tu_0^x(s,y)ds,
\eeno
where the zero mode of velocity satisfies 
\begin{equation}
\partial _{t}u^x_0-\pa_{yy}u^x_0=-\pa_y<\pa_x\psi u^x>.  \label{eq:u_1_Equation}
\end{equation}

In order to get enough decay of $g$ in lower regularity, we introduce the equation of the stream function. By the fact that $\int u_{in}(x,y)dxdy=\int \th_{in}(x,y)dxdy {=0}$, we have for any $t\geq 0$, $\int u(t,x,y)dxdy=0$. Thus the average of the stream function $\psi_0(t,y)=\f{1}{2\pi}\int\psi(t,x,y)dx$ satisfies the 1D nonlinear heat equation 
\begin{equation}
\partial _{t}\psi_0-\pa_{yy}\psi_0=-<\pa_x\psi u^x> \label{eq:psi_0}
\end{equation}
and $u_0^x(t,y)=-\pa_y\psi_0(t,y)$.

\subsection{Main Energy estimate}
We will use a carefully designed time-dependent norms written as
\begin{equation}\label{Norm_K}
\|\rmA(t,\na)K\|_2^2=\sum_k\int_{\eta}|\rmA_{k}(t,\eta)\widehat{K}_k(t,\eta)|^2d\eta,
\end{equation}
and
\begin{equation}\label{Norm_rho}
\|\rmA(t,\na)\rho\|_2^2=\sum_k\int_{\eta}|\rmA_{k}(t,\eta)\widehat{\rho}_k(t,\eta)|^2d\eta.
\end{equation}

The multiplier $\rmA$ has several components
\begin{equation}\label{Multiplier}
\rmA_{k}(t,\eta)=e^{\la(t)|k,\eta|^s}\langle k,\eta\rangle^{\s}\rmJ_k(t,\eta)\cM_k(t,\eta)\rmB_k(t,\eta). 
\end{equation}
The index $\la(t)$ is the bulk Gevrey-$\f{1}{s}$ regularity and will be chosen 
to satisfy (see \cite{Bedrossian_2015})
\begin{equation}
\lambda(t)=\frac{3}{4}\lambda_0+\frac{1}{4}\lambda^\prime,\qquad t\leq 1
\end{equation}
and 
\begin{equation}\label{lambda_Decay}
\frac{d}{dt}\lambda(t)=-\frac{\delta_\lambda}{\langle t\rangle^{2\tilde{q}}}(1+\lambda(t)),\quad t>1,
\end{equation}
where $\delta_\lambda\approx \lambda_0-\lambda^\prime$ is a small parameter that  ensures $\lambda(t)>\lambda_0/2+\lambda^\prime/2$ and $\tilde{q}$ is a parameter that will be determined by the proof. Let us also remark here that to study analytic data, $s=1$, we would need to add an additional Gevrey-$\f{1}{s'}$ correction to $\rmA$ with $s'\in (\f13,1)$ as an intermediate regularity so that we may take advantage of certain beneficial properties of Gevrey spaces.

The main multipliers for dealing with the nonlinear interaction are
\ben
\rmJ_{k}(t,\eta)=\f{e^{\mu|\eta|^{\f13}}}{\Theta_k(t,\eta)}+e^{\mu|k|^{\f13}},
\een
and
\ben
\cM_k(t,\eta)=\f{e^{4\pi\d_{\rmL}^{-1}|\eta|^{\f13}}}{\rmg(t,\eta)}+e^{4\pi \d_{\rmL}^{-1} |k|^{\f13}},
\een
with $\d_{\rmL}>0$ being a small enough constant that will be determined by the linear nonlocal term. The weights
 $\Theta_k(t,\eta)$ and $\rmg(t,\eta)$ are constructed in Section \ref{Section_3}. 

The multiplier $\rmB$ is defined as follows:
\beno
\rmB_k(t,\eta)=\exp\Big(\d_{\rmB}^{-1}\int_0^t\f {b\big(s,{k},\eta\big)}{1+(s-\f{\eta}{k})^2}ds\Big).
\eeno   
Here 
\beno
b\big(t,{k},\eta\big)=\chi\big(\f{100}{t}\big)\chi\big(\f{\eta}{t^3}\big)\chi\big(\f{\eta}{k^3}\big)\chi_1\big(\f{\eta}{kt}\big)
\eeno
where $0\leq \chi(x)\leq 1$ is a smooth cut-off function satisfying $\chi(x)\equiv 1$ for $|x|\leq 8$ and $\mathrm{supp}\, \chi\subset [-10,10]$ and $0\leq \chi_1(x)\leq 1$  is a smooth cut-off function satisfying $\chi_1(x)\equiv 1$ when $\f12\leq |x|\leq \f32$ and $\mathrm{supp}\, \chi\subset [\f13,\f52]$. 
 
Thus we get that $\rmB_0(t,\eta)\equiv 1$ and 
\ben
\rmB_k(t,\eta)\approx_{\d_{\rmB}} 1. 
\een

Note that 
\ben
\sup_{t,k}|\pa_{\eta}b(t,k,\eta)| \lesssim \f{1}{\langle \eta\rangle}. 
\een

With this special norm, we can define our main energy:
\begin{align}
\rmE(t)=\f12\|\rmA(t,\na)K\|_2^2+\f12\|\rmA(t,\na)\rho\|_2^2,\quad
\rmE_{d}(t)=\f12\langle t\rangle\|\rmA \langle\pa_{v}\rangle^2h\|_{2}^2
\end{align}
and
\begin{align}
\rmE_{lo,f_0}(t)&=\sum\limits_{k=0}^{3}\f{t^{k}}{4^k}\|\pa_v^{k}f_0\|_{\cG^{\la(t),\b;s}}^2,\\
\rmE_{lo,g}(t)&=\langle t\rangle^{4}\|\pa_v^{3}g\|_{\cG^{\la(t),\b;s}}^2, \quad
\rmE_{lo,h}(t)=\langle t\rangle^{2}\|\pa_v^{2}h\|_{\cG^{\la(t),\b;s}}^2,\quad 
\end{align}
and the assistant energy
\begin{align}
\rmE_{as,\psi_0}(t)&=\max\Big(
\langle t\rangle \|\psi_0\|_{L^{\infty}}^2,\ 
\langle t\rangle^4\|\pa_{yyy}\psi_0\|_{L^{\infty}}^2,\ 
\langle t\rangle^{\f52}\|\psi_0\|_{\dot{H}^2}^2,\ 
\langle t\rangle^{\f92}\|\psi_0\|_{\dot{H}^4}^2\Big),\\
\rmE_{as,g}(t)&=\max\Big(\f{\langle t\rangle^4\|g\|_{L^{\infty}}^2}{(1+\ln \langle t\rangle)^2},\ \langle t\rangle^{\f{7}{2}}\|g\|_{L^2}^2,\ 
\ep_1^2\langle t\rangle^{4-\ep_1}\|g\|_{\dot{H}^{\f12-\ep_1}}^2,\ \ep_2^2\langle t\rangle^{4}\|g\|_{\dot{H}^{\f12+\ep_2}}^2\Big),
\end{align}
where $0<\sqrt{\ep}\leq \ep_1^2$ with $0<\ep_1\leq 3s+1-4\tilde{q}$ and $0<\ep_2<\f12$. 

Let us make a technical remark about the assistant energy here: We use the inviscid change of coordinate which brings us new challenge since the equation are not both inviscid. It leads to the appearance of the term $\f1tv'\pa_vf_0$ in the equation of $g$. Then it holds that $\|g\|_{L^2}\lesssim \ep\langle t\rangle^{-\f74}$, which decays not fast enough to control the main energy. However, the assistant energy together with Proposition \ref{prop: g,f_0,h} shows that the homogeneous norms $\|g\|_{\dot{H}^{\f12-\ep_1}}, \|g\|_{\dot{H}^{\f12+\ep_2}}, \|\pa_vg\|_{\cG^{\la,\b;s}}$, the Fourier $L^1$ norms $\|\hat{g}\|_{L^1}, \|g\|_{\cG_1^{\la,\b;s}}$ and $\|g\|_{L^{\infty}}$ have better decay rate. Luckily we can use these norms of $g$ to close our energy estimate. 

It is natural to compute the time evolution of $\|\rmA(t,\na)K\|_2^2$ and $\|\rmA(t,\na)\rho\|_2^2$.
To lighten  notations, we define for $\varphi\in\{K,\rho\}$, the following:
\begin{equation}\label{Ck_Terms_rho}
\begin{aligned}
\rmCK_{\la,\varphi}=&\,-
\dt{\lambda}(t)\Vert |\nabla|^{\f{s}{2}} \rmA\varphi\Vert_{L^2}^2,\\
\rmCK_{\Theta,\varphi}=&\, \sum_k\int_\eta\frac{\partial_t \Theta_{k}(t,\eta)}{\Theta_{k}(t,\eta)}  e^{\la(t)|k,\eta|^s}\langle k,\eta\rangle^{\s} \f{e^{\mu|\eta|^{\f13}}}{\Theta_k(t,\eta)}\cM_k(t,\eta)\rmB_k(t,\eta)\rmA_{k}(t,\eta)|\widehat{\varphi}_k(t,\eta)|^2d\eta,\\
\rmCK_{M,\varphi}=&\,\sum_k\int_\eta\frac{\partial_t \rmg(t,\eta)}{\rmg(t,\eta)}e^{\la(t)|k,\eta|^s}\langle k,\eta\rangle^{\s} \f{e^{4\pi\d_{\rmL}^{-1}|\eta|^{\f13}}}{\rmg(t,\eta)}\rmJ_{k}(t,\eta)\rmB_k(t,\eta)\rmA_{k}(t,\eta)|\widehat{\varphi}_k(t,\eta)|^2d\eta, \\
\rmCK_{B,\varphi}=& \d_{\rmB}^{-1}\sum_k\int_\eta\frac{b\big(t,{k},\eta\big)}{1+(t-\f{\eta}{k})^2}|\rmA_k(t,\eta)\widehat{\varphi}_k(t,\eta)|^2d\eta, 
\end{aligned}    
\end{equation}
where $\rmCK$ stands for `Cauchy-Kovalevskaya'. 
In what follows, we define 
\begin{equation}\label{A_tilde}
\tilde{\rmA}_{k}(t,\eta)=e^{\la(t)|k,\eta|^s}\langle k,\eta\rangle^{\s} \f{e^{\mu|\eta|^{\f13}}}{\Theta_k(t,\eta)}\cM_k(t,\eta)\rmB_k(t,\eta)
\end{equation}
and
\begin{equation}\label{A_tildetilde}
\tilde{\tilde{\rmA}}_{k}(t,\eta)=e^{\la(t)|k,\eta|^s}\langle k,\eta\rangle^{\s} \f{e^{4\pi\d_{\rmL}^{-1}|\eta|^{\f13}}}{\rmg(t,\eta)}\rmJ_k(t,\eta)\rmB_k(t,\eta)
\end{equation}
which satisfy $\tilde{\rmA}\leq \rmA$, $\tilde{\tilde{\rmA}}\leq \rmA$. In particular if $|k|\leq |\eta|$ then $\rmA \lesssim  \tilde{\rmA}$ and $\rmA \lesssim \tilde{\tilde{\rmA}}$. 

First, we have 
\begin{equation}\label{rho_Energy_1}
\begin{aligned}
\f12\f{d}{dt}\int|\rmA\rho|^2dzdv
&=-\rmCK_{\la,\rho}-\rmCK_{\Theta,\rho}-\rmCK_{M,\rho}-\rmCK_{B,\rho}-\int \rmA\rho \rmA\big({\bf{u}}\cdot \na \rho\big) dzdv\\
&\quad+\int\rmA\rho\rmA\big(\pa_z\phi)dzdv\\
&=-\rmCK_{\la,\rho}-\rmCK_{\Theta,\rho}-\rmCK_{M,\rho}-\rmCK_{B,\rho}-\rmNL_{\rho}+\Pi_{\rho}.
\end{aligned}
\end{equation}
Similarly, we have 
\begin{align}
\f12\f{d}{dt}\int|\rmA K|^2dzdv
\nonumber=&\,-\rmCK_{\la,K}-\rmCK_{\Theta,K}-\rmCK_{M,K}-\rmCK_{B,K}+\int \rmA K \rmA(\Delta_t K)  dzdv
\\
\nonumber&-\int \rmA K \rmA\big({\bf{u}}\cdot \na K \big) dzdv+\int \rmA K \rmA (\mathbf{H}) dzdv\\
\nonumber&+\gamma^2\gamma_1\int \rmA K \rmA( \partial _{zz}\phi)dzdv-2\int \rmA K\rmA( (\partial_{v}-t\partial_z) \partial_zf) dzdv\\
\label{dE_K_dt}
=&\, -\rmCK_{\la,K}-\rmCK_{\Theta,K}-\rmCK_{M,K}-\rmCK_{B,K}\\
&+\rmE-\rmNL_K^1-\rmNL_K^2+\Pi_K^1+\Pi_K^2. 
\end{align}

\subsection{Bootstrap argument and main propositions}
We prove the theorem by a bootstrap argument. In order to avoid discussing the fake singularity in the equations, let us first give the following local-wellposedness theory. 
\begin{lemma}
Under the same assumptions of Theorem \ref{Thm:main} or \ref{Thm:main-2}, we have
\beno
\sup_{t\in [0,1]}\rmE(t)+\rmE_d(t)\leq \ep^2.
\eeno
\end{lemma}
The lemma is easy to obtain by using the classical energy method to the transport equation and we omit the proof here. 

The goal is next to prove by a continuity argument that this energy $\rmE(t)+\rmE_d(t)$ together with $\rmE_{lo,g}(t), \rmE_{lo,h}(t), \rmE_{as,\psi_0}(t)$ and $\rmE_{as,f_0}(t)$ are uniformly bounded for all time if $\ep$ is small enough. We define the following controls referred to in the sequel as the bootstrap hypotheses for $t\geq 1$ and some constants $C_0,\ \rmK_d\geq 1$  independent of $\ep$ and determined in the proof, 
\begin{align}
\label{eq:B1}&\rmE(t)+\rmK_d^{-1}\rmE_d(t)\leq 10C_0\ep^2,\\
\label{eq:B2}&\int_1^t\Big(\rmCK_{\la,K}+\rmCK_{\Theta,K}+\rmCK_{M,K}+\rmCK_{B,K}
+\|\na_{L}\rmA K\|_2^{2}\\
\nonumber&\qquad\qquad\qquad\qquad+\rmCK_{\la,\rho}+\rmCK_{\Theta,\rho}+\rmCK_{M,\rho}+\rmCK_{B,\rho}\Big)(s)ds\leq 10C_0\ep^2,\\
\label{eq:B3}&\int_1^t\Big(
\rmCK_{\la,h}+\rmCK_{\Theta,h}+\rmCK_{M,h}+
\left\|\rmA\langle\pa_v\rangle^2h\right\|_2^2\Big)(s)ds\leq 10\rmK_dC_0\ep^2
\end{align}
where
\begin{align}
\rmCK_{\la,h}(t)&=-\dt{\la}(t)\langle t\rangle\left\||\pa_v|^{\f s2}\rmA \langle\pa_v\rangle^2h\right\|_{2}^2,\\
\rmCK_{\Theta,h}(t)&=\langle t\rangle\left\|\rmA \langle\pa_v\rangle^2\sqrt{\f{\pa_t\Theta}{\Theta}}h\right\|_{2}^2,\quad
\rmCK_{M,h}(t)=\langle t\rangle\left\|\rmA \langle\pa_v\rangle^2\sqrt{\f{\pa_t\rmg}{\rmg}}h\right\|_{2}^2. 
\end{align}

The main proposition of this paper is as follows:
\begin{proposition}[Bootstrap]\label{prop:bostp} There exists an $\ep_0\in (0,\f12)$ depending only on $\la,\, \la', \, s$ and $\s$ such that if $\ep<\ep_0$, and on $[1,\rmT^*]$ the bootstrap hypotheses \eqref{eq:B1}-\eqref{eq:B3} hold, then for $\forall t\in [1,T^*]$, 
\begin{align*}
&\rmE(t)+\rmK_{d}^{-1}\rmE_d(t)\leq 8C_0\ep^2\\
&\int_1^t\Big(\rmCK_{\la,K}+\rmCK_{\Theta,K}+\rmCK_{M,K}++\rmCK_{B,K}
+\|\na_{L}\rmA K\|_2^{2}\\
&\qquad\qquad\qquad\qquad+\rmCK_{\la,\rho}+\rmCK_{\Theta,\rho}+\rmCK_{M,\rho}+\rmCK_{B,\rho}\Big)(s)ds\leq 8C_0\ep^2,\\
&\int_1^t\Big(
\rmCK_{\la,h}+\rmCK_{\Theta,h}+\rmCK_{M,h}+
\left\|\rmA\langle\pa_v\rangle^2h\right\|_2^2\Big)(s)ds\leq 8\rmK_dC_0\ep^2,
\end{align*}
from which it follows that $\rmT^*=+\infty$.
\end{proposition}
The main purpose of this paper is to prove Proposition \ref{prop:bostp}, which follows from the following propositions by taking $\ep,\, \d_{\rmL},\, \d_{\rmB}$ sufficiently small and $M_0$ sufficiently large. (See Section \ref{section: pi_k^2} for the determination of $M_0$.)

First we control the lower energy and the assistant energy and obtain the following proposition, where its proof is given in Section \ref{Zero_mod_Section}.  

\begin{proposition}\label{prop: g,f_0,h}
Under the bootstrap hypotheses, for any $0<\sqrt{\ep}\leq \ep_1^2$ with $0<\ep_1\leq 3s+1-4\tilde{q}$ and $0<\ep_2<\f12$, it holds that
\beno
\langle t\rangle^2 \left\|\rmA\pa_{v}^3g\right\|_2^2+\int_1^t\langle s\rangle^2 \left\|\rmA\pa_{v}^3g\right\|_2^2(s)ds \lesssim \ep^2.
\eeno
and
\beno
\rmE_{as,\psi_0}(t)\lesssim \ep^2,\quad \rmE_{as,g}(t)\lesssim \ep^2,
\eeno
As a corollary, it holds that 
\beno
\|h\|_{L^2}\lesssim \f{\ep}{\langle t\rangle},
\quad \|f_0\|_{L^2}\lesssim {\ep}{\langle t\rangle^{-\f54}},
\quad \|\pa_vf_0\|_{L^2}\lesssim {\ep}{\langle t\rangle^{-\f74}}.
\eeno
We also get that
\beno
\rmE_{lo,f_0}(t)+\rmE_{lo,g}(t)+\rmE_{lo,h}(t)\lesssim \ep^2.
\eeno
As a corollary, it holds that
\begin{align*}
\|g\|_{\mathcal{G}_1^{\la,\b;s}}\eqdef \int_{\R}\langle \xi\rangle^{\b}e^{\la(t)|\xi|^{s}}|\hat{g}(t,\xi)|d\xi
&\lesssim \f{\ep}{\ep_1^{3/2}}\langle t\rangle^{-2+\f{\ep_1}{2}}\lesssim \sqrt{\ep}\langle t\rangle^{-2+\f{\ep_1}{2}}.
\end{align*}
\end{proposition}

\begin{proposition}\label{prop: NL}
Under the bootstrap hypotheses, it holds that
\begin{align*}
|\rmN\rmL_{\rho}|+|\rmN\rmL_{K}^1|
&\lesssim \f{\ep^3}{\langle t\rangle^2}+\ep^2\langle t\rangle^2 \|\rmA\pa_v^3g\|_2^2+
\sqrt{\ep}\rmCK_{\la,\rho}+\ep\rmCK_{\Theta,\rho}+\ep\rmCK_{M,\rho}\\
&\quad+\sqrt{\ep}\rmCK_{\la,K}+\ep\rmCK_{\Theta,K}+\ep\rmCK_{M,K}\\
&\quad+\ep\left\|\left\langle\f{\pa_v}{t\pa_z}\right\rangle^{-1}\left(\f{|\na|^{\f{s}{2}}}{\langle t\rangle^{\f{3s}{2}}}\rmA+\sqrt{\f{\pa_t\rmg}{\rmg}}\tilde{\tilde{\rmA}}+\sqrt{\f{\pa_t\Theta}{\Theta}}\tilde{\rmA}\right)\pa_z^{-1}\Delta_{L}^2 P_{\neq}\phi\right\|_2^2.
\end{align*}
\end{proposition}

The proof of Proposition \ref{prop: NL} is given in Section \ref{Section_Main_Nonlinear_terms} and it is one of the main parts of the proof.  

\begin{proposition}\label{prop: NL_2}
Under the bootstrap hypotheses, it holds that
\begin{align*}
|\rmN\rmL_{K}^2|
&\lesssim \f{\ep^3}{\langle t\rangle^2}+\ep\|\na_L\rmA K\|_2^2+\f{\ep}{\langle t\rangle^2}\left\|
\rmA\pa_z^{-1}\Delta_{L}P_{\neq}f\right\|_2^2\\
&\quad+\ep\left\|\left\langle\f{\pa_v}{t\pa_z}\right\rangle^{-1}\left(\f{|\na|^{\f{s}{2}}}{\langle t\rangle^{\f{3s}{2}}}\rmA+\sqrt{\f{\pa_t\rmg}{\rmg}}\tilde{\tilde{\rmA}}+\sqrt{\f{\pa_t\Theta}{\Theta}}\tilde{\rmA}\right)\pa_z^{-1}\Delta_{L}^2 P_{\neq}\phi\right\|_2^2.
\end{align*}
\end{proposition}
The proof of Proposition \ref{prop: NL_2} is the subject of Section  \ref{Section_NL_K_2}.

In Propositions \ref{prop:Linear_Terms}, \ref{Proposition_E_Estimate} and  \ref{prop:Linear_Terms_K}, we state the estimates of the linear terms $\Pi_{\rho},\, \rmE,\, \Pi_K^1$ and $\Pi_K^2$ appearing in \eqref{rho_Energy_1} and \eqref{dE_K_dt}. The proofs of such estimates  are given in Section \ref{Section_Linear_Terms}.

\begin{proposition}\label{prop:Linear_Terms}
Under the bootstrap hypotheses, it holds that
\begin{equation}
|\Pi_{\rho}|\leq\f{C\ep}{\langle t\rangle^2}\left\|\left\langle\f{\pa_v}{t\pa_z}\right\rangle^{-1}\pa_z^{-1}\Delta_{L}^2\rmA P_{\neq}\phi\right\|_2
+C_1\d_{\rmL}\rmCK_{M,\rho}
+\d_{\rmL}\left\|\left\langle\f{\pa_v}{t\pa_z}\right\rangle^{-1}\sqrt{\f{\pa_t\rmg}{\rmg}}\pa_z^{-1}\Delta_{L}^2\tilde{\tilde{\rmA}} P_{\neq}\phi\right\|_2^2, 
\end{equation}
where $C_1$ is a constant independent of $\delta_\rmL$. 
\end{proposition}

\begin{proposition}
\label{Proposition_E_Estimate}
Under the bootstrap hypotheses, it holds that
\ben
\rmE\leq -\f78\|\na_L\rmA K\|_2^2+\f{C\ep^3}{\langle t\rangle^2}. 
\een
\end{proposition}
\begin{proposition}\label{prop:Linear_Terms_K}
Under the bootstrap hypotheses, it holds that
\begin{equation}
\begin{aligned}
|\Pi_K^1|+|\Pi_K^2|\lesssim&\, \f{C}{\langle t\rangle^4}\left\|\left\langle\f{\pa_v}{t\pa_z}\right\rangle^{-1}\pa_z^{-1}\Delta_{L}^2\rmA P_{\neq}\phi\right\|_2^2+\f{C}{\langle t\rangle^4}\left\|\left\langle \f{\pa_v}{t\pa_z}\right\rangle^{-1}\rmA\pa_z^{-1}\Delta_{L}P_{\neq }f\right\|_2^2\\
&+\f{1}{8}\|\na_L\rmA K\|_2^2+C_1\d_{\rmL}\left\|\left\langle\f{\pa_v}{t\pa_z}\right\rangle^{-1}\sqrt{\f{\pa_t\rmg}{\rmg}}\pa_z^{-1}\Delta_{L}^2\tilde{\tilde{\rmA}} P_{\neq}\phi\right\|_2^2\\
&+C_1\left\|1_{t\geq M_0}\f{|\na|^{\f s 2}}{\langle t\rangle^{\f{3s}{2}}}\left\langle \f{\pa_v}{t\pa_z}\right\rangle^{-1}\rmA\pa_z^{-1}\Delta_{L}P_{\neq }f\right\|_2^2\\
&+C_1\d_{\rmB}\left\|\sqrt{\f{b(t,\na)\pa_{zz}}{\Delta_{L}}}\left\langle \f{\pa_v}{t\pa_z}\right\rangle^{-1}\rmA\pa_z^{-1}\Delta_{L}P_{\neq }f\right\|_2^2,
\end{aligned}
\end{equation}
where $C_1$ is a constant independent of $\delta_\rmL$, $\d_{\rmB}$ and $M_0$.
\end{proposition}

\begin{proposition}\label{prop: elliptic-3}
Under the bootstrap hypotheses, it holds that
\begin{equation}\label{Elliptic_f_1}
\left\|
\rmA\pa_z^{-1}\Delta_{L}P_{\neq}f\right\|_2\lesssim \ep.
\end{equation}     
\end{proposition}

The proof of Proposition \ref{prop: elliptic-3} is the subject of Section \ref{Section_Proof_elliptic-3}. 

\begin{proposition}\label{prop: elliptic-4}
Under the bootstrap hypotheses, for some $C_1\geq 1$  independent of $M_0$, it holds that
\begin{equation}\label{Elliptic_f_lambda}
\begin{aligned}
\left\|\mathbf{1}_{t\geq M_0}\f{|\na|^{\f s 2}}{\langle t\rangle^{\f{3s}{2}}}\left\langle \f{\pa_v}{t\pa_z}\right\rangle^{-1}\rmA\pa_z^{-1}\Delta_{L}P_{\neq }f\right\|_2^2\leq C_1 M_0^{2\tilde{q}-3s}(\rmCK_{\la,K}+\rmCK_{\la,\rho})+C\ep^2 \rmCK_{\la,h}
\end{aligned}
\end{equation}
and for some $C_2\geq1$ independent of $\d_{\rmL}$, it holds that
\begin{equation}\label{Elliptic_g_lambda}
\begin{aligned}
&\left\|\sqrt{\f{\pa_t\rmg}{\rmg}}\left\langle \f{\pa_v}{t\pa_z}\right\rangle^{-1}\tilde{\tilde{\rmA}}\pa_z^{-1}\Delta_{L}P_{\neq }f\right\|_2^2\\
&\leq C_2(\rmCK_{M,K}+\rmCK_{\la,K}+\rmCK_{M,\rho}+\rmCK_{\la,\rho})+C\ep^2 \|\rmA \langle\pa_{v}\rangle^2h\|_2^2,    
\end{aligned}
\end{equation}
and for some $C_3\geq 1$ independent of $\d_{\rmB}$, it holds that
\beno
\left\|\sqrt{\f{b(t,\na)\pa_{zz}}{\Delta_{L}}}\left\langle \f{\pa_v}{t\pa_z}\right\rangle^{-1}\rmA\pa_z^{-1}\Delta_{L}P_{\neq }f\right\|_2^2
\leq C_3(\rmCK_{B,K}+\rmCK_{B,\rho})+C\ep^2 \|\rmA \langle\pa_{v}\rangle^2h\|_2^2. 
\eeno
\end{proposition}

The proof of Proposition \ref{prop: elliptic-4} is given in Section \ref{Section_Proof_elliptic-4}.

\begin{proposition}\label{prop: elliptic-1}
Under the bootstrap hypotheses, it holds that
\beno
\left\|\left\langle \f{\pa_v}{t\pa_z}\right\rangle^{-1} \rmA\pa_z^{-1}\Delta^{2}_{L}P_{\neq}\phi\right\|_2\lesssim \ep. 
\eeno
\end{proposition}
The proof of Proposition \ref{prop: elliptic-1} is detailed in Section \ref{Section_Proof_Proposition_elliptic-1}

\begin{proposition}\label{prop: elliptic-2}
Under the bootstrap hypotheses, it holds that
\begin{align}\label{Main_elliptic_Estimate}
&\left\|\left\langle\f{\pa_v}{t\pa_z}\right\rangle^{-1}\left(\f{|\na|^{\f{s}{2}}}{\langle t\rangle^{\f{3s}{2}}}\rmA+\sqrt{\f{\pa_t\rmg}{\rmg}}\tilde{\tilde{\rmA}}+\sqrt{\f{\pa_t\Theta}{\Theta}}\tilde{\rmA}\right)\pa_z^{-1}\Delta_{L}^2 P_{\neq}\phi\right\|_2^2\\
&\leq C_1( \rmCK_{\la,K}+\rmCK_{\Theta,K}+\rmCK_{M,K}
+\rmCK_{\la,\rho}+\rmCK_{\Theta,\rho}+\rmCK_{M,\rho})\\
&\quad+C_2(\ep^2 \rmCK_{\la,h}+\ep^2\|\rmA\langle\pa_v\rangle^2h\|_2^2), 
\end{align}
where the constant $C_1$ is independent of $\d_{\rmL}$ and the constant $C_2$ may depend on $\delta_L$. 
\end{proposition}

The proof of Proposition \ref{prop: elliptic-2} is the subject of Section \ref{Section_Proof_elliptic-2}.

\subsection{Conclusion of the proof}\label{Section_Main_proofs}

Now, to prove the estimates of the main Theorems \ref{Thm:main} and \ref{Thm:main-2}, we need to recover the estimates on the original systems \eqref{eq:NSB3} and \eqref{eq:NSB4}. So, we need to undo the coordinate transform and transfer the uniform energy bound on $\rmE(t)+\rmE_d(t)$ in the $(z,v)$ variables   into the original $(x,y)$ variables. This requires the use of an inverse function theorem in Gevrey spaces. We refer to \cite[Remark 7 and Section 2.4]{Bedrossian_2015} for more details. Here we only focus on the scattering results. 

Our goal now is to prove \eqref{eq: scatting}.  We have by using the second equation in \eqref{Main_New_system_3}, and for some $\la'<\la_1<\lambda_0/2+\lambda^\prime/2<\la$
\begin{equation}\label{rho_Scat}
\begin{aligned}
\Vert\rho(t)- \rho_\infty\Vert_{\cG^{\la_1,0;s}} \lesssim&\, \left\Vert\int_t^\infty g\pa_v\rho ds \right\Vert_{\cG^{\la_1,0;s}}+ \left\Vert \int_t^\infty v'\nabla _{z,\upsilon }^{\perp }P_{\neq }\phi\cdot 
\nabla _{z,\upsilon } \rho(s)ds\right\Vert_{\cG^{\la_1,0;s}}\\
&+\left\Vert \int_t^\infty\pa_z\phi(s)ds \right\Vert _{\cG^{\la_1,0;s}}. 
\end{aligned}
\end{equation}
 We estimate first   
$\Vert g\pa_v\rho\Vert_{\cG^{\la_1,0;s}}$. We have by using the physical definition of the Gevrey spaces  (see \eqref{Gevrey_Ph}) together with  Lemma \ref{Product_Gevrey_Lemma}
\begin{equation}
\begin{aligned}
\Vert g\pa_v\rho\Vert_{\cG^{\la_1;s}}=\,\|g\pa_v\rho\|_{\ell^2\rmL^2;\la_1}\lesssim & \, \|g\pa_v\rho\|_{\ell^1\rmL^2;\la_1}\\
\lesssim &\, 
 \Vert \pa_v\rho\Vert_{\ell^1\rmL^2;\la_1}\Vert g\Vert_{\ell^1\rmL^\infty;\la_1}. 
\end{aligned}
\end{equation} 
We have that
\begin{equation}\label{g_Generalized_Gevrey}
\begin{aligned}
\|g\|_{\ell^1\rmL^{\infty};\la_1}\lesssim&\, \|g\|_{\rmL^{\infty}}+\|\pa_v^3g\|_{\ell^1\rmL^{\infty};\la_1}  \\
\lesssim&\,\|g\|_{\rmL^{\infty}}+\|\pa_v^3g\|_{\ell^2\rmL^{\infty};\la_2} 
\lesssim\,\|g\|_{\rmL^{\infty}}+\|\pa_v^3g\|_{\cG^{\la_2,3;s}}
\end{aligned}
\end{equation}
where we have used \eqref{Estimate_Generalized_Gevrey} for some  $\lambda_0/2+\lambda^\prime/2>\la_3>\la_2>\la_1$.  

We also have  by \eqref{Estimate_Generalized_Gevrey}
\begin{equation}\label{rho_Generalized_Gevrey}
\Vert \pa_v\rho\Vert_{\ell^1\rmL^2;\la_1}\lesssim_{\la_2-\la_1} \Vert \pa_v\rho\Vert_{\ell^2\rmL^2;\la_2}. 
\end{equation}

Hence, \eqref{g_Generalized_Gevrey} and \eqref{rho_Generalized_Gevrey} together with the bootstrap assumption, yields  
\begin{equation}\label{g_Term_Sca}
\|g\pa_v\rho\|_{\cG^{\la_1,0;s}}\lesssim \|g\|_{L^{\infty}}\|\rho\|_{\cG^{\la_2,0;s}}+\|\pa_v^3g\|_{\cG^{\la_3,0;s}} \|\rho\|_{\cG^{\la_2,0;s}}\lesssim \ep^2\f{\ln (e+t)}{\langle t\rangle ^2}. 
\end{equation}

The last two terms can be easily estimated using the lossy elliptic estimates (see Lemma \ref{Lemma_Lossy}). Indeed, we have 
\begin{equation}\label{phi_sca_Term}
 \left\Vert \int_t^\infty v'\nabla _{z,\upsilon }^{\perp }P_{\neq }\phi\cdot 
\nabla _{z,\upsilon } \rho(s)ds\right\Vert_{\cG^{\la_1,0;s}}\lesssim \ep^2\int_t^\infty \f{1}{\langle s\rangle^4} ds\lesssim \f{\ep^2}{\langle t\rangle^3}. 
\end{equation}
and 
\begin{equation}\label{phi_sca_Term_2}
\left\Vert \int_t^\infty\pa_z\phi(s)ds \right\Vert _{\cG^{\la_1,0;s}}\lesssim \f{\ep}{\langle t\rangle^3}. 
\end{equation}  
Hence, keeping in mind \eqref{rho_Scat} and collecting \eqref{g_Term_Sca}, \eqref{phi_sca_Term} and \eqref{phi_sca_Term_2}, we obtain 
\begin{equation}
\Vert\rho(t)- \rho_\infty\Vert_{\cG^{\la_1,0;s}}\lesssim \f{\ep^2}{\langle t\rangle}\ln (e+t)+\f{\ep}{\langle t\rangle^{3}}. 
\end{equation}     
Again by the same argument in \cite[Section 2.3]{Bedrossian_2015}, we have the estimate in $(x,y)$ coordinate and get \eqref{eq: scatting}.

\section{Growth mechanism and construction of weights}\label{Section_3}
In this section, we will construct the key weights  $\Theta$ and $\rmg$ which come from a toy model that capture the growth of the solutions (the worst case scenario) in the nonlinear interaction. 

From the argument of the linearized equation, due to the dissipation term, the good unknown $K$ decays. We may focus on the nonlinear interaction in the equation of $\rho$. 
Since we must pay regularity to deduce decay on the velocity $u$, it is natural to consider the frequency interactions in the product $u\cdot\na \rho$ with the frequencies of $u$ much larger than $\rho$, which is called the reaction. This leads us to study a simpler model 
\begin{equation}\label{rho_toy}
\partial _{t} \rho-\partial_{v }\phi_{\neq}\partial _{z} \rho_{lo}=0. 
\end{equation}
This contribution was chosen as $\Delta_t$ loses ellipticity in $v$, not in $z$. Suppose that instead of $\Delta_t\phi=f$ and $\Delta_tf-\g^2\pa_z\rho=K$, we had
$\Delta_L\phi=f$ and $\Delta_Lf-\g^2\pa_z\rho=0$, then on the Fourier side: 
\begin{equation}
\pa_t\hat{\rho}(t,k,\eta)=\frac{i\g^2}{2\pi }\sum_{l\neq 0}
\int_{\xi}\frac{\xi l (k-l) \hat{\rho}\left( t,l ,\xi \right) }{(l ^{2}+\left\vert \xi
-l t\right\vert ^{2})^2} \hat{\rho}_{lo}\left( t,k-l
,\eta-\xi \right) d\xi .  \label{eq:Toy model1}
\end{equation}
Since ${\rho}_{lo}$ weakens interactions between well-separated frequencies, let us consider a discrete model with $\eta$ as a fixed parameter:
\begin{equation}
\pa_t\hat{\rho}(t,k,\eta)=\frac{i\g^2}{2\pi }\sum_{l\neq 0}
\frac{\eta l (k-l) \hat{\rho}\left( t,l ,\eta \right) }{(l ^{2}+\left\vert \eta-l t\right\vert ^{2})^2} \hat{\rho}_{lo}\left( t,k-l
,0 \right) .  \label{eq: Toy model 2}
\end{equation}

This is an infinite system of ODEs. The coefficient $\frac{\eta l}{(l ^{2}+\left\vert \eta-l t\right\vert ^{2})^2}$ of $\hat{\rho}\left( t,l ,\eta \right)$ in the summation on the right hand side of \eqref{eq: Toy model 2} becomes large when $|t-\f{\eta}{l}|\leq \f{\eta}{l^3}$. Meanwhile, other coefficients $\frac{\eta \ell}{(k ^{2}+\left\vert \eta-\ell t\right\vert ^{2})^2}, \, \ell\neq l$ is small $|t-\f{\eta}{l}|\leq \f{\eta}{l^3}$. Thus for any $l$ and $|t-\f{\eta}{l}|\leq \f{\eta}{l^3}$, we divide $\{\hat{\rho}(t,\ell,\eta)\}_{\ell\neq 0}$ into two parts $\rho_{\rmR}=\hat{\rho}(t,l,\eta)$ and $\rho_{\rmNR}=\hat{\rho}(t,\ell,\eta), \, \ell\neq l$ where $  \rmR$ and $\rmNR$ stand for resonance and non resonance. Due to the existence of $(k-l)$ in the summation of \eqref{eq: Toy model 2}, the resonance part will never act on the resonance part. By choosing the largest coefficient of all non resonance, we have for $k=1,2,...,\rmE(|\eta|^{\f13})$ and $|t-\f{\eta}{k}|\lesssim \f{\eta}{k^3}$, the growth of the solutions to \eqref{eq: Toy model 2} is captured by the following ODE system: 
\beq\label{eq:toy-model-ode_0}
\left\{
\begin{aligned}
&\pa_t\rho_{\rmR}=\kappa\f{k^5}{\eta^3}\rho_{\rmN\rmR},\\
&\pa_t\rho_{\rmN\rmR}=\kappa\f{\eta k}{(k^2+(\eta-kt)^2)^2}\rho_{\rmR},
\end{aligned}
\right.
\eeq
This leads to the weight $\Theta$ which is constructed in the next section. 

We will also take into account  the growth for $|t-\f{\eta}{k}|\gtrsim \f{\eta}{k^3}$ with $k=1,2,...,\rmE(|\eta|^{\f23})$, which leads to the weight $\rmg$.

\subsection{Construction of $\Theta$}
The construction is done backward in time, starting with $k=1$. For $t\in \rmI_{k,\eta}=[t_{k,\eta}^{-},t_{k,\eta}^+]$  with $|k|\leq \rmE(|\eta|^{\f13})$. Let $\Theta_{\rmN\rmR}$ be a non-decreasing function of time with $\Theta_{\rmN\rmR}(t,\eta)=1$ for $t\geq 2|\eta|$. For definiteness, we remark here that for $|\eta|<1$, $\Theta_{\rmN\rmR}(t,\eta)=1$, which will be consequence of the ensuing the definition. Hence we may safely assume $|\eta|>1$ for the duration of the section. For $k\geq 1$, we assume that $\Theta_{\rmN\rmR}(\f{2\eta}{2k-1},\eta)$ was computed.

 To compute $\Theta_{\rmN\rmR}$ on the interval $\bar{\rmI}_{k,\eta}$, we define for $k=1,2,3,...,\rmE(|\eta|^{\f13})$ and $t\in [t_{k,\eta}^{+},\f{2\eta}{2k-1}]$, 
\beno
\Theta_{\rmR}(t,\eta)=\Theta_{\rmN\rmR}(t,\eta)=\Theta_{\rmN\rmR}\left(\f{2\eta}{2k-1},\eta\right).
\eeno

For $t\in [t_{k,\eta}^{-},t_{k,\eta}^+]$, we define
\beq\label{eq:w_NR+}\begin{split}
\Theta_{\rmN\rmR}(t,\eta)&=\left(\f{k^3}{\eta}\left[1+b_{k,\eta}\left|t-\f{\eta}{k}\right|\right]\right)^{C\kappa}\Theta_{\rmN\rmR}(t_{k,\eta}^+,\eta),\\
\forall t\in \rmI_{k,\eta}^{\rmR}&=\left[\f{\eta}{k},t_{k,\eta}^+\right],
\end{split}
\eeq
\beq\label{eq:w_NR-}\begin{split}
\Theta_{\rmN\rmR}(t,\eta)&=\left(1+a_{k,\eta}\left|t-\f{\eta}{k}\right|\right)^{-1-C\kappa}\Theta_{\rmN\rmR}\left(\f{\eta}{k},\eta\right),\\
\forall t\in \rmI_{k,\eta}^{\rmL}&=\left[t_{k,\eta}^-,\f{\eta}{k}\right].
\end{split}
\eeq

The constant $b_{k,\eta}$ is chosen to ensure that $\f{k^3}{\eta}\left[1+b_{k,\eta}\f{\eta}{2k^3}\right]=1$, hence for $k\geq 1$, we have
\ben
b_{k,\eta}=2(1-\f{k^3}{\eta}). 
\een
The constant $a_{k,\eta}$ is chosen to ensure that
 $\Theta_{\rmN\rmR}(t_{k,\eta}^+,\eta)=\Theta_{\rmN\rmR}(t_{k,\eta}^-,\eta)(\f{\eta}{k^3})^{1+2C\kappa}$, hence for $k\geq 1$, we have
\ben
a_{k,\eta}=2(1-\f{k^3}{\eta}).
\een
On each interval $\rmI_{k,\eta}$, we define $\Theta_{\rmR}(t,\eta)$ by
\ben
\Theta_{\rmR}(t,\eta)=\f{k^3}{\eta}\left(1+b_{k,\eta}\left|t-\f{\eta}{k}\right|\right)\Theta_{\rmN\rmR}(t,\eta),\quad \forall t\in \rmI_{k,\eta}^{\rmR}&=\left[\f{\eta}{k},t_{k,\eta}^+\right],
\een
\ben
\Theta_{\rmR}(t,\eta)=\f{k^3}{\eta}\left(1+a_{k,\eta}\left|t-\f{\eta}{k}\right|\right)\Theta_{\rmN\rmR}(t,\eta),\quad
\forall t\in \rmI_{k,\eta}^{\rmL}&=\left[t_{k,\eta}^-,\f{\eta}{k}\right]. 
\een

For $t\in [\f{2\eta}{2k+1},t_{k,\eta}^{-}]$, we define 
\beno
\Theta_{\rmR}(t,\eta)=\Theta_{\rmN\rmR}(t,\eta)=\Theta_{\rmNR}(t_{k,\eta}^-,\eta).
\eeno

Due to the choice of $a_{k,\eta}$ and $b_{k,\eta}$, we get that $\Theta_{\rmR}(t_{k,\eta}^{\pm},\eta)=\Theta_{\rmN\rmR}(t_{k,\eta}^{\pm},\eta)$, $\Theta_{\rmR}(\f{\eta}{k},\eta)=\f{k^3}{\eta}\Theta_{\rmN\rmR}(\f{\eta}{k},\eta)$ and for $t\in \rmI_{k,\eta}$, 
\beq\begin{split}
&\pa_{t}\Theta_{\rmR}\approx \f{k^3}{\eta}\Theta_{\rmN\rmR},\\
&\pa_t\Theta_{\rmN\rmR}\approx \f{\kappa\eta}{k^3(1+|t-\f{\eta}{k}|^2)}\Theta_{\rmR}.
\end{split}
\eeq
Then we define 
\beq
\Theta_{k}(t,\eta)=
\left\{
\begin{aligned}\label{Total_Multipl}
&\Theta_k(t_{\rmE(|\eta|^{\f13}),\eta},\eta),\quad t<t_{\rmE(|\eta|^{\f13}),\eta}\\
&\Theta_{\rmNR}(t,\eta), \quad t\in \left[t_{\rmE(|\eta|^{\f13}),\eta},2\eta\right]\setminus \rmI_{k,\eta}\\
&\Theta_{\rmR}(t,\eta),\quad t\in \rmI_{k,\eta}\\
&1, \quad t\geq 2\eta  
\end{aligned}
\right.
\eeq
 
\subsection{Construction of $\rmg(t,\eta)$}   

Again the construction is done backward in time, starting with $k=1$. For $t\in \bar{\rmI}_{k,\eta}=[\f{2\eta}{2k+1},\f{2\eta}{2k-1}]$ for $1\leq |k|\leq \rmE(|\eta|^{\f23})$. Let $\rmg$ be a non-decreasing function of time with $\rmg(t,\eta)=1$ for $t\geq 2|\eta|$. For definiteness, we remark here for $|\eta|<1$, $\rmg(t,\eta)=1$, which will be consequence of the ensuing the definition. Hence we may safely assume $|\eta|>1$ for the duration of the section. For $k\geq 1$, we assume that $\rmg(\f{2\eta}{2k-1},\eta)$ was computed. To compute $\rmg$ on the interval $\bar{\rmI}_{k,\eta}$, we define for $k=1,2,3,...,\rmE(|\eta|^{\f13})$ and $t\in [\f{2\eta}{2k+1},\f{2\eta}{2k-1}]$, 
\beno\label{Weight_g}
\pa_t\rmg=\f{\d_{\rmL}^{-1}}{1+|t-\f{\eta}{k}|^2}\rmg, \quad \rmg\big|_{t=\f{2\eta}{2k-1}}=\rmg\Big(\f{2\eta}{2k-1},\eta\Big),
\eeno
and for $k=\rmE(|\eta|^{\f13})+1,...,\rmE(|\eta|^{\f23})$ and 
$t\in [\f{2\eta}{2k+1},\f{2\eta}{2k-1}]$
\begin{equation}\label{Weight_g_2}
\pa_t\rmg=\f{\d_{\rmL}^{-1}\f{\eta}{k^3}}{1+|t-\f{\eta}{k}|^2}\rmg, \quad \rmg\big|_{t=\f{2\eta}{2k-1}}=\rmg\Big(\f{2\eta}{2k-1},\eta\Big).
\end{equation}

\subsection{The total growth predicted by the toy model}  
\begin{lemma}\label{lem:3.1}
For $\eta>1$, there exists $\mu=6(1+2C\kappa)$  such that 
\ben\label{Total_Growth_w}
\frac{\Theta_k(2\eta,\eta)}{\Theta_k(0,\eta)}=\f{1}{\Theta_k(0,\eta)}=\f{1}{\Theta_{k}\left(t_{\rmE(|\eta|^{\f13}),\eta},\eta\right)}\approx \f{1}{|\eta|^{\f{\mu}{12}}}e^{\f{\mu}{2}|\eta|^{\f13}}. 
\een
\end{lemma}
\begin{proof}
The proof of Lemma \ref{lem:3.1} can be done as in \cite{Bedrossian_2015}. Indeed, the total growth over $\bigcup_{k=N}^2 \rmI_{k,\eta} $ is given by the product ($N=\rmE(|\eta|^{\f13})$) 
\begin{equation} \label{Growth_formula} 
\Big(\frac{\eta}{N^3}\Big)^c\Big(\frac{\eta}{(N-1)^3}\Big)^c\dots \Big(\frac{\eta}{1^3}\Big)^c =\Big[\frac{\eta^N}{(N!)^3}\Big]^c
\end{equation}
with $c=2C\kappa+1$. Using the Stirling formula: $N!\approx \sqrt{2\pi N} (\frac{N}{e})^N$, we get 
\begin{equation}
\frac{\eta^N}{(N!)^3} \sim \frac{1}{(2\pi)^{3/2}\sqrt{\eta}}e^{3\eta^{1/3}}\Big[\frac{\eta^{1/2}}{N^{3/2}}\Big(\frac{\eta}{N^3}\Big)^{N}e^{3N-3\eta^{1/3}}\Big]. 
\end{equation}
Since $|N-\eta^{1/3}|\leq 1$, then it holds that 
\begin{equation}
\frac{\eta^N}{(N!)^3}\sim  \frac{1}{\sqrt{\eta}}e^{3\eta^{1/3}}.
\end{equation}
Hence, this together with \eqref{Growth_formula}, yields \eqref{Total_Growth_w}. 
\end{proof}

\begin{lemma}\label{lem:3.1-1}
For $\eta>1$, it holds that
\begin{align*}
&\f{1}{\rmg(t_{\rmE(|\eta|^{\f23}),\eta},\eta)}=\f{\rmg(2\eta, \eta)}{\rmg(t_{\rmE(|\eta|^{\f23}),\eta},\eta)}
=\f{\rmg(t_{\rmE(|\eta|^{\f13}),\eta},\eta)}{\rmg(t_{\rmE(|\eta|^{\f23}),\eta},\eta)}\f{\rmg(2\eta, \eta)}{\rmg(t_{\rmE(|\eta|^{\f13}),\eta},\eta)}\\
&=\,\exp\left(\d_{\rmL}^{-1}\sum_{\rmE(|\eta|^{\f13})+1}^{\rmE(|\eta|^{\f23})}\f{\eta}{k^3}\Big(\arctan\big(\f{|\eta|}{|k|(2|k|+1)}\big)+\arctan\big(\f{|\eta|}{|k|(2|k|-1)}\big)\Big)\right)\\
&\quad \times \exp\left(\d_{\rmL}^{-1}\sum_{k=1}^{\rmE(|\eta|^{\f13})}\Big(\arctan\big(\f{|\eta|}{|k|(2|k|+1)}\big)+\arctan\big(\f{|\eta|}{|k|(2|k|-1)}\big)\Big)\right).
\end{align*}
This gives us for any $t\geq 0$, 
\begin{equation}\label{Estimate_g}
1\leq \f{1}{\rmg(t,\eta)}\leq e^{3\pi\d_{\rmL}^{-1}|\eta|^{\f13}}. 
\end{equation}
\end{lemma}
\begin{proof} 
We have for $t\in \bar{\rmI}_{k,\eta}$, if $|k|\leq \rmE(|\eta|^{\f13})$, 
\beno
\rmg(t,\eta)=\exp\left(\d_{\rmL}^{-1}\Big(\arctan\big(t-\f{|\eta|}{|k|}\big)-\arctan\big(\f{|\eta|}{|k|(2|k|-1)}\big)\Big)\right)\rmg\Big(\f{2\eta}{2k-1},\eta\Big),
\eeno
and if $\rmE(|\eta|^{\f23})\geq |k|\geq \rmE(|\eta|^{\f13})+1$, 
\beno
\rmg(t,\eta)=\exp\left(\d_{\rmL}^{-1}\f{\eta}{k^3}\Big(\arctan\big(t-\f{|\eta|}{|k|}\big)-\arctan\big(\f{|\eta|}{|k|(2|k|-1)}\big)\Big)\right)\rmg\Big(\f{2\eta}{2k-1},\eta\Big),
\eeno
This yields for $|k|\leq \rmE(|\eta|^{\f13})$,
\begin{equation}\label{eq:g(k,eta)/g(k,eta)1}
\begin{aligned}
\frac{\rmg(t_{k-1,\eta},\eta)}{\rmg(t_{k,\eta},\eta)}=&\,\exp\left(\d_{\rmL}^{-1}\Big(\arctan\big(t_{k-1,\eta}-\f{|\eta|}{|k|}\big)-\arctan\big(t_{k,\eta}-\f{|\eta|}{|k|}\big)\Big)\right)\\
=&\,\exp\left(\d_{\rmL}^{-1}\Big(\arctan\big(\f{|\eta|}{|k|(2|k|+1)}\big)+\arctan\big(\f{|\eta|}{|k|(2|k|-1)}\big)\Big)\right)
\end{aligned}
\end{equation}
and $\rmE(|\eta|^{\f23})\geq |k|\geq \rmE(|\eta|^{\f13})+1$
\begin{equation}\label{eq:g(k,eta)/g(k,eta)2}
\frac{\rmg(t_{k-1,\eta},\eta)}{\rmg(t_{k,\eta},\eta)}
=\exp\left(\d_{\rmL}^{-1}\f{\eta}{k^3}\Big(\arctan\big(\f{|\eta|}{|k|(2|k|+1)}\big)+\arctan\big(\f{|\eta|}{|k|(2|k|-1)}\big)\Big)\right).
\end{equation}
Hence, we obtain (recall that $t_{0,\eta}=2\eta$)
\begin{align*}
\f{\rmg(2\eta, \eta)}{\rmg(t_{\rmE(|\eta|^{\f23}),\eta},\eta)}=&\,\f{\rmg(t_0,\eta, \eta)}{\rmg(t_{t_1,\eta},\eta)}\frac{\rmg(t_{t_1,\eta},\eta)}{\rmg(t_{t_2,\eta},\eta)}\dots \frac{\rmg(t_{\rmE(|\eta|^{\f23})-1,\eta},\eta)}{\rmg(t_{t_\rmE(|\eta|^{\f23}),\eta},\eta)}\\
=&\,\prod_{k=1}^{\rmE(|\eta|^{\f13})} \exp\left(\d_{\rmL}^{-1}\Big(\arctan\big(t_{k-1,\eta}-\f{|\eta|}{|k|}\big)-\arctan\big(t_{k,\eta}-\f{|\eta|}{|k|}\big)\Big)\right)\\
\times& \prod_{\rmE(|\eta|^{\f13})+1}^{\rmE(|\eta|^{\f23})} \exp\left(\d_{\rmL}^{-1}\f{\eta}{k^3}\Big(\arctan\big(t_{k-1,\eta}-\f{|\eta|}{|k|}\big)-\arctan\big(t_{k,\eta}-\f{|\eta|}{|k|}\big)\Big)\right)
\end{align*}
which gives the lemma.
\end{proof}

\begin{lemma}\label{Lem:compare}
Let $\xi$, $\eta$ be such that there exists some $\al\geq 1$ with $\f{1}{\al}|\xi|\leq |\eta|\leq \al|\xi|$ and let $k, n$ be such that $t\in \bar{\rmI}_{k,\eta}$ and $t\in \bar{\rmI}_{n,\xi}$ (note that $k\approx n$). Then at least one of the following holds:
\begin{itemize}
\item[(a)] $k=n$ and $t\in \rmI_{k,\xi}\cap \rmI_{k,\eta}$;
\item[(b)] $k=n$ and $|t-\f{\eta}{k}|\geq \f{1}{10\al}\f{|\eta|}{k^3}$ and $|t-\f{\xi}{k}|\geq \f{1}{10\al}\f{|\xi|}{k^3}$;
\item[(c)] $k=n$ and $|\xi-\eta|\gtrsim_{\al} \f{|\eta|}{k^2}$; 
\item[(d)] $|t-\f{\eta}{k}|\geq \f{1}{10\al}\f{|\eta|}{k^2}$ and $|t-\f{\xi}{n}|\geq \f{1}{10\al}\f{|\xi|}{n^2}$;
\item[(e)] $|\xi-\eta|\gtrsim_{\al} \f{|\eta|}{|n|}$.
\end{itemize}
Moreover if $t\in \rmI_{k,\eta}\cap \rmI_{n,\xi}$, then at least one of the following holds:
\begin{itemize}
\item[(a')] $k=n$; 
\item[(b')] $|t-\f{\eta}{k}|\gtrsim_{\al}\f{|\eta|}{k^3}$ and $|t-\f{\xi}{n}|\gtrsim_{\al} \f{|\xi|}{n^3}$;
\item[(c')] $|\xi-\eta|\gtrsim_{\al} \f{|\eta|}{|n|^2}$.
\end{itemize}
\end{lemma}
\begin{proof}
To see that $k\approx n$ note that
\ben
\f{|k|}{|n|}=\f{|\eta|}{|\xi|}\f{|tk|}{|\eta|}\f{|\xi|}{|tn|}\approx_{\al}1. 
\een
If $n=k$ and $t\in \rmI_{k,\xi}\cap \rmI_{k,\eta}$, then there is nothing to prove. Suppose $n=k$ but (a) and (b) are false, which means one of the two inequalities in (b) fails. Without loss of generality suppose $|t-\f{\xi}{k}|<\f{1}{10\al}\f{|\xi|}{k^3}$. Then, $t\notin \rmI_{k,\eta}$ gives us that
\beno
|\xi-\eta|\geq k\left|\f{\xi}{k}-\f{\eta}{k}\right|\geq k\left|\f{\eta}{k}-t\right|-k\left|t-\f{\xi}{k}\right|
\geq \f{|\eta|}{2k^2}-\f{1}{10\al}\f{|\xi|}{k^2}\gtrsim\f{|\eta|}{k^2}.
\eeno
This proves (c).\\
Suppose $n\neq k$ and (d) is false. Without loss of generality suppose that $|t-\f{\xi}{n}|< \f{1}{10\al}\f{|\xi|}{n^2}$.  Then, $t\in \rmI_{k,\eta}$ gives us $t\notin \rmI_{n,\eta}$, which gives us that
\beno
\left|\f{\xi}{n}-\f{\eta}{n}\right|\geq \left|\f{\eta}{n}-t\right|-\left|t-\f{\xi}{n}\right|
\geq \f{|\eta|}{2n(n+1)}-\f{1}{10\al}\f{|\xi|}{n^2}\gtrsim\f{|\eta|}{n^2}.
\eeno
This proves (e). The proof of (a'), (b') and (c') is similar. 
\end{proof}

By the construction of $\Theta$, it is easy to check that the following lemma holds. 
\begin{lemma}\label{Lemma_Growth_1}
For $t\in \rmI_{k,\eta}$ and $t>\rmE(|\eta|^{\f13})$, we have the following with $\tau=t-\f{\eta}{k}$: 
\ben\label{Weight_ineq_1}
\f{\pa_t\Theta_{\rmNR}(t,\eta)}{\Theta_{\rmNR}(t,\eta)}\approx \f{1}{(1+|\tau|)}\approx \f{\pa_t\Theta_{\rmR}(t,\eta)}{\Theta_{\rmR}(t,\eta)}.
\een
\end{lemma}

\begin{lemma}\label{lem:g/g}
For all $t,\xi,\eta$, it holds that
\begin{align}\label{eq:g/g}
\f{\rmg(t,\xi)}{\rmg(t,\eta)}+\f{\rmg(t,\eta)}{\rmg(t,\xi)}
\lesssim  e^{C\d_{\rmL}^{-1}|\eta-\xi|^{\f13}}.
\end{align}
\end{lemma}
\begin{proof}
It is enough to prove 
\begin{equation}\label{Equiv_Estimate}
e^{-C\d_{\rmL}^{-1}|\eta-\xi|^{\f13}}\lesssim \f{\rmg(t,\xi)}{\rmg(t,\eta)}\lesssim  e^{C\d_{\rmL}^{-1}|\eta-\xi|^{\f13}}.
\end{equation}
Without loss of generality, we assume that $|\xi|\leq |\eta|$.
If $|\eta|\leq 1$, we have $\rmg(t,\xi) =\rmg(t,\eta)=1$, so there is nothing to prove. If $|\xi|<1\leq |\eta|$, then we have $|\eta|\leq |\eta-\xi|$ and hence we have from \eqref{Estimate_g} 
\begin{equation}
e^{C\d_{\rmL}^{-1}|\eta-\xi|^{\f13}}\lesssim \f{\rmg(t,\xi)}{\rmg(t,\eta)}=\f{1}{\rmg(t,\eta)}\lesssim  e^{C\d_{\rmL}^{-1}|\eta|^{\f13}}\lesssim e^{C\d_{\rmL}^{-1}|\eta-\xi|^{\f13}}. 
\end{equation}
So from now on, we assume that $\min(|\xi|,|\eta|)\geq 1. $

 If $|\xi|<\f{|\eta|}{2}$, then it holds that $|\eta|\leq 2|\eta-\xi|$ and hence $|\xi|\leq |\eta-\xi|$. Therefore, we have by using \eqref{Estimate_g}, 
 \begin{equation}
e^{-C\d_{\rmL}^{-1}|\eta-\xi|^{\f13}}\lesssim e^{-C\d_{\rmL}^{-1}|\xi|^{\f13}}\lesssim \f{\rmg(t,\xi)}{\rmg(t,\eta)}\lesssim e^{C\d_{\rmL}^{-1}|\eta|^{\f13}}\lesssim e^{C\d_{\rmL}^{-1}|\eta-\xi|^{\f13}}. 
\end{equation}

Now, we may focus on the case $=|\eta|/2\leq |\xi|\leq |\eta|$. We need to discuss the following time regime: 
\begin{description}
\item[Case 1] $t\geq 2|\eta|$;
\item[Case 2] $t\leq \min (t_{\rmE(|\xi|^{\f23}),|\xi|}, t_{\rmE(|\eta|^{\f23}),|\eta|}) $;
\item [Case 3] $\min (t_{\rmE(|\xi|^{\f23}),|\xi|}, t_{\rmE(|\eta|^{\f23}),|\eta|})\leq t\leq \max  (t_{\rmE(|\xi|^{\f23}),|\xi|}, t_{\rmE(|\eta|^{\f23}),|\eta|})$;
\item [Case 4] $\max  (t_{\rmE(|\xi|^{\f23}),|\xi|}, t_{\rmE(|\eta|^{\f23}),|\eta|})\leq t\leq \min (t_{\rmE(|\xi|^{\f13}),|\xi|}, t_{\rmE(|\eta|^{\f13}),|\eta|})$;
\item [Case 5] $\min (t_{\rmE(|\xi|^{\f13}),|\xi|}, t_{\rmE(|\eta|^{\f13}),|\eta|})\leq t\leq \max  (t_{\rmE(|\xi|^{\f13}),|\xi|}, t_{\rmE(|\eta|^{\f13}),|\eta|})$;
\item [Case 6] $\max  (t_{\rmE(|\xi|^{\f13}),|\xi|}, t_{\rmE(|\eta|^{\f13}),|\eta|})\leq t\leq 2|\xi|$;
\item [Case 7] $2|\xi|\leq t\leq 2|\eta|$.
\end{description}

Now, we discuss each of the above cases separately.  
 Throughout  the proof, we will use the following notations:  
\begin{equation}
\rmF(k,\eta)=\arctan\big(\f{|\eta|}{|k|(2|k|+1)}\big)+\arctan\big(\f{|\eta|}{|k|(2|k|-1)}\big)
\end{equation}
and 
\begin{equation}\label{F_tilde}
\tilde{\rmF}(k,\eta)=\f{|\eta|}{|k|^3}\rmF(k,\eta). 
\end{equation}

\textbf{Case 1.}
For $t\geq 2|\eta|$, then $\rmg(t,\xi) =\rmg(t,\eta)=1$, so there is nothing to prove.

\textbf{Case 2.} For $t\leq \min (t_{\rmE(|\xi|^{\f23}),|\xi|}, t_{\rmE(|\eta|^{\f23}),|\eta|}), $ then we have  
\begin{equation}\label{Main_term_1}
\begin{aligned}
\f{\rmg(t,\xi)}{\rmg(t,\eta)}=\f{\rmg(0,\xi)}{\rmg(0,\eta)}=&\,\exp(G_1(k,\eta)-G_1(k,\xi))\exp(G_2(k,\eta)-G_2(k,\xi))\\
\end{aligned}
\end{equation}
with    
\begin{equation} 
\begin{aligned}  
G_1(k,\eta)=\d_{\rmL}^{-1}\sum_{k=1}^{\rmE(|\eta|^{\f13})}\rmF(k,\eta)\quad \text{and}\quad 
G_2(k,\eta)=\d_{\rmL}^{-1}\sum_{\rmE(|\eta|^{\f13})+1}^{\rmE(|\eta|^{\f23})}\tilde{\rmF}(k,\eta). 
\end{aligned}
\end{equation}

We have 
\begin{align*}
&\exp(G_1(k,\eta)-G_1(k,\xi))\\
\leq &\,\exp\bigg(\d_{\rmL}^{-1}\sum_{k=\min\{\rmE(|\xi|^{\f13}),\rmE(|\eta|^{\f13})\}+1}^{\max\{\rmE(|\xi|^{\f13}),\rmE(|\eta|^{\f13})\}}\Big(\arctan\big(\f{\max\{|\xi|,|\eta|\}}{|k|(2|k|+1)}\big)+\arctan\big(\f{\max\{|\xi|,|\eta|\}}{|k|(2|k|-1)}\big)\Big)\bigg)\\
&\qquad \times\exp\bigg(\d_{\rmL}^{-1}\sum_{k=1}^{\min\{\rmE(|\xi|^{\f13}),\rmE(|\eta|^{\f13})\}}|\rmF(k,\xi)-\rmF(k,\eta)|\bigg).
\end{align*}
By the fact that 
\beno
|\rmF(k,\xi)-\rmF(k,\eta)|\lesssim \f{|\xi-\eta|^{\f13}|k^2|}{|\xi|^{\f43}},
\eeno
we obtain that
\begin{equation}\label{G_1_Term}
\begin{aligned}
&\exp(G_1(k,\eta)-G_1(k,\xi))\\
&\leq \exp\bigg(\pi\d_{\rmL}^{-1}\left|\rmE(|\eta|^{\f13})-\rmE(|\xi|^{\f13})\right|\bigg)
\exp\bigg(C\d_{\rmL}^{-1}\sum_{k=1}^{\min\{\rmE(|\xi|^{\f13}),\rmE(|\eta|^{\f13})\}}\f{|\xi-\eta|^{\f13}|k^2|}{|\xi|}\bigg)\\
&\lesssim e^{C\d_{\rmL}^{-1}|\eta-\xi|^{\f13}}. 
\end{aligned}
\end{equation}
Similarly, we have 
\begin{align*}  
&\exp(G_2(k,\eta)-G_2(k,\xi))\\
\leq&\, \exp\bigg(\d_{\rmL}^{-1}\sum_{k=\min\{\rmE(|\xi|^{\f23}),\rmE(|\eta|^{\f23})\}+1}^{\max\{\rmE(|\xi|^{\f23}),\rmE(|\eta|^{\f23})\}}\f{|\eta|}{|k|^3}\Big(\arctan\big(\f{\max\{|\xi|,|\eta|\}}{|k|(2|k|+1)}\big)+\arctan\big(\f{\max\{|\xi|,|\eta|\}}{|k|(2|k|-1)}\big)\Big)\bigg)\\
& \times\exp\bigg(\d_{\rmL}^{-1}\sum_{k=\max\{\rmE(|\xi|^{\f13}),\rmE(|\eta|^{\f13})\}}^{\min\{\rmE(|\xi|^{\f23}),\rmE(|\eta|^{\f23})\}}|\tilde{\rmF}(k,\xi)-\tilde{\rmF}(k,\eta)|)\bigg)
\end{align*}
We have  
\begin{equation}\label{F_tilde_Estimate}
|\tilde{\rmF}(k,\xi)-\tilde{\rmF}(k,\eta)|\leq \f{|\eta-\xi|}{|k|^3}.  
\end{equation}
Consequently, we obtain 
\begin{equation}\label{G_2_Term}
\begin{aligned}
&\exp(G_2(k,\eta)-G_2(k,\xi))\\
\leq&\,\exp\bigg(\pi\d_{\rmL}^{-1}\left|\rmE(|\eta|^{\f23})-\rmE(|\xi|^{\f23})\right||\eta|^{-1}\bigg)
\exp\bigg(C\d_{\rmL}^{-1}\sum_{k=\max\{\rmE(|\xi|^{\f13}),\rmE(|\eta|^{\f13})\}}^{\min\{\rmE(|\xi|^{\f23}),\rmE(|\eta|^{\f23})\}}\f{|\eta-\xi|}{|k|^3}\bigg)\\
&\lesssim e^{C\d_{\rmL}^{-1}|\eta-\xi|^{\f13}}.
\end{aligned}
\end{equation}
Plugging \eqref{G_1_Term} and \eqref{G_2_Term} into \eqref{Main_term_1}, we obtain the desired result. 

\textbf{Case 3.} 
If $t_{\rmE(|\xi|^{\f23}),\xi}\leq t\leq t_{\rmE(|\eta|^{\f23}),\eta}$, we have
\begin{equation}\label{Eq_Compa_Case_3_1}
\f{\rmg(t_{\rmE(|\eta|^{\f23}),\eta},\xi)}{\rmg(t_{\rmE(|\eta|^{\f23}),\eta},\eta)}\geq \f{\rmg(t,\xi)}{\rmg(t,\eta)}\geq \f{\rmg(0,\xi)}{\rmg(0,\eta)},
\end{equation}
and if $t_{\rmE(|\eta|^{\f23}),\eta}\leq t\leq t_{\rmE(|\xi|^{\f23}),\xi}$, we have
\begin{equation}\label{g_Case_3_Estimate}
\f{\rmg(t_{\rmE(|\xi|^{\f23}),\xi},\xi)}{\rmg(t_{\rmE(|\xi|^{\f23}),\xi},\eta)}\leq \f{\rmg(t,\xi)}{\rmg(t,\eta)}\leq \f{\rmg(0,\xi)}{\rmg(0,\eta)}.  
\end{equation}
Thus we can deduce the estimates in Case 3 by the estimate of Case 4 and Case 2. Similar  argument can also apply to Case 7. Indeed, we have:

\textbf{Case 7.} If $2|\xi|\leq t\leq 2|\eta|$, we have $1\geq \rmg(t,\eta)\geq \rmg(2|\xi|,\eta)$ and $\rmg(t,\xi)=\rmg(2|\xi|,\xi)=1$, which implies
\beno
1\geq \f{\rmg(t,\eta)}{\rmg(t,\xi)}\geq \f{\rmg(2|\xi|,\eta)}{\rmg(2|\xi|,\xi)}. 
\eeno
Thus we can deduce the estimates in Case 7 by the estimate of Case 6.

For the cases $4$, $5$ and $6$, we have  $\max(t_{\rmE(|\eta|^{\f23}),\eta},t_{\rmE(|\xi|^{\f23}),\xi})\leq t\leq 2|\xi|$.  Let $j$ and $n$ be such that $t\in \bar{\rmI}_{n,\eta}\cap \bar{\rmI}_{j,\xi}$. Then we have $j\approx n\geq j$.

\textbf{Case 6.}
 Let $t$ be such that   $\max  (t_{\rmE(|\xi|^{\f13}),|\xi|}, t_{\rmE(|\eta|^{\f13}),|\eta|})\leq t\leq 2|\xi|$.  
We have in this case 
\begin{align*}
\rmg(t,\eta)=&\,\exp\left(-\d_{\rmL}^{-1}\sum_{k=1}^{n-1}\rmF(k,\eta)\right)\\
&\times \exp\left(\d_{\rmL}^{-1}\Big(\arctan\big(t-\f{|\eta|}{|n|}\big)-\arctan\big(\f{|\eta|}{|n|(2|n|-1)}\big)\Big)\right)
\end{align*}
and 
\begin{align*}
\rmg(t,\xi)=&\,\exp\left(-\d_{\rmL}^{-1}\sum_{k=1}^{j-1}\rmF(k,\xi)\right)\\
&\times \exp\left(\d_{\rmL}^{-1}\Big(\arctan\big(t-\f{|\xi|}{|j|}\big)-\arctan\big(\f{|\xi|}{|j|(2|j|-1)}\big)\Big)\right). 
\end{align*}  
By the fact that $n\geq j$ we have   
\begin{align*}
\f{\rmg(t,\xi)}{\rmg(t,\eta)}
&=\f{\exp\left(\d_{\rmL}^{-1}\sum\limits_{k=1}^{j-1}\rmF(k,\eta)\right)}{\exp\left(\d_{\rmL}^{-1}\sum\limits_{k=1}^{j-1}\rmF(k,\xi)\right)}
\underbrace{\exp\left(\d_{\rmL}^{-1}\sum_{k=j}^{n-1}\rmF(k,\eta)\right)}_{\text{when}\ j\leq n-1 }\\
&\quad\times \f{\exp\left(\d_{\rmL}^{-1}\big(\arctan\big(t-\f{|\xi|}{|j|}\big)-\arctan\big(\f{|\xi|}{|j|(2|j|-1)}\big)\big)\right)}{\exp\left(\d_{\rmL}^{-1}\big(\arctan\big(t-\f{|\eta|}{|n|}\big)-\arctan\big(\f{|\eta|}{|n|(2|n|-1)}\big)\big)\right)}
\end{align*}
We get 
\begin{equation}\label{Mean_Value_1}
\Big|\arctan\big(\f{|\eta|}{|k|(2|k|\pm1)}\big)-\arctan\big(\f{|\xi|}{|k|(2|k|\pm1)}\big)\Big|\lesssim \f{|\eta-\xi|}{\langle \xi\rangle}. 
\end{equation}
Since $t\in \bar{\rmI}_{n,\eta}\cap \bar{\rmI}_{j,\xi}$, then it holds that  $|jt-\xi|=|j||t-\f{\xi}{j}|\leq \f{|\xi|}{|j|}\approx t $ and also $|\eta-nt|\leq \f{|\eta|}{|n|}\approx t$, hence, we obtain the inequality 
\begin{equation}\label{Alg_Ineq_Im}
|j-n|\leq \f{|jt-\xi|+|\eta-\xi|+|\eta-nt|}{t}\lesssim 1+\f{|\eta-\xi||n|}{|\eta|}\lesssim 1+|\eta-\xi|^{\f13},
\end{equation}
Hence, we obtain 
\begin{equation}
\begin{aligned}
\f{\rmg(t,\xi)}{\rmg(t,\eta)} &\lesssim \exp\Big(C\d_{\rmL}^{-1}\f{|\eta-\xi|}{\langle\xi\rangle}n\Big)\exp(|n-j|),   
\end{aligned}
\end{equation}
and similarly we have $\f{\rmg(t,\eta)}{\rmg(t,\xi)} \lesssim \exp\Big(C\d_{\rmL}^{-1}\f{|\eta-\xi|}{\langle\xi\rangle}n\Big)\exp(|n-j|)$. 

By the fact that $n\lesssim |\xi|^{\f13}$ and $|\eta-\xi|\lesssim |\xi|$, we get 
\begin{align*}
\f{\rmg(t,\eta)}{\rmg(t,\xi)}+\f{\rmg(t,\xi)}{\rmg(t,\eta)}\lesssim\exp\Big(C\d_{\rmL}^{-1}|\eta-\xi|^{\f13}\Big).
\end{align*}

\textbf{Case 4.}
In the proof of Lemma 3.10, you may need more precise formula: (similar for the Case 5.)
Let $j$ and $n$ be such that $t\in \bar{\rmI}_{n,\eta}\cap \bar{\rmI}_{j,\xi}$.   
Using the fact that $|\eta|\geq |\xi|$ which gives that $\rmE(|\eta|^{\f13})\geq \rmE(|\xi|^{\f13})$, then $|n|\geq \rmE(|\eta|^{\f13})$ and $|j|\geq \rmE(|\xi|^{\f13})$. By the definition of $\rmg(t,\eta)$, we get that
\begin{equation}\label{g_eta_case_4}
\begin{aligned}
\rmg(t,\eta)=&\exp\left(-\d_{\rmL}^{-1}\sum_{k=1}^{\rmE(|\xi|^{\f13})}\rmF(k,\eta)\right) 
\underbrace{\exp\left(-\d_{\rmL}^{-1}\sum_{\rmE(|\xi|^{\f13})+1}^{\rmE(|\eta|^{\f13})}\rmF(k,\eta)\right)}_{\text{when}\ \rmE(|\eta|^{\f13})\geq \rmE(|\xi|^{\f13})+1}\\
 &\times\exp\left(-\d_{\rmL}^{-1}\sum_{k=\rmE(|\eta|^{\f13})+1}^{n-1}\tilde{\rmF}(k,\eta)\right)\\
&\times \exp\left(\d_{\rmL}^{-1}\f{|\eta|}{|n|^3}\Big(\arctan\big(t-\f{|\eta|}{|n|}\big)-\arctan\big(\f{|\eta|}{|n|(2|n|-1)}\big)\Big)\right)
\end{aligned}
\end{equation}
and if $\rmE(|\eta|^{\f13})\geq |j| \geq \rmE(|\xi|^{\f13})$, 
\begin{align*}
\rmg(t,\xi)=&\exp\left(-\d_{\rmL}^{-1}\sum_{k=1}^{\rmE(|\xi|^{\f13})}\rmF(k,\xi)\right)\exp\left(-\d_{\rmL}^{-1}\sum_{k=\rmE(|\xi|^{\f13})+1}^{j-1}\tilde{\rmF}(k,\xi)\right)\\
&\times \exp\left(\d_{\rmL}^{-1}\f{|\xi|}{|j|^3}\Big(\arctan\big(t-\f{|\xi|}{|j|}\big)-\arctan\big(\f{|\xi|}{|j|(2|j|-1)}\big)\Big)\right)
\end{align*}
and if $\rmE(|\eta|^{\f13})+1\leq |j|$
\begin{equation}\label{g_xi_Case_4_2}
\begin{aligned}
\rmg(t,\xi)=&\exp\left(-\d_{\rmL}^{-1}\sum_{k=1}^{\rmE(|\xi|^{\f13})}\rmF(k,\xi)\right) 
\exp\left(-\d_{\rmL}^{-1}\sum_{k=\rmE(|\xi|^{\f13})+1}^{\rmE(|\eta|^{\f13})}\tilde{\rmF}(k,\xi)\right) \\
&\times\exp\left(-\d_{\rmL}^{-1}\sum_{k=\rmE(|\eta|^{\f13})+1}^{j-1}\tilde{\rmF}(k,\xi)\right)\\
&\times \exp\left(\d_{\rmL}^{-1}\f{|\xi|}{|j|^3}\Big(\arctan\big(t-\f{|\xi|}{|j|}\big)-\arctan\big(\f{|\xi|}{|j|(2|j|-1)}\big)\Big)\right)
\end{aligned}  
\end{equation}
Thus we have that
\begin{align*}
&\f{\rmg(t,\eta)}{\rmg(t,\xi)}+\f{\rmg(t,\xi)}{\rmg(t,\eta)}\\
&\lesssim \exp\left(\d_{\rmL}^{-1}\sum_{k=1}^{\rmE(|\xi|^{\f13})}|\rmF(k,\eta)-\rmF(k,\xi)|\right)
\exp\big(\rmE(|\eta|^{\f13})-\rmE(|\xi|^{\f13})\big)\\
&\quad \times \exp\left(\d_{\rmL}^{-1}\sum_{k=\rmE(|\eta|^{\f13})+1}^{j-1}|\tilde{\rmF}(k,\eta)-\tilde{\rmF}(k,\xi)|\right)
\exp\left(\d_{\rmL}^{-1}\sum_{k=j}^{n-1}\tilde{\rmF}(k,\eta)\right)\\
&\lesssim \exp\left(C\d_{\rmL}^{-1}|\xi|^{\f13}\f{|\xi-\eta|}{\xi}\right)\exp\big(\rmE(|\eta|^{\f13})-\rmE(|\xi|^{\f13})\big)\exp\left(\d_{\rmL}^{-1}\sum_{k=\rmE(|\eta|^{\f13})+1}^{j-1}\f{|\xi-\eta|}{k^3}\right)\exp\Big(\f{\d_{\rmL}^{-1}|n-j|\eta}{n^3}\Big).
\end{align*}
By the fact that $|\eta-\xi|\lesssim |\xi|\approx|\eta|$, $|n|\gtrsim |\eta|^{\f13}$ and
\beno
|j-n|\leq \f{|jt-\xi|+|\eta-\xi|+|\eta-nt|}{t}\lesssim 1+\f{|\eta-\xi||n|}{|\eta|},
\eeno
we get 
\begin{align*}
\f{\rmg(t,\eta)}{\rmg(t,\xi)}+\f{\rmg(t,\xi)}{\rmg(t,\eta)}
\lesssim \exp\left(C\d_{\rmL}^{-1}|\xi-\eta|^{\f13}\right)\exp\Big(\f{\d_{\rmL}^{-1}|\eta-\xi|}{n^2}\Big)\lesssim \exp\left(C\d_{\rmL}^{-1}|\xi-\eta|^{\f13}\right).
\end{align*}

\textbf{Case 5.} Let $j$ and $n$ be such that $t\in \bar{\rmI}_{n,\eta}\cap \bar{\rmI}_{j,\xi}$. If $t_{\rmE(|\xi|^{\f13}),|\xi|}\leq t\leq t_{\rmE(|\eta|^{\f13}),|\eta|}$, we have $|j|\leq \rmE(|\xi|^{\f13})\leq \rmE(|\eta|^{\f13})\leq |n|\approx |j|\approx |\xi|^{\f13}\approx|\eta|^{\f13}$, we write 
\begin{align*}
\rmg(t,\eta)=&\exp\left(-\d_{\rmL}^{-1}\sum_{k=1}^{j-1}\rmF(k,\eta)\right) 
\exp\left(-\d_{\rmL}^{-1}\sum_{k=j}^{\rmE(|\eta|^{\f13})}\rmF(k,\eta)\right)
 \exp\left(-\d_{\rmL}^{-1}\sum_{k=\rmE(|\eta|^{\f13})+1}^{n-1}\tilde{\rmF}(k,\eta)\right)\\
&\times \exp\left(\d_{\rmL}^{-1}\f{|\eta|}{|n|^3}\Big(\arctan\big(t-\f{|\eta|}{|n|}\big)-\arctan\big(\f{|\eta|}{|n|(2|n|-1)}\big)\Big)\right)
\end{align*}
and
\begin{align*}
\rmg(t,\xi)=&\exp\left(-\d_{\rmL}^{-1}\sum_{k=1}^{j-1}\rmF(k,\xi)\right) \exp\left(\d_{\rmL}^{-1}\Big(\arctan\big(t-\f{|\eta|}{|j|}\big)-\arctan\big(\f{|\eta|}{|j|(2|j|-1)}\big)\Big)\right)
\end{align*}
By the fact that
\beno
|j-n|\leq \f{|jt-\xi|+|\eta-\xi|+|\eta-nt|}{t}\lesssim 1+\f{|\eta-\xi||n|}{|\eta|},
\eeno
we have
\begin{align*}
\f{\rmg(t,\xi)}{\rmg(t,\eta)}+\f{\rmg(t,\eta)}{\rmg(t,\xi)}\lesssim \exp\left(\d_{\rmL}^{-1}\sum_{k=1}^{j-1}|\rmF(k,\xi)-\rmF(k,\eta)|\right)\exp(|j-n|)\lesssim \exp(C\d_{\rmL}^{-1}|\xi-\eta|^{\f13}). 
\end{align*}
If $t_{\rmE(|\xi|^{\f13}),|\xi|}\geq t\geq t_{\rmE(|\eta|^{\f13}),|\eta|}$, then $\rmE(|\xi|^{\f13})\leq |j|\leq |n|\leq \rmE(|\eta|^{\f13})$, we write
\begin{align*}
\rmg(t,\eta)=&\exp\left(-\d_{\rmL}^{-1}\sum_{k=1}^{\rmE(|\xi|^{\f13})}\rmF(k,\eta)\right) 
\exp\left(-\d_{\rmL}^{-1}\sum_{k=\rmE(|\xi|^{\f13})+1}^{n-1}\rmF(k,\eta)\right)\\
&\times \exp\left(\d_{\rmL}^{-1}\Big(\arctan\big(t-\f{|\eta|}{|n|}\big)-\arctan\big(\f{|\eta|}{|n|(2|n|-1)}\big)\Big)\right)
\end{align*}
and
\begin{align*}
\rmg(t,\xi)=&
\exp\left(-\d_{\rmL}^{-1}\sum_{k=1}^{\rmE(|\xi|^{\f13})}\rmF(k,\eta)\right)
 \exp\left(-\d_{\rmL}^{-1}\sum_{k=\rmE(|\xi|^{\f13})+1}^{j-1}\tilde{\rmF}(k,\eta)\right)\\
&\times \exp\left(\d_{\rmL}^{-1}\f{|\xi|}{|j|^3}\Big(\arctan\big(t-\f{|\xi|}{|j|}\big)-\arctan\big(\f{|\xi|}{|j|(2|j|-1)}\big)\Big)\right).
\end{align*}
By the fact that
\beno
|j-n|\leq \f{|jt-\xi|+|\eta-\xi|+|\eta-nt|}{t}\lesssim 1+\f{|\eta-\xi||n|}{|\eta|},
\eeno
we have
\begin{align*}
\f{\rmg(t,\xi)}{\rmg(t,\eta)}+\f{\rmg(t,\eta)}{\rmg(t,\xi)}\lesssim \exp\left(\d_{\rmL}^{-1}\sum_{k=1}^{\rmE(|\xi|^{\f13})}|\rmF(k,\xi)-\rmF(k,\eta)|\right)\exp(|j-n|)\lesssim \exp(C\d_{\rmL}^{-1}|\xi-\eta|^{\f13}). 
\end{align*}
Thus we proved the lemma. 
\end{proof}

\begin{lemma}\label{lem:3.5}
For all $t,\eta,\xi,$ we have
\beno
\f{\Theta_{\rmNR}(t,\eta)}{\Theta_{\rmNR}(t,\xi)}\lesssim e^{\mu |\eta-\xi|^{\f13}}. 
\eeno
\end{lemma}
\begin{proof}
One may follow the proof of {\it Lemma 3.5} or the proof of Lemma \ref{lem:g/g} and deduce to prove for $\eta,\xi\geq 0$, $\eta/2\leq \xi\leq \eta$ and $t\in \max(t_{\rmE(|\eta|^{\f13}),\eta},t_{\rmE(|\xi|^{\f13}),\xi})\leq t\leq 2|\xi|$, it holds that
\ben\label{eq:up-loww/w}
e^{-\mu |\eta-\xi|^{\f13}}\lesssim \f{\Theta_{\rmNR}(t,\eta)}{\Theta_{\rmNR}(t,\xi)}\lesssim e^{\mu |\eta-\xi|^{\f13}}. 
\een
Let $j$ and $n$ be such that $t\in \bar{I}_{n,\eta}\cap \bar{I}_{j,\xi}$. Then we have $n\approx j\leq n$. We consider the following cases:

Case $j=n$ and $t\in \rm{I}_{n,\eta}^{\rmR}\cap \rm{I}_{j,\xi}^{\rmR}$ or $t\in \rm{I}_{n,\eta}^{\rmL}\cap {I}_{j,\xi}^{\rmL}$ or $t\in \rm{I}_{n,\eta}^{\rmL}\cap \rm{I}_{j,\xi}^{\rmR}$: We can use the same argument in the proof of {\it Lemma 3.5} to prove \eqref{eq:up-loww/w} and we omit the details here.

Case $j=n$, $|t-\f{\eta}{n}|\gtrsim \f{|\eta|}{n^3}$ and $|t-\f{\xi}{j}|\gtrsim \f{|\xi|}{j^3}$: It holds that 

for $t\geq \f{\eta}{n}$, $\Theta_{\rmNR}(t,\eta)\approx \Theta_{\rmNR}(t^+_{n,\eta},\eta)$; 

for $t\geq \f{\xi}{n}$, $\Theta_{\rmNR}(t,\xi)\approx \Theta_{\rmNR}(t^+_{n,\xi},\xi)$;

for $t\leq \f{\eta}{n}$, 
$
\Theta_{\rmNR}(t,\eta) \approx \left(\f{\eta}{n^3}\right)^{-1-2C\kappa}\Theta_{\rmNR}(t^+_{n,\eta},\eta)\approx \Theta_{\rmNR}(t^-_{n,\eta},\eta);
$

for $t\leq \f{\xi}{n}$,
$
\Theta_{\rmNR}(t,\xi) \approx \left(\f{\xi}{n^3}\right)^{-1-2C\kappa}\Theta_{\rmNR}(t^+_{n,\xi},\xi)\approx \Theta_{\rmNR}(t^-_{n,\xi},\xi). 
$

Thus if $t\geq \f{\eta}{n}$ and $t\leq \f{\xi}{n}$ or $t\leq \f{\eta}{n}$ and $t\geq \f{\xi}{n}$, then we have $|\eta-\xi|\geq \f{|\eta|}{n^2}$ and 
\begin{align*}
\f{\Theta_{\rmNR}(t,\xi)}{\Theta_{\rmNR}(t,\eta)}
&\lesssim \left(\f{\eta}{\xi}\right)^{c(n-1)+C\kappa}\left(\f{\eta}{n^3}\right)^{1+2C\kappa}\lesssim \left(1+\f{\eta-\xi}{\xi}\right)^{c|\eta|^{\f13}}\left(\f{1+|\eta-\xi|}{n}\right)^{1+2C\kappa}\\
&\lesssim e^{c|\eta-\xi|^{\f13}} 
\end{align*}
and if $t\geq \f{\eta}{n}$ and $t\geq \f{\xi}{n}$ or $t\leq \f{\eta}{n}$ and $t\leq \f{\xi}{n}$, then we have 
\begin{align*}
\f{\Theta_{\rmNR}(t,\xi)}{\Theta_{\rmNR}(t,\eta)}
\lesssim \left(\f{\eta}{\xi}\right)^{c(n-1)+C\kappa}\lesssim e^{c|\eta-\xi|^{\f13}}.
\end{align*}

Case $j=n$, $|\eta-\xi|\gtrsim \f{|\eta|}{n^2}$, similarly we have 
\begin{align*}
\f{\Theta_{\rmNR}(t,\xi)}{\Theta_{\rmNR}(t,\eta)}
&\lesssim \left(\f{\eta}{\xi}\right)^{c(n-1)+C\kappa}\left(\f{\eta}{n^3}\right)^{1+2C\kappa}\lesssim e^{c|\eta-\xi|^{\f13}}. 
\end{align*}

Case $j=n-1$: If $t\in \bar{\rmI}_{n,\eta}^{\rmL}$ then $t_{n-1,\xi}\leq \f{\eta}{n}$. If $t\in \bar{\rmI}_{n-1,\xi}^{\rmR}$, then $\f{\xi}{n-1}<t_{n-1,\eta}$. In either one of these case, we deduce that $\f{|\xi|}{n^2}\lesssim \f{\eta-\xi}{n}$ and thus
\begin{align}\label{eq:xi-eta}
\f{\Theta_{\rmNR}(t,\xi)}{\Theta_{\rmNR}(t,\eta)}
&\lesssim \left(\f{\eta}{\xi}\right)^{c(n-1)+C\kappa}\left(\f{\eta}{n^3}\right)^{1+2C\kappa}\lesssim e^{c|\eta-\xi|^{\f13}}. 
\end{align}
Next assume that $t\in \bar{\rmI}_{n,\eta}^{\rmR}\cap \bar{\rmI}_{n-1,\xi}^{\rmL}$ there are only two possibilities: One is $\f{|\xi|}{n^2}\lesssim \f{\eta-\xi}{n}$ which we can conclude in a similar way to the above \eqref{eq:xi-eta}; The other one is $|t-\f{\eta}{n}|\gtrsim \f{\eta}{n^2}\geq \f{\eta}{n^3}$ and $|t-\f{\xi}{j}|\gtrsim \f{\xi}{j^2}\geq \f{\xi}{j^3}$, which gives us that
$\Theta_{\rmNR}(t,\eta)\approx \Theta_{\rmNR}(t^+_{n,\eta},\eta)$ and $\Theta_{\rmR}(t,\xi)\approx \Theta_{\rmNR}(t^{-}_{n-1,\xi},\xi)=\Theta_{\rmNR}(t^{+}_{n,\xi},\xi)$. Thus we have
\begin{align*}
\f{\Theta_{\rmNR}(t,\xi)}{\Theta_{\rmNR}(t,\eta)}
\lesssim \left(\f{\eta}{\xi}\right)^{c(n-1)+C\kappa}\lesssim e^{c|\eta-\xi|^{\f13}}.
\end{align*}

Case $j<n-1$: In this case, it is easy to see that $\f{\xi}{n^2}\lesssim \f{\eta-\xi}{n}$ and we can conclude in a similar way to \eqref{eq:xi-eta}. 
\end{proof}

\begin{lemma}\label{lem: g/g-1-1}
Let $t\leq \f12\min\{|\xi|^{\f13},|\eta|^{\f13}\}$. Then
\beno
\left|\f{\cM_k(t,\eta)}{\cM_l(t,\xi)}-1\right|\lesssim \f{\langle k-l,\xi-\eta\rangle }{(|k|+|l|+|\eta|+|\xi|)^{\f23}}e^{C\d_{L}^{-1}|k-l,\xi-\eta|^{\f13}}
\eeno
\end{lemma}
\begin{proof}
By Lemma \ref{lem:g/g}, we get that
\begin{equation}\label{Eq_M/M}
\f{\cM_k(t,\eta)}{\cM_l(t,\xi)}\lesssim \f{\rmg(t,\xi)}{\rmg(t,\eta)}e^{C\d_{\rmL}^{-1}|\xi-\eta|^{\f13}}+e^{C\d_{\rmL}^{-1}|k-l|^{\f13}}\lesssim e^{C\d_{\rmL}^{-1}|k-l,\xi-\eta|^{\f13}}.
\end{equation}
If $(|k|+|l|+|\eta|+|\xi|)^{\f23}\lesssim \max\{1,|k-l|+|\xi-\eta|\}$, we get the lemma. 

From now on, we assume that 
\beno
|k-l|+|\xi-\eta|\leq \f{1}{100}\Big(|k|^{\f23}+|l|^{\f23}+|\eta|^{\f23}+|\xi|^{\f23}\Big)\leq \f{1}{100}\Big(|k|+|l|+|\eta|+|\xi|\Big). 
\eeno
Case 1: $\f{1}{10}(|k|+|l|)\leq |\eta|+|\xi|\leq 10(|k|+|l|)$: In this case, we have
\beno
|k-l,\xi-\eta|\lesssim |k|\approx |l|\approx |\xi|\approx |\eta|,
\eeno
and
\begin{align*}
\left|\f{\cM_k(t,\eta)}{\cM_l(t,\xi)}-1\right|
&\leq \f{\left|\f{e^{4\pi\d_{\rmL}^{-1}|\eta|^{\f13}}}{\rmg(t,\eta)}-\f{e^{4\pi\d_{\rmL}^{-1}|\xi|^{\f13}}}{\rmg(t,\xi)}\right|}{\f{e^{4\pi\d_{\rmL}^{-1}|\xi|^{\f13}}}{\rmg(t,\xi)}+e^{4\pi\d_{\rmL}^{-1} |l|^{\f13}}}
+\f{\left|e^{4\pi\d_{\rmL}^{-1}|k|^{\f13}}-e^{4\pi\d_{\rmL}^{-1}|l|^{\f13}}\right|}{\f{e^{4\pi\d_{\rmL}^{-1}|\xi|^{\f13}}}{\rmg(t,\xi)}+e^{4\pi\d_{\rmL}^{-1} |l|^{\f13}}}\\
&\leq \f{\rmg(t,\xi)}{\rmg(t,\eta)}\left||\xi|^{\f13}-|\eta|^{\f13}\right|e^{4\pi\d_{\rmL}^{-1}|\eta-\xi|^{\f13}}+\left|\f{\rmg(t,\xi)}{\rmg(t,\eta)}-1\right|
+\left||k|^{\f13}-|l|^{\f13}\right|e^{4\pi\d_{\rmL}^{-1}|k-l|^{\f13}}\\
&\lesssim \f{|\xi-\eta|}{|\xi|^{\f23}+|\eta|^{\f23}}e^{C\d_{\rmL}^{-1}|\xi-\eta|^{\f13}}
+\left|\f{\rmg(0,\xi)}{\rmg(0,\eta)}-1\right|
+\f{|l-k|}{|l|^{\f23}+|k|^{\f23}}e^{4\pi\d_{\rmL}^{-1}|l-k|^{\f13}}. 
\end{align*}
Recall that \eqref{eq:g(k,eta)/g(k,eta)1} and \eqref{eq:g(k,eta)/g(k,eta)2}, we have 

\beno
\begin{aligned}
\f{1}{\rmg(0,\eta)}
&=\,\exp\left(\d_{\rmL}^{-1}\sum_{\rmE(|\eta|^{\f13})+1}^{\rmE(|\eta|^{\f23})}\tilde{\rmF}(k,\eta)\right)
 \exp\left(\d_{\rmL}^{-1}\sum_{k=1}^{\rmE(|\eta|^{\f13})}\rmF(k,\eta)\right).  
\end{aligned}
\eeno
and 
\beno
\begin{aligned}
\f{1}{\rmg(0,\xi)}
&=\,\exp\left(\d_{\rmL}^{-1}\sum_{\rmE(|\xi|^{\f13})+1}^{\rmE(|\xi|^{\f23})}\tilde{\rmF}(k,\xi)\right)
 \exp\left(\d_{\rmL}^{-1}\sum_{k=1}^{\rmE(|\xi|^{\f13})}\rmF(k,\xi)\right).
\end{aligned}
\eeno
This gives (if $|\xi|\leq |\eta|$)
\begin{align*}
\f{\rmg(0,\xi)}{\rmg(0,\eta)}=&\,\exp\left(\d_{\rmL}^{-1}\sum_{k=1}^{\rmE(|\xi|^{\f13})}(\rmF(k,\eta)-\rmF(k,\xi))\right)\exp\left(\d_{\rmL}^{-1}\sum_{k=\rmE(|\xi|^{\f13})+1}^{\rmE(|\eta|^{\f13})}\rmF(k,\eta)\right)\\
&\times \exp\left(\d_{\rmL}^{-1}\sum_{\rmE(|\eta|^{\f13})+1}^{\rmE(|\xi|^{\f23})}(\tilde{\rmF}(k,\eta)-\tilde{\rmF}(k,\xi))\right)\exp\left(-\d_{\rmL}^{-1}\sum_{\rmE(|\xi|^{\f13})+1}^{\rmE(|\eta|^{\f13})}\tilde{\rmF}(k,\xi)\right)\\ 
&\times \exp\left(\d_{\rmL}^{-1}\sum_{\rmE(|\xi|^{\f23})+1}^{\rmE(|\eta|^{\f23})}\tilde{\rmF}(k,\eta)\right)\\
\f{\rmg(0,\eta)}{\rmg(0,\xi)}=&\,\exp\left(-\d_{\rmL}^{-1}\sum_{k=1}^{\rmE(|\xi|^{\f13})}(\rmF(k,\eta)-\rmF(k,\xi))\right)\exp\left(-\d_{\rmL}^{-1}\sum_{k=\rmE(|\xi|^{\f13})+1}^{\rmE(|\eta|^{\f13})}\rmF(k,\eta)\right)\\
&\times \exp\left(-\d_{\rmL}^{-1}\sum_{\rmE(|\eta|^{\f13})+1}^{\rmE(|\xi|^{\f23})}(\tilde{\rmF}(k,\eta)-\tilde{\rmF}(k,\xi))\right)\exp\left(\d_{\rmL}^{-1}\sum_{\rmE(|\xi|^{\f13})+1}^{\rmE(|\eta|^{\f13})}\tilde{\rmF}(k,\xi)\right)\\ 
&\times \exp\left(-\d_{\rmL}^{-1}\sum_{\rmE(|\xi|^{\f23})+1}^{\rmE(|\eta|^{\f23})}\tilde{\rmF}(k,\eta)\right)
\end{align*}
By using the fact that $|e^ae^b-1|\lesssim (|a|+|b|)e^{a+b}$, we have 
\begin{align*}
&\left|\f{\rmg(0,\xi)}{\rmg(0,\eta)}-1\right|+\left|\f{\rmg(0,\eta)}{\rmg(0,\xi)}-1\right|\\
&\lesssim \bigg(\d_{\rmL}^{-1}\sum_{k=1}^{\rmE(|\xi|^{\f13})}|\rmF(k,\eta)-\rmF(k,\xi)|+\d_{\rmL}^{-1}\sum_{\rmE(|\eta|^{\f13})+1}^{\rmE(|\xi|^{\f23})}|\tilde{\rmF}(k,\eta)-\tilde{\rmF}(k,\xi)|\\
&\quad\quad+\sum_{\min\{\rmE(|\eta|^{\f13}),\rmE(|\xi|^{\f13})\}}^{\max\{\rmE(|\eta|^{\f13}),\rmE(|\xi|^{\f13})\}}1
+\sum_{\min\{\rmE(|\eta|^{\f23}),\rmE(|\xi|^{\f23})\}}^{\max\{\rmE(|\eta|^{\f23}),\rmE(|\xi|^{\f23})\}}\f{\max\{|\xi|,|\eta|\}}{|k|^3}\bigg)
\max\left\{\f{\rmg(0,\xi)}{\rmg(0,\eta)},\f{\rmg(0,\eta)}{\rmg(0,\xi)}\right\}.  
\end{align*}
 Now applying Lemma \ref{lem:g/g}, we deduce that  
\begin{align*}
&\left|\f{\rmg(0,\xi)}{\rmg(0,\eta)}-1\right|+\left|\f{\rmg(0,\eta)}{\rmg(0,\xi)}-1\right|
\\
&\lesssim \bigg(\sum_{k=1}^{\rmE(|\xi|^{\f13})}|\rmF(k,\eta)-\rmF(k,\xi)|+\sum_{\rmE(|\eta|^{\f13})+1}^{\rmE(|\xi|^{\f23})}|\tilde{\rmF}(k,\eta)-\tilde{\rmF}(k,\xi)|\\
&\quad\quad+|\rmE(|\eta|^{\f13})-\rmE(|\xi|^{\f13})|
+\f{|\rmE(|\eta|^{\f23})-\rmE(|\xi|^{\f23})|}{|\xi|}\bigg)e^{C\d_{\rmL}^{-1}|\xi-\eta|^{\f13}}.
\end{align*}

We have for $|\xi|\approx |\eta|$ and $1\leq |k|\leq \min\{\rmE(|\eta|^{\f13}),\rmE(|\xi|^{\f13})\}$ that  
\begin{equation}\label{Ineq_F_Mean_Value}
|\rmF(k,\eta)-\rmF(k,\xi)|\lesssim \f{|\xi-\eta|}{k^2+|\xi|^2},
\end{equation}
and for $\max\{\rmE(|\eta|^{\f13}),\rmE(|\xi|^{\f13})\} \leq |k|\leq \min\{\rmE(|\eta|^{\f23}),\rmE(|\xi|^{\f23})\}$ that
\begin{equation}\label{F_tilde_Mean_Value}
|\tilde{\rmF}(k,\eta)-\tilde{\rmF}(k,\xi)|\lesssim \f{\eta}{k^3}\f{|\xi-\eta|}{k^2+|\xi|^2}+\f{|\eta-\xi|}{k^3}\lesssim \f{|\eta-\xi|}{k^3},
\end{equation}
which gives us that
\begin{equation}\label{g_0_g_0}
\left|\f{\rmg(0,\eta)}{\rmg(0,\xi)}-1\right|+\left|\f{\rmg(0,\xi)}{\rmg(0,\eta)}-1\right|\lesssim \f{|\xi-\eta|}{|\xi|^{\f23}+|\eta|^{\f23}}e^{C\d_{\rmL}^{-1}|\xi-\eta|^{\f13}}
\end{equation}
Case 2: $|\eta|+|\xi|\geq 10(|k|+|l|)$: We have $|\xi|\geq 4(|k|+|l|)$ and
\begin{align*}
\left|\f{\cM_k(t,\eta)}{\cM_l(t,\xi)}-1\right|
&\leq \f{\left|\f{e^{4\pi\d_{\rmL}^{-1}|\eta|^{\f13}}}{\rmg(t,\eta)}-\f{e^{4\pi\d_{\rmL}^{-1}|\xi|^{\f13}}}{\rmg(t,\xi)}\right|}{\f{e^{4\pi\d_{\rmL}^{-1}|\xi|^{\f13}}}{\rmg(t,\xi)}+e^{4\pi \d_{\rmL}^{-1}|l|^{\f13}}}
+\f{\left|e^{4\pi\d_{\rmL}^{-1}|k|^{\f13}}-e^{4\pi\d_{\rmL}^{-1}|l|^{\f13}}\right|}{\f{e^{4\pi\d_{\rmL}^{-1}|\xi|^{\f13}}}{\rmg(t,\xi)}+e^{4\pi\d_{\rmL}^{-1} |l|^{\f13}}}\\
&\leq \f{\rmg(t,\xi)}{\rmg(t,\eta)}\left||\xi|^{\f13}-|\eta|^{\f13}\right|e^{4\pi\d_{\rmL}|\eta-\xi|^{\f13}}+\left|\f{\rmg(t,\xi)}{\rmg(t,\eta)}-1\right|
+\f{\left||k|^{\f13}-|l|^{\f13}\right|e^{4\pi\d_{\rmL}^{-1}|k-l|^{\f13}}}{e^{4\pi\d_{\rmL}^{-1}|\xi|^{\f13}-4\pi \d_{\rmL}^{-1}|l|^{\f13}}}\\
&\lesssim \f{|\xi-\eta|}{|\xi|^{\f23}+|\eta|^{\f23}}e^{C\d_{\rmL}^{-1}|\xi-\eta|^{\f13}}
+\f{\left||k|-|l|\right|e^{4\pi\d_{\rmL}^{-1}|k-l|^{\f13}}}{|\xi|}
\lesssim \f{|k-l, \xi-\eta|}{|\xi|^{\f23}+|\eta|^{\f23}}e^{C\d_{\rmL}^{-1}|\xi-\eta|^{\f13}}. 
\end{align*}
Case 3: $|\eta|+|\xi|\leq \f{1}{10}(|k|+|l|)$: We have 
\begin{align*}
|k-l|+|\xi-\eta|\leq \f{11}{1000}(|k|+|l|),\quad \f{1011}{989}|l|\geq |k|\geq \f{989}{1011}|l|,
\end{align*}
thus $2|\xi|\leq |\xi-\eta|+|\eta|+|\xi|\leq \f{111}{1000}(|k|+|l|)\leq \f{222}{989}|l|$, which implies $|l|\geq 8|\xi|$. Thus
\beno
\f{e^{4\pi\d_{\rmL}^{-1}|\xi|^{\f13}}}{\rmg(t,\xi)}
\leq Ce^{7\pi\d_{\rmL}^{-1}|\xi|^{\f13}}\leq Ce^{3.5\pi\d_{\rmL}^{-1}|l|^{\f13}},
\eeno
and then 
\begin{align*}
\left|\f{\cM_k(t,\eta)}{\cM_l(t,\xi)}-1\right|
&\leq \left(\f{\rmg(t,\xi)}{\rmg(t,\eta)}\left||\xi|^{\f13}-|\eta|^{\f13}\right|e^{4\pi\d_{\rmL}|\eta-\xi|^{\f13}}+\left|\f{\rmg(t,\xi)}{\rmg(t,\eta)}-1\right|\right)e^{-0.5\pi\d_{\rmL}^{-1}|l|^{\f13}}\\
&\quad+\left||k|^{\f13}-|l|^{\f13}\right|e^{4\pi\d_{\rmL}^{-1}|k-l|^{\f13}}\\
&\lesssim \f{|\xi-\eta|}{|\xi|^{\f23}+|\eta|^{\f23}}e^{C\d_{\rmL}^{-1}|\xi-\eta|^{\f13}}e^{-0.5\pi\d_{\rmL}^{-1}|l|^{\f13}}
+\f{\left||k|-|l|\right|e^{4\pi\d_{\rmL}^{-1}|k-l|^{\f13}}}{|l|^{\f23}+|k|^{\f23}}\\
&\lesssim \f{|k-l, \xi-\eta|}{|l|^{\f23}+|k|^{\f23}}e^{C\d_{\rmL}^{-1}|\xi-\eta|^{\f13}}. 
\end{align*}
Thus we proved the lemma. 
\end{proof}
\begin{lemma}\label{Lemma_Comm_M_0}
Let $t\leq \f12\min\{|\xi|^{\f23},|\eta|^{\f23}\}$. Then
\beno
\left|\f{\cM_k(t,\eta)}{\cM_k(t,\xi)}-1\right|\lesssim \f{\langle\xi-\eta\rangle }{(|k|+|\eta|+|\xi|)^{\f13}}e^{C\d_{\rmL}^{-1}|\xi-\eta|^{\f13}}.
\eeno
\end{lemma}
\begin{proof}
By the argument in the proof of Lemma \ref{lem: g/g-1-1} (see Case 3) and Lemma \ref{lem:g/g}, we have for $|\eta|+|\xi|\leq \f{1}{5}|k|$, 
\begin{align*}
\left|\f{\cM_k(t,\eta)}{\cM_k(t,\xi)}-1\right|
&\leq \left(\f{\rmg(t,\xi)}{\rmg(t,\eta)}\left||\xi|^{\f13}-|\eta|^{\f13}\right|e^{4\pi\d_{\rmL}|\eta-\xi|^{\f13}}+\left|\f{\rmg(t,\xi)}{\rmg(t,\eta)}-1\right|\right)e^{-0.5\pi\d_{\rmL}^{-1}|k|^{\f13}}\\
&\lesssim \f{1}{|k|^{\f13}}e^{C\d_{\rmL}^{-1}|\xi-\eta|^{\f13}}\lesssim \f{\langle\xi-\eta\rangle }{(|k|+|\eta|+|\xi|)^{\f13}}e^{C\d_{\rmL}^{-1}|\xi-\eta|^{\f13}}. 
\end{align*}
Now we focus on the case $|\eta|+|\xi|>\f{1}{5}|k|$. 
We have from Lemma \ref{lem:g/g} that 
\begin{equation}
\Big|\f{\cM_k(t,\eta)}{\cM_k(t,\xi)}-1\Big|\lesssim 1+e^{C\d_{\rmL}^{-1}|\xi-\eta|^{\f13}}
\end{equation}
So, if $(|\xi|+|\eta|)^{\f13}\lesssim \max\{1,|\eta-\xi|\}$, then we obtain the lemma. So from now on, we assume that $|\eta|^{\f13}+|\xi|^{\f13}\gtrsim 1$ and 
\begin{equation}\label{xi_eta_ineq}
|\xi-\eta|\leq \frac{1}{100} (|\eta|^{\f13}+|\xi|^{\f 13})\leq \frac{1}{100}(|\eta|+|\xi|). 
\end{equation}
  
We have 
\begin{align}\label{M_0_main_Estimate}
\left|\f{\cM_k(t,\eta)}{\cM_k(t,\xi)}-1\right|
&\leq \Big|\f{\rmg(t,\xi)}{\rmg(t,\eta)}\Big[e^{4\pi\d_{\rmL}^{-1}(|\xi|^{1/3}-|\eta|^{\f13})}-1\Big]\Big|+\left|\f{\rmg(t,\xi)}{\rmg(t,\eta)}-1\right|. 
\end{align}
Using Lemma \ref{lem:g/g}, we estimate the first term as 
\begin{equation}
\Big|\f{\rmg(t,\xi)}{\rmg(t,\eta)}\Big[e^{4\pi\d_{\rmL}^{-1}(|\xi|^{1/3}-|\eta|^{\f13})}-1\Big]\Big|\lesssim \f{|\eta-\xi|}{|\eta|^{\f23}+|\xi|^{\f23}}e^{C\d_{\rmL}^{-1}|\eta-\xi|^{\f13}}. 
\end{equation}

Now, we estimate the second term in \eqref{M_0_main_Estimate}. 
We first consider the time regime $t\leq \min (t_{\rmE(|\xi|^{\f23}),|\xi|}, t_{\rmE(|\eta|^{\f23}),|\eta|}). $ In this case, we have (see  \eqref{Main_term_1})
\begin{equation}\label{Estim_Case_2_1}
\Big|\f{\rmg(t,\xi)}{\rmg(t,\eta)}-1\Big|
=\Big|\f{\rmg(0,\xi)}{\rmg(0,\eta)}-1\Big|\lesssim \f{|\xi-\eta|}{|\xi|^{\f23}+|\eta|^{\f23}}e^{C\d_{\rmL}^{-1}|\xi-\eta|^{\f13}}.  
\end{equation}

Now, we consider the time regime: $\min (t_{\rmE(|\xi|^{\f23}),|\xi|}, t_{\rmE(|\eta|^{\f23}),|\eta|})\leq t\leq \max  (t_{\rmE(|\xi|^{\f23}),|\xi|}, t_{\rmE(|\eta|^{\f23}),|\eta|})$
 which corresponds to Case 3 in the proof of Lemma \ref{lem:g/g}. Let $t\in \bar{\rmI}_{n,\eta}\cap \bar{\rmI}_{j,\xi}$
 
 If $t_{\rmE(|\xi|^{\f23}),\xi}\leq t\leq t_{\rmE(|\eta|^{\f23}),\eta}$, we have
\begin{equation}\label{Eq_Compa_Case_3_1}
\f{\rmg(t_{\rmE(|\eta|^{\f23}),\eta},\xi)}{\rmg(t_{\rmE(|\eta|^{\f23}),\eta},\eta)}\geq \f{\rmg(t,\xi)}{\rmg(t,\eta)}\geq \f{\rmg(0,\xi)}{\rmg(0,\eta)},
\end{equation}
This gives that if $\f{\rmg(t,\xi)}{\rmg(t,\eta)}\leq 1$, then we have
\begin{equation}
\Big|\f{\rmg(t,\xi)}{\rmg(t,\eta)}-1\Big|\leq\Big|\f{\rmg(0,\xi)}{\rmg(0,\eta)}-1\Big|\lesssim \f{|\xi-\eta|}{|\xi|^{\f23}+|\eta|^{\f23}}e^{C\d_{\rmL}^{-1}|\xi-\eta|^{\f13}}
\end{equation}
and if 
$\f{\rmg(t,\xi)}{\rmg(t,\eta)}\geq 1$, then we have
\begin{equation}
\Big|\f{\rmg(t,\xi)}{\rmg(t,\eta)}-1\Big|\leq\Big|\f{\rmg(t_{\rmE(|\eta|^{\f23}),\eta},\xi)}{\rmg(t_{\rmE(|\eta|^{\f23}),\eta},\eta)}-1\Big|.
\end{equation}
Similarly, we have for $t_{\rmE(|\eta|^{\f23}),\eta}\leq t\leq t_{\rmE(|\xi|^{\f23}),\xi}$,
\beno
\Big|\f{\rmg(t,\xi)}{\rmg(t,\eta)}-1\Big|\leq \Big|\f{\rmg(t_{\rmE(|\xi|^{\f23}),\xi},\xi)}{\rmg(t_{\rmE(|\xi|^{\f23}),\xi},\eta)}-1\Big|+\f{|\xi-\eta|}{|\xi|^{\f23}+|\eta|^{\f23}}e^{C\d_{\rmL}^{-1}|\xi-\eta|^{\f13}}.
\eeno
Thus we deduce the problem to consider the time regime  $\max\left\{t_{\rmE(|\eta|^{\f23}),\eta},t_{\rmE(|\xi|^{\f23}),\xi}\right\}\leq t\leq \f12\min\{|\xi|^{\f23},|\eta|^{\f23}\}$, 
which corresponds to Case 4 in proof of Lemma \ref{lem:g/g}. 

Without loss of generality, let us assume $|\eta|\geq |\xi|$ and prove 
\ben
\left|\f{\rmg(t,\xi)}{\rmg(t,\eta)}-1\right|+\left|\f{\rmg(t,\eta)}{\rmg(t,\xi)}-1\right|\lesssim \f{\langle\eta-\xi\rangle}{|\eta|^{\f13}+|\xi|^{\f13}}e^{C\d_{\rmL}^{-1}|\eta-\xi|^{\f13}}
\een

So, let $n$ and $j$ be such that $t\in \bar{\rmI}_{n,\eta}\cap \bar{\rmI}_{j,\xi}$. 
Then we have $|n|\geq |j|$,

{\bf Claim}: It holds that $\f{2|\eta|}{2|n|-1}< \f{1.5|\xi|}{1.5n-1}<\f{|\xi|}{n-1}$ which implies $||n|-|j||\leq 1$. 

Indeed if not, then $\f{2|\eta|}{2|n|-1}\geq \f{1.5|\xi|}{1.5|n|-1}$ which implies $||\eta|-|\xi||\geq \f{2|\eta|-1.5|\xi|}{3|n|}$. It holds from \eqref{xi_eta_ineq} that 
\beno
\f{99}{101}|\xi|\leq |\eta|\leq \f{101}{99}|\xi|,\quad |\eta-\xi|\leq \f{3}{100}|\xi|^{\f13}.
\eeno
Thus we get that 
\beno
\f{3}{100}|\xi|^{\f13}\geq |\eta-\xi|\geq \f{2|\eta|-1.5|\xi|}{3|n|}\geq \f{|\xi|}{7\rmE(|\eta|^{\f23})}\geq \f{|\xi|}{12|\xi|^{\f23}}=\f{1}{12}|\xi|^{\f13},
\eeno
which leads a contradiction. 

Therefore, it holds that $n\geq \rmE(|\eta|^{\f13})$, $j\geq \rmE(|\xi|^{\f13})$ and $\rmE(|\eta|^{\f13})\geq \rmE(|\xi|^{\f13})$. 
For $j=n$ (hence in this case $|j|\geq \rmE(|\eta|^{\f13})+1$), thus  recalling \eqref{g_eta_case_4} and \eqref{g_xi_Case_4_2} and 
using the inequality $|e^ae^b-1|\lesssim (|a|+|b|)e^{a+b}$, we get  
\begin{align*}
&\Big|\f{\rmg(t,\xi)}{\rmg(t,\eta)}-1\Big|+\Big|\f{\rmg(t,\eta)}{\rmg(t,\xi)}-1\Big|\\
&\lesssim \bigg(\sum_{k=1}^{\rmE(|\xi|^{\f13})}\d_{\rmL}^{-1}|\rmF(k,\eta)-\rmF(k,\xi)| 
 + \sum_{k=\rmE(|\xi|^{\f13})+1}^{\rmE(|\eta|^{\f13})}\d_{\rmL}^{-1}|\tilde{\rmF}(k,\eta)-\tilde{\rmF}(k,\xi)|\bigg.
\\
&\bigg.+ \sum_{k=\rmE(|\eta|^{\f13})+1}^{j-1}\d_{\rmL}^{-1}|\tilde{\rmF}(k,\eta)-\tilde{\rmF}(k,\xi)|\bigg.
+\d_{\rmL}^{-1}\Big|\f{|\xi|}{|j|^3}\arctan\big(t-\f{|\xi|}{|j|}\big)-\f{|\eta|}{|j|^3}\arctan\big(t-\f{|\eta|}{|j|}\big)\Big|\bigg.\\
&+\bigg.\d_{\rmL}^{-1}\Big|\f{|\eta|}{|j|^3}\arctan\big(\f{|\eta|}{|j|(2|j|-1)}\big)-\f{|\xi|}{|j|^3}\arctan\big(\f{|\xi|}{|j|(2|j|-1)}\big)\Big|
\bigg)\max\left\{\f{\rmg(t,\xi)}{\rmg(t,\eta)},\f{\rmg(t,\eta)}{\rmg(t,\xi)}\right\} 
\end{align*}

First applying \eqref{Ineq_F_Mean_Value}, we have 
\begin{equation}\label{First_Term_1}
\sum_{k=1}^{\rmE(|\xi|^{\f13})}\d_{\rmL}^{-1}|\rmF(k,\eta)-\rmF(k,\xi)|\lesssim |\xi|^{\f13}\f{|\xi-\eta|}{k^2+|\xi|^2}\lesssim \f{\langle \xi-\eta\rangle}{|\xi|^{
\f53}+|\eta|^{\f53}}. 
\end{equation}  
Similarly, using \eqref{F_tilde_Mean_Value}, we have 
\begin{equation}\label{Second_Term_1}
\begin{aligned}
\sum_{k=\rmE(|\xi|^{\f13})+1}^{\rmE(|\eta|^{\f13})}\d_{\rmL}^{-1}|\tilde{\rmF}(k,\eta)-\tilde{\rmF}(k,\xi)|\lesssim |\rmE(|\eta|^{\f13})-\rmE(|\xi|^{\f13})|\f{|\eta-\xi|}{|\xi|}\lesssim \f{\langle \xi-\eta\rangle}{|\xi|^{
\f23}+|\eta|^{\f23}} 
\end{aligned}
\end{equation}
and 
\begin{equation}\label{Third_Term_1}
 \sum_{k=\rmE(|\eta|^{\f13})+1}^{j-1}\d_{\rmL}^{-1}|\tilde{\rmF}(k,\eta)-\tilde{\rmF}(k,\xi)|\lesssim \f{|\eta-\xi|}{|\eta|^{\f23}}\lesssim \f{\langle \xi-\eta\rangle}{|\xi|^{
\f23}+|\eta|^{\f23}}.  
 \end{equation}
Also, we have 
\begin{equation}\label{Fourth_Term_1}
\begin{aligned}
&\Big|\f{|\xi|}{|j|^3}\arctan\big(t-\f{|\xi|}{|j|}\big)-\f{|\eta|}{|j|^3}\arctan\big(t-\f{|\eta|}{|j|}\big)\Big|\\
\lesssim&\, \f{|\xi|}{|j|^3} \Big|\arctan\big(t-\f{|\xi|}{|j|}\big)-\arctan\big(t-\f{|\eta|}{|j|}\big)\Big|
+\Big|\arctan\big(t-\f{|\eta|}{|j|}\big)\Big| \Big|\f{|\xi|}{|j|^3}-\f{|\eta|}{|j|^3}\Big|\\
\lesssim&\,\f{|\xi|}{|j|^3} \f{|\eta-\xi|}{|j|}+\f{|\eta-\xi|}{|j|^3}
\lesssim \,\f{\langle \xi-\eta\rangle}{|\xi|^{
\f13}+|\eta|^{\f13}}. 
\end{aligned}
\end{equation}
Similarly, for the last term, we have by applying \eqref{Mean_Value_1} 
\begin{equation}\label{Fifth_Term_1}
\begin{aligned}
&\Big|\f{|\eta|}{|j|^3}\arctan\big(\f{|\eta|}{|j|(2|j|-1)}\big)-\f{|\xi|}{|j|^3}\arctan\big(\f{|\xi|}{|j|(2|j|-1)}\big)\Big|\\
&\lesssim \f{|\eta|}{|j|^3} \Big|\arctan\big(\f{|\eta|}{|j|(2|j|-1)}\big)-\arctan\big(\f{|\xi|}{|j|(2|j|-1)}\big)\Big|\\
&+\Big|\arctan\big(\f{|\xi|}{|j|(2|j|-1)}\big)\Big|\Big|\f{|\xi|}{|j|^3}-\f{|\eta|}{|j|^3}\Big| \\
&\lesssim \f{|\eta|}{|j|^3}\f{|\eta-\xi|}{\langle \xi\rangle}+\f{|\eta-\xi|}{|j|^3}\lesssim \f{\langle \xi-\eta\rangle}{|\xi|+|\eta|}.   
\end{aligned}
\end{equation}
Collecting \eqref{First_Term_1}- \eqref{Fifth_Term_1}  together with Lemma \ref{lem:g/g} yield Lemma \ref{Lemma_Comm_M_0} in the above time regime with $j=n$. 

Now, we discuss the case $n\neq j$. We then have
\begin{equation}\label{Time_Regime} 
\f{|\eta|}{|n|}\leq \f{2|\xi|}{2|n|-1}\leq t\leq \f{2|\eta|}{2|n|-1}<\f{1.5|\xi|}{1.5n-1}<\f{|\xi|}{n-1}
\end{equation}
 It is clear that since $t\approx \f{|\eta|}{|n|}\leq \f12\min\{|\xi|^{\f23},|\eta|^{\f23}\}$, then it holds that $j=n-1\geq 2|\eta|^{\f23}-1\geq |\eta|^{\f23}+1$. Hence, using \eqref{g_eta_case_4} and \eqref{g_xi_Case_4_2}, we have  

\begin{align}\label{Estimate_g/g-1_2}
&\Big|\f{\rmg(t,\xi)}{\rmg(t,\eta)}-1\Big|+\Big|\f{\rmg(t,\eta)}{\rmg(t,\xi)}-1\Big|\\
\nonumber&\lesssim \d_{\rmL}^{-1}\bigg(\sum_{k=1}^{\rmE(|\xi|^{\f13})}|\rmF(k,\eta)-\rmF(k,\xi)| 
 + \sum_{k=\rmE(|\xi|^{\f13})+1}^{\rmE(|\eta|^{\f13})}\rmF(k,\eta)+\sum_{k=\rmE(|\xi|^{\f13})+1}^{\rmE(|\eta|^{\f13})}\tilde{\rmF}(k,\xi)
\\
\nonumber&\quad+ \sum_{k=\rmE(|\eta|^{\f13})+1}^{n-1}|\tilde{\rmF}(k,\eta)-\tilde{\rmF}(k,\xi)|
+\f{|\eta|}{|n|^3}\Big|\arctan\big(t-\f{|\eta|}{|n|}\big)-\arctan\big(\f{|\eta|}{|n|(2|n|-1)}\big)\Big|\bigg.\\
\nonumber&\quad+\Big|\f{|\xi|}{|n-1|^3}\Big(\arctan\big(t-\f{|\xi|}{|n-1|}\big)-\arctan\big(\f{|\xi|}{|n-1|(2n-3)}\big)\Big)+\tilde{\rmF}(n-1,\xi)\Big|\bigg)\\
\nonumber&\qquad\times\max\left\{\f{\rmg(t,\xi)}{\rmg(t,\eta)},\f{\rmg(t,\eta)}{\rmg(t,\xi)}\right\}.
\end{align}
The first four terms can be estimated as in \eqref{First_Term_1},  \eqref{Second_Term_1} and \eqref{Third_Term_1}. 
 
 We also have in the time regime \eqref{Time_Regime}, 
 \begin{equation*}
 \begin{aligned}
&\Big|\arctan\big(t-\f{|\eta|}{|n|}\big)-\arctan\big(\f{|\eta|}{|n|(2|n|-1)}\big)\Big|\\
\lesssim &\,\Big|\arctan\big(\f{2|\xi|}{2|n|-1}-\f{|\eta|}{|n|}\big)-\arctan\big(\f{2|\eta|}{2|n|-1}-\f{|\eta|}{|n|}\big)\Big|\\
\lesssim &\,\f{|\xi-\eta|}{|n|}\lesssim\,\f{\langle \xi-\eta\rangle}{|\xi|^{
\f13}+|\eta|^{\f13}}. 
\end{aligned}
\end{equation*}
Now, using \eqref{F_tilde_Estimate}, we have 
\begin{equation}
\begin{aligned}
|\tilde{\rmF}(n-1,\eta)-\tilde{\rmF}(n-1,\xi)|\lesssim&\, \f{|\eta-\xi|}{(|n|-1)^3}
\lesssim\f{|\eta-\xi|}{|\xi|}\lesssim  \f{\langle \xi-\eta\rangle}{|\xi|^{
\f13}+|\eta|^{\f13}}.  
\end{aligned}
\end{equation}

Now, we need to estimate the last term in \eqref{Estimate_g/g-1_2}. 

Recalling   \eqref{F_tilde}, we have 
\begin{equation}
\begin{aligned}
&\Big|\f{|\xi|}{|n-1|^3}\Big(\arctan\big(t-\f{|\xi|}{|n-1|}\big)-\arctan\big(\f{|\xi|}{(|n|-1)(2n-3)}\big)\Big)+\tilde{\rmF}(n-1,\xi)\Big|\\
=&\,\Big|\f{|\xi|}{|n-1|^3}\Big(\arctan\big(t-\f{|\xi|}{|n-1|}\big)-\arctan\big(\f{|\xi|}{|n-1|(2n-3)}\big)\Big.\\
&\Big.+\arctan\big(\f{|\xi|}{|n-1|(2n-1)}\big)+\arctan\big(\f{|\xi|}{|n-1|(2n-3)}\big)\Big)\Big|\\
=&\,\Big|\f{|\xi|}{|n-1|^3}\Big(\arctan\big(t-\f{|\xi|}{|n-1|}\big)+\arctan\big(\f{|\xi|}{|n-1|(2n-1)}\big)\Big)\Big|. 
\end{aligned}
\end{equation}

We have by exploiting \eqref{Time_Regime}, 
\begin{equation}
\begin{aligned}
&\f{|\xi|}{|n-1|^3}\Big|\arctan\big(t-\f{|\xi|}{|n-1|}\big)+\arctan\big(\f{|\xi|}{|n-1|(2n-1)}\big)\Big|\\
=&\,\f{|\xi|}{|n-1|^3}\Big|\arctan\big(\f{|\xi|}{|n-1|}-\f{2|\xi|}{2n-1}\big)-\arctan\big(\f{|\xi|}{|n-1|}-t\big)\Big|\\
\lesssim&\f{|\xi|}{|n-1|^3}\Big|\arctan\big(\f{|\xi|}{|n-1|}-\f{2|\xi|}{2n-1}\big)-\arctan\big(\f{|\xi|}{|n-1|}-\f{2|\eta|}{2|n|-1}\big)\Big|\\
\lesssim& \f{|\xi|}{|n-1|^3} \f{|\eta-\xi|}{|n|}\lesssim \f{\langle \eta-\xi\rangle}{|\xi|^{1/3}+|\eta|^{1/3}}. 
\end{aligned}
\end{equation}

Putting the estimates together, we proved the lemma.
\end{proof}

\begin{lemma}\label{Lemma_Cmmutator_J}
Let $t\leq \f12\min\{|\xi|^{\f23},|\eta|^{\f23}\}$. Then
\beno
\left|\f{\rmJ_k(t,\eta)}{\rmJ_l(t,\xi)}-1\right|\lesssim \f{\langle k-l,\xi-\eta\rangle }{(|k|+|l|+|\eta|+|\xi|)^{\f23}}e^{C\mu|k-l,\xi-\eta|^{\f13}}.
\eeno
\end{lemma}
\begin{proof}
Keeping \eqref{Total_Multipl} in mind and due to our assumption on the range of $t$, then it holds that 
$\rmJ_k(t,\eta)=\rmJ_k(t_{\rmE(|\eta|^{\f13}),\eta},\eta)=\rmJ_k(0,\eta)$ and 
$\rmJ_k(t,\xi)=\rmJ_k(t_{\rmE(|\xi|^{\f13}),\xi},\xi)=\rmJ_k(0,\xi)$. First, we have from Lemma \ref{lem:3.5}:
\begin{equation}\label{J_Main_Estimate}
\begin{aligned}
\Big|\f{\rmJ_k(t,\eta)}{\rmJ_l(t,\xi)}\Big|\lesssim e^{\mu |\eta-\xi|^{\f13}}+e^{\mu|k-l|^{\f13}}\lesssim e^{C\mu |\eta-\xi,k-l|^{\f13}} . 
\end{aligned}
\end{equation}
Hence, we obtain 
 \begin{equation}
\begin{aligned}
\left|\f{\rmJ_k(t,\eta)}{\rmJ_l(t,\xi)}-1\right|\leq  \Big|\f{\rmJ_k(t,\eta)}{\rmJ_l(t,\xi)}\Big|+1. 
\end{aligned}
\end{equation}
Hence, the lemma holds for $(|k|+|l|+|\eta|+|\xi|)^{\f23}\lesssim |k-l|+|\xi-\eta|$.  Also, if $(|k|+|l|+|\eta|+|\xi|)^{\f23}\leq 1$, then Lemma \ref{Lemma_Cmmutator_J} holds since due to \eqref{J_Main_Estimate} by allowing the constant $C$ in the exponent to be large enough.  
So, from now on we restrict to the case  
$|k|^{\f23}+|l|^{\f23}+|\eta|^{\f23}+|\xi|^{\f23}\gtrsim 1$.

We assume now that 
\begin{equation}
|k|^{\f23}+|l|^{\f23}+|\eta|^{\f23}+|\xi|^{\f23}\geq  100(|k-l|+|\xi-\eta|)
\end{equation}
That is, we have in this case 
\begin{equation}\label{Ineq_int}
|k|+|l|+|\eta|+|\xi|\geq  100(|k-l|+|\xi-\eta|). 
\end{equation}

 We define the multiplier $\tilde{\rmJ}_k(t,\eta)$ as 
\begin{equation}
\tilde{\rmJ}_k(t,\eta)=\f{e^{\mu|\eta|^{\f13}}}{\Theta_k(t,\eta)}.
\end{equation}
Case 1. If $|k,l|\approx |\eta,\xi|$. That is if for instance: $\frac{1}{10} (|k|+|l|)\leq |\eta|+|\xi|\leq 10 (|k|+|l|)$. Then this together with \eqref{Ineq_int} implies that 
 \beno
|k-l,\xi-\eta|\lesssim |k|\approx |l|\approx |\xi|\approx |\eta|. 
\eeno

Hence, we have
\begin{equation}\label{J_Commu_1}
\left|\f{\rmJ_k(t,\eta)}{\rmJ_l(t,\xi)}-1\right|\leq  \Big|\f{\tilde{\rmJ}_k(t,\eta)-\tilde{\rmJ}_l(t,\xi)}{\rmJ_l(t,\xi)}\Big|+\Big|\frac{e^{\mu|k|^{\f13}}-e^{\mu|l|^{\f13}}}{\rmJ_l(t,\xi)}\Big|. 
\end{equation}
The second term on the right-hand side of \eqref{J_Commu_1} can be estimate as 
\begin{equation}
\begin{aligned}
\Big|\frac{e^{\mu|k|^{\f13}}-e^{\mu|l|^{\f13}}}{\rmJ_l(t,\xi)}\Big|\lesssim&\, |e^{\mu(|k|^{\f13}-|l|^{\f13})}-1|\\
\lesssim &\,\,\frac{|k-l|}{|k|^{2/3}+|k|^{1/3}|l|^{1/3}+|l|^{2/3}}e^{\mu(|k-l|^{1/3}}\\
\lesssim&\,\f{\langle k-l,\eta-\xi\rangle }{(|k|+|l|+|\eta|+|\xi|)^{\f23}}e^{\mu|k-l,\eta-\xi|^{\f13}}. 
\end{aligned}
\end{equation}

 For the first term, we have 

 \begin{align}\label{J_Comm_Estimate}
\Big|\f{\tilde{\rmJ}_k(t,\eta)-\tilde{\rmJ}_l(t,\xi)}{\rmJ_l(t,\xi)}\Big|\leq&\, \Big|\f{\tilde{\rmJ}_k(t,\eta)-\tilde{\rmJ}_l(t,\xi)}{\tilde{\rmJ}_l(t,\xi)}\Big|\\
\nonumber\lesssim&\, \f{\Theta_l(0,\xi)} {\Theta_k(0,\eta)} |e^{\mu(|\eta|^{1/3}-|\xi|^{1/3})}-1|+\Big|\f{\Theta_l(0,\xi)} {\Theta_k(0,\eta)}-1\Big|. 
\end{align}

Then, we control the first term as, by using the mean value theorem  
\begin{equation}
\begin{aligned}
|e^{\mu(|\eta|^{1/3}-|\xi|^{1/3})}-1|\leq &\,\mu ||\eta|^{1/3}-|\xi|^{1/3}|e^{\mu(|\eta|^{1/3}-|\xi|^{1/3}}\\ 
\lesssim&\,\frac{|\eta-\xi|}{|\eta|^{2/3}+|\eta|^{1/3}|\xi|^{1/3}+|\xi|^{2/3}}e^{\mu(|\eta|^{1/3}-|\xi|^{1/3}}\\
\lesssim&\,\f{\langle k-l,\eta-\xi\rangle }{(|k|+|l|+|\eta|+|\xi|)^{\f23}}e^{\mu|k-l,\eta-\xi|^{\f13}},
\end{aligned}
\end{equation}
 which together with Lemma \ref{lem:3.5} implies the control of the first term in \eqref{J_Comm_Estimate}. 
 
 Our goal now is to control the second term on the right-hand side of \eqref{J_Comm_Estimate}. Due to \eqref{Ineq_int} and the fact that $ |k|+|l|\approx |\eta|+|\xi|$, it holds that $|\rmE(\eta)-\rmE(\xi)|\leq 1$.
 
 Keeping in mind \eqref{Growth_formula} and if $\rm E(|\eta|^{1/3})=\rm E(|\xi|^{1/3})$, then we have  
 \begin{equation}
\Big|\f{\Theta_l(0,\xi)} {\Theta_k(0,\eta)}-1\Big|=\Big|\Big(\f{|\eta|}{|\xi|}\Big)^{c\rmE(|\eta|^{1/3})}-1\Big|
\end{equation}
with $2C\kappa+1$. This implies 
\begin{equation}
\Big|\Big(\f{|\eta|}{|\xi|}\Big)^{c\rmE(|\eta|^{1/3})}-1\Big|\leq \Big|\Big(1+\f{|\eta-\xi|}{|\xi|}\Big)^{c\rmE(|\eta|^{1/3})}-1\Big|. 
\end{equation}
Using the inequality 
\begin{equation*}
\Big(1+\f{a}{b^3}\Big)^{b}-1\leq C\frac{|a|}{b^2},\qquad |a|<1,\quad b>1,
\end{equation*}
we obtain 
\begin{equation}
\Big|\Big(1+\f{|\eta-\xi|}{|\xi|}\Big)^{c\rmE(|\eta|^{1/3})}-1\Big|\lesssim \frac{|\eta-\xi|}{|\xi|^{\f23}}.   
\end{equation}

Next, if $\rmE(|\eta|^{1/3})=\rmE(|\xi|^{1/3})+1 $, we have $|\xi|^{1/3}\leq \rmE(|\eta|^{1/3})\leq |\eta|^{1/3}$ 
\begin{equation}
\begin{aligned}
\Big|\f{\Theta_l(0,\xi)} {\Theta_k(0,\eta)}-1\Big|=&\Big|\Big(\f{|\eta|}{|\xi|}\Big)^{c\rmE(|\xi|^{1/3})}\Big(\frac{|\eta|}{(\rmE(|\eta|^{1/3}))^3}\Big)^c-1\Big|\\
\lesssim &\frac{|\eta-\xi|}{|\xi|^{\f23}}+ \Big|\Big(\frac{|\eta|}{(\rmE(|\eta|^{1/3}))^3}\Big)^c-1\Big|
\end{aligned}
\end{equation}
Since, 
\begin{equation}
\Big|\Big(\frac{|\eta|}{(\rmE(|\eta|^{1/3}))^3}\Big)^c-1\Big|\leq \Big(\frac{|\eta|}{|\xi|}\Big)^c-1\Big|\lesssim \frac{\langle \eta-\xi\rangle}{|\xi|}. 
\end{equation}
We omit the case $\rmE(|\eta|^{1/3})=\rmE(|\xi|^{1/3})-1 $ since it can be treated similarly. 

The other two cases $|\xi|+|\eta|\geq 10(|k|+|\eta|)$ and $|k|+|l|\geq 10 (|\xi|+|\eta|)$ can be treated by modifying slightly the above argument and using the fact that the first condition together with \eqref{Ineq_int} implies that 
$|\xi|\geq \frac{989}{200}(|l|+|k|)$ and the second condition together with \eqref{Ineq_int} implies that 
$|\xi|\geq \frac{989}{200}(|l|+|k|)$ and the second condition together with \eqref{Ineq_int} implies $|l|\geq \frac{989}{200}(|\xi|+|\eta|)$

\end{proof}

\section{Zero mode and coordinate system}\label{Zero_mod_Section}
In this section, we deal with the coordinate system and prove Proposition \ref{prop: g,f_0,h}. 

\subsection{Assistant estimate}
In this section, we prove the estimate of $\psi_0$. Recall \eqref{eq:psi_0}, by the Duhamel’s principle, we have
\begin{equation}\label{Duhamel_Formula}
<\psi >=e^{t\pa_{yy}}<\psi_{\mathrm{in}} >-\int_0^t e^{(t-s)\pa_{yy}}<\pa_x\psi_{\neq} u^x_{\neq}>(s)ds. 
\end{equation}
Under the bootstrap hypotheses, by the fact that for any $f_1(y)=f_2(v)$ we have for any $k\geq 0$ and $1\leq p\leq \infty$, $\|\langle\pa_v\rangle^{k+2}h\|_{L^{\infty}}\lesssim \ep$ and then 
\ben\label{eq:Sobolev-equal}
\|f_1\|_{W^{k,p}_y}\eqdef \sum_{i=0}^{k}\left(\int_{\R}|\pa_y^if_1(y)|^pdy\right)^{\f1p} \approx \|f_2\|_{W^{k,p}_v}\eqdef \sum_{i=0}^{k}\left(\int_{\R}|\pa_v^if_2(v)|^pdv\right)^{\f1p}
\een
and for any $\g\in (0,1)$
\ben\label{eq:Sobolev-equal-2}
\begin{aligned}
\|f_1\|_{\dot{H}^{\g}_y}
&\eqdef \left(\int_{\R^2}\f{|f_1(y_1)-f_1(y_2)|^2}{|y_1-y_2|^{1+2\g}}dy_1dy_2\right)^{\f12}\\
&\approx \left(\int_{\R^2}\f{|f_2(v_1)-f_2(v_2)|^2}{|v_1-v_2|^{1+2\g}}dv_1dv_2\right)^{\f12}\eqdef \|f_2\|_{\dot{H}^{\g}_v}. 
\end{aligned}
\een
which together with the elliptic estimate, we have for any $k\geq 0$
\begin{equation}\label{Decay_avr_psi}
\|<\pa_z\psi_{\neq} u^x_{\neq}>\|_{W^{k,p}_y}\lesssim 
\|<\pa_z\phi_{\neq} \tilde{u}^x_{\neq}>\|_{W^{k,p}_v}
\lesssim \|\phi_{\neq}\|_{H^{k+3}}\|\tilde{u}^x_{\neq}\|_{H^{k+3}}\lesssim \f{\ep^2}{1+t^4}.
\end{equation}
Thus we get that for $0\leq j\leq 7$
\begin{align*}
&\|\pa_{y}^j<\psi>\|_{L^{p}}\\
&\lesssim \f{1}{\langle t\rangle^{\f12(1-\f1p)+\f{j}{2}}}\Big(\|<\psi_{in}>\|_{L^1}+\|<\psi_{in}>\|_{W^{j,\infty}}\Big)
+\int_0^t\f{1}{\langle t-s\rangle^{\f12(1-\f1p)+\f{j}{2}}}\f{\ep^2}{1+s^4}ds\\
&\lesssim \f{\ep}{\langle t\rangle^{\f12(1-\f1p)+\f{j}{2}}}.
\end{align*}
By choosing different $j$ and $p$, we get $\rmE_{as,\psi_0}(t)\lesssim \ep^2$.

By using the fact that $\om_0=\pa_{yy}<\psi>$ and $<u^x>=-\pa_y<\psi>$, we have
\begin{align*}
&\|\om_0\|_{L^2}\lesssim \ep\langle t\rangle^{-\f54}, \quad \|\pa_y\om_0\|_{L^2}\lesssim \ep\langle t\rangle^{-\f74},\quad 
\|\pa_{yy}\om_0\|_{L^2}\lesssim \ep\langle t\rangle^{-\f94},\quad
\|\pa_{yy}<u^x>\|_{L^{\infty}}\lesssim \f{\ep}{1+t^2},
\end{align*}
which together with \eqref{eq:Sobolev-equal}, \eqref{Decay_avr_psi} and \eqref{eq:u_1_Equation} gives us that
\begin{align*}
&\|f_0\|_{L^2}\lesssim \ep\langle t\rangle^{-\f54},\quad 
\|\pa_vf_0\|_{L^2}\lesssim \ep\langle t\rangle^{-\f74},\quad
\|\pa_tu_0^x\|_{L^{\infty}}\lesssim \f{\ep}{1+t^2},\\
&\|\pa_tu_0^x\|_{L^2_y}\lesssim \ep\langle t\rangle^{-\f74},\quad 
\|\pa_tu^x_0\|_{\dot{H}^{\f12-\ep_1}}\lesssim \ep\langle t\rangle^{-2+\f{\ep_1}{2}},\quad
\|\pa_tu^x_0\|_{\dot{H}^{\f12+\ep_2}}\lesssim \ep\langle t\rangle^{-2-\f{\ep_2}{2}}.
\end{align*}
Moreover recall that $K_0(t,v)=\pa_{yy}\om_0(t,y)$, we get that
\ben\label{eq:K_0L^2}
\|K_0\|_{L^2}\lesssim \ep\langle t\rangle^{-\f94}.
\een
Recall that $g(t,v)=\pa_tv(t,y)$ and $h(t,v)=\pa_yv(t,y)-1$ and that $\pa_tv(t,y)=\f{1}{t^2}\int_0^ts\pa_tu_0^x(s,y)ds$ and $\pa_yv(t,y)-1=-\f{1}{t}\int_0^t\om_0(s,y)ds$, we get that
\begin{align*}
&\|\pa_tv(t,y)\|_{L^{\infty}_y}\lesssim \f{1}{\langle t\rangle^2}\int_0^t s\|\pa_tu_0^x\|_{L^{\infty}_y}ds
\lesssim \ep(\ln\langle t\rangle+1)\langle t\rangle^{-2},\\
&\|\pa_tv(t,y)\|_{L^2}\lesssim \f{1}{\langle t\rangle^2}\int_0^t s\|\pa_tu_0^x\|_{L^2_y}ds\lesssim \ep\langle t\rangle^{-\f74},\\
&\|\pa_tv(t,y)\|_{\dot{H}^{\f12-\ep_1}}\lesssim \f{\ep}{\ep_1}\langle t\rangle^{-2+\f{\ep_1}{2}},\quad
\|\pa_tv(t,y)\|_{\dot{H}^{\f12+\ep_2}}\lesssim \f{1}{\langle t\rangle^2}\int_0^t\f{\ep}{\langle s\rangle^{1+\f{\ep_2}{2}}}ds\lesssim \f{\ep}{\ep_2}\langle t\rangle^{-2}
\end{align*}
and
\beno
\|\pa_yv(t,y)-1\|_{L^2_y}\lesssim \f{1}{\langle t\rangle}\int_0^t\|\om_0\|_{L^2}ds\lesssim \f{\ep}{\langle t\rangle},
\eeno

Thus by \eqref{eq:Sobolev-equal} and \eqref{eq:Sobolev-equal-2}, we get that 
\ben\label{eq: h-L^2}
\|h\|_{L^2}\lesssim \f{\ep}{\langle t\rangle},
\een 
and $\rmE_{as,g}(t)\lesssim \ep^2$.

\subsection{Estimate of $g(t,y)$}

\begin{proof}
Recall \eqref{eq: bar h}. Under the bootstrap hypotheses we have 
\beno
\|\rmA K_0\|_2\leq 5\ep,\quad 
\int_1^t\|\pa_v\rmA K_0\|_2^2ds\leq 8\ep^2
\eeno
By the fact that $K_0=\pa_{vv}f_0+h(h+2)\pa_{vv}f_0+\pa_vh(1+h)\pa_vf_0$, we have
\begin{align*}
\|\rmA\pa_{vv}f_0\|_2
&\lesssim \|\rmA K_0\|_2+\|\rmA h\|_2(2+\|\rmA h\|_2)\|\rmA\pa_{vv}f_0\|_2+\|\rmA\pa_vh\|_2(1+\|\rmA h\|_2)\|\rmA\pa_vf_0\|_2\\
&\lesssim \|\rmA K_0\|_2+\ep\|\rmA\pa_{vv}f_0\|_2
+\ep(\|\pa_vf_0\|_{L^2}+\|\rmA\pa_{vv}f_0\|_2),
\end{align*}
which gives us that
\ben\label{eq: f_0-1}
\|\rmA\pa_{vv}f_0\|_2\lesssim \|\rmA K_0\|_2+\ep^2\langle t\rangle^{-\f74}\lesssim \ep.
\een
By the fact that 
\ben\label{eq: K_0-1}
\|\rmA K_0\|_2\leq\|K_0\|_{L^2}+\|\pa_v\rmA K_0\|_{L^2}, 
\een
we have by \eqref{eq:K_0L^2},
\beno
\int_1^t\|\rmA\pa_{vv}f_0\|_2^2ds\lesssim \ep^2.
\eeno
Thus we get that
\begin{align*}
t\|\pa_{vv}\rmA\bar{h}\|_{L^2}\lesssim \|\rmA\pa_{vv}f_0\|_2+\|\rmA\pa_{vv}h\|_{L^2}\lesssim \ep.
\end{align*}
and
\begin{align*}
\int_1^ts\|\pa_{vv}\rmA\bar{h}(s)\|_{L^2}^2ds\lesssim \int_1^t\|\rmA\pa_{vv}f_0(s)\|_2^2ds+\int_1^t\|\rmA\pa_{vv}h(s)\|_{L^2}^2ds\lesssim \ep^2.
\end{align*}
Note that 
\begin{align*}
\|\rmA\pa_{vvv} g\|_2
&\leq \|\rmA\pa_{vv}\bar{h}\|_2
+\|\rmA\pa_{vv}(h\pa_v g)\|_2\\
&\lesssim \|\rmA\pa_{vv}\bar{h}\|_2
+\|\rmA h\|_{2}\|\rmA\pa_{vvv} g)\|_2+\|\rmA\pa_{vv}h\|_2\|\rmA\pa_vg\|_2\\
&\lesssim \|\rmA\pa_{vv}\bar{h}\|_2
+\|\rmA h\|_{2}\|\rmA\pa_{vvv} g)\|_2+\|\rmA\pa_{vv}h\|_2(\|\rmA\pa_{vvv}g\|_2+\|g\|_{L^2}),
\end{align*}
which together with the bootstrap assumption implies that
\beno
\|\rmA\pa_{vvv} g\|_2\lesssim \|\rmA\pa_{vv}\bar{h}\|_2+\ep^2\langle t\rangle^{-\f74},
\eeno
which implies the first inequality in Proposition \ref{prop: g,f_0,h}. 

Moreover by the fact that $\|f_0\|_{L^2}\lesssim \ep\langle t\rangle^{-\f54}$ and $\|h\|_{L^2}\lesssim \f{\ep}{\langle t\rangle}$, we have $\|\bar{h}\|_{L^2}\lesssim \f{\ep}{\langle t\rangle}$, thus under the bootstrap hypotheses, we get that 
\ben\label{eq:g-H1}
\|\pa_vg\|_{L^2}\lesssim \f{\ep}{\langle t\rangle^2}.
\een 
\end{proof}

\subsection{Low Gevrey norm estimate}
In this section we prove the rest parts of Proposition \ref{prop: g,f_0,h}. It is natural to compute the time evolution of $\rmE_{lo,f_0}$ and $\rmE_{lo,h}$. Recall that $f_0$ satisfies \eqref{Equ_Av_f} and $h$ satisfies \eqref{eq: h-equ}. We then have
\begin{align*}
&\f{1}{2}\f{d}{dt}\Big(\sum_{k=0}^{3}\f{t^{k}}{4^k}\|\pa_v^kf_0\|_{\cG^{\la,\b;s}}^2\Big)-\dt{\la}(t)\sum_{k=0}^{3}\f{t^{k}}{4^k}\||\pa_v|^{\f{s}{2}}\pa_v^kf_0\|_{\cG^{\la,\b;s}}^2+\rmD_{lo,f_0}(t)\\
&=\sum_{k=1}^{3}\f{k}{2}\f{t^{k-1}}{4^k}\|\pa_v^kf_0\|_{\cG^{\la,\b;s}}^2
+\sum_{k=0}^{3}\f{t^{k}}{4^k}<\pa_v^kf_0,v'<\na^{\bot}\phi_{\neq}\cdot\na f_{\neq}>>_{\cG^{\la,\b;s}}^2\\
&\quad-<f_0,g\pa_vf_0>_{\cG^{\la,\b;s}}
+<f_0,((v')^2-1)\pa_{vv}f_0>_{\cG^{\la,\b;s}}
+<f_0,v''\pa_vf_0>_{\cG^{\la,\b;s}}\\
&\quad+\sum_{k=1}^{3}\Big(
-\f{t^{k}}{4^k}<\pa_v^kf_0,\pa_v^k(g\pa_vf_0)>_{\cG^{\la,\b;s}}
+\f{t^{k}}{4^k}<\pa_v^kf_0,\pa_v^k\big(((v')^2-1)\pa_{vv}f_0\big)>_{\cG^{\la,\b;s}}\\
&\qquad\quad\quad+\f{t^{k}}{4^k}<\pa_v^kf_0,\pa_v^k(v''\pa_vf_0)>_{\cG^{\la,\b;s}}
\Big)
=\sum_{j=1}^{8}\Pi_{f_0,j},
\end{align*}
where
\beno
\rmD_{lo,f_0}(t)=\sum_{k=0}^{3}\f{t^{k}}{4^k}\|\pa_v^{k+1}f_0\|_{\cG^{\la,\b;s}}^2,
\eeno 
and the inner product is defined as follows:
\beno
<f_1,f_2>_{\cG^{\la,\b;s}}=\f{1}{2\pi}\int_{\R}\langle \eta\rangle^{2\b}e^{2\la(t)|\eta|^s}\overline{\hat{f_1}}(\eta)\hat{f_2}(\eta)d\eta.
\eeno
And we also have
\begin{align*}
&\f12\f{d}{dt}(t^2\|\pa_v^2h\|_{\cG^{\la,\b;s}}^2)-\dt{\la}(t)t^2\|\pa_v^2h\|_{\cG^{\la,\b;s}}^2\\
&=-t^2<\pa_{vv}h,\pa_{vv}(g\pa_vh)>_{\cG^{\la,\b;s}}
-t<\pa_{vv}f_0,\pa_{vv}h>_{\cG^{\la,\b;s}}=\Pi_{h,1}+\Pi_{h,2}.
\end{align*}
By \eqref{eq: h-L^2}, it is easy to check that
\beno
t\|h\|_{\cG^{\la,\b+2;s}}\lesssim \ep+\rmE_{lo,h}^{\f12}
\eeno
and by \eqref{h_bar_Equation_2}, \eqref{eq:g-H1} and using the same argument as above section, we have
\beno
t^2\|\pa_vg\|_{\cG^{\la,\b+2;s}}\leq \ep+\rmE_{lo,g}^{\f12},\quad 
\rmE_{lo,g}\lesssim \rmE_{lo,f_0}+\rmE_{lo,h}
\eeno
It is easy to check that 
\beno
|\Pi_{f_0,1}|\leq \f{3}{8}\rmD_{lo,f_0}(t).
\eeno
By the elliptic estimate, we have
\beno
|\Pi_{f_0,2}|\lesssim t^k\|\pa_v^kf_0\|_{\cG^{\la,\b;s}}(1+\|h\|_{\cG^{\la,\b;s}})\|\na P_{\neq}\phi\|_{\cG^{\la,\b;s}}\|\na f_{\neq}\|_{\cG^{\la,\b;s}}\lesssim \f{\ep^2}{\langle t\rangle^{\f92}}\rmE_{lo,f_0}^{\f12}.
\eeno
By the fact that $\|g\|_{L^2}\lesssim \f{\ep}{\langle t\rangle^{\f74}}$, we get that
\begin{align*}
|\Pi_{f_0,3}|
&\lesssim \|f_0\|_{\cG^{\la,\b;s}}
\|g\|_{\cG^{\la,\b;s}}\|\pa_vf_0\|_{\cG^{\la,\b;s}}\\
&\lesssim \|f_0\|_{\cG^{\la,\b;s}}
(\|g\|_{L^2}+\|\pa_vg\|_{\cG^{\la,\b;s}})\|\pa_vf_0\|_{\cG^{\la,\b;s}}\lesssim
\f{\ep}{\langle t\rangle^{2}}\rmE_{lo,f_0}+\f{1}{\langle t\rangle^2}\rmE_{lo,f_0}\rmE_{lo,g}^{\f12}.
\end{align*}
Similarly, we have 
\begin{align*}
&|\Pi_{f_0,4}|\leq \|f_0\|_{\cG^{\la,\b;s}}\|h\|_{\cG^{\la,\b;s}}(2+\|h\|_{\cG^{\la,\b;s}})\|\pa_{vv}f_0\|_{\cG^{\la,\b;s}}
\lesssim \f{\ep}{\langle t\rangle^2}\rmE_{lo,f_0}
+\f{1}{\langle t\rangle^2}\rmE_{lo,f_0}\rmE_{lo,h}^{\f12},\\
&|\Pi_{f_0,5}|\lesssim \|f_0\|_{\cG^{\la,\b;s}}\|\pa_vh\|_{\cG^{\la,\b;s}}(1+\|h\|_{\cG^{\la,\b;s}})\|\pa_{v}f_0\|_{\cG^{\la,\b;s}}
\lesssim \f{\ep}{\langle t\rangle^{\f32}}\rmE_{lo,f_0}
+\f{1}{\langle t\rangle^{\f32}}\rmE_{lo,f_0}\rmE_{lo,h}^{\f12}.
\end{align*}
By using the fact that $<f_1,\pa_vf_2>_{\cG^{\la,\b;s}}=-<\pa_vf_1,f_2>_{\cG^{\la,\b;s}}$, we have 
\begin{align*}
|\Pi_{f_0,6}|
&\lesssim \sum_{k=1}^3\Big(t^k\|\pa_v^{k+1}f_0\|_{\cG^{\la,\b;s}}\|g\|_{\cG^{\la,\b;s}}\|\pa_v^kf_0\|_{\cG^{\la,\b;s}}+t^k\|\pa_v^{k+1}f_0\|_{\cG^{\la,\b;s}}\|\pa_v^{2}g\|_{\cG^{\la,\b;s}}\|\pa_vf_0\|_{\cG^{\la,\b;s}}\Big)\\
&\lesssim \rmD_{lo,f_0}^{\f12}\rmE_{lo,f_0}^{\f12}\left(\f{\ep}{\langle t\rangle^{\f74}}+\f{1}{\langle t\rangle^2}\rmE_{lo,g}^{\f12}\right)
+\rmD_{lo,f_0}^{\f12}\rmE_{lo,f_0}^{\f12}\left(\f{\ep}{\langle t\rangle}+\f{\rmE_{lo,g}^{\f12}}{\langle t\rangle}\right),\\
|\Pi_{f_0,7}|
&\lesssim \sum_{k=1}^3\Big(\ep t^k\|\pa_v^{k+1}f_0\|_{\cG^{\la,\b;s}}^2+t^{\f{k}{2}}\|\pa_v^{k+1}f_0\|_{\cG^{\la,\b;s}}
t\|h\|_{\cG^{\la,\b+2;s}}t^{\f12}\|\pa_{vv}f_0\|_{\cG^{\la,\b;s}}\Big)\\
&\lesssim \ep\rmD_{lo,f_0}+\rmE_{lo,h}^{\f12}\rmD_{lo,f_0}.
\end{align*}
Now we deal with $\Pi_{f_0,8}$, for $k=3$, we have
\begin{align*}
&t^3|<\pa_v^3f_0,\pa_v^3\big(\pa_v((v')^2-1)\pa_vf_0\big)>_{\cG^{\la,\b;s}}|\\
&\lesssim t^{\f32}\|\pa_v^4f_0\|_{\cG^{\la,\b;s}}
t^{\f32}\left\|\pa_v^2\big(\pa_v((v')^2-1)\pa_vf_0\big)\right\|_{\cG^{\la,\b;s}}\\
&\lesssim \rmD_{lo,f_0}^{\f12}\sqrt{t}\|h\|_{\cG^{\la,\b+1;s}}t\|\pa_v^3f_0\|_{\cG^{\la,\b;s}}
+\rmD_{lo,f_0}^{\f12}\sqrt{t}\|\pa_v^2h\|_{\cG^{\la,\b+1;s}}
t\|\pa_vf_0\|_{\cG^{\la,\b;s}}\\
&\lesssim \ep\rmD_{lo,f_0}+\rmE_{lo,h}^{\f12}\rmD_{lo,f_0}
+\rmD_{lo,f_0}^{\f12}\sqrt{t}\|\rmA\pa_v^2h\|_2
t\Big(\|\pa_vf_0\|_{L^2}+\|\pa_{vv}f_0\|_{\cG^{\la,\b;s}}\Big)\\
&\lesssim \ep\rmD_{lo,f_0}+\rmE_{lo,h}^{\f12}\rmD_{lo,f_0}
+\Big(\ep+\rmE_{lo,f_0}^{\f12}\Big)\Big(\rmD_{lo,f_0}
+\langle t\rangle\|\rmA\pa_v^2h\|_2^2\Big),
\end{align*}
for $k=1,2$, we have
\begin{align*}
&t^k|<\pa_v^kf_0,\pa_v^k\big(\pa_v((v')^2-1)\pa_vf_0\big)>_{\cG^{\la,\b;s}}|\\
&\lesssim t^{\f k2}\|\pa_v^{k+1}f_0\|_{\cG^{\la,\b;s}}
t\|h\|_{\cG^{\la,\b+2;s}}\Big(\|\pa_{v}f_0\|_{\cG^{\la,\b;s}}+\|\pa_{vv}f_0\|_{\cG^{\la,\b;s}}\Big)\\
&\lesssim \ep\rmD_{lo,f_0}+\rmE_{lo,h}^{\f12}\rmD_{lo,f_0}. 
\end{align*}
Next we turn to $\Pi_{h,1}$ and $\Pi_{h,2}$. We get that
\begin{align*}
|\Pi_{h,1}|
&\lesssim t\|\pa_{vv}h\|_{\cG^{\la,\b;s}}
t\|g\|_{\cG^{\la,\b;s}} \|\pa_{vv}h\|_{\cG^{\la,\b+1;s}}
+t\|\pa_{vv}h\|_{\cG^{\la,\b;s}}t\|\pa_{v}^3g\|_{\cG^{\la,\b;s}} \|h\|_{\cG^{\la,\b;s}}\\
&\lesssim \rmE_{lo,h}^{\f12}\sqrt{t}\big(\|g\|_{L^2}+\|\pa_v^3g\|_{\cG^{\la,\b;s}}\big)\sqrt{t}\|\rmA\pa_{vv}h\|_2
+\f{1}{\langle t\rangle}\rmE_{lo,h}^{\f12}\rmE_{lo,g}^{\f12}\big(\|h\|_{L^2}+\|\pa_{vv}h\|_{\cG^{\la,\b;s}}\big)\\
&\lesssim \f{\ep+\rmE_{lo,g}^{\f12}}{\langle t\rangle^{\f54}}\big(\rmE_{lo,h}+\ep\big)
+\f{\ep+\rmE_{lo,h}^{\f12}}{\langle t\rangle^2}\rmE_{lo,h}^{\f12}\rmE_{lo,g}^{\f12}.
\end{align*}
Thus we conclude by taking $\ep$ small enough that
\begin{align*}
&\f12\f{d}{dt}\big(\rmE_{lo,f}+\rmE_{lo,h}\big)+\f12\rmD_{lo,f_0}\\
&\lesssim \f{\ep^3}{\langle t\rangle^{\f54}}
+\f{1}{\langle t\rangle^{\f54}}\big(\rmE_{lo,f}+\rmE_{lo,h}\big)^{\f32}
+\big(\rmE_{lo,f}+\rmE_{lo,h}\big)^{\f12}\rmD_{lo,f_0}
+\ep\langle t\rangle\|\rmA\pa_v^2h\|_2^2
+\rmE_{lo,f_0}^{\f12}\langle t\rangle\|\rmA\pa_v^2h\|_2^2,
\end{align*}
which implies that $\rmE_{lo,g}\lesssim \ep^2$ and
\ben
\big(\rmE_{lo,f}+\rmE_{lo,h}\big)+\f{1}{3}\int_0^t\rmD_{lo,f_0}(s)ds\lesssim \ep^2.
\een

We have 
\begin{align*}
\|g\|_{\cG_1^{\la,\b;s}}
&\lesssim \int_{|\xi|\leq 1}|\hat{g}(t,\xi)|d\xi+\|\pa_{vvv}g\|_{\cG_1^{\la,\b;s}}\\
&\lesssim \left\||\xi|^{-\f12+\ep_1}\right\|_{L^2}\|g\|_{\dot{H}^{\f12-\ep_1}}+\ep\langle t\rangle^{-2}\\
&\lesssim
 \f{1}{\sqrt{\ep_1}}\|g\|_{\dot{H}^{\f12-\ep_1}}+\ep\langle t\rangle^{-2}\lesssim
 \f{\ep}{\sqrt{\ep_1^3}\langle t\rangle^{2-\f{\ep_1}{2}}}. 
\end{align*}
which gives the proposition. 

\begin{remark}
It holds that
\beno
\|\pa_vg\|_{\cG^{\la,\b+2;s}}\lesssim \f{\ep}{\langle t\rangle^2}.
\eeno
\end{remark}
\section{Elliptic estimate}
In this section, we study the elliptic estimate. We give the proofs of Propositions \ref{prop: elliptic-1}-\ref{prop: elliptic-4}. Before giving the proof of these propositions, we start with fundamental estimate on the stream function $\phi$ in a lower norm. The regularity gap between the higher regularity  norms and this lower regularity norm allows us to trade the regularity of $f$ in higher  norms for the decay of the stream function in lower norms. In other word, we prove the following lemma. 
\begin{lemma}\label{Lemma_Lossy}
Under the bootstrap hypothesis and for $\ep$ sufficiently small, it holds that  
\begin{equation}\label{Lossy_1}
\begin{aligned}
&\Vert P_{\neq} \pa_z^{-1}\Delta_L\phi(t)\Vert_{\cG^{\lambda, \sigma-2;s}}+\langle t \rangle^2\Vert P_{\neq} \pa_z^{-1}\phi(t)\Vert_{\cG^{\lambda, \sigma-4;s}}\lesssim \Vert \pa_z^{-1}P_{\neq}f(t)\Vert_{\cG^{\lambda, \sigma-2;s}}
\end{aligned}  
\end{equation} 
and
\begin{equation}\label{Decay_f_ep}
\|P_{\neq}\pa_{z}^{-1}\Delta_{L}f\|_{\cG^{\la,\sigma;s}}+{\langle t\rangle^{2}}\Vert \pa_z^{-1}P_{\neq}f(t)\Vert_{\cG^{\lambda, \sigma-2;s}}\lesssim \ep. 
\end{equation}
\end{lemma}
\begin{proof}
We have for any $\varphi$ and $\s'>0$
\begin{align*}
\Vert P_{\neq} \pa_z^{-1}\varphi(t)\Vert_{\cG^{\lambda, \sigma';s}}^2=&\sum_{k\neq0}\int_\eta e^{2\lambda |k,\eta|^s}\langle k,\eta \rangle^{2\sigma'}|\rhat{\pa_z^{-1}\phi}(k,\eta)|^2 d\eta\\
=&\sum_{k\neq0}\int_\eta e^{2\lambda |k,\eta|^s}\f{\langle k,\eta \rangle^{2\sigma'+4}}{\langle k,\eta \rangle^4(k^2+(\eta-kt)^2)^2}\\
&\times (k^2+(\eta-kt)^2)^2|\rhat{\pa_z^{-1}\varphi}(k,\eta)|^2 d\eta\\
\lesssim &\, \f{1}{\langle t \rangle^4}\Vert P_{\neq} \Delta_L\pa_z^{-1}\varphi(t)\Vert_{\cG^{\lambda, \sigma'+2;s}}^2,
\end{align*}

On the other hand, we have
\begin{equation}\label{Phi_Identity}
 \begin{aligned}
\Delta_LP_{\neq}\phi=&\,P_{\neq}f+(1-(v^\prime)^2)(\pa_v-t\pa_z)^2P_{\neq}\phi-v^{\prime\prime}(\pa_v-t\pa_z)P_{\neq}\phi. 
\end{aligned}
\end{equation}
and
\begin{equation}\label{Eq_f_Delta_L}
\Delta_L f=\gamma^2\partial_z \rho+K-((v')^2-1)(\pa_v-t\pa_z)^2f-v''(\pa_v-t\pa_z)f.
\end{equation}
Hence, it holds that by using the algebra property of Gevrey spaces together with the bootstrap assumption   
\begin{equation}
\begin{aligned}
\Vert P_{\neq} \pa_z^{-1}\Delta_L\phi(t)\Vert_{\cG^{\lambda, \sigma-2;s}}\lesssim&\, \Vert P_{\neq} \pa_z^{-1}f(t)\Vert_{\cG^{\lambda, \sigma-2;s}}+ \Vert 1-(v^\prime)^2\Vert_{\cG^{\lambda, \sigma-2;s}}\Vert P_{\neq}\pa_z^{-1} \Delta_L\phi(t)\Vert_{\cG^{\lambda, \sigma-2;s}}\\
&+ \Vert v^{\prime\prime}\Vert_{\cG^{\lambda, \sigma-2;s}}\Vert P_{\neq} \pa_z^{-1}\Delta_L\phi(t)\Vert_{\cG^{\lambda, \sigma-2;s}}\\
\lesssim&\, \Vert P_{\neq}\pa_z^{-1} f(t)\Vert_{\cG^{\lambda, \sigma-2;s}}+\ep\Vert P_{\neq} \pa_z^{-1}\Delta_L\phi(t)\Vert_{\cG^{\lambda, \sigma-2;s}}.   
\end{aligned}
\end{equation}
Thus by taking $\varphi=\phi$, $\s'=\s-4$ and $\ep$ sufficiently small, we get  \eqref{Lossy_1}.  

Similarly, we have from \eqref{Eq_f_Delta_L} 
\begin{equation}
\begin{aligned}
\|P_{\neq}\pa_{z}^{-1}\Delta_{L}f\|_{\cG^{\la,\sigma;s}}\lesssim&\,\|\rmA P_{\neq}\rho\|_2+\|\rmA \pa_{z}^{-1}P_{\neq}K\|_2+\Vert G_1\Vert_{\cG^{\la,\sigma;s}}\|P_{\neq}\pa_{z}^{-1}\Delta_{L}f\|_{\cG^{\la,\sigma;s}}\\
&+\Vert\pa_vG_1\Vert_{\cG^{\la,\sigma;s}}\|P_{\neq}\pa_{z}^{-1}\Delta_{L}f\|_{\cG^{\la,\sigma;s}}\\
\lesssim&\, \ep+\ep \|P_{\neq}\pa_{z}^{-1}\Delta_{L}f\|_{\cG^{\la,\sigma;s}}.
\end{aligned}
\end{equation}
Therefore by taking $\varphi=f$, $\s'=\s-2$ and $\ep$ sufficiently small, we get \eqref{Decay_f_ep}. 
\end{proof}

\subsection{Proof of Proposition \ref{prop: elliptic-3}}
\label{Section_Proof_elliptic-3}
\begin{proof}
It is also easy to check that
\begin{equation}\label{f_Ellip_Estimate}
\begin{aligned}
\|\rmA P_{\neq}\pa_z^{-1}\Delta_L f\|_2
\leq &\g^2\|\rmA P_{\neq}\rho\|_2+\|\rmA \pa_{z}^{-1}P_{\neq}K\|_2\\
&+\|\rmA P_{\neq}\pa_{z}^{-1}\mathcal{M}_{1,f}\|_2
+\|\rmA P_{\neq}\pa_{z}^{-1}\mathcal{M}_{2,f}\|_2.
\end{aligned}
\end{equation}
where 
\beno
\mathcal{M}_{1,f}=((v')^2-1)(\pa_v-t\pa_z)^2f,\quad 
\mathcal{M}_{2,f}=v''(\pa_v-t\pa_z)f.
\eeno
Hence,  by dividing each via a paraproduct decomposition in the $v$ variable only we have that
\begin{align*}
\widehat{\cM}_{1,f}(t,k,\eta)
&=-\f{1}{2\pi}\sum_{\rmM\geq 8}\int \widehat{G_1}(\xi)_{\rmM}((\eta-\xi)-kt)^2\hat{f}_k(\eta-\xi)_{<\rmM/8}d\xi\\
&\quad-\f{1}{2\pi}\sum_{\rmM\geq 8}\int\widehat{G_1}(\eta-\xi)_{<\rmM/8}(\xi-kt)^2\hat{f}_k(\xi)_{\rmM}d\xi\\
&\quad-\f{1}{2\pi}\sum_{\rmM\in \mathbf{D}}\sum_{\f18\rmM\leq \rmM'\leq 8\rmM}\int \widehat{G_1}(\xi)_{\rmM'}((\eta-\xi)-kt)^2\hat{f}_k(\eta-\xi)_{\rmM}d\xi\\
&=\widehat{\cM}_{1,f;\rmH\rmL}+\widehat{\cM}_{1,f;\rmL\rmH}+\widehat{\cM}_{1,f;\rmH\rmH},
\end{align*}
and 
\begin{align*}
\widehat{\cM}_{2,f}(t,k,\eta)
&=\f{i}{2\pi}\sum_{\rmM\geq 8}\int\widehat{v''}(\xi)_{\rmM}((\eta-\xi)-kt)\hat{f}_k(\eta-\xi)_{<\rmM/8}d\xi\\
&\quad+\f{i}{2\pi}\sum_{\rmM\geq 8}\int\widehat{v''}(\eta-\xi)_{<\rmM/8}(\xi-kt)\hat{f}_k(\xi)_{\rmM}d\xi\\
&\quad+\f{i}{2\pi}\sum_{\rmM\in \mathbf{D}}\sum_{\f18\rmM\leq \rmM'\leq 8\rmM}\int\widehat{v''}(\eta-\xi)_{\rmM'}(\xi-kt)\hat{f}_k(\xi)_{\rmM}d\xi\\
&=\widehat{\cM}_{2,f;\rmH\rmL}+\widehat{\cM}_{2,f;\rmL\rmH}+\widehat{\cM}_{2,f;\rmH\rmH}.
\end{align*}
where $G_1=(v')^2-1$ and $v''=\f12\pa_vG_1$. 

The treatment of $\widehat{\cM}_{1,f;\rmL\rmH}$ and $\widehat{\cM}_{2,f;\rmL\rmH}$ is similar.
 By the fact that
\beno
\f{\rmJ_k(\eta)}{\rmJ_k(\xi)}\lesssim e^{C|\eta-\xi|^{\f13}}
\eeno
We get that 
\begin{equation}\label{f_1_LH}
\begin{aligned}
\|\rmA P_{\neq}\pa_{z}^{-1}\mathcal{M}_{1,f;\rmL\rmH}\|_2^2
&\lesssim \sum_{\rmM\geq 8}\|G_1\|_{\cG^{\la,0;s}}^2
\|\rmA\pa_z^{-1}(\pa_v-t\pa_z)^2P_{\neq}f_{\rmM}\|_2^2,\\
&\lesssim \ep^2\|\rmA P_{\neq}\pa_z^{-1}\Delta_{L}f\|_2^2,
\end{aligned}
\end{equation}
and
\begin{equation}\label{f_2_LH}
\begin{aligned}
\|\rmA P_{\neq}\pa_{z}^{-1}\mathcal{M}_{2,f;\rmL\rmH}\|_2^2
&\lesssim \sum_{\rmM\geq 8}\|v''\|_{\cG^{\la,0;s}}^2
\|\rmA\pa_z^{-1}(\pa_v-t\pa_z)P_{\neq}f_{\rmM}\|_2^2\\
&\lesssim \ep^2\|\rmA P_{\neq}\pa_z^{-1}\Delta_{L}f\|_2^2.
\end{aligned}
\end{equation}

Next we consider the high-low interaction. The notation is deceptive: the frequency in $z$ could be very
large and hence more ‘derivatives’ are appearing on $f$ and we will be in a situation like the low-high interaction. Hence we break into two cases: 
\begin{align*}
\widehat{\cM}_{1,f;\rmH\rmL}
&=-\f{1}{2\pi}\sum_{\rmM\geq 8}\int [\mathbf{1}_{|k|\geq \f{1}{16}|\eta|}+\mathbf{1}_{|k|< \f{1}{16}|\eta|}]\widehat{G_1}(\xi)_{\rmM}((\eta-\xi)-kt)^2\hat{f}_k(\eta-\xi)_{<\rmM/8}d\xi\\
&=\widehat{\cM}_{1,f;\rmH\rmL}^{z}+\widehat{\cM}_{1,f;\rmH\rmL}^{v}, \\
\widehat{\cM}_{2,f;\rmH\rmL}
&=\f{i}{2\pi}\sum_{\rmM\geq 8}\int [\mathbf{1}_{|k|\geq \f{1}{16}|\eta|}+\mathbf{1}_{|k|< \f{1}{16}|\eta|}]\widehat{v''}(\xi)_{\rmM}((\eta-\xi)-kt)\hat{f}_k(\eta-\xi)_{<\rmM/8}d\xi\\
&=\widehat{\cM}_{2,f;\rmH\rmL}^{z}+\widehat{\cM}_{2,f;\rmH\rmL}^{v}. 
\end{align*}
Let us first treat $\widehat{\cM}_{1,f;\rmH\rmL}^{z}$ and $\widehat{\cM}_{2,f;\rmH\rmL}^{z}$. 
On the support of the integrand, we get that there is some $c\in (0,1)$ such that,
\begin{equation}
|k,\eta|^s\leq |k,\eta-\xi|^s+c|\xi|^{s}
\end{equation}

We also have 
\begin{equation}\label{eq_J_eta_J_xi}
\f{\rmJ_k(\eta)}{\rmJ_k(\xi)}\lesssim e^{C|k,\eta-\xi|^{\f13}}.
\end{equation}
Thus we get that 
\begin{equation}\label{f_1_HL}
\begin{aligned}
&\|\rmA P_{\neq}\pa_{z}^{-1}\mathcal{M}_{1,f;\rmH\rmL}^z\|_2^2
+\|\rmA P_{\neq}\pa_{z}^{-1}\mathcal{M}_{2,f;\rmH\rmL}^z\|_2^2\\
\lesssim &\,\sum_{\rmM\geq 8}\big(\|G_1\|_{\cG^{\la,0;s}}^2+\|v''\|_{\cG^{\la,0;s}}^2\big)
\|\rmA\pa_z^{-1}\Delta_{L}P_{\neq}f_{\rmM}\|_2^2
\lesssim \ep^2\|\rmA\pa_z^{-1}\Delta_{L}P_{\neq}f\|_2^2, 
\end{aligned}
\end{equation}
Next we consider $\mathcal{M}_{1,f;\rmH\rmL}^v$ and $\mathcal{M}_{2,f;\rmH\rmL}^v$. Due to the fact that $G_1$ admits two more derivate. By the fact that
\begin{equation}\label{J_J_0_Estimate_1}
\f{\rmJ_k(\eta)}{\rmJ_0(\xi)}\lesssim \langle \xi\rangle e^{C|k,\eta-\xi|^{\f13}},
\end{equation}

we have
\begin{equation}\label{f_1_2_HL}
\begin{aligned}
&\|\rmA P_{\neq}\pa_{z}^{-1}\mathcal{M}_{1,f;\rmH\rmL}^v\|_2^2
+\|\rmA P_{\neq}\pa_{z}^{-1}\mathcal{M}_{2,f;\rmH\rmL}^v\|_2^2\\
&\lesssim \sum_{\rmM\geq 8}\big(\|\langle \pa_v\rangle \rmA (v'')_{\rmM}\|_2^2+\|\langle \pa_v\rangle \rmA (G_1)_{\rmM}\|_2^2\big)\|P_{\neq}\pa_{z}^{-1}\Delta_{L}f\|_{\cG^{\la,0;s}}^2
\lesssim \ep^4. 
\end{aligned}
\end{equation}
The high-high interaction is easy to treat. We show the results and omit the proof. 
\begin{align}\label{f_1_HH}
\|\rmA P_{\neq}\pa_{z}^{-1}\mathcal{M}_{1,f;\rmH\rmH}^v\|_2^2
+\|\rmA P_{\neq}\pa_{z}^{-1}\mathcal{M}_{2,f;\rmH\rmH}^v\|_2^2\lesssim \ep^4
\end{align}
Plugging \eqref{f_1_LH}, \eqref{f_2_LH}, \eqref{f_1_HL}, \eqref{f_1_2_HL} and \eqref{f_1_HH} into \eqref{f_Ellip_Estimate} using the bootstrap assumption and  taking $\ep$ small enough, we get the proposition. 
\end{proof}

\subsection{Proof of Proposition \ref{prop: elliptic-1}}\label{Section_Proof_Proposition_elliptic-1}
\begin{proof}
We write 
 \begin{equation}
 \begin{aligned}
\Delta_L\phi=&\,f+(1-(v^\prime)^2)(\pa_v-t\pa_z)^2\phi-v^{\prime\prime}(\pa_v-t\pa_z)\phi. 
\end{aligned}
\end{equation}
This yields by using the fact that 
$\pa_v (v^\prime)^2=2\pa_v h(h+1)$ and 
\begin{equation}\label{Delta_L_phi}
 \begin{aligned}
\Delta_L^2\phi=&\,\Delta_L f-G_1(\pa_v-t\pa_z)^2\Delta_L\phi-\pa_vG_1(\pa_v-t\pa_z)\Big(\f52(\pa_v-t\pa_z)^2+\f12\pa_z^2\Big)\phi\\
&-2\pa_{vv} G_1(\pa_v-t\pa_z)^2\phi-\f12\pa_{vvv}G_1(\pa_v-t\pa_z)\phi\\
=&\,\mathcal{M}_{1,\phi}+\mathcal{M}_{2,\phi}+\mathcal{M}_{3,\phi}+\mathcal{M}_{4,\phi}. 
\end{aligned}
\end{equation}

Similarly by following the proof of Lemma \ref{Lemma_Lossy}, it is easy to check that under the bootstrap assumption, it holds for $\s_0\leq \s-1$.  
\begin{equation}\label{Gev_phi}
\|P_{\neq}\pa_{z}^{-1}\Delta_{L}^2\phi\|_{\cG^{\la,\sigma_0;s}}\lesssim \ep
\end{equation}

Hence, it holds that by using \eqref{Elliptic_f_1} 
\begin{equation}\label{phi_Ellip_Main}
\begin{aligned}
\Big\|\Big\langle \f{\pa_v}{t\pa_z}\Big\rangle^{-1} \rmA\pa_z^{-1}\Delta^{2}_{L}P_{\neq}\phi\Big\|_2\lesssim &\,\Big\|\Big\langle \f{\pa_v}{t\pa_z}\Big\rangle^{-1}\rmA P_{\neq}\pa_z^{-1}\Delta_L f\Big\|_2+\sum_{i=1}^4 \Big\|\Big\langle \f{\pa_v}{t\pa_z}\Big\rangle^{-1} \rmA P_{\neq}\pa_z^{-1}\mathcal{M}_{i,\phi}\Big\|_2\\
\lesssim &\, \ep+\sum_{i=1}^4 \Big\|\Big\langle \f{\pa_v}{t\pa_z}\Big\rangle^{-1} \rmA P_{\neq}\pa_z^{-1}\mathcal{M}_{i,\phi}\Big\|_2. 
\end{aligned}
\end{equation}
We write    
\begin{align*}
\widehat{\cM}_{i, \phi}(t,k,\eta)
&=\widehat{\cM}_{i,\phi;\rmH\rmL}+\widehat{\cM}_{i,\phi;\rmL\rmH}+\widehat{\cM}_{i,\phi;\rmH\rmH},\quad i=1,\dots 4.  
\end{align*}
To estimate $\widehat{\cM}_{1,\phi;\rmL\rmH}$, we proceed as in the estimate involving $\widehat{\cM}_{1,f;\rmL\rmH}$ and   using the fact that on the support of the integrand, we have for the  $\langle \f{\eta}{tk}\rangle^{-1}\approx \langle \f{\xi}{tk}\rangle^{-1} $, which means that we can move this factor to $\phi$ and obtain as in \eqref{f_1_LH} 
 \begin{equation}
 \begin{aligned}
\Big\|\Big\langle \f{\pa_v}{t\pa_z}\Big\rangle^{-1} \rmA P_{\neq}\pa_z^{-1}\mathcal{M}_{1,\phi;\rmL\rmH}\Big\|_2\lesssim \ep \Big\|\Big\langle \f{\pa_v}{t\pa_z}\Big\rangle^{-1} \rmA\pa_z^{-1}\Delta^{2}_{L}P_{\neq}\phi\Big\|_2. 
\end{aligned}
\end{equation}
We also have similarly 
\begin{equation}
 \begin{aligned}
\Big\|\Big\langle \f{\pa_v}{t\pa_z}\Big\rangle^{-1} \rmA P_{\neq}\pa_z^{-1}\mathcal{M}_{2,\phi;\rmL\rmH}\Big\|_2\lesssim \ep \Big\|\Big\langle \f{\pa_v}{t\pa_z}\Big\rangle^{-1} \rmA\pa_z^{-1}\Delta^{2}_{L}P_{\neq}\phi\Big\|_2
\end{aligned}
\end{equation}
Both terms can be absorbed by the left-hand side of \eqref{phi_Ellip_Main} for sufficiently small $\ep$.

 The other two terms $\widehat{\cM}_{3,\phi;\rmL\rmH}$ and $\widehat{\cM}_{4,\phi;\rmL\rmH}$  can be treated similarly. We omit the details and write the result
 \begin{equation}
 \begin{aligned}
&\Big\|\Big\langle \f{\pa_v}{t\pa_z}\Big\rangle^{-1} \rmA P_{\neq}\pa_z^{-1}\mathcal{M}_{3,\phi;\rmL\rmH}\Big\|_2+\Big\|\Big\langle \f{\pa_v}{t\pa_z}\Big\rangle^{-1} \rmA P_{\neq}\pa_z^{-1}\mathcal{M}_{4,\phi;\rmL\rmH}\Big\|_2\\
&\lesssim\, \ep \Big\|\Big\langle \f{\pa_v}{t\pa_z}\Big\rangle^{-1} \rmA\pa_z^{-1}\Delta^{2}_{L}P_{\neq}\phi\Big\|_2. 
\end{aligned}
\end{equation}

 For the high-low interaction, we write 
 \begin{align*}
\widehat{\cM}_{1,\phi;\rmH\rmL}
&=-\f{1}{2\pi}\sum_{\rmM\geq 8}\int \Big\langle \f{\eta}{tk}\Big\rangle^{-1}[\mathbf{1}_{|k|\geq \f{1}{16}|\eta|}+\mathbf{1}_{|k|< \f{1}{16}|\eta|}]\widehat{G_1}(\xi)_{\rmM}((\eta-\xi)-kt)^2\rhat{\Delta_L \phi}_k(\eta-\xi)_{<\rmM/8}d\xi\\
&=\widehat{\cM}_{1, \phi;\rmH\rmL}^{z}+\widehat{\cM}_{1,\phi;\rmH\rmL}^{v}, \\
\widehat{\cM}_{2,\phi;\rmH\rmL}
&=\f{i}{2\pi}\sum_{\rmM\geq 8}\int \Big\langle \f{\eta}{tk}\Big\rangle^{-1}[\mathbf{1}_{|k|\geq \f{1}{16}|\eta|}+\mathbf{1}_{|k|< \f{1}{16}|\eta|}]\widehat{v''}(\xi)_{\rmM}((\eta-\xi)-kt)\rhat{\Delta_L \phi}_k(\eta-\xi)_{<\rmM/8}d\xi\\
&=\widehat{\cM}_{2,\phi;\rmH\rmL}^{z}+\widehat{\cM}_{2,\phi;\rmH\rmL}^{v}. 
\end{align*}
  
As in \eqref{f_1_2_HL} we have the estimate 
\begin{equation}\label{f_1_2_HL}
\begin{aligned}
&\Big\|\Big\langle \f{\pa_v}{t\pa_z}\Big\rangle^{-1} \rmA P_{\neq}\pa_{z}^{-1}\mathcal{M}_{1,\phi;\rmH\rmL}^z\Big\|_2
+\Big\|\Big\langle \f{\pa_v}{t\pa_z}\Big\rangle^{-1} \rmA P_{\neq}\pa_{z}^{-1}\mathcal{M}_{2,\phi;\rmH\rmL}^z\Big\|_2\\
\lesssim &\,\ep\Big\|\Big\langle \f{\pa_v}{t\pa_z}\Big\rangle^{-1} \rmA\pa_z^{-1}\Delta^{2}_{L}P_{\neq}\phi\Big\|_2  
\end{aligned}
\end{equation}
which again can be absorbed by the left-hand side of \eqref{phi_Ellip_Main}.

 The estimate of the terms $\widehat{\cM}_{1,\phi;\rmH\rmL}^{v}$ and $\widehat{\cM}_{2,\phi;\rmH\rmL}^{v}$ can also be done as in \eqref{f_1_2_HL} and by using \eqref{Gev_phi},  we get 
 \begin{equation}\label{phi_1_2_HL}
\begin{aligned}
&\Big\|\Big\langle \f{\pa_v}{t\pa_z}\Big\rangle^{-1} \rmA P_{\neq}\pa_{z}^{-1}\mathcal{M}_{1,\phi;\rmH\rmL}^v\Big\|_2^2
+\Big\|\Big\langle \f{\pa_v}{t\pa_z}\Big\rangle^{-1} \rmA P_{\neq}\pa_{z}^{-1}\mathcal{M}_{2,\phi;\rmH\rmL}^v\Big\|_2^2\\
&\lesssim \sum_{\rmM\geq 8}\big(\|\langle \pa_v\rangle \rmA (v'')_{\rmM}\|_2^2+\|\langle \pa_v\rangle \rmA (G_1)_{\rmM}\|_2^2\big)\|P_{\neq}\pa_{z}^{-1}\Delta_{L}^2\phi\|_{\cG^{\la,0;s}}^2
\lesssim \ep^4. 
\end{aligned}
\end{equation}

 Now, we prove estimates for $\widehat{\cM}_{4,\phi;\rmH\rmL}$. The one of $\widehat{\cM}_{3,\phi;\rmH\rmL}$ is easier compared to $\widehat{\cM}_{4,\phi;\rmH\rmL}$.
 We write as above  $\widehat{\cM}_{4,\phi;\rmH\rmL}=\widehat{\cM}_{4,\phi;\rmH\rmL}^{z}+\widehat{\cM}_{4,\phi;\rmH\rmL}^v$. The term $\widehat{\cM}_{4,\phi;\rmH\rmL}^{z}$ can be treated as the low-high interaction and we have 
 \begin{equation}
\Big\|\Big\langle \f{\pa_v}{t\pa_z}\Big\rangle^{-1} \rmA P_{\neq}\pa_{z}^{-1}\mathcal{M}_{1,\phi;\rmH\rmL}^z\Big\|_2\lesssim \ep\Big\|\Big\langle \f{\pa_v}{t\pa_z}\Big\rangle^{-1} \rmA\pa_z^{-1}\Delta^{2}_{L}P_{\neq}\phi\Big\|_2
\end{equation}
 So, the most changeling term is the term  $\widehat{\cM}_{4,\phi;\rmH\rmL}^v$ since in this term all the derivatives are landing on the term $\pa_{vvv}G_1$ which will have a regularity loss. Here, where we need to use the factor $\langle \xi/(lt)  \rangle^{-1}$ to absorb one derivative by paying time decay.  
 
 By using the estimate 
 \begin{equation}\label{J_J_0_Estimate}
\f{\rmJ_k(\eta)}{\rmJ_0(\xi)}\lesssim \langle t\rangle e^{C|k,\eta-\xi|^{\f13}},
\end{equation}
we get by using the fact that on the support of the integrand $|\eta|\approx |\xi|$, and since $s>\f13$, 
\begin{equation}
\begin{aligned}
&\Big|\Big\langle \f{\pa_v}{t\pa_z}\Big\rangle^{-1} \rmA P_{\neq}\pa_{z}^{-1}\mathcal{M}_{4,\phi;\rmH\rmL}^v\Big|\\
&\lesssim\sum_{\rmM\geq 8}\int \Big\langle \f{\eta}{tk}\Big\rangle^{-1}\mathbf{1}_{|k|< \f{1}{16}|\eta|}\rmA_0 (\xi)|\xi|^3|\widehat{G_1}(\xi)_{\rmM}|\langle t\rangle|(\eta-\xi)-kt|e^{c\la|k,\eta-\xi|^{s}}e^{C|k,\eta-\xi|^{\f13}}|\rhat{\phi}_k(\eta-\xi)_{<\rmM/8}|d\xi\\
&\lesssim
\sum_{\rmM\geq 8}\int \mathbf{1}_{|k|< \f{1}{16}|\eta|}\langle \xi\rangle^{2}\rmA_0 (\xi)|\widehat{G_1}(\xi)_{\rmM}|\langle t\rangle^3\langle k,\eta-\xi\rangle^2e^{c\la|k,\eta-\xi|^{s}}e^{C|k,\eta-\xi|^{\f13}}|\rhat{\phi}_k(\eta-\xi)_{<\rmM/8}|d\xi
\end{aligned}
\end{equation}
This yields by using the bootstrap assumption together with Lemma \ref{Lemma_Lossy},   
\begin{equation}
\begin{aligned}
\Big\|\Big\langle \f{\pa_v}{t\pa_z}\Big\rangle^{-1} \rmA P_{\neq}\pa_{z}^{-1}\mathcal{M}_{4,\phi;\rmH\rmL}^v\Big\|_2^2\lesssim&\,\langle t\rangle^6 \sum_{\rmM\geq 8}\big(\|\langle \pa_v\rangle^2 \rmA h_{\rmM}\|_2\big)\|_2^2\big)\|P_{\neq}\phi\|_{\cG^{\la,0;s}}^2
\lesssim\ep^4. 
\end{aligned}
\end{equation}

For the terms $\mathcal{M}_{i,\phi;\rmH\rmH},\, i=1,\dots,4$, we have by using the fact that on the support of the integrand, we have  
$|\eta|^s \leq c|\eta-\xi|^s+c|\xi|^s,\, c\in(0,1)$ together with \eqref{J_J_0_Estimate_1}, we have, by absorbing all the possible loss of derivatives by the Gevery term (since $c<1$), we have 
\begin{equation}
\begin{aligned}
\sum_{i=1}^4\|\rmA P_{\neq}\pa_{z}^{-1}\mathcal{M}_{i,\phi;\rmH\rmH}^v\|_2^2\lesssim \ep^4. 
\end{aligned}
\end{equation}
Collecting all the above estimates and taking $\ep$ sufficiently small, we get the desired result. 
\end{proof}

\subsection{Proof of Proposition \ref{prop: elliptic-4}}. 
\label{Section_Proof_elliptic-4}
In this section, we prove the estimate the estimate in Proposition \ref{prop: elliptic-4} .
\subsubsection{Proof of \eqref{Elliptic_f_lambda}}

We have by using \eqref{Eq_f_Delta_L} (see \eqref{f_Ellip_Estimate})
\begin{equation}\label{f_M_0_Estimate}
\begin{aligned}
\bigg\|\mathbf{1}_{t\geq M_0}\f{|\na|^{\f s 2}}{\langle t\rangle^{\f{3s}{2}}}\rmA\pa_z^{-1}\Delta_{L}P_{\neq }f\bigg\|_2^2\leq&\, \g^2\bigg\|\mathbf{1}_{t\geq M_0}\f{|\na|^{\f s 2}}{\langle t\rangle^{\f{3s}{2}}}\rmA P_{\neq}\rho\bigg\|_2^2+\bigg\|\mathbf{1}_{t\geq M_0}\f{|\na|^{\f s 2}}{\langle t\rangle^{\f{3s}{2}}}\rmA \pa_{z}^{-1}P_{\neq}K\bigg\|_2^2\\
&+\bigg\|\mathbf{1}_{t\geq M_0}\f{|\na|^{\f s 2}}{\langle t\rangle^{\f{3s}{2}}}\rmA P_{\neq}\pa_{z}^{-1}\mathcal{M}_{1,f}\bigg\|_2^2
+\bigg\|\mathbf{1}_{t\geq M_0}\f{|\na|^{\f s 2}}{\langle t\rangle^{\f{3s}{2}}}\rmA P_{\neq}\pa_{z}^{-1}\mathcal{M}_{2,f}\bigg\|_2^2. 
\end{aligned}   
\end{equation}
 Keeping in mind \eqref{lambda_Decay}, the first two terms in \eqref{f_M_0_Estimate} can be estimates using $\rmCK_{\la,K}$ and $\rmCK_{\la,\rho}$ as follows:  
 \begin{equation}
\begin{aligned}
\bigg\|\mathbf{1}_{t\geq M_0}\f{|\na|^{\f s 2}}{\langle t\rangle^{\f{3s}{2}}}\rmA P_{\neq}\rho\bigg\|_2^2\leq&\, -C_1\f{\dt{\lambda}(t)}{t^{3s-2\tilde{q}}}\mathbf{1}_{t\geq M_0}\Big\||\na|^{\f s 2} \rmA P_{\neq}\rho\Big\|_2^2\leq C_1 M_0^{2\tilde{q}-3s}\rmCK_{\la,\rho}. 
\end{aligned}
\end{equation}
Similarly, we have
\begin{equation}
\bigg\|\mathbf{1}_{t\geq M_0}\f{|\na|^{\f s 2}}{\langle t\rangle^{\f{3s}{2}}}\rmA \pa_{z}^{-1}P_{\neq}K\bigg\|_2^2\leq C_1 M_0^{2\tilde{q}-3s}\rmCK_{\la,K}. 
\end{equation}
The above   constant $C_1$ depends on $\delta_\lambda$ but it is independent of $M_0$.   
 
 Now, we estimate the last two terms in \eqref{f_M_0_Estimate}. We simply first write  
 \begin{equation}\label{Decom_M_i}
\mathcal{M}_{i,f}=\mathcal{M}_{i,f;\rmH\rmL}+\mathcal{M}_{i,f;\rmL\rmH}+\mathcal{M}_{i,f;\rmH\rmH}
\end{equation}
as in the proof of Proposition \ref{prop: elliptic-3}. Following similar ideas as in the proof of Proposition \ref{prop: elliptic-3}, we have ((by using the same notation) 
 \begin{equation}
\begin{aligned}
&\sum_{i=1}^2\bigg\|\mathbf{1}_{t\geq M_0}\f{|\na|^{\f s 2}}{\langle t\rangle^{\f{3s}{2}}}\rmA P_{\neq}\pa_{z}^{-1}\mathcal{M}_{i,f;\rmL\rmH}\bigg\|_2+\sum_{i=1}^2\bigg\|\mathbf{1}_{t\geq M_0}\f{|\na|^{\f s 2}}{\langle t\rangle^{\f{3s}{2}}}\rmA P_{\neq}\pa_{z}^{-1}\mathcal{M}^z_{i,f;\rmH\rmL}\bigg\|_2\\
\lesssim &\,\ep\bigg\|\mathbf{1}_{t\geq M_0}\f{|\na|^{\f s 2}}{\langle t\rangle^{\f{3s}{2}}}\rmA\pa_z^{-1}\Delta_{L}P_{\neq }f\bigg\|_2
\end{aligned}
\end{equation}
which will be absorbed by the left-hand side of \eqref{f_M_0_Estimate}, provided that $\ep$ is sufficiently small. 
 
 Next, the terms involving $\mathcal{M}^v_{i,f;\rmH\rmL}$ can be estimated as (we omit the details)  
 \begin{equation}
\begin{aligned}
\sum_{i=1}^2\bigg\|\mathbf{1}_{t\geq M_0}\f{|\na|^{\f s 2}}{\langle t\rangle^{\f{3s}{2}}}\rmA P_{\neq}\pa_{z}^{-1}\mathcal{M}^v_{i,f;\rmH\rmL}\bigg\|_2^2\lesssim \ep^2 \rmCK_{\la,h}(t). 
\end{aligned}
\end{equation}
 
 \subsubsection{Proof of \eqref{Elliptic_g_lambda}}
 In this section, we prove the estimate \eqref{Elliptic_g_lambda}.   
  Using \eqref{Eq_f_Delta_L}, we have 
  \begin{equation}\label{g_Ellip_1}
\begin{aligned}
&\left\|\sqrt{\f{\pa_t\rmg}{\rmg}}\left\langle \f{\pa_v}{t\pa_z}\right\rangle^{-1}\tilde{\tilde{\rmA}}\pa_z^{-1}\Delta_{L}P_{\neq }f\right\|_2^2\\
\leq &\,\g^2\bigg\|\sqrt{\f{\pa_t\rmg}{\rmg}}\left\langle \f{\pa_v}{t\pa_z}\right\rangle^{-1}\tilde{\tilde{\rmA}} P_{\neq}\rho\bigg\|_2^2+\bigg\|\sqrt{\f{\pa_t\rmg}{\rmg}}\left\langle \f{\pa_v}{t\pa_z}\right\rangle^{-1}\tilde{\tilde{\rmA}}\pa_{z}^{-1}P_{\neq}K\bigg\|_2^2\\
&+\bigg\|\sqrt{\f{\pa_t\rmg}{\rmg}}\left\langle \f{\pa_v}{t\pa_z}\right\rangle^{-1}\tilde{\tilde{\rmA}} P_{\neq}\pa_{z}^{-1}\mathcal{M}_{1,f}\bigg\|_2^2
+\bigg\|\sqrt{\f{\pa_t\rmg}{\rmg}}\left\langle \f{\pa_v}{t\pa_z}\right\rangle^{-1}\tilde{\tilde{\rmA}} P_{\neq}\pa_{z}^{-1}\mathcal{M}_{2,f}\bigg\|_2^2.
\end{aligned}
\end{equation}
A direct calculation shows that 
\begin{equation}
\begin{aligned}
\bigg\|\sqrt{\f{\pa_t\rmg}{\rmg}}\left\langle \f{\pa_v}{t\pa_z}\right\rangle^{-1}\tilde{\tilde{\rmA}} P_{\neq}\rho\bigg\|_2^2+\bigg\|\sqrt{\f{\pa_t\rmg}{\rmg}}\left\langle \f{\pa_v}{t\pa_z}\right\rangle^{-1}\tilde{\tilde{\rmA}}\pa_{z}^{-1}P_{\neq}K\bigg\|_2^2\leq C_2( \rmCK_{\Theta,\rho} +\rmCK_{\Theta,K}) 
\end{aligned}
\end{equation}
for some $C_2>0$ independent of $\delta_L$. 

Similarly we write 
\beno
\mathcal{M}_{i,f}=\mathcal{M}_{i,f;\rmL\rmH}
+\mathcal{M}_{i,f;\rmH\rmL}
+\mathcal{M}_{i,f;\rmH\rmH},\quad i=1,2.
\eeno
We write 
\begin{align*}
&\bigg|\mathcal{F}\bigg(\sqrt{\f{\pa_t\rmg}{\rmg}}\left\langle \f{\pa_v}{t\pa_z}\right\rangle^{-1}\tilde{\tilde{\rmA}} P_{\neq}\pa_{z}^{-1}\mathcal{M}_{1,f;\rmL\rmH}\bigg)\bigg|\\
\lesssim &\,\sum_{\rmM\geq 8}\sum_{k\neq 0}\int \sqrt{\f{\pa_t\rmg(t,\eta)}{\rmg(t,\eta)}}\left\langle \f{\eta}{tk}\right\rangle^{-1}\tilde{\tilde{\rmA}}(\eta)\widehat{G_1}(\eta-\xi)_{<\rmM/8}(\xi-kt)^2\rhat{\pa_z^{-1}f}(k,\xi)_{\rmM}d\xi
\end{align*} 

To simplify our proof, let us take advantage of the $\f1t$ decay of $G_1$ and $\pa_v G_1$. By the fact that on the support of integrand $|\eta|^{\f13}\lesssim t\lesssim |\eta|\approx |\xi|$, 
 $\tilde{\tilde{\rmA}}\leq \rmA$ and $\sqrt{\f{\pa_t\rmg(t,\eta)}{\rmg(t,\eta)}}\lesssim 1$, we have by using the fact that   $1\leq\langle t\rangle^{2s} \f{|\eta|^s}{\langle t\rangle^{3s}}$
\begin{align*}
&\bigg\|\sqrt{\f{\pa_t\rmg}{\rmg}}\left\langle \f{\pa_v}{t\pa_z}\right\rangle^{-1}\tilde{\tilde{\rmA}} P_{\neq}\pa_{z}^{-1}\mathcal{M}_{1,f;\rmL\rmH}\bigg\|_2^2
+\bigg\|\sqrt{\f{\pa_t\rmg}{\rmg}}\left\langle \f{\pa_v}{t\pa_z}\right\rangle^{-1}\tilde{\tilde{\rmA}} P_{\neq}\pa_{z}^{-1}\mathcal{M}_{2,f;\rmL\rmH}\bigg\|_2^2\\
&\lesssim \langle t\rangle^{2s}(\|G_1\|_{\cG^{\la,1;s}}^{2}+{\|\pa_vG_1\|_{\cG^{\la,1;s}}^2}) \left\|\f{|\na|^{\f s2}}{\langle t\rangle^{\f{3s}{2}}}\left\langle \f{\pa_v}{t\pa_z}\right\rangle^{-1}\rmA\pa_z^{-1}\Delta_{L}P_{\neq }f\right\|_2^2\\
&\lesssim \ep^{2} \left\|\f{|\na|^{\f s2}}{\langle t\rangle^{\f{3s}{2}}}\left\langle \f{\pa_v}{t\pa_z}\right\rangle^{-1}\rmA\pa_z^{-1}\Delta_{L}P_{\neq }f\right\|_2^2. 
\end{align*}
The high-low and high-high interactions are easy. 
Again by using the fact $\sqrt{\f{\pa_t\rmg(t,\eta)}{\rmg(t,\eta)}}\lesssim 1$ and following the proof of Proposition \ref{prop: elliptic-3}, we have 
\begin{align*}
&\bigg\|\sqrt{\f{\pa_t\rmg}{\rmg}}\left\langle \f{\pa_v}{t\pa_z}\right\rangle^{-1}\tilde{\tilde{\rmA}} P_{\neq}\pa_{z}^{-1}\mathcal{M}_{1,f;\rmH\rmL}\bigg\|_2^2
+\bigg\|\sqrt{\f{\pa_t\rmg}{\rmg}}\left\langle \f{\pa_v}{t\pa_z}\right\rangle^{-1}\tilde{\tilde{\rmA}} P_{\neq}\pa_{z}^{-1}\mathcal{M}_{2,f;\rmH\rmH}\bigg\|_2^2\\
&\quad+\bigg\|\sqrt{\f{\pa_t\rmg}{\rmg}}\left\langle \f{\pa_v}{t\pa_z}\right\rangle^{-1}\tilde{\tilde{\rmA}} P_{\neq}\pa_{z}^{-1}\mathcal{M}_{1,f;\rmH\rmH}\bigg\|_2^2
+\bigg\|\sqrt{\f{\pa_t\rmg}{\rmg}}\left\langle \f{\pa_v}{t\pa_z}\right\rangle^{-1}\tilde{\tilde{\rmA}} P_{\neq}\pa_{z}^{-1}\mathcal{M}_{2,f;\rmH\rmL}\bigg\|_2^2\\
&\lesssim \ep^{2} \left\|\f{|\na|^{\f s2}}{\langle t\rangle^{\f{3s}{2}}}\left\langle \f{\pa_v}{t\pa_z}\right\rangle^{-1}\rmA\pa_z^{-1}\Delta_{L}P_{\neq }f\right\|_2^2
+\ep^2 \big(\|\langle \pa_v\rangle \rmA G_1\|_2^2+\|\langle \pa_v\rangle \rmA \pa_vG_1\|_2^2\big)\\
&\lesssim \ep^{2} \left\|\f{|\na|^{\f s2}}{\langle t\rangle^{\f{3s}{2}}}\left\langle \f{\pa_v}{t\pa_z}\right\rangle^{-1}\rmA\pa_z^{-1}\Delta_{L}P_{\neq }f\right\|_2^2
+\ep^2\|\rmA\langle \pa_v\rangle^2 h\|_2^2.
\end{align*}
Similarly, by using the fact that $\tilde{\rmA}\leq \rmA$ and
\beno
\sqrt{\f{b(t,k,\eta)k^2}{k^2+(\eta-kt)^2}}+\sqrt{\f{\pa_t\Theta_k(t,\eta)}{\Theta_k(t,\eta)}}\lesssim 1.
\eeno
We get that 
\begin{align}
\nonumber&\left\|\left\langle \f{\pa_v}{t\pa_z}\right\rangle^{-1}\left(\sqrt{\f{b(t,\na)\pa_{zz}}{\Delta_{L}}}\rmA+\sqrt{\f{\pa_t\Theta}{\Theta}}\tilde{\rmA}\right)\pa_z^{-1}\Delta_{L}P_{\neq }\cM_{1,f}\right\|_2^2\\
\nonumber&+\left\|\left\langle \f{\pa_v}{t\pa_z}\right\rangle^{-1}\left(\sqrt{\f{b(t,\na)\pa_{zz}}{\Delta_{L}}}\rmA+\sqrt{\f{\pa_t\Theta}{\Theta}}\tilde{\rmA}\right)\pa_z^{-1}\Delta_{L}P_{\neq }\cM_{2,f}\right\|_2^2\\
\label{f_B_estimate}&\lesssim 
\ep^{2} \left\|\f{|\na|^{\f s2}}{\langle t\rangle^{\f{3s}{2}}}\left\langle \f{\pa_v}{t\pa_z}\right\rangle^{-1}\rmA\pa_z^{-1}\Delta_{L}P_{\neq }f\right\|_2^2
+\ep^2\|\rmA\langle \pa_v\rangle^2 h\|_2^2.
\end{align}
Thus by taking $\ep$ small enough, we proved Proposition \ref{prop: elliptic-4}. 
 
\subsection{Proof of Proposition \ref{prop: elliptic-2}}\label{Section_Proof_elliptic-2}

Recalling \eqref{Delta_L_phi}, we have 
\begin{align}
\nonumber&\Big\|\Big\langle \f{\pa_v}{t\pa_z}\Big\rangle^{-1} \f{|\na|^{\f{s}{2}}}{\langle t\rangle^{\f{3s}{2}}}\rmA\pa_z^{-1}\Delta^{2}_{L}P_{\neq}\phi\Big\|_2\\
\label{phi_Ellip_Main_2}&\lesssim \,\Big\|\f{|\na|^{\f{s}{2}}}{\langle t\rangle^{\f{3s}{2}}}\rmA P_{\neq}\pa_z^{-1}\Delta_L f\Big\|_2
+\sum_{i=1}^4 \Big\|\Big\langle \f{\pa_v}{t\pa_z}\Big\rangle^{-1} \f{|\na|^{\f{s}{2}}}{\langle t\rangle^{\f{3s}{2}}}\rmA P_{\neq}\pa_z^{-1}\mathcal{M}_{i,\phi}\Big\|_2\\
\nonumber&\lesssim\,\Big\|\Big\langle \f{\pa_v}{t\pa_z}\Big\rangle^{-1}\f{|\na|^{\f{s}{2}}}{\langle t\rangle^{\f{3s}{2}}}\rmA P_{\neq}\pa_z^{-1}\Delta_L f\Big\|_2 \\
\nonumber&\quad+ \sum_{i=1}^4 \Big\|\Big\langle \f{\pa_v}{t\pa_z}\Big\rangle^{-1} \f{|\na|^{\f{s}{2}}}{\langle t\rangle^{\f{3s}{2}}}\rmA P_{\neq}\pa_z^{-1}\Big(\mathcal{M}_{i,\phi;\rmH\rmL}+\mathcal{M}_{i,\phi;\rmL\rmH}+\mathcal{M}_{i,\phi;\rmH\rmH}\Big)\Big\|_2. 
\end{align}
Applying Proposition \ref{prop: elliptic-4}, we estimate the first term in \eqref{phi_Ellip_Main_2} as  
\begin{equation}\label{C_K_lambda}
\,\Big\|\f{|\na|^{\f{s}{2}}}{\langle t\rangle^{\f{3s}{2}}}\rmA P_{\neq}\pa_z^{-1}\Delta_L f\Big\|_2\lesssim C_1 (\rmCK_{\la,K}+\rmCK_{\la,\rho})+C\ep \rmCK_{\la,h}. 
\end{equation}

Now, we estimate the second term in \eqref{phi_Ellip_Main_2}. For the low-high terms and since on the support of the integrand, we have $|k,\eta|\approx |k,\xi|$ and $|\eta|\approx |\xi|$, then we can  move  the term $\Big\langle \f{\pa_v}{t\pa_z}\Big\rangle^{-1}\f{|\na|^{\f{s}{2}}}{\langle t\rangle^{\f{3s}{2}}}\rmA P_{\neq}$ to land on $\phi$ and get as in the proof of Proposition \ref{prop: elliptic-1}, and under the bootstrap assumption  
\begin{equation}
\begin{aligned}
\sum_{i=1}^4 \Big\|\Big\langle \f{\pa_v}{t\pa_z}\Big\rangle^{-1} \f{|\na|^{\f{s}{2}}}{\langle t\rangle^{\f{3s}{2}}}\rmA P_{\neq}\pa_z^{-1}\mathcal{M}_{i,\phi;\rmL\rmH}\Big\|_2\lesssim \ep\Big\|\Big\langle \f{\pa_v}{t\pa_z}\Big\rangle^{-1} \f{|\na|^{\f{s}{2}}}{\langle t\rangle^{\f{3s}{2}}}\rmA\pa_z^{-1}\Delta^{2}_{L}P_{\neq}\phi\Big\|_2.
\end{aligned}
\end{equation}

Next, we treat the high-low part. We use the decomposition 
\begin{equation}\label{f_CK_Terms_phi}
\rhat{\mathcal{M}}_{i,\phi;\rmL\rmH}=\rhat{\mathcal{M}}_{i,\phi;\rmH\rmL}^z+\rhat{\mathcal{M}}_{i,\phi;\rmH\rmL}^v,\quad i=1,\dots,4. 
\end{equation}
 The terms involving $\rhat{\mathcal{M}}_{i,\phi;\rmH\rmL}^z$ can be treated as the low-high part and we have 
\begin{equation}
\begin{aligned}
\sum_{i=1}^4 \Big\|\Big\langle \f{\pa_v}{t\pa_z}\Big\rangle^{-1} \f{|\na|^{\f{s}{2}}}{\langle t\rangle^{\f{3s}{2}}}\rmA P_{\neq}\pa_z^{-1}\mathcal{M}_{i,\phi;\rmL\rmH}^z\Big\|_2\lesssim \ep\Big\|\Big\langle \f{\pa_v}{t\pa_z}\Big\rangle^{-1} \f{|\na|^{\f{s}{2}}}{\langle t\rangle^{\f{3s}{2}}}\rmA\pa_z^{-1}\Delta^{2}_{L}P_{\neq}\phi\Big\|_2. 
\end{aligned}
\end{equation}
Now, we treat the term   involving 
$\rhat{\mathcal{M}}_{4,\phi;\rmH\rmL}^v$, which is the most challenging terms.  We have by using \eqref{J_J_0_Estimate},  with the fact that on the support of the integrand we have $|\eta|\approx \xi$ and $|k,\eta|\lesssim |\xi|$ and by making use of the bootstrap assumption 
\begin{equation}      
\begin{aligned}
&\Big|\f{|\na|^{\f{s}{2}}}{\langle t\rangle^{\f{3s}{2}}}\Big\langle \f{\pa_v}{t\pa_z}\Big\rangle^{-1} \rmA P_{\neq}\pa_{z}^{-1}\mathcal{M}_{4,\phi;\rmH\rmL}^v\Big|\\
\lesssim&\sum_{\rmM\geq 8}\int  \langle \xi\rangle^2 \f{|\xi|^{\f{s}{2}}}{\langle t\rangle^{\f{3s}{2}}}\mathbf{1}_{|k|< \f{1}{16}|\eta|}\rmA_0 (\xi)\langle t\rangle^3|  \widehat{G_1}(\xi)_{\rmM}||k,\eta-\xi| e^{c\la |k,\eta-\xi|^{s}}e^{C|k,\eta-\xi|^{\f13}}|\rhat{\phi}_k(\eta-\xi)_{<\rmM/8}|d\xi.  
\end{aligned}
\end{equation}
This yields since  
\begin{equation}
\begin{aligned}
\Big\|\f{|\na|^{\f{s}{2}}}{\langle t\rangle^{\f{3s}{2}}}\Big\langle \f{\pa_v}{t\pa_z}\Big\rangle^{-1} \rmA P_{\neq}\pa_{z}^{-1}\mathcal{M}_{4,\phi;\rmH\rmL}^v\Big\|_2^2\lesssim\ep^2 \rmCK_{\la,h}(t).  
\end{aligned}
\end{equation}
Similarly, we cam prove that  
\begin{equation}
\begin{aligned}
\sum_{i=1}^3\Big\|\f{|\na|^{\f{s}{2}}}{\langle t\rangle^{\f{3s}{2}}}\Big\langle \f{\pa_v}{t\pa_z}\Big\rangle^{-1} \rmA P_{\neq}\pa_{z}^{-1}\mathcal{M}_{i,\phi;\rmH\rmL}^v\Big\|_2^2\lesssim\ep^2 \rmCK_{\la,h}(t).  
\end{aligned}
\end{equation}
Hence, collecting all the above estimates, we obtain 
\begin{equation}\label{Estimate_Ck_Lambda_Phi}
\begin{aligned}
\Big\|\Big\langle \f{\pa_v}{t\pa_z}\Big\rangle^{-1} \f{|\na|^{\f{s}{2}}}{\langle t\rangle^{\f{3s}{2}}}\rmA\pa_z^{-1}\Delta^{2}_{L}P_{\neq}\phi\Big\|_2
\lesssim \,C_1 (\rmCK_{\la,K}+\rmCK_{\la,\rho})+C\ep \rmCK_{\la,h}.
\end{aligned}
\end{equation}
Next, we treat the $\Theta$-term on the left-hand side of \eqref{Main_elliptic_Estimate}. The $\rmg$-term is similar and we omit the details. We have by using \eqref{Delta_L_phi}, 
\begin{equation}\label{Estimate_CKw_Lambda_Phi}
\begin{aligned}
\left\|\left\langle\f{\pa_v}{t\pa_z}\right\rangle^{-1}\sqrt{\f{\pa_t\Theta}{\Theta}}\tilde{\rmA}\pa_z^{-1}\Delta_{L}^2 P_{\neq}\phi\right\|_2^2
\lesssim  &\,\bigg\|\sqrt{\f{\pa_t\Theta}{\Theta}}\left\langle \f{\pa_v}{t\pa_z}\right\rangle^{-1}\tilde{\rmA}\pa_z^{-1}\Delta_L P_{\neq} f\bigg\|_2^2\\
&+\sum_{i=1}^4\bigg\|\sqrt{\f{\pa_t\Theta}{\Theta}}\left\langle \f{\pa_v}{t\pa_z}\right\rangle^{-1}\tilde{\rmA} P_{\neq}\pa_{z}^{-1}\mathcal{M}_{i,\phi}\bigg\|_2^2. 
\end{aligned}
\end{equation}
The  first term on the right-hand side of   \eqref{Estimate_CKw_Lambda_Phi} has been already estimated in \eqref{f_B_estimate}.  

 Using the fact that $\tilde{\rmA}\lesssim \rmA$ together with  the bootstrap assumption, we have 
 \begin{equation}
\bigg\|\sqrt{\f{\pa_t\Theta}{\Theta}}\left\langle \f{\pa_v}{t\pa_z}\right\rangle^{-1}\tilde{\rmA} P_{\neq}\rho\bigg\|_2^2+\bigg\|\sqrt{\f{\pa_t\Theta}{\Theta}}\left\langle \f{\pa_v}{t\pa_z}\right\rangle^{-1}\tilde{\rmA}\pa_{z}^{-1}P_{\neq}K\bigg\|_2^2\lesssim C_2( \rmCK_{\Theta,\rho} +\rmCK_{\Theta,K}). 
\end{equation}

   By the fact that on the support of integrand $|\eta|^{\f23}\lesssim t\lesssim |\eta|\approx |\xi|$, 
 $\tilde{\rmA}\leq \rmA$ and $\sqrt{\f{\pa_t\Theta(t,\eta)}{\Theta(t,\eta)}}\lesssim 1$, we have by using the fact that   $1\leq\langle t\rangle^{2s} \f{|\eta|^s}{\langle t\rangle^{3s}}$,
 \begin{equation}
\begin{aligned}
&\sum_{i=1}^4\bigg\|\sqrt{\f{\pa_t\Theta}{\Theta}}\left\langle \f{\pa_v}{t\pa_z}\right\rangle^{-1}\tilde{\rmA} P_{\neq}\pa_{z}^{-1}\mathcal{M}_{i,\phi;\rmL\rmH}\bigg\|_2^2\\ &\lesssim \langle t\rangle^{2s}\Big(\sum_{j=0}^3\|\pa_v^jG_1\|_{\cG^{\la,1;s}}^{2}\Big) \left\|\f{|\na|^{\f s2}}{\langle t\rangle^{\f{3s}{2}}}\left\langle \f{\pa_v}{t\pa_z}\right\rangle^{-1}\rmA\pa_z^{-1}\Delta_{L}^2P_{\neq }\phi\right\|_2^2\\
&\lesssim \ep^{2} \left\|\f{|\na|^{\f s2}}{\langle t\rangle^{\f{3s}{2}}}\left\langle \f{\pa_v}{t\pa_z}\right\rangle^{-1}\rmA\pa_z^{-1}\Delta_{L}^2P_{\neq }\phi\right\|_2^2.
\end{aligned}  
\end{equation}

Similarly, we have  
\begin{equation}
\begin{aligned}
&\sum_{i=1}^4\bigg\|\sqrt{\f{\pa_t\Theta}{\Theta}}\left\langle \f{\pa_v}{t\pa_z}\right\rangle^{-1}\tilde{\rmA} P_{\neq}\pa_{z}^{-1}\mathcal{M}^z_{i,\phi;\rmH\rmL}\bigg\|_2^2 \lesssim \, \ep^{2} \left\|\f{|\na|^{\f s2}}{\langle t\rangle^{\f{3s}{2}}}\left\langle \f{\pa_v}{t\pa_z}\right\rangle^{-1}\rmA\pa_z^{-1}\Delta_{L}^2P_{\neq }\phi\right\|_2^2. 
\end{aligned}
\end{equation}

For the terms involving $\mathcal{M}^v_{i,\phi;\rmH\rmL}$, we treat the most problematic term which is the term  $\mathcal{M}^v_{4,\phi;\rmH\rmL}$ which contains the loss of three  derivatives.  The other terms can be treated by the same method (even easier). We have 
by the fact that $\f{\rmJ_k(t,\eta)}{\rmJ_0(\xi)}\lesssim \langle t\rangle e^{C|k,\eta-\xi|^{\f13}}$ and $|\eta-\xi-kt|\lesssim \langle k,\eta-\xi\rangle \langle t\rangle$
\begin{align*}
&\Big|\sqrt{\f{\pa_t\Theta}{\Theta}}\Big\langle \f{\pa_v}{t\pa_z}\Big\rangle^{-1} \tilde{\rmA} P_{\neq}\pa_{z}^{-1}\mathcal{M}_{4,\phi;\rmH\rmL}^v\Big|\\
&\lesssim\sum_{\rmM\geq 8}\int \Big\langle \f{\eta}{tk}\Big\rangle^{-1}\mathbf{1}_{|k|< \f{1}{16}|\eta|}\rmA_0 (\xi)|\xi|^3|\widehat{G_1}(\xi)_{\rmM}|\langle t\rangle\\
&\quad\times \langle k,\eta-\xi\rangle \langle t\rangle^2 e^{c\la |k,\eta-\xi|^s}e^{C|k,\eta-\xi|^{\f13}}|\rhat{\phi}_k(\eta-\xi)_{<\rmM/8}|d\xi\\
&\lesssim\sum_{\rmM\geq 8}\int \mathbf{1}_{|k|< \f{1}{16}|\eta|}\rmA_0 (\xi)|\xi|^2|\widehat{G_1}(\xi)_{\rmM}|\langle t\rangle\\
&\quad\times \langle k,\eta-\xi\rangle^2 \langle t\rangle^3 e^{c\la |k,\eta-\xi|^s}e^{C|k,\eta-\xi|^{\f13}}|\rhat{\phi}_k(\eta-\xi)_{<\rmM/8}|d\xi,
\end{align*}
which together with the fact that $\|P_{\neq}\phi\|_{\cG^{\la,0;s}}\lesssim \f{\ep}{\langle t\rangle^4}$, implies that 

\beno
\Big\|\sqrt{\f{\pa_t\Theta}{\Theta}}\Big\langle \f{\pa_v}{t\pa_z}\Big\rangle^{-1} \tilde{\rmA} P_{\neq}\pa_{z}^{-1}\mathcal{M}_{4,\phi;\rmH\rmL}^v\Big\|_2^2
\lesssim \|\rmA\langle \pa_v\rangle^2G_1\|_2^2\langle t\rangle^6\|P_{\neq}\phi\|_{\cG^{\la,0;s}}^2\lesssim \ep^2\|\rmA\langle \pa_v\rangle^2h\|_2^2. 
\eeno
Similarly we have 
\beno
\Big\|\sqrt{\f{\pa_t\Theta}{\Theta}}\Big\langle \f{\pa_v}{t\pa_z}\Big\rangle^{-1} \tilde{\rmA} P_{\neq}\pa_{z}^{-1}\mathcal{M}_{4,\phi;\rmH\rmH}^v\Big\|_2^2
\lesssim \ep^2\|\rmA\langle \pa_v\rangle^2h\|_2^2. 
\eeno
Thus we have proved the proposition.

\section{Estimate of $\rmNL_{\rho}$ and $\rmNL_{K}^1$}
\label{Section_Main_Nonlinear_terms}
This section is devoted to the proof of Proposition \ref{prop: NL} and hence we estimate $\rmNL_{\rho}$ and $\rmNL_{K}^1$. Let focus on $\rmNL_{\rho}$, the estimate of 
$\rmNL_{K}^1$ can be obtained by easily replacing $\rho$ by $K$. 

We have
\begin{equation}\label{Main_Identiy}
\begin{aligned}
\rmNL_{\rho}
&=\int \rmA\rho \big[\rmA({\bf{u}}\cdot\na\rho)-{\bf{u}}\cdot\na\rmA\rho\big]dzdv
-\f12\int \na\cdot {\bf{u}}|\rmA\rho|^2 dzdv\\
&=\rmNL_{1}-\f12\int \na\cdot {\bf{u}}|\rmA\rho|^2 dzdv.   
\end{aligned}
\end{equation}
The second term in \eqref{Main_Identiy} can be estimated as,   
\begin{equation}
\begin{aligned}
\Big|\int \na\cdot {\bf{u}}|\rmA\rho|^2 dzdv\Big|\leq \Vert \nabla \mathbf{u}\Vert_{L^\infty} \Vert \rmA\rho \Vert_{L^2}^2.  
\end{aligned}
\end{equation}
Recall the fact that 
\begin{equation}\label{u_formula}
\mathbf{u}(t,z,\upsilon)=(0,g)^T+v^\prime\nabla _{z,\upsilon }^{\perp }P_{\neq }\phi=(0,g)^T+h\nabla _{z,\upsilon }^{\perp }P_{\neq }\phi+\nabla _{z,\upsilon }^{\perp }P_{\neq }\phi,
\end{equation}
then we have 
\begin{equation}
\Vert  \na \mathbf{u} \Vert_{L^{\infty}}\lesssim (\Vert \pa_vg\Vert_{L^{\infty}}+(1+\Vert h\Vert_{H^2})\Vert P_{\neq} \phi \Vert_{H^3})\lesssim \f{\ep}{\langle t\rangle^{2}}. 
\end{equation}

To handle $\rmNL_{1}$, we use a paraproduct decomposition. Precisely, we define three main contributions: {\it transport} (low-high interaction),  {\it reaction} (high-low interaction) and a {\it remainder}:
\begin{align*}
\rmNL_{1}&=\int \rmA\rho \big[\rmA\big({\bf{u}}\cdot\na_{z,v}\rho\big)-{\bf{u}}\cdot\na_{z,v}\rmA\rho\big]dzdv\\
&=\f{1}{2\pi}\sum_{\rmN\geq 8}\rmT_{1;\rmN}
+\f{1}{2\pi}\sum_{\rmN\geq 8}\rmR_{1;\rmN}
+\f{1}{2\pi}\cR_{1},
\end{align*}
where
\begin{align*}
\rmT_{1;\rmN}&=2\pi\int \rmA\rho \big[\rmA\big({\bf{u}}_{<\rmN/8}\cdot\na_{z,v}\rho_{\rmN}\big)-{\bf{u}}_{<\rmN/8}\cdot\na_{z,v}\rmA\rho_{\rmN}\big]dzdv\\
\rmR_{1;\rmN}&=2\pi\int \rmA\rho \big[\rmA\big({\bf{u}}_{\rmN}\cdot\na_{z,v}\rho_{<\rmN/8}\big)-{\bf{u}}_{\rmN}\cdot\na_{z,v}\rmA\rho_{<\rmN/8}\big]dzdv\quad \\
\cR_{1}&=2\pi\sum_{\rmN\in \bbD}\sum_{\f18\rmN\leq \rmN'\leq 8\rmN}\int \rmA\rho \big[\rmA\big({\bf{u}}_{\rmN}\cdot\na_{z,v}\rho_{\rmN'}\big)-{\bf{u}}_{\rmN}\cdot\na_{z,v}\rmA\rho_{\rmN'}\big]dzdv. 
\end{align*}

 \subsection{Reaction term $\rmR_{1;\rmN}$} \label{Section_Reaction}
Recall \eqref{u_formula}, we write 
\begin{align*}
\rmR_{1;\rmN}&=\rmR_{1;\rmN}^{1}+\rmR_{1;\rmN}^{\ep,1}+\rmR_{1;\rmN}^{2}+\rmR_{1;\rmN}^3  
\end{align*}
where 
\begin{align*}
\rmR_{1;\rmN}^{1}&=\sum_{k,l\neq 0}\int_{\eta,\xi}\rmA \overline{\widehat{\rho}}_{k}(\eta)\rmA_{k}(\eta)(\eta l-\xi k)\hat{\phi}_{l}(\xi)_{\rmN}\widehat{\rho}_{k-l}(\eta-\xi)_{<\rmN/8}d\eta d\xi\\
\rmR_{1;\rmN}^{\ep,1}&=\sum_{k,l\neq 0}\int_{\eta,\xi}\rmA\overline{\widehat{\rho}}_{k}(\eta)\rmA_{k}(\eta)\left[\widehat{h\na^{\bot}\phi_l}\right](\xi)_{\rmN}\cdot\widehat{\na \rho}_{k-l}(\eta-\xi)_{<\rmN/8}d\eta d\xi\\
\rmR_{1;\rmN}^{2}&=-\sum_{k}\int_{\eta,\xi}\rmA\overline{\widehat{\rho}}_{k}(\eta)\rmA_{k}(\eta)\widehat{g}(\xi)_N\cdot\widehat{\pa_v \rho}_{k}(\eta-\xi)_{<\rmN/8}d\eta d\xi\\
\rmR_{1;\rmN}^{3}&=-\sum_{k,l}\int_{\eta,\xi}\rmA\overline{\hat{\rho}}_{k}(\eta)\rmA_{k-l}(\eta-\xi)\widehat{{\bf u}}_{l}(\xi)_{N}\widehat{\na\rho}_{k-l}(\eta-\xi)_{<N/8}d\eta d\xi.
\end{align*}
Here $\epsilon$  stands for the smallness of the term coefficient $h$.  

In this section we will prove the following propositions
\begin{proposition}
Under the bootstrap hypotheses, it holds that,
\begin{align*}
\sum_{\rmN\geq 8}\rmR_{1;\rmN}&\lesssim 
\f{\ep^{2}}{\langle t\rangle^2}+\ep\langle t\rangle \|\rmA\pa_{v}g\|_{2}^2+
\ep\rmCK_{\la,\rho}+\ep\rmCK_{\Theta,\rho}+\ep\rmCK_{M,\rho}\\
&\quad+\ep\left\|\left\langle\f{\pa_v}{t\pa_z}\right\rangle^{-1}\left(\f{|\na|^{\f{s}{2}}}{\langle t\rangle^{\f{3s}{2}}}\rmA+\sqrt{\f{\pa_t\rmg}{\rmg}}\tilde{\tilde{\rmA}}+\sqrt{\f{\pa_t\Theta}{\Theta}}\tilde{\rmA}\right)\pa_z^{-1}\Delta_{L}^2 P_{\neq}\phi\right\|_2^2.
\end{align*}
\end{proposition}

\subsubsection{Main contribution}
\label{Sec_Main_Cont}
The main contribution comes from $\rmR_{1;\rmN}^{1}$. We subdivide this integral depending on whether or not $(l,\xi)$ and/or $(k,\eta)$ are resonant as each combination requires a slightly different treatment.

Define the partition
\begin{equation}\label{Partition_Unity_Reaction}
1=\mathbf{1}_{t\notin \rmI_{k,\eta},t\notin \rmI_{l,\xi}}+\mathbf{1}_{t\notin \rmI_{k,\eta},t\in \rmI_{l,\xi}}+\mathbf{1}_{t\in \rmI_{k,\eta},t\notin \rmI_{l,\xi}}+\mathbf{1}_{t\in \rmI_{k,\eta},t\in \rmI_{l,\xi}}. 
\end{equation}
Correspondingly, denote
\begin{align*}
\rmR_{1;\rmN}^{1}&=\sum_{k,l\neq 0}\int_{\eta,\xi}[\mathbf{1}_{t\notin \rmI_{k,\eta},t\notin \rmI_{l,\xi}}+\mathbf{1}_{t\notin \rmI_{k,\eta},t\in \rmI_{l,\xi}}+\mathbf{1}_{t\in \rmI_{k,\eta},t\notin \rmI_{l,\xi}}+\mathbf{1}_{t\in \rmI_{k,\eta},t\in \rmI_{l,\xi}}]\\
&\qquad\qquad \times\rmA \overline{\widehat{\rho}}_{k}(\eta)\rmA_{k}(\eta)(\eta l-\xi k)\hat{\phi}_{l}(\xi)_{\rmN}\widehat{\rho}_{k-l}(\eta-\xi)_{<\rmN/8}d\eta d\xi\\
&=\rmR_{1;\rmN;\rmNR,\rmNR}^{1}
+\rmR_{1;\rmN;\rmNR,\rmR}^{1}
+\rmR_{1;\rmN;\rmR,\rmNR}^{1}
+\rmR_{1;\rmN;\rmR,\rmR}^{1}.
\end{align*}
On the support of the integrand of $\rmR_{1;\rmN}^{1}$, it holds that
\begin{subequations}\label{Support_Reaction}
\beq\label{eq: support}
\left||l,\xi|-|k,\eta|\right|\leq |k-l,\eta-\xi|\leq \f{6}{32}|l,\xi|.
\eeq
This implies that $|k,\eta|\approx |l,\xi|$ since $|l,\xi|\leq \frac{32}{28}|k,\eta|$ and
\begin{equation}\label{eq: support_2}
|k,\eta|\leq |k-l,\eta-\xi|+|l,\xi|\leq \frac{38}{32}|l,\xi|.
\end{equation}
\end{subequations}

\no{\bf Treatment of $\rmR_{1;\rmN;\rmNR,\rmNR}^{1}$}\\
We write first 
\begin{align*}
\rmR_{1;\rmN;\rmNR,\rmNR}^{1}=\sum_{k,l\neq 0}\int_{\eta,\xi}\mathbf{1}_{t\notin \rmI_{k,\eta},t\notin \rmI_{l,\xi}}\rmA \overline{\widehat{\rho}}_{k}(\eta)\rmA_{k}(\eta)\frac{i\  l(\eta l-\xi k)}{(l^2+(\xi-  lt)^2)^2}\widehat{\pa_z^{-1}\Delta_{L}^2\phi}_{l}(\xi)_{\rmN}\widehat{\rho}_{k-l}(\eta-\xi)_{<\rmN/8}d\eta d\xi.
\end{align*}

First, if $l\xi<0$, we do not  have resonances for positive times. In this case, we have 
\begin{align*}
&\f{|l||l,\xi|^{}}{(l^2+|\xi-lt|^2)^2}\approx \f{|l||l,\xi|^{}}{l^4\langle t\rangle^4+\xi^4}\lesssim \f{1}{\langle t\rangle^2}|k,\eta|^{\f s2}|l,\xi|^{\f s 2}\left\langle \f{\xi}{lt}\right\rangle^{-1},
\end{align*}
we have 
\beno
|\rmR_{1;\rmN;\rmNR,\rmNR}^{1}|+|\rmR_{1;\rmN;\rmR,\rmNR}^{1}|
\lesssim \f{1}{\langle t\rangle^{2}}\left\||\na|^{\f s 2}\rmA \rho_{\sim \rmN}\right\|_2
\left\|1_{\rmNR}\left\langle\f{\pa_v}{t\pa_z}\right\rangle^{-1}|\na|^{\f s 2}\pa_z^{-1}\Delta_{L}^2\rmA P_{\neq}\phi_{\rmN}\right\|_2\|\rho\|_{\mathcal{G}^{s}},
\eeno
which together with the bootstrap hypotheses implies that
\beq\label{eq:R-1}
\sum_{\rmN\geq 8}|\rmR_{1;\rmN;\rmNR,\rmNR}^{1}|+|\rmR_{1;\rmN;\rmR,\rmNR}^{1}|
\lesssim \ep\rmCK_{\la,\rho}
+\ep\left\|1_{\rmNR}\left\langle\f{\pa_v}{t\pa_z}\right\rangle^{-1}\f{|\na|^{\f s 2}}{\langle t\rangle^{\f{3s}{2}}}\pa_z^{-1}\Delta_{L}^2\rmA P_{\neq}\phi\right\|_2^2.
\eeq

Second, let us assume that  $l\xi>0$.  
Now we consider the following two cases: $|\xi|\geq 6|l|$ and $|\xi|<6|l|$ and in each case, we consider several sub-cases, depending on the time regime.  

\textbf{Case 1:} $|\xi|\geq 6|l|$, then it holds by \eqref{eq: support} that 
\begin{equation}\label{equiv_eta_xi}
|\eta-\xi|\leq \f{7}{32}|\xi|\qquad  \text{and} \qquad \f{25}{32}|\xi|\leq |\eta|\leq \f{39}{32}|\xi|.
\end{equation}
we obtain 
\begin{align*}
\f{\rmJ_k(\eta)}{\rmJ_l(\xi)}
&\leq \f{\Theta_l(t,\xi)}{\Theta_k(t,\eta)}e^{\mu |\eta-\xi|^{\f13}}+e^{\mu |k-l|^{\f13}}
\leq \f{\Theta_{\rmNR}(t,\xi)}{\Theta_{\rmNR}(t,\eta)}e^{\mu |\eta-\xi|^{\f13}}+e^{\mu |k-l|^{\f13}}.
\end{align*}
By Lemma \ref{lem:3.5}, we have
\ben\label{Mult_J_Ineq_0}
\f{\rmJ_k(\eta)}{\rmJ_l(\xi)}\lesssim e^{10\mu |k-l,\eta-\xi|^{\f13}}
\een
We also have 
\begin{equation}
\begin{aligned}\label{M_Multp_Ineq_0}
\f{\cM_k(\eta)}{\cM_l(\xi)}
\leq \f{\rmg(t,\xi)}{\rmg(t,\eta)}e^{4\pi\d_{L}^{-1} |\eta-\xi|^{\f13}}+e^{4\pi\d_{L}^{-1} |k-l|^{\f13}}
\lesssim e^{C\d_{L}^{-1} |k-l,\eta-\xi|^{\f13}}.
\end{aligned}
\end{equation}
By using \eqref{eq: support}, we have $|k,\eta|\approx |\xi,l|$, this means that we can freely interchange between $(k,\eta)$ and $(l,\xi)$ in the Sobolev correction as well as in the Gevrey part   in $\rmA_{k}(t,\eta)$. We get that
\begin{equation}
\begin{aligned}\label{R_NR_NR_Main_Term}
|\rmR_{1;\rmN;\rmNR,\rmNR}^1|
\lesssim&\, \sum_{k,l\neq 0}\int_{\eta,\xi}\mathbf{1}_{t\notin \rmI_{k,\eta},t\notin \rmI_{l,\xi}}|\rmA \overline{\widehat{\rho}}_{k}(\eta)|\f{\rmJ_k(\eta)}{\rmJ_l(\xi)}\f{\cM_k(\eta)}{\cM_l(\xi)}\rmA_{l}(\xi)\\
&\times \frac{|\xi|}{|l|^3(1+(\frac{\xi}{l}-  t)^2)^2}|\widehat{\pa_z^{-1}\Delta_{L}^2\phi}_{l}(\xi)_{\rmN}|
|\widehat{\na\rho}_{k-l}(\eta-\xi)_{<\rmN/8}|d\eta d\xi.
\end{aligned}
\end{equation}

Case 1.1: $1\leq t\leq \max\left\{\f{|\eta|}{2\rmE(|\eta|^{\f23})+1},\ \f{|\xi|}{2\rmE(|\xi|^{\f23})+1}\right\}$: We have $t\lesssim |\xi|^{\f13}$ and hence, we obtain for $|\xi|\leq 2lt$, 
\begin{align}
&\f{|\xi|}{l^3(1+(t-\f{\xi}{l})^2)^2}\lesssim 1\lesssim 
\left\{
\begin{aligned}\label{Inequality_1_0}
&\f{|\xi|^{s}}{\langle t\rangle^{3s}}\left\langle\f{\xi}{lt}\right\rangle^{-1}\quad \text{if} \quad |l|\geq \rmE(|\xi|^{\f13})+1\\
&\f{|\xi|^{s}}{\langle t\rangle^{3s}}\left\langle\f{\xi}{lt}\right\rangle^{-1}\quad \text{if} \quad 1\leq |l|\leq \rmE(|\xi|^{\f13}),
\end{aligned}
\right.
\end{align}
and for $\xi\geq 2lt$, then $\f{\xi}{l}\geq 2t$ which implies $|\f{\xi}{l}-t|\gtrsim \f{\xi}{l}$
\ben
\f{|\xi|}{l^3(1+(t-\f{\xi}{l})^2)^2}\lesssim \f{|l|}{|\xi|^3}\lesssim \f{|\xi|^{s}}{\langle t\rangle^{3s}}\left\langle\f{\xi}{lt}\right\rangle^{-1}
\een

Consequently, in this case, we get from \eqref{R_NR_NR_Main_Term}, 
\begin{equation}\label{Estimate_R_NR_NR_Case 1.1}
|\rmR_{1;\rmN;\rmNR,\rmNR}^1|
\lesssim\f{1}{\langle t\rangle^{3s}}\left\||\na|^{\f{s}{2}}\rmA \rho_{\sim \rmN}\right\|_2
\left\||\na|^{\f{s}{2}}1_{\rmNR}\left\langle\f{\pa_v}{t\pa_z}\right\rangle^{-1}\pa_z^{-1}\Delta_{L}^2\rmA P_{\neq}\phi_{\rmN}\right\|_2\|\rho\|_{\mathcal{G}^{s}}.  
\end{equation}

Case 1.2: $\max\left\{\f{|\eta|}{2\rmE(|\eta|^{\f23})+1},\ \f{|\xi|}{2\rmE(|\xi|^{\f23})+1}\right\}\leq t\leq \f{|\xi|}{2\rmE(|\xi|^{\f13})+1}$: 
In this case, Let 
 $j$ and $n$ be such that $t\in \bar{\rmI}_{j,\eta}\cap \bar{\rmI}_{n,\xi}$. Then we have form \eqref{equiv_eta_xi} that $|\xi|^{\f13}\approx |\eta|^{\f13}\lesssim t\lesssim |\xi|^{\f23}\approx |\eta|^{\f23}$ and 
if $|\xi|\geq 2lt$, we have $|\f{\xi}{l}-t|\gtrsim \f{\xi}{l}\gtrsim t$ which implies that
\ben
\f{|\xi|}{l^3(1+(t-\f{\xi}{l})^2)^2}\lesssim \f{l}{|\xi|^3}\lesssim \f{|\xi|^{s}}{\langle t\rangle^{3s}}\left\langle\f{\xi}{lt}\right\rangle^{-1},
\een
and if $|\xi|<2lt$, then 
\begin{align}\label{Inequality_Case_1_2}
\hspace{0.7cm}
&\f{|\xi|}{l^3(1+(t-\f{\xi}{l})^2)^2}\lesssim 
\left\{
\begin{aligned}
&\f{1}{\langle \xi\rangle}\lesssim \f{|\xi|^{s}}{\langle t\rangle^{3s}}\left\langle\f{\xi}{lt}\right\rangle^{-1}\quad \text{if} \quad |l|\geq \rmE(|\xi|^{\f23})+1\\
&\f{{|\xi|}/{n^3}}{(1+(t-\f{\xi}{n})^2)^{2}}\left\langle\f{\xi}{lt}\right\rangle^{-1}\quad \text{if} \quad\rmE(|\xi|^{\f13})+1\leq |l|=|n|\leq \rmE(|\xi|^{\f23})\\
&\f{1}{\sqrt{\xi}}\lesssim \f{|\xi|^{s}}{\langle t\rangle^{3s}}\left\langle\f{\xi}{lt}\right\rangle^{-1}
\quad \text{if}\quad  |l|\neq |n| \quad \text{and} \quad \left\{\begin{aligned} &\text{if}\quad \rmE(|\xi|^{\f13})+1\leq |l|\leq |\xi|^{\f12}\\
&\text{if}\quad  |\xi|^{\f12}+1\leq |l|\leq \rmE(|\xi|^{\f23})
\end{aligned}
\right.\\
&\f{1}{(1+(t-\f{\xi}{l})^2)^{\f32}}\lesssim \f{l^2}{\xi}\lesssim \f{|\xi|^{s}}{\langle t\rangle^{3s}}\left\langle\f{\xi}{lt}\right\rangle^{-1}\quad \text{if} \quad 1\leq |l|\leq \rmE(|\xi|^{\f13}).
\end{aligned}
\right.
\end{align}
By Lemma \ref{Lem:compare}, for the second case, we have if $|k|\approx |j|\approx |n|=|l|\gtrsim |\xi|^{\f13}$, $\eta\approx \xi$ then
\begin{align*}
\f{\xi/n^3}{(1+(t-\f{\xi}{n})^2)^{2}}\lesssim \f{\sqrt{\xi/n^3}}{1+|t-\f{\xi}{n}|}\f{\langle \xi-\eta\rangle\sqrt{\eta/j^3}}{1+|t-\f{\eta}{j}|}+\f{|\xi|^{s}}{\langle t\rangle^{3s}}\langle \xi-\eta\rangle^3,
\end{align*}
 In summary, we get that
\begin{align*}
|\rmR_{1;\rmN;\rmNR,\rmNR}^1|
&\lesssim \ep\left\|\f{|\na|^{\f{s}{2}}}{\langle t\rangle^{\f{3s}{2}}}\rmA \rho_{\sim \rmN}\right\|_2
\left\|1_{\rmNR}\f{|\na|^{\f{s}{2}}}{\langle t\rangle^{\f{3s}{2}}}\left\langle\f{\pa_v}{t\pa_z}\right\rangle^{-1}\pa_z^{-1}\pa_z^{-1}\Delta_{L}^2\rmA P_{\neq}\phi_{\rmN}\right\|_2\\
&\quad+\ep\left\|\sqrt{\f{\pa_t\rmg}{\rmg}}\tilde{\tilde{\rmA}} \rho_{\sim \rmN}\right\|_2
\left\|1_{\rmNR}\sqrt{\f{\pa_t\rmg}{\rmg}}\left\langle\f{\pa_v}{t\pa_z}\right\rangle^{-1}\pa_z^{-1}\pa_z^{-1}\Delta_{L}^2\tilde{\tilde{\rmA}}P_{\neq}\phi_{\rmN}\right\|_2.
\end{align*}
Here we use that fact that $|k|\leq |\eta|$, $|l|\leq |\xi|$ and then $\rmA\lesssim \tilde{\tilde{\rmA}}$. 

Case 1.2'(if possible): $\f{|\xi|}{2\rmE(|\xi|^{\f13})+1}\leq t\leq \max\left\{\f{|\xi|}{2\rmE(|\xi|^{\f13})+1},\f{|\eta|}{2\rmE(|\eta|^{\f13})+1}\right\}$: Let $j$ and $n$ be such that $t\in \bar{\rmI}_{j,\eta}\cap \bar{\rmI}_{n,\xi}$. Then we have $t\approx  |\xi|^{\f23}\approx |\eta|^{\f23}$, $n\approx j\approx |\xi|^{\f13}\approx |\eta|^{\f13}$ and 
\begin{align*}
&\f{|\xi|}{l^3(1+(t-\f{\xi}{l})^2)^2}\lesssim 
\left\{
\begin{aligned}
&\f{1}{\langle \xi\rangle}\lesssim \f{|\xi|^{s}}{\langle t\rangle^{3s}}\left\langle\f{\xi}{lt}\right\rangle^{-1}\quad \text{if} \quad |l|\geq \rmE(|\xi|^{\f23})+1\\
&\f{1}{\sqrt{\xi}}\lesssim \f{|\xi|^{s}}{\langle t\rangle^{3s}}\left\langle\f{\xi}{lt}\right\rangle^{-1}
\quad \text{if}\quad  |l|\neq |n| \quad \text{and} \quad \left\{\begin{aligned} &\text{if}\quad \rmE(|\xi|^{\f13})+1\leq |l|\leq |\xi|^{\f12}\\
&\text{if}\quad  |\xi|^{\f12}+1\leq |l|\leq \rmE(|\xi|^{\f23})
\end{aligned}
\right.\\
&\f{1}{(1+(t-\f{\xi}{n})^2)^{2}}\left\langle\f{\xi}{lt}\right\rangle^{-1}\quad \text{if} \quad |\xi|^{\f13}\approx |l|=|n|\leq \rmE(|\xi|^{\f13})\\
&\f{|l|^5}{|\xi|^3}\lesssim \f{|\xi|^{s}}{\langle t\rangle^{3s}}\left\langle\f{\xi}{lt}\right\rangle^{-1} \quad \text{if}\quad  |l|\leq \rmE(|\xi|^{\f13})\quad  \text{and}\quad |n|\neq |l|
\end{aligned}
\right.
\end{align*}
By Lemma \ref{Lem:compare}, for the third case, we have $|j|\approx |k|\approx |n|= |l|\approx  |\xi|^{\f13}$, $\eta\approx \xi$ and 
\begin{align*}
\f{1}{(1+(t-\f{\xi}{n})^2)^{2}}\lesssim \f{\sqrt{\xi/n^3}}{1+|t-\f{\xi}{n}|}\f{\langle \xi-\eta\rangle\sqrt{\eta/j^3}}{1+|t-\f{\eta}{j}|}+\f{|\xi|^{s}}{\langle t\rangle^{3s}}\langle \xi-\eta\rangle^3,
\end{align*}
which implies that 
\begin{align*}
|\rmR_{1;\rmN;\rmNR,\rmNR}^1|
&\lesssim \ep\left\|\f{|\na|^{\f{s}{2}}}{\langle t\rangle^{\f{3s}{2}}}\rmA \rho_{\sim \rmN}\right\|_2
\left\|1_{\rmNR}\f{|\na|^{\f{s}{2}}}{\langle t\rangle^{\f{3s}{2}}}\left\langle\f{\pa_v}{t\pa_z}\right\rangle^{-1}\pa_z^{-1}\Delta_{L}^2\rmA P_{\neq}\phi_{\rmN}\right\|_2\\
&\quad+\ep\left\|\sqrt{\f{\pa_t\rmg}{\rmg}}\tilde{\tilde{\rmA}} \rho_{\sim \rmN}\right\|_2
\left\|1_{\rmNR}\sqrt{\f{\pa_t\rmg}{\rmg}}\left\langle\f{\pa_v}{t\pa_z}\right\rangle^{-1}\pa_z^{-1}\Delta_{L}^2\tilde{\tilde{\rmA}}P_{\neq}\phi_{\rmN}\right\|_2.
\end{align*}

Case 1.3: $t\geq \min\{2|\eta|,2|\xi|\}$: 
We have 
$|l|t\gtrsim |lt-\xi|\gtrsim |l|t$, and
thus by the fact that $\rmg(t,\eta)=\rmg(t,\xi)$ and $\Theta_l(t,\xi)=\Theta_k(t,\eta)$, we obtain that
\begin{align*}
&|\rmR_{1;\rmN;\rmNR,\rmNR}^1|
\lesssim \left|\sum_{k,l\neq 0}\int_{\eta,\xi}1_{t\notin \rmI_{k,\eta},t\notin \rmI_{l,\xi}}\rmA \overline{\widehat{\rho}}_{k}(\eta)\rmA_{k}(\eta)(\eta l-\xi k)\hat{\phi}_{l}(\xi)_{N}\widehat{\rho}_{k-l}(\eta-\xi)_{<N/8}d\eta d\xi\right|\\
&\lesssim \sum_{k,l\neq 0}\int_{\eta,\xi}1_{t\notin \rmI_{k,\eta},t\notin \rmI_{l,\xi}}|\rmA \overline{\widehat{\rho}}_{k}(\eta)|\rmA_{l}(\xi)|\xi||\widehat{\phi}_{l}(\xi)_{N}|{e^{C\d_{L}^{-1}|k-l,\eta-\xi|^{\f13}}}|\widehat{\na\rho}_{k-l}(\eta-\xi)_{<N/8}|d\eta d\xi\\
&\lesssim \f{1}{\langle t\rangle^2}\left\||\na|^{\f s2}\rmA \rho_{\sim \rmN}\right\|_2  
\left\|1_{\rmNR}|\na|^{\f{s}{2}}\left\langle\f{\pa_v}{t\pa_z}\right\rangle^{-1}\pa_z^{-1}\Delta_{L}^2\rmA P_{\neq}\phi_{\rmN}\right\|_2\|\rho\|_{\mathcal{G}^{s}}.
\end{align*}

Case 1.4: $\max\left\{\f{|\eta|}{2\rmE(|\eta|^{\f13})+1},\ \f{|\xi|}{2\rmE(|\xi|^{\f13})+1}\right\}< t< \min\{2|\eta|,2|\xi|\}$:
Let $j$ and $n$ be such that $t\in \bar{\rmI}_{n,\eta}\cap \bar{\rmI}_{j,\xi}$ and we may first consider the case $|lt|\geq 2|\xi|$, then we have 
\beno
\f{|\xi|}{l^3(1+(t-\f{\xi}{l})^2)^2}\lesssim\f{|\xi||l|}{(lt-\xi)^4}\lesssim \f{1}{|l|^2|t|^2}\lesssim \f{|\xi|^{s}}{\langle t\rangle^{3s}}\left\langle\f{\xi}{lt}\right\rangle^{-1}.
\eeno
Then we focus on the case $|lt|<2|\xi|$, which implies $|l|\lesssim |\xi|^{\f13}$, if $|l|\geq \rmE(|\xi|^{\f13})+1$, then $l\neq j$, $\left|t-\f{\xi}{l}\right|\gtrsim \f{\xi}{l^2}\gtrsim |\xi|^{\f13}$ and thus
\beno
\f{|\xi|}{l^3(1+(t-\f{\xi}{l})^2)^2}\lesssim \f{1}{|\xi|^{\f43}}
\lesssim  \f{|\xi|^{s}}{\langle t\rangle^{3s}}\left\langle\f{\xi}{lt}\right\rangle^{-1}.
\eeno
Now we consider the case $|l|\leq \rmE(|\xi|^{\f13})$, and we have for $|l|\leq \f{1}{10}|j|$, 
\beno
\Big|\f{\xi}{l}-t\Big|\gtrsim \f{\xi}{l}\gtrsim \f{\xi}{j}\approx t,
\eeno
which implies that
\beno
\f{|\xi|}{l^3(1+(t-\f{\xi}{l})^2)^2}
\lesssim  \f{|\xi|^{s}}{\langle t\rangle^{3s}}\left\langle\f{\xi}{lt}\right\rangle^{-1}.
\eeno

For $|l|\geq \f{1}{10}|j|$, we have $|lt|\gtrsim |jt|\approx |\xi|$ and 
\beno
\left|t-\f{\xi}{l}\right|\geq \f{\xi}{l^3},\quad
\left|t-\f{\xi}{j}\right|\leq \left|t-\f{\xi}{l}\right|,\quad 
\left|t-\f{\eta}{n}\right|\leq \left|t-\f{\eta}{k}\right|,
\eeno
and
\begin{align*}
\f{|\xi|}{l^3(1+|t-\f{\xi}{l}|^2)^2}
&\lesssim \f{l^3}{\xi}\f{1}{(1+|t-\f{\xi}{l}|)^2}
\lesssim \f{1}{(1+|t-\f{\xi}{j}|)^2}\left\langle\f{\xi}{lt}\right\rangle^{-1}
\end{align*}
Case 1.3.1 $j=n$: We have
\begin{align*}
\f{1}{(1+|t-\f{\xi}{j}|)^2}\lesssim \f{1}{(1+|t-\f{\xi}{j}|)}\f{1}{(1+|t-\f{\eta}{j}|)}\langle \xi-\eta\rangle,
\end{align*}
which implies 
\begin{align*}
&|\rmR_{1;\rmN;\rmNR,\rmNR}^1|
\lesssim \left|\sum_{k,l\neq 0}\int_{\eta,\xi}1_{t\notin \rmI_{k,\eta},t\notin \rmI_{l,\xi}}\rmA \overline{\widehat{\rho}}_{k}(\eta)\rmA_{k}(\eta)(\eta l-\xi k)\hat{\phi}_{l}(\xi)_{\rmN}\widehat{\rho}_{k-l}(\eta-\xi)_{<\rmN/8}d\eta d\xi\right|\\
&\lesssim \sum_{k,l\neq 0}\int_{\eta,\xi}1_{t\notin \rmI_{k,\eta},t\notin \rmI_{l,\xi}}|\rmA \overline{\widehat{\rho}}_{k}(\eta)|\sqrt{\f{\pa_t\rmg(t,\eta)}{\rmg(t,\eta)}}\sqrt{\f{\pa_t\rmg(t,\xi)}{\rmg(t,\xi)}}\rmA_{l}(\xi)|\pa_z^{-1}\Delta_{L}^2\widehat{\phi}_{l}(\xi)_{\rmN}|\\
&\qquad\qquad\qquad\qquad\qquad\qquad \qquad\times e^{C\d_{L}^{-1}|k-l,\eta-\xi|^{\f13}}|\widehat{\na^2\rho}_{k-l}(\eta-\xi)_{<N/8}|d\eta d\xi\\
&\lesssim \left\|\sqrt{\f{\pa_t\rmg}{\rmg}}\tilde{\tilde{\rmA}} \rho_{\sim \rmN}\right\|_2\left\|\sqrt{\f{\pa_t\rmg}{\rmg}}\left\langle\f{\pa_v}{t\pa_z}\right\rangle^{-1}1_{\rmNR}\pa_z^{-1}\Delta_{L}^2\tilde{\tilde{\rmA}} P_{\neq}\phi_{\rmN}\right\|_2\|\rho\|_{\mathcal{G}^{s}}.  
\end{align*}

Case 1.3.2 $|t-\f{\xi}{j}|\gtrsim\f{\xi}{j^2}$ and $|t-\f{\eta}{n}|\gtrsim\f{\eta}{n^2}$: We have
\begin{align*}
\f{1}{(1+|t-\f{\xi}{j}|^2)}\lesssim \f{1}{(1+|t-\f{\xi}{j}|)}\f{1}{(1+|t-\f{\eta}{n}|)},
\end{align*}
which gives us that
\begin{align*}
|\rmR_{1;\rmN;\rmNR,\rmNR}^1|
\lesssim \left\|\sqrt{\f{\pa_t\rmg}{\rmg}}\tilde{\tilde{\rmA}} \rho_{\sim \rmN}\right\|_2\left\|\sqrt{\f{\pa_t\rmg}{\rmg}}\left\langle\f{\pa_v}{t\pa_z}\right\rangle^{-1}1_{\rmNR}\pa_z^{-1}\Delta_{L}^2\tilde{\tilde{\rmA}} P_{\neq}\phi_{\rmN}\right\|_2\|\rho\|_{\mathcal{G}^{s}}.
\end{align*}

Case 1.3.3 $|\xi-\eta|\gtrsim \f{|\eta|}{n}\approx t$: We have
\begin{align*}
\f{|\xi|}{l^3(1+|t-\f{\xi}{l}|)}\lesssim \f{\langle \xi-\eta\rangle^2}{\langle t\rangle^2},
\end{align*}
which gives us that
\begin{align*}
|\rmR_{1;\rmN;\rmNR,\rmNR}^1|
\lesssim \f{1}{\langle t\rangle^2}\left\|\rmA \rho_{\sim \rmN}\right\|_2
\left\|1_{\rmNR}\left\langle\f{\pa_v}{t\pa_z}\right\rangle^{-1}\pa_z^{-1}\Delta_{L}^{2}\rmA P_{\neq}\phi_{\rmN}\right\|_2\|\rho\|_{\mathcal{G}^{s}}. 
\end{align*}

\no \textbf{Case 2:} $|\xi|<6|l|$: We have
\beno
\f{|l|}{l^2+(\xi-lt)^2}\lesssim \f{1}{\langle t\rangle^2}
\eeno
which implies
\begin{align*}
&|\rmR_{1;\rmN;\rmNR,\rmNR}^1|
\lesssim \left|\sum_{k,l\neq 0}\int_{\eta,\xi}1_{t\notin \rmI_{k,\eta},t\notin \rmI_{l,\xi}}\rmA \overline{\widehat{\rho}}_{k}(\eta)\rmA_{k}(\eta)(\eta l-\xi k)\hat{\phi}_{l}(\xi)_{N}\widehat{\rho}_{k-l}(\eta-\xi)_{<N/8}d\eta d\xi\right|\\
&\lesssim \sum_{k,l\neq 0}\int_{\eta,\xi}1_{t\notin \rmI_{k,\eta},t\notin \rmI_{l,\xi}}|\rmA \overline{\widehat{\rho}}_{k}(\eta)|\rmA_{l}(\xi)|l||\widehat{\phi}_{l}(\xi)_{N}|e^{C\d_{L}^{-1}|k-l,\eta-\xi|^{\f13}}|\widehat{\na\rho}_{k-l}(\eta-\xi)_{<N/8}|d\eta d\xi\\
&\lesssim \f{1}{\langle t\rangle^2}\left\||\na|^{\f s 2}\rmA \rho_{\sim \rmN}\right\|_2
\left\|\left\langle\f{\pa_v}{t\pa_z}\right\rangle^{-1}1_{\rmNR}|\na|^{\f s 2}\pa_{z}^{-1}\Delta_{L}^2\rmA P_{\neq}\phi_{\rmN}\right\|_2\|\rho\|_{\mathcal{G}^{s}}.
\end{align*} 

\no{\bf Treatment of $\rmR_{1;\rmN;\rmR,\rmNR}^{1}$}\\
We have 
\begin{align*}
|\rmR_{1;\rmN;\rmR,\rmNR}^{1}|\lesssim &\,\sum_{k,l\neq 0}\int_{\eta,\xi} \mathbf{1}_{t\in \rmI_{k,\eta},t\notin \rmI_{l,\xi}}|\rmA \overline{\widehat{\rho}}_{k}(\eta)\rmA_{k}(\eta)||\eta l-\xi k||\hat{\phi}_{l}(\xi)_{N}\widehat{\rho}_{k-l}(\eta-\xi)_{<N/8}|d\eta d\xi\\
\lesssim &\,\sum_{k,l\neq 0}\int_{\eta,\xi} \mathbf{1}_{t\in \rmI_{k,\eta},t\notin \rmI_{l,\xi}}\rmA \overline{\widehat{\rho}}_{k}(\eta)\f{\rmJ_k(t,\eta)}{\rmJ_l(t,\xi)}\f{\cM(t,\eta)}{\cM(t,\xi)}\\
&\times\rmA_{l}(\xi)|\eta l-\xi k|\frac{|\xi|}{|l|^3(1+(\frac{\xi}{l}-  t)^2)^2}|\widehat{\pa_z^{-1}\Delta_{L}^2\phi}_{l}(\xi)_{N}||\widehat{\rho}_{k-l}(\eta-\xi)_{<N/8}|d\eta d\xi.
\end{align*}
Note that for $t\in \rmI_{k,\eta}$, we have 
\begin{equation}
\Theta_{\rmR}(t,\eta)\approx \Theta_{\rmN\rmR}(t,\eta)\Big[\frac{k^3}{|\eta|}\Big(1+|t-\frac{|\eta|}{|k|}|\Big)\Big], 
\end{equation}
thus we obtain that
\begin{align*}
\f{\rmJ(t,\eta)}{\rmJ(t,\xi)}\leq&\, \frac{\Theta_{\rmN\rmR}(t,\xi)}{\Theta_{\rmR}(t,\eta)}e^{\mu|\eta-\xi|^{1/3}}+e^{\mu |k-l|^{1/3}}\\
\leq&\, \frac{\Theta_{\rmN\rmR}(t,\xi)}{\Theta_{\rmN\rmR} (t,\eta)}\frac{\Theta_{\rmN\rmR}(t,\eta)}{\Theta_{\rmR}(t,\eta)} e^{\mu|\eta-\xi|^{1/3}}+e^{\mu |k-l|^{1/3}}\\
\lesssim&\, e^{C\mu|\eta-\xi|^{1/3}}\frac{|\eta|}{|k|^3(1+|t-\frac{|\eta|}{|k|}|)}+e^{\mu |k-l|^{1/3}}\\
\lesssim&\,\frac{|\eta|}{|k|^3(1+|t-\frac{|\eta|}{|k|}|)} e^{C\mu|\eta-\xi,k-l|^{1/3}} 
\end{align*}
Also using the fact that 
\begin{equation}
\f{\cM_k(\eta)}{\cM_l(\xi)}\lesssim e^{C\d_{L}^{-1} |k-l,\eta-\xi|^{\f13}} 
\end{equation}
Hence, we obtain that 
\begin{equation}\label{R_1_R_NR_Main_term}
\begin{aligned}
|\rmR_{1;\rmN;\rmR,\rmNR}^{1}|\lesssim& \,\sum_{k,l\neq 0}\int_{\eta,\xi} \mathbf{1}_{t\in \rmI_{k,\eta},t\notin \rmI_{l,\xi}}
 \rmA \overline{\widehat{\rho}}_{k}(\eta)
\rmA_{l}(\xi)\frac{|\xi|}{|l|^3(1+(\frac{\xi}{l}-  t)^2)^2}\frac{|\eta|}{|k|^3(1+|t-\frac{|\eta|}{|k|}|)}\\ &\times |\widehat{\pa_z^{-1}\Delta_{L}^2\phi}_{l}(\xi)_{N}|
|k-l,\eta-\xi|e^{c\la|\eta-\xi,k-l|^{s}}|\widehat{\rho}_{k-l}(\eta-\xi)_{<N/8}|d\eta d\xi. 
\end{aligned}
\end{equation}
with $c\in (0,1)$ and $s>\f13$. 
Our goal now is to estimate the symbol  
\begin{equation}\label{Factor_Reson}
\frac{|\xi|}{|l|^3(1+(\frac{\xi}{l}-  t)^2)^2}\frac{|\eta|}{|k|^3(1+|t-\frac{|\eta|}{|k|}|)}
\end{equation}
in different time regimes in order to  absorb the large  factor 
\begin{equation}\label{Prob_term}
\frac{|\eta|}{|k|^3(1+|t-\frac{|\eta|}{|k|}|)}.
\end{equation}

Case 1. First, for $lt\geq 2|\xi|$ {and $|l|\geq 1$}, it holds that 
\begin{equation}\label{t_2xi}
\Big|t-\frac{\xi}{l}\Big|\gtrsim \f{|\xi|}{l}\qquad\text{and}\qquad \Big|t-\frac{\xi}{l}\Big|\geq  t-\Big|\frac{\xi}{l}\Big|\geq t-\frac{t}{2l}\gtrsim t. 
\end{equation}
Hence, keeping in mind the fact that that $|\xi|\approx |\eta|$,  we estimate the factor in \eqref{Factor_Reson} as 
\begin{equation}
\begin{aligned}
\frac{|\xi|}{|l|^3(1+(\frac{\xi}{l}-  t)^2)^2}\frac{|\eta|}{|k|^3(1+|t-\frac{|\eta|}{|k|}|)}\lesssim \frac{1}{1+(\frac{\xi}{l}-  t)^2}\lesssim \frac{1}{\langle t \rangle^2}\left\langle\f{\xi}{lt}\right\rangle^{-1}. 
\end{aligned}
\end{equation}
Consequently, it holds from \eqref{R_1_R_NR_Main_term}, that
\begin{equation}\label{R_R_NR_First_Estimate}
|\rmR_{1;\rmN;\rmR,\rmNR}^{1}|\lesssim  \f{1}{\langle t\rangle^2}\left\||\na|^{\f s 2}\rmA \rho_{\sim \rmN}\right\|_2
\left\||\na|^{\f s 2}\left\langle \f{\pa_v}{t\pa_z}\right\rangle^{-1}\pa_z^{-1}\Delta_{L}^2\rmA P_{\neq}\phi_{\rmN}\right\|_2\|\rho\|_{\mathcal{G}^{s}}. 
\end{equation}

Now we always assume that $|l|t\leq 2|\xi|$. 

 Case 2. Now, for $\f{2|\xi|}{2\rmE(|\xi|^{\f23})+1}< t< 2|\xi|$:
then there exists $n$ such that $t\in \bar{\rmI}_{n,\xi}$. Hence, we have $t\in \rmI_{k,\eta}\cap \bar{\rmI}_{n,\xi}\subset \bar{\rmI}_{k,\eta}\cap \bar{\rmI}_{n,\xi}$, $|k|\approx |n|\lesssim |\xi|^{\f13}$ and $\left|t-\f{\xi}{n}\right|\lesssim \left|t-\f{\xi}{l}\right|$, $\Big|t- \f{\xi}{n}\Big|\leq\f{\xi}{n^2}$. 

If $|l|\leq \f{1}{10}|n|\lesssim |\xi|^{\f13}$, then $\Big|\f{\xi}{l}-t\Big|\gtrsim \f{\xi}{l}\gtrsim \f{\xi}{n}\approx t\approx \f{\eta}{k}$, thus we have
\beno
\frac{|\xi|}{|l|^3(1+(\frac{\xi}{l}-  t)^2)^2}\frac{|\eta|}{|k|^3(1+|t-\frac{|\eta|}{|k|}|)}\lesssim \f{|\xi|^{s}}{\langle t\rangle^{3s}}\left\langle\f{\xi}{lt}\right\rangle^{-1},
\eeno
which implies that
\beno
|\rmR_{1;\rmN;\rmR,\rmNR}^{1}|\lesssim  \left\|\f{|\na|^{\f s 2}}{\langle t\rangle^{3s}}\rmA \rho_{\sim \rmN}\right\|_2
\left\|\f{|\na|^{\f s 2}}{\langle t\rangle^{3s}}\left\langle \f{\pa_v}{t\pa_z}\right\rangle^{-1}\pa_z^{-1}\Delta_{L}^2\rmA P_{\neq}\phi_{\rmN}\right\|_2\|\rho\|_{\mathcal{G}^{s}}. 
\eeno
Now we focus on the case $|l|\geq \f{1}{10}|n|$, thus $|lt|\gtrsim nt\approx |\xi|$ which gives us that
\beno
\left\langle\f{\xi}{lt}\right\rangle^{-1}\approx 1. 
\eeno
We are in a position to apply  Lemma \ref{Lem:compare}.  

Case 2.1 $k=n$: Since $t\notin \rmI_{l,\xi}$ we have the following two cases:
\beno
|l|\geq \rmE(|\xi|^{\f13})+1\ \text{ and } \ |l|\leq \rmE(|\xi|^{\f13}),\quad \text{ {with}}\quad\ \Big|t-\f{\xi}{l}\Big|\geq\f{\xi}{2l^3}.
\eeno

For the first one, if $t\leq \f{2|\xi|}{2\rmE(|\xi|^{\f13})+1}$, then by the fact that $t\in \rmI_{k,\eta}\cap \bar{\rmI}_{k,\xi}$, {then $t\approx |\eta|^{\f23}$, hence  }
we get $|k|\approx |\eta|^{\f13}$; and 
if $t\geq \f{2|\xi|}{2\rmE(|\xi|^{\f13})+1}$ then $|k|=|n|\leq \rmE(|\xi|^{\f13})$, $\left|t-\f{\xi}{l}\right|\geq \f{\xi}{l^2}\geq |\xi|^{\f13}$. Therefore in both sub-cases, we get that
\begin{equation}
\begin{aligned}
\frac{|\xi|}{|l|^3(1+(\frac{\xi}{l}-  t)^2)^2}\frac{|\eta|}{|k|^3(1+|t-\frac{|\eta|}{|k|}|)}\lesssim &\,\frac{1}{1+|t-\frac{\xi}{n}|}\frac{1}{1+|t-\f{\eta}{k}|}\\
\lesssim  &\, \sqrt{\frac{\partial_t \rmg (t,\xi)}{\rmg(t,\xi)}}\sqrt{\frac{\partial_t \rmg (t,\eta)}{\rmg(t,\eta)}}. 
\end{aligned}
\end{equation}

For the second one, 
and since $t\notin \rmI_{l,\xi}$, then  if $|l|=|n|$ which means that $t\in \bar{\rmI}_{n,\xi}\setminus \rmI_{n,\xi}$, recall that $|k|=|n|$, then we get $\left|t-\f{\xi}{l}\right|\geq\f{\xi}{2k^3}$; and if $|l|\neq |n|$ thus $\left|t-\f{\xi}{l}\right|\gtrsim \f{\xi}{l^2}$. Therefore in both sub-cases, we get that
\begin{equation}
\begin{aligned}
\frac{|\xi|}{|l|^3(1+(\frac{\xi}{l}-  t)^2)^2}\frac{|\eta|}{|k|^3(1+|t-\frac{\eta}{k}|)} \lesssim &\,
\frac{1}{1+|t-\frac{\xi}{n}|}\frac{1}{1+|t-\f{\eta}{k}|}
\\
\lesssim  &\, \sqrt{\frac{\partial_t \rmg (t,\xi)}{\rmg(t,\xi)}}\sqrt{\frac{\partial_t \rmg (t,\eta)}{\rmg(t,\eta)}}.
\end{aligned}
\end{equation}

Case 2.2 $|t-\f{\xi}{n}|\gtrsim \f{\xi}{n^2}$ and $|t-\f{\eta}{k}|\gtrsim \f{\xi}{k^2}$: We have by using the inequality $|t-\f{\xi}{n}|\lesssim |t-\f{\xi}{l}|$, 
\begin{equation}
\begin{aligned}
\frac{|\xi|}{|l|^3(1+(\frac{\xi}{l}-  t)^2)^2}\frac{|\eta|}{|k|^3(1+|t-\frac{\eta}{k}|)}\lesssim  &\,
\frac{1}{1+|t-\frac{\xi}{n}|}\frac{1}{1+|t-\f{\eta}{n}|}\\
\lesssim  &\, \sqrt{\frac{\partial_t \rmg (t,\xi)}{\rmg(t,\xi)}}\sqrt{\frac{\partial_t \rmg (t,\eta)}{\rmg(t,\eta)}}
\end{aligned}
\end{equation}

Case 2.3 $|\xi-\eta|\gtrsim \f{\xi}{n}\gtrsim |\xi|^{\f23}\gtrsim \langle t\rangle^{\f23}$, 
where we have used the fact that $|k|\approx |n|\leq |\xi|^{1/3}$ and $t\leq 2|\xi|$ in the above time regime. Hence, 
 we have
\begin{equation}
\frac{|\xi|}{|l|^3(1+(\frac{\xi}{l}- t)^2)^2}\frac{|\eta|}{|k|^3(1+|t-\frac{\eta}{k}|)}
\lesssim \langle \eta-\xi \rangle^{2}\sqrt{\frac{\partial_t \rmg (t,\xi)}{\rmg(t,\xi)}}\sqrt{\frac{\partial_t \rmg (t,\eta)}{\rmg(t,\eta)}},
\end{equation}
Consequently, in all the above cases, we have the following estimate:
\begin{equation}\label{Second_Estimate_R_NR}
|\rmR_{1;\rmN;\rmR,\rmNR}^{1}|\lesssim \left\|\sqrt{\f{\pa_t\rmg}{\rmg}}\tilde{\tilde{\rmA}} \rho_{\sim \rmN}\right\|_2\left\|\sqrt{\f{\pa_t\rmg}{\rmg}}\left\langle \f{\pa_v}{t\pa_z}\right\rangle^{-1}\pa_z^{-1}1_{\rmNR}\pa_z^{-1}\Delta_{L}^2\tilde{\tilde{\rmA}} P_{\neq}\phi_{\rmN}\right\|_2\|\rho\|_{\mathcal{G}^{s}}
\end{equation}

\no{\bf Treatment of $\rmR_{1;\rmN;\rmNR,\rmR}^{1}$}\\
By Lemma \ref{lem:g/g} and Lemma \ref{lem:3.5}, it holds that $1\leq |l|\leq \rmE(|\xi|^{\f13})$ and 
\begin{align*}
|\rmR_{1;\rmN;\rmNR,\rmR}^{1}|\lesssim &\,\sum_{k,l\neq 0}\int_{\eta,\xi} \mathbf{1}_{t\notin \rmI_{k,\eta},t\in \rmI_{l,\xi}}|\rmA \overline{\widehat{\rho}}_{k}(\eta)\rmA_{k}(\eta)||\xi||\hat{\phi}_{l}(\xi)_{N}\widehat{\na\rho}_{k-l}(\eta-\xi)_{<N/8}|d\eta d\xi\\
\lesssim &\sum_{k,l\neq 0}\int_{\eta,\xi} \mathbf{1}_{t\notin \rmI_{k,\eta},t\in \rmI_{l,\xi}}|\rmA \overline{\widehat{\rho}}_{k}(\eta)\f{\rmJ(t,\eta)}{\rmJ(t,\xi)}\f{\cM(t,\eta)}{\cM(t,\xi)}\rmA_{l}(\xi)||\xi||\hat{\phi}_{l}(\xi)_{N}\widehat{\na\rho}_{k-l}(\eta-\xi)_{<N/8}|d\eta d\xi\\
\lesssim &\sum_{k,l\neq 0}\int_{\eta,\xi} \mathbf{1}_{t\notin \rmI_{k,\eta},t\in \rmI_{l,\xi}}|\rmA \overline{\widehat{\rho}}_{k}(\eta)\f{|l|^3}{|\xi|}(1+|t-\f{\xi}{l}|)\rmA_{l}(\xi)||\xi||\hat{\phi}_{l}(\xi)_{N}\\
&\qquad\qquad\times e^{C\d_{\rmL}|k-l,\xi-\eta|^{\f13}}\widehat{\na\rho}_{k-l}(\eta-\xi)_{<N/8}|d\eta d\xi.
\end{align*}
By the fact that \eqref{eq: support}, we have $t\geq \f{\eta}{2\rmE(|\eta|^{\f23})+1}$. Suppose $\f{3}{2}|\xi|\geq t> 2|\eta|$, then 
\beno
|\xi-\eta|\geq |\xi|-\f{3}{4}|\xi|\geq \f14|\xi|\approx t,
\eeno 
which implies that 
\beno
|\rmR_{1;\rmN;\rmNR,\rmR}^{1}|\lesssim \f{1}{\langle t\rangle^2}\| |\na|^{\f s 2}\rmA\rho_{\sim \rmN}\|_2
\left\|\left\langle\f{\pa_v}{t\pa_z}\right\rangle^{-1}|\na|^{\f s 2}\pa_z^{-1}\Delta_{L}^2\rmA P_{\neq}\phi_{\rmN}\right\|_2\|\rho\|_{\mathcal{G}^s}.
\eeno
Note that in this case it holds that 
\beno
\left\langle\f{\xi}{lt}\right\rangle^{-1}\approx 1
\eeno
Now let focus on the case $t\leq 2|\eta|$. 
Let $n$ be such that $t\in \bar{\rmI}_{n,\eta}\cap \rmI_{l,\xi}\subset \bar{\rmI}_{n,\eta}\cap \bar{\rmI}_{l,\xi}$, we have the following cases:

Case $l=n$: We have $(1+|t-\f{\eta}{n}|)\lesssim (1+|t-\f{\xi}{l}|)\langle \xi-\eta\rangle$.  

Case $|t-\f{\eta}{n}|\gtrsim \f{\eta}{n^2}$ and $|t-\f{\xi}{l}|\gtrsim \f{\xi}{l^2}$: Then by the fact $t\in \rmI_{l,\xi}$ we have $|n|\approx |l|\approx 1$. Thus it still holds that 
$(1+|t-\f{\eta}{n}|)\lesssim (1+|t-\f{\xi}{l}|)\langle \xi-\eta\rangle$.

Therefore, we obtain that
\begin{align*}
|\rmR_{1;\rmN;\rmNR,\rmR}^{1}|
\lesssim &\sum_{k,l\neq 0}\int_{\eta,\xi} \mathbf{1}_{t\notin \rmI_{k,\eta},t\in \rmI_{l,\xi}}\f{|\rmA \overline{\widehat{\rho}}_{k}(\eta)|}{1+|t-\f{\eta}{n}|}\f{|l|^3(1+|t-\f{\xi}{l}|)^3\rmA_{l}(\xi)|\hat{\phi}_{l}(\xi)_{N}|}{1+|t-\f{\xi}{l}|}\\
&\qquad\qquad\times e^{c\la|k-l,\xi-\eta|^{s}}\widehat{\na^2\rho}_{k-l}(\eta-\xi)_{<N/8}|d\eta d\xi\\
&\lesssim \left\|\sqrt{\f{\pa_t\rmg}{\rmg}}\tilde{\tilde{\rmA}}\rho_{\sim\rmN}\right\|_2\left\|\left\langle \f{\pa_v}{t\pa_z}\right\rangle^{-1}\sqrt{\f{\pa_t\Theta}{\Theta}}1_{\rmR}\tilde{\rmA} \pa_z^{-1}\Delta_{L}^2P_{\neq}\phi_{\rmN}\right\|_2\|\rho\|_{\mathcal{G}^s}.
\end{align*}

Case $|\xi-\eta|\gtrsim \f{\xi}{l}$: We have $|\xi-\eta|\gtrsim |\xi|^{\f23}\gtrsim \langle t\rangle^{\f23}$ and thus 
\beno
|\rmR_{1;\rmN;\rmNR,\rmR}^{1}|\lesssim \f{1}{\langle t\rangle^2}\||\na|^{\f s 2}\rmA\rho_{\sim \rmN}\|_2\left\|1_{\rmR}\left\langle \f{\pa_v}{t\pa_z}\right\rangle^{-1}|\na|^{\f s 2}\pa_z^{-1}\Delta_{L}^2\rmA P_{\neq}\phi_{\rmN}\right\|_2\|\rho\|_{\mathcal{G}^s}.
\eeno

\no{\bf Treatment of $\rmR_{1;\rmN;\rmR,\rmR}^{1}$}\\
By Lemma \ref{Lem:compare}, we consider the following three cases:

Case 1: $k=l$, then 
\begin{align*}
|\rmR_{1;\rmN;\rmR,\rmR}^{1}|&=\left|\sum_{l\neq 0}\int_{\eta,\xi}\mathbf{1}_{t\in \rmI_{l,\eta},t\in \rmI_{l,\xi}}
\rmA \overline{\widehat{\rho}}_{k}(\eta)\rmA_{k}(\eta)(\eta l-\xi l)\hat{\phi}_{l}(\xi)_{N}\widehat{\rho}_{0}(\eta-\xi)_{<N/8}d\eta d\xi\right|\\
&\lesssim \sum_{l\neq 0}\int_{\eta,\xi}\mathbf{1}_{t\in \rmI_{l,\eta},t\in \rmI_{l,\xi}}
\rmA |\overline{\widehat{\rho}}_{k}(\eta)|\f{\rmJ(t,\eta)}{\rmJ(t,\xi)}\f{\cM(t,\eta)}{\cM(t,\xi)}\rmA_{l}(\xi)|l||\hat{\phi}_{l}(\xi)_{N}||\widehat{\na \rho}_{0}(\eta-\xi)_{<N/8}|d\eta d\xi
\end{align*}
By Lemma \ref{lem:g/g} and the fact that 
\begin{align}\label{eq:J_R/J_R}\nonumber
\f{\rmJ(t,\eta)}{\rmJ(t,\xi)}&\lesssim \f{\Theta_{\rmNR}(t,\xi)}{\Theta_{\rmNR}(t,\eta)}\f{\Theta_{\rmR}(t,\xi)}{\Theta_{\rmNR}(t,\xi)}\f{\Theta_{\rmNR}(t,\eta)}{\Theta_{\rmR}(t,\eta)}e^{\mu|\eta-\xi|^{\f13}}\\
&\lesssim e^{C\mu|\eta-\xi|^{\f13}}\f{1+|t-\f{\xi}{l}|}{1+|t-\f{\eta}{l}|}\lesssim e^{C\mu|\eta-\xi|^{\f13}}
\end{align}
and that
\ben
\f{1}{1+|t-\f{\xi}{l}|}\lesssim \f{\langle \xi-\eta\rangle}{1+|t-\f{\eta}{l}|},
\een
we obtain that 
\begin{align*}
|\rmR_{1;\rmN;\rmR,\rmR}^{1}|
\lesssim \left\|\sqrt{\f{\pa_t\rmg}{\rmg}}\tilde{\tilde{\rmA}} \rho_{\sim \rmN}\right\|_2\left\|\left\langle \f{\pa_v}{t\pa_z}\right\rangle^{-1}\sqrt{\f{\pa_t\rmg}{\rmg}}1_{\rmR}\pa_z^{-1}\Delta_{L}^2\tilde{\tilde{\rmA}} P_{\neq}\phi_{\rmN}\right\|_2\|\rho\|_{\mathcal{G}^{s}}.
\end{align*}

Case 2: $|t-\f{\xi}{l}|\gtrsim \f{\xi}{l^3}$ and $|t-\f{\eta}{k}|\gtrsim \f{\eta}{k^3}$: By Lemma \ref{lem:g/g} and \eqref{eq:J_R/J_R} and the fact that 
\beno
\f{\xi}{l^3(1+|t-\f{\xi}{l}|^4)}\lesssim \f{1}{(1+|t-\f{\xi}{l}|^3)}\lesssim \f{{l^3}/{\xi}}{(1+|t-\f{\xi}{l}|)} \f{{\eta}/{k^3}}{(1+|t-\f{\eta}{k}|)}\lesssim \sqrt{\f{\pa_t\rmg(t,\xi)}{\rmg(t,\xi)}}\sqrt{\f{\pa_t\rmg(t,\eta)}{\rmg(t,\eta)}}.
\eeno
we obtain that 
\begin{align*}
&|\rmR_{1;\rmN;\rmR,\rmR}^1|
\lesssim \left\|\sqrt{\f{\pa_t\rmg}{\rmg}}\tilde{\tilde{\rmA}} \rho_{\sim \rmN}\right\|_2\left\|\sqrt{\f{\pa_t\rmg}{\rmg}}\left\langle \f{\pa_v}{t\pa_z}\right\rangle^{-1}1_{\rmR}\pa_z^{-1}\Delta_{L}^2\tilde{\tilde{\rmA}} P_{\neq}\phi_{\rmN}\right\|_2\|\rho\|_{\mathcal{G}^{s}}
\end{align*}

Case 3: $|\xi-\eta|\gtrsim \f{|\xi|}{l^2}$: We have $|l|\lesssim |\xi|^{\f13}$, $|\xi|^{\f23}\lesssim \langle t\rangle \lesssim |\xi|$ which gives us that  
\beno
|\xi-\eta|\gtrsim \langle t\rangle^{\f13}. 
\eeno
Thus we get that
\begin{align*}
|\rmR_{1;\rmN;\rmR,\rmR}^{1}|
&\lesssim \sum_{k,l\neq 0}\int_{\eta,\xi} \mathbf{1}_{t\in \rmI_{k,\eta},t\in \rmI_{l,\xi}}|\rmA \overline{\widehat{\rho}}_{k}(\eta)\rmA_{k}(\eta)||\eta l-\xi k||\hat{\phi}_{l}(\xi)_{N}\widehat{\rho}_{k-l}(\eta-\xi)_{<N/8}|d\eta d\xi\\
&\lesssim \sum_{k,l\neq 0}\int_{\eta,\xi} \mathbf{1}_{t\in \rmI_{k,\eta},t\in \rmI_{l,\xi}}|\rmA \overline{\widehat{\rho}}_{k}(\eta)|e^{C\d_{\rmL}|k-l,\eta-\xi|^{\f13}}\rmA_{l}(\xi)|\xi||\hat{\phi}_{l}(\xi)_{N}\widehat{\na\rho}_{k-l}(\eta-\xi)_{<N/8}|d\eta d\xi\\
&\lesssim \f{1}{\langle t\rangle^{2}} \sum_{k,l\neq 0}\int_{\eta,\xi} |\rmA \overline{\widehat{\rho}}_{k}(\eta)|e^{C\d_{\rmL}|k-l,\eta-\xi|^{\f13}}\rmA_{l}(\xi)|\hat{\phi}_{l}(\xi)_{N}\langle\xi-\eta\rangle^{11}\widehat{\na\rho}_{k-l}(\eta-\xi)_{<N/8}|d\eta d\xi\\
&\lesssim \f{1}{\langle t\rangle^2}\|\rmA\rho_{\sim \rmN}\|_2\left\|\left\langle \f{\pa_v}{t\pa_z}\right\rangle^{-1}1_{\rmR}\rmA P_{\neq}\phi_{\rmN}\right\|_2\|\rho\|_{\mathcal{G}^s}. 
\end{align*}

We conclude this subsection by introduce the following lemma. 
\begin{lemma}
Under the bootstrap hypotheses, it holds that,
\begin{align*}
\sum_{\rmN\geq 8}\rmR^1_{1;\rmN}&\lesssim 
\ep\rmCK_{\la,\rho}+\ep\rmCK_{\Theta,\rho}+\ep\rmCK_{M,\rho}\\
&\quad+\ep\left\|\left\langle\f{\pa_v}{t\pa_z}\right\rangle^{-1}\left(\f{|\na|^{\f{s}{2}}}{\langle t\rangle^{\f{3s}{2}}}\rmA+\sqrt{\f{\pa_t\rmg}{\rmg}}\tilde{\tilde{\rmA}}+\sqrt{\f{\pa_t\Theta}{\Theta}}\tilde{\rmA}\right)\pa_z^{-1}\Delta_{L}^2 P_{\neq}\phi\right\|_2^2.
\end{align*}
\end{lemma}

\subsubsection{The term $
\rmR_{1;\rmN}^{\ep,1}$}\label{Section_Term_R_1_epsilon}
In this section, we treat the term $
\rmR_{1;\rmN}^{\ep,1}$. We also use a paraproduct in $v$ to linearize the high frequencies around the low frequencies of the product $h\nabla^\perp \phi_l$. Hence, we have 
\begin{equation}\label{R_1_epsilon_Main}
\begin{aligned}
\rmR_{1;\rmN}^{\ep,1}=&\,\f{1}{2\pi}\sum_{\rmM\geq 8}\sum_{k,l\neq 0}\int_{\eta,\xi,\xi^\prime}\rmA\overline{\widehat{\rho}}_{k}(\eta)\rmA_{k}(\eta)\big((\eta-\xi)l-\xi^\prime (k-l)\big)\hat{h}(\xi-\xi^\prime)_{<\rmM/8}\\
&\times\varphi_{\rmN}(l,\xi)\hat{\phi}_l(\xi^\prime)_{\rmM}\hat{ \rho}_{k-l}(\eta-\xi)_{<\rmN/8}d\eta d\xi d\xi^\prime\\
&+\f{1}{2\pi}\sum_{\rmM\geq 8}\sum_{k,l\neq 0}\int_{\eta,\xi,\xi^\prime}\rmA\overline{\widehat{\rho}}_{k}(\eta)\rmA_{k}(\eta)\big((\eta-\xi)l-\xi^\prime (k-l)\big)\hat{h}(\xi-\xi^\prime)_{\rmM}\\
&\times\varphi_{\rmN}(l,\xi)\hat{\phi}_l(\xi^\prime)_{<\rmM/8}\hat{ \rho}_{k-l}(\eta-\xi)_{<\rmN/8}d\eta d\xi d\xi^\prime\\
&+\f{1}{2\pi}\sum_{\rmM\in \mathbf{D}}\sum_{\f18\rmM\leq \rmM^\prime\leq 8\rmM}\sum_{k,l\neq 0}\int_{\eta,\xi,\xi^\prime}\rmA\overline{\widehat{\rho}}_{k}(\eta)\rmA_{k}(\eta)\big((\eta-  \xi)l-\xi^\prime (k-l)\big)\hat{h}(\xi-\xi^\prime)_{\rmM^\prime}\\
&\times\varphi_{\rmN}(l,\xi)\hat{\phi}  _l(\xi^\prime)_{\rmM }\hat{ \rho}_{k-l}(\eta-\xi)_{<\rmN/8}d\eta d\xi d\xi^\prime\\
=&\, \rmR_{1;\rmN;\rmL\rmH}^{\ep,1}+\rmR_{1;\rmN;\rmH\rmL}^{\ep,1}+\rmR_{1;\rmN;\rmH\rmH}^{\ep,1}, 
\end{aligned}
\end{equation}
where $\varphi_\rmN$ denotes the cut-off associated to the N-th dyadic shell in $\Z\times\R$.

Let us first treat the term $\rmR_{1;\rmN;\rmL\rmH}^{\ep,1}$. 
Since $h$ is in low frequency, then it is natural to expect that $\rmR_{1;\rmN;\rmL\rmH}^{\ep,1}$ behaves somehow like $\rmR_{1;\rmN}^{\ep,1}$, since $\hat{h}_{<\rm M/8}$ provides only a modulation of the term $(\hat{\phi}_l)_{\rmM}$ for large frequencies. 
 
On the support of the integrand we have (see \eqref{eq: support}). 
\begin{equation}\label{Support_eq_2}
\begin{aligned}
\left||l,\xi|-|k,\eta|\right|\leq |k-l,\eta-\xi|\leq \f{6}{32}|l,\xi|\\
\left||l,\xi^\prime|-|l,\xi|\right|\leq |\xi-\xi^\prime|\leq \f{6}{32}|l,\xi^\prime|
\end{aligned}
\end{equation}
This implies $|k,\eta|\approx |l,\xi|\approx |l,\xi^\prime|$. As in \eqref{eq: support}, it holds that 
\begin{equation}\label{Ineq_loss_Deriva}
\big|(\eta-  \xi)l-\xi^\prime (k-l)\big|\lesssim |l,\xi^\prime||k-l,\eta-\xi|. 
\end{equation}
Hence, the estimates goes the same as in the one of $\rmR_{1;\rmN}^1$ where $(l,\xi^\prime)$ will play the role of $(l,\xi)$.  Therefore, when we switch from $\rmA_k(t,\eta)$ to $\rmA_l(t,\xi^\prime)$, we will pay a Gevrey-3 regularity for $h$ as shown by the following inequality: 
\begin{equation}\label{Exponential_Ineq_2}
\begin{aligned}
 e^{\la |k,\eta|^s}\lesssim e^{\la |l,\xi^\prime|^s+c\lambda |k-l,\eta-\xi|^s+c\la |\xi-\xi^\prime|^s}
\end{aligned}
\end{equation}
with $0<c<1$ which will help to absorb the Sobolev exponent  in the estimates. Hence, with the estimate \eqref{Exponential_Ineq_2} in hand, we may rewrite  $\rmR_{1;\rmN;\rmL\rmH}^{\ep,1}$ as 
\begin{equation}\label{R_1_epsilon_1}
\begin{aligned}
|\rmR_{1;\rmN;\rmL\rmH}^{\ep,1}|\lesssim&\, \sum_{\rmM\geq 8}\sum_{k,l\neq 0}\int_{\eta,\xi,\xi^\prime}\rmA|  \overline{\widehat{\rho}}_{k}(\eta)|\rmJ_{k}(\eta)\cM_{k}(\eta)\langle l,\xi^\prime\rangle^\sigma\big|(\eta-\xi)l-\xi^\prime (k-l)\big|\varphi_\rmN(l,\xi)\\    
&\times e^{c\lambda |\xi-\xi^\prime|^s}|\hat{h}(\xi-\xi^\prime)_{<\rmM/8}|
 e^{\la |l,\xi^\prime|^s}\hat{\phi}_l(\xi^\prime)_{\rmM}e^{c\lambda |k-l,\eta-\xi|^s}\hat{ \rho}_{k-l}(\eta-\xi)_{<\rmN/8}d\eta d\xi d\xi^\prime. 
\end{aligned}
\end{equation}

Thus we have
\begin{align*}
|\rmR_{1;\rmN;\rmL\rmH}^{\ep,1}|
&\lesssim \ep^2\left\|\f{|\na|^{\f{s}{2}}}{\langle t\rangle^{\f{3s}{2}}}\rmA \rho_{\sim \rmN}\right\|_2
\left\|\f{|\na|^{\f{s}{2}}}{\langle t\rangle^{\f{3s}{2}}}\left\langle\f{\pa_v}{t\pa_z}\right\rangle^{-1}\pa_z^{-1}\Delta_{L}^2\rmA P_{\neq}\phi_{\rmN}\right\|_2\\
&\quad+\ep^2\left\|\sqrt{\f{\pa_t\rmg}{\rmg}}\tilde{\tilde{\rmA}} \rho_{\sim \rmN}\right\|_2
\left\|\sqrt{\f{\pa_t\rmg}{\rmg}}\left\langle\f{\pa_v}{t\pa_z}\right\rangle^{-1}\pa_z^{-1}\Delta_{L}^2\tilde{\tilde{\rmA}}P_{\neq}\phi_{\rmN}\right\|_2\\
&\quad+\ep^2\left\|\sqrt{\f{\pa_t\rmg}{\rmg}}\tilde{\tilde{\rmA}}\rho_{\sim\rmN}\right\|_2\left\|\left\langle \f{\pa_v}{t\pa_z}\right\rangle^{-1}\sqrt{\f{\pa_t\Theta}{\Theta}}1_{\rmR}\tilde{\rmA} \pa_z^{-1}\Delta_{L}^2P_{\neq}\phi_{\rmN}\right\|_2
\end{align*}

Let us now turn to the term $\rmR_{1;\rmN;\rmH\rmL}^{\ep,1}$. We use the cut-off $1=\mathbf{1}_{16|l|\geq |\xi|}+\mathbf{1}_{16|l|\leq |\xi|}$ and split the term $\rmR_{1;\rmN;\rmH\rmL}^{\ep,1}$ according to this cut-off and write 
\begin{equation}
\begin{aligned}
\rmR_{1;\rmN;\rmH\rmL}^{\ep,1}=&\,\f{1}{2\pi}\sum_{\rmM\geq 8}\sum_{k,l\neq 0}\int_{\eta,\xi,\xi^\prime}\rmA\overline{\widehat{\rho}}_{k}(\eta)\rmA_{k}(\eta)\big((\eta-\xi)l-\xi^\prime (k-l)\big)\hat{h}(\xi-\xi^\prime)_{\rmM}\\
&\times\varphi_{\rmN}(l,\xi)\hat{\phi}_l(\xi^\prime)_{<\rmM/8}\hat{ \rho}_{k-l}(\eta-\xi)_{<\rmN/8}(\mathbf{1}_{16|l|\geq |\xi|}+\mathbf{1}_{16|l|\leq |\xi|})d\eta d\xi d\xi^\prime\\
=&\,\rmR_{1;\rmN;\rmH\rmL}^{\ep,1;z}+\rmR_{1;\rmN;\rmH\rmL}^{\ep,1;v}. 
\end{aligned}
\end{equation}
To estimate $\rmR_{1;\rmN;\rmH\rmL}^{\ep,1;z}$, we have first form \eqref{Support_eq_2} and from the fact that 
\begin{equation*}
\f\rmM 2\leq |\xi-\xi^\prime|\leq \f{3\rmM}{2}\quad \text{and}\quad \f{|\xi^\prime|}{\rmM/8}\leq \f{3}{4}
\end{equation*}
together with 
$|\xi|\leq 16|l|$, we have 
\begin{subequations}\label{Support}
\begin{equation}\label{Support_eq_3}
\left||l,\xi|-|k,\eta|\right|\leq |k-l,\eta-\xi|\leq \f{6}{32}|l,\xi|\\
\end{equation}
\begin{equation}\label{Support_4}
\left||l,\xi^\prime|-|l,\xi|\right|\leq |\xi-\xi^\prime|\leq \f{16}{13}|\xi|\lesssim |l|.   
\end{equation}
\end{subequations}  
We discuss two cases: $|l|\geq 16|\xi|$ and $\f{1}{16}|\xi|\leq |l|\leq 16 |\xi|$. For the case $|l|\geq 16|\xi|$ First, if $|l|\geq 16|\xi|$, then it holds that $\f{16}{13}|\xi|\leq\f{1}{13}|l|\leq \f{1}{13}|l,\xi|$, hence \eqref{Exponential_Ineq_2} holds. 
 
 Now, for $\f{1}{16}|\xi|\leq |l|\leq 16 |\xi|$, then, we have $|\xi-\xi^\prime|\approx |l,\xi|$ and hence, 
 we can obtain 
 \begin{equation}
e^{\lambda|l,\xi|^s}\leq e^{c(\lambda |l,\xi^\prime|+|\xi-\xi^\prime|)} \leq e^{\lambda|l,\xi^\prime|^s}e^{c\lambda|\xi-\xi^\prime|^s}
\end{equation}
Hence, in both cases\eqref{Exponential_Ineq_2}, holds.  Hence, using the fact that $\langle k,\eta\rangle\approx \langle l,\xi^\prime \rangle\leq \langle l\rangle $, we obtain as in \eqref{R_1_epsilon_1}, 
\begin{equation}\label{R_1_epsilon_1_z}
\begin{aligned}
|\rmR_{1;\rmN;\rmH\rmL}^{\ep,1;z}|\lesssim&\, \sum_{\rmM\geq 8}\sum_{k,l\neq 0}\int_{\eta,\xi,\xi^\prime}\mathbf{1}_{16|l|\geq |\xi|}\rmA|  \overline{\widehat{\rho}}_{k}(\eta)|\rmJ_{k}(\eta)\cM_{k}(\eta)\langle l\rangle^\sigma\big|(\eta-\xi)l-\xi^\prime (k-l)\big|\varphi_\rmN(l,\xi)\\    
&\times e^{c\lambda |\xi-\xi^\prime|^s}|\hat{h}(\xi-\xi^\prime)_{\rmM}|
 e^{\la |l,\xi^\prime|^s}\hat{\phi}_l(\xi^\prime)_{<\rmM/8}e^{c\lambda |k-l,\eta-\xi|^s}\hat{ \rho}_{k-l}(\eta-\xi)_{<\rmN/8}d\eta d\xi d\xi^\prime. 
\end{aligned}
\end{equation}
Now, in the above integral, we will take advantages for $|l|$ being large to exclude the resonant interval. For this reason, we fix $\rmM_0$ large enough and split the above term into two terms:  
 \begin{equation}\label{R_1_epsilon_2_z}
\begin{aligned}
|\rmR_{1;\rmN;\rmH\rmL}^{\ep,1;z}|\lesssim&\,\Big( \sum_{\rmM\geq \rmM_0}+\sum_{\rmM\leq \rmM_0}\Big)
\sum_{k,l\neq 0}\int_{\eta,\xi,\xi^\prime}\rmA|  \overline{\widehat{\rho}}_{k}(\eta)|\rmJ_{k}(\eta)\cM_{k}(\eta)\langle l\rangle^\sigma\big|(\eta-\xi)l-\xi^\prime (k-l)\big|\varphi_\rmN(l,\xi)\\    
&\times \mathbf{1}_{16|l|\geq |\xi|}e^{c\lambda |\xi-\xi^\prime|^s}|\hat{h}(\xi-\xi^\prime)_{\rmM}|
 e^{\la |l,\xi^\prime|^s}\hat{\phi}_l(\xi^\prime)_{<\rmM/8}e^{c\lambda |k-l,\eta-\xi|^s}\hat{ \rho}_{k-l}(\eta-\xi)_{<\rmN/8}d\eta d\xi d\xi^\prime\\
 =& \rmR_{1;\rmN;\rmH\rmL;\rmH}^{\ep,1;z}+\rmR_{1;\rmN;\rmH\rmL;\rmL}^{\ep,1;z}. 
\end{aligned}
\end{equation}
  For the term $\rmR_{1;\rmN;\rmH\rmL;\rmH}^{\ep,1;z}$, we have $|l|$ is large compared to $|\xi^\prime|$ and since $|k,\eta|\approx |l,\xi^\prime|$, then both $(k,\eta)$ and $(l,\xi)$ are both non-resonant. Then, it holds that 
\begin{equation}\label{J/J_M/M}
\f{\rmJ_{k}(\eta)\cM_{k}(\eta)}{\rmJ_{l}(\xi^\prime)\cM_{l}(\xi^\prime)}\lesssim e^{C|k-l,\eta-\xi^\prime|^{\f13}}\lesssim e^{C|k-l,\eta-\xi|^{\f13}+C|\xi-\xi^\prime|^{\f13}}. 
\end{equation}
Hence, it holds that by applying \eqref{Ineq_loss_Deriva}  
\begin{equation}
\begin{aligned}
|\rmR_{1;\rmN;\rmH\rmL;\rmH}^{\ep,1;z}| \lesssim&\, \sum_{\rmM\geq \rmM_0}\sum_{k,l\neq 0}\int_{\eta,\xi,\xi^\prime}\rmA |k,\eta|^{\f s2}|  \overline{\widehat{\rho}}_{k}(\eta)|\rmJ_{l}(\xi^\prime)\cM_{l}(\xi^\prime) \f{|l,\xi'|}{l^3(1+|t-\f{\xi'}{l}|^2)^2}\\    
&\times \langle l\rangle^\sigma\varphi_\rmN(l,\xi)\mathbf{1}_{16|l|\geq |\xi|}e^{c\lambda |\xi-\xi^\prime|^s}|\hat{h}(\xi-\xi^\prime)_{\rmM}|\mathbf{1}_{\rmN\rmR}\\
&\times
 e^{\la |l,\xi^\prime|^s}\widehat{\pa_z^{-1}\Delta_{L}^2\phi}_l(\xi^\prime)_{<\rmM/8}|k-l,\eta-\xi|e^{c\lambda |k-l,\eta-\xi|^s}\hat{ \rho}_{k-l}(\eta-\xi)_{<\rmN/8}d\eta d\xi d\xi^\prime\\    
 \end{aligned}
\end{equation}
On the support of the integrand, we have 
 \begin{equation}\label{Decay_phi}
 \begin{aligned}
\f{|l,\xi'|}{l^3(1+|t-\f{\xi'}{l}|^2)^2}\lesssim \f{1}{l^2+|lt-\xi'|^2}\lesssim \f{1}{l^2\langle t\rangle ^2}\left\langle\f{\xi}{lt}\right\rangle^{-1}
\end{aligned}
\end{equation}
Consequently, it holds from above that   
\begin{align*}
|\rmR_{1;\rmN;\rmH\rmL;\rmH}^{\ep,1;z}|
 \lesssim &\f{1}{\langle t\rangle ^2}\left\|\rmA \rho_{\sim \rmN}\right\|_2\left\|\left\langle\f{\pa_v}{t\pa_z}\right\rangle^{-1}\mathbf{1}_{\rmN\rmR}\pa_z^{-1}\Delta_{L}^2\rmA P_{\neq}\phi_{\rmN}\right\|_2\|\rho_{<\rmN/8}\|_{\mathcal{G}^{s}}\sum_{\rmM\geq \rmM_0}\f{1}{\rmM}\Vert h _{\rmM}\Vert_{\mathcal{G}^{s}}\\
 \lesssim&\,\f{\ep^2}{\langle t\rangle ^2}\left\|\rmA \rho_{\sim \rmN}\right\|_2\left\|\left\langle\f{\pa_v}{t\pa_z}\right\rangle^{-1}\mathbf{1}_{\rmN\rmR}\pa_z^{-1}\Delta_{L}^2\rmA P_{\neq}\phi_{\rmN}\right\|_2.
\end{align*}
The treatment of $\rmR_{1;\rmN;\rmH\rmL;\rmL}^{\ep,1;z}$ is easy since $|\xi-\xi'|+|\xi'|+|l|\lesssim 2^{\rmM_0}$. 
Thus we have
\begin{align*}
\sum_{\rmN\geq 8}|\rmR_{1;\rmN;\rmH\rmL;\rmL}^{\ep,1;z}|
\lesssim \ep^2\rmCK_{\la,\rho}+\ep^2\left\|\f{|\na|^{\f s2}}{\langle t\rangle^{\f{3s}{2}}}\left\langle\f{\pa_v}{t\pa_z}\right\rangle^{-1}\mathbf{1}_{\rmN\rmR}\pa_z^{-1}\Delta_{L}^2\rmA P_{\neq}\phi\right\|_2.
\end{align*}

Now let us treat $\rmR_{1;\rmN;\rmH\rmL}^{\ep,1;v}$. 
On the support of the integrand, \eqref{Support_eq_3} holds. In addition, we have 
\begin{subequations}\label{support_v}
\begin{equation}\label{Support_eq_3_0}
\left||l,\xi|-|k,\eta|\right|\leq |k-l,\eta-\xi|\leq \f{6}{32}|l,\xi|\\
\end{equation}
\begin{equation}\label{Support_eq_3_1_0}
\big||\xi-\xi^\prime|-|l,\xi|\big|\leq |l,\xi^\prime|\leq \f{67}{256}|\xi-\xi^\prime|. 
\end{equation}
\end{subequations}
Hence, it holds that  
\begin{equation}\label{Exponential_v_part}
e^{\la |k,\eta|^s}\leq e^{\lambda |\xi-\xi^\prime|^s+c\lambda |l,\xi^\prime|^s+c\lambda |k-l,\eta-\xi|^s}. 
\end{equation}
Also it holds that from \eqref{support_v} that $\f{189}{256}|\xi-\xi^\prime|\leq |l,\xi|\leq  \f{323}{256}|\xi-\xi^\prime|$. Hence, it holds that, by using \eqref{Ineq_loss_Deriva} and the fact that $\langle k,\eta\rangle\approx \langle l,\xi\rangle $    
\begin{align*}
|\rmR_{1;\rmN;\rmH\rmL}^{\ep,1;v}|\lesssim&\,\sum_{\rmM\geq 8}\sum_{k,l\neq 0}\int_{\eta,\xi,\xi^\prime}\rmA\overline{\widehat{\rho}}_{k}(\eta)\rmJ_{k}(\eta)\cM_{k}(\eta)\langle \xi-\xi^\prime\rangle^\sigma e^{\lambda |\xi-\xi^\prime|^s|}|\hat{h}(\xi-\xi^\prime)_{\rmM}|\\
&\times\varphi_{\rmN}(l,\xi)e^{c\lambda |l,\xi^\prime|^s}\hat{\phi}_l(\xi^\prime)_{<\rmM/8}|l,\xi^\prime||k-l,\eta-\xi|\\
&\times e^{c\lambda |k-l,\eta-\xi|^s}\hat{ \rho}_{k-l}(\eta-\xi)_{<\rmN/8}\mathbf{1}_{16|l|\leq |\xi|}d\eta d\xi d\xi^\prime
\end{align*}
Now, we use the estimate 
\begin{align*}
\f{\rmJ_k(\eta)}{\rmJ_0(\xi-\xi^\prime)}
&\lesssim \langle \xi-\xi^\prime\rangle e^{C|k,\eta-\xi-\xi^\prime|^{\f13}}\lesssim  \langle \xi-\xi^\prime\rangle e^{C|l,\xi^\prime|^{\f13}+C|k-l,\eta-\xi|^{\f13} } \\
\f{\cM_{k}(\eta)}{\cM_{0}(\xi-\xi^\prime)}
&\lesssim e^{C|k,\eta-\xi-\xi^\prime|^{\f13}}\lesssim e^{C|l,\xi^\prime|^{\f13}+C|k-l,\eta-\xi|^{\f13} },
\end{align*}
which gives us that
\begin{align*}
|\rmR_{1;\rmN;\rmH\rmL}^{\ep,1;v}|
\lesssim\left\|\rmA \rho_{\sim \rmN}\right\|_2\left\| P_{\neq}\phi\right\|_{\mathcal{G}^{\la,\b;s}}\Vert \rmA_0\pa_vh_{\sim \rmN} \Vert_{2}\|\rho\|_{\mathcal{G}^{\la,\b;s}}.  
\end{align*}

Now, we estimate the term $\rmR_{1;\rmN;\rmH\rmH}^{\ep,1}$. We have 
\begin{equation}
\begin{aligned}
\rmR_{1;\rmN;\rmH\rmH}^{\ep,1}=&\,\f{1}{2\pi}\sum_{\rmM\in \mathbf{D}}\sum_{\f18\rmM\leq \rmM^\prime\leq 8\rmM}\sum_{k,l\neq 0}\int_{\eta,\xi,\xi^\prime}\rmA\overline{\widehat{\rho}}_{k}(\eta)\rmA_{k}(\eta)\big((\eta-  \xi)l-\xi^\prime (k-l)\big)\hat{h}(\xi-\xi^\prime)_{\rmM^\prime}\\
&\times\varphi_{\rmN}(l,\xi)\hat{\phi}  _l(\xi^\prime)_{\rmM }\hat{ \rho}_{k-l}(\eta-\xi)_{<\rmN/8}d\eta (\mathbf{1}_{|l|\geq 100|\xi^\prime|}+\mathbf{1}_{|l|\leq 100|\xi^\prime|}) d\xi d\xi^\prime\\
=&\, \rmR_{1;\rmN;\rmH\rmH}^{\ep,1;z}+\rmR_{1;\rmN;\rmH\rmH}^{\ep,1;v}. 
\end{aligned}  
\end{equation}
As above, we always have
\begin{subequations}  
\begin{equation}\label{Support_eq3_1}
\left||l,\xi|-|k,\eta|\right|\leq |k-l,\eta-\xi|\leq \f{6}{32}|l,\xi|. \\
\end{equation}
Also on the support of the integrand, we have 
\begin{equation}\label{Support_M_M_prime}
\f{\rmM^\prime} {2}\leq |\xi-\xi^\prime|\leq \f{3\rmM^\prime}{2}\qquad \text{and}\qquad \f{\rmM} {2}\leq |\xi^\prime|\leq \f{3\rmM}{2}
\end{equation}
 This together with $\f18\rmM\leq \rmM^\prime\leq 8\rmM$ implies 
 \begin{equation} \label{Support_Integrand_Rem}
\big| |l,\xi|-|l,\xi^\prime|\big|\leq |\xi-\xi^\prime|\leq 24 |\xi^\prime|\leq \f{24}{100} |l,\xi^\prime|. 
\end{equation}
\end{subequations}
Thus it holds that
\begin{equation*}
e^{\la |k,\eta|^s}\leq e^{\lambda|l,\xi|^s}e^{c\lambda |k-l,\eta-\xi|}\leq e^{\lambda |l,\xi^\prime|^s+c\lambda |\xi-\xi^\prime|^s+c\lambda |k-l,\eta-\xi|^s}. 
\end{equation*}
Hence, applying \eqref{Decay_phi}, we get  
\begin{equation}
\begin{aligned}
|\rmR_{1;\rmN;\rmH\rmH}^{\ep,1;z}|
\lesssim&\f{1}{\langle t\rangle ^2}
\left\|\rmA \rho_{\sim \rmN}\right\|_2
\left\|\left\langle\f{\pa_v}{t\pa_z}\right\rangle^{-1}\mathbf{1}_{\rmN\rmR}\pa_z^{-1}\Delta_{L}^2\rmA P_{\neq}\phi_{\rmN}\right\|_2\|\rho_{<\rmN/8}\|_{\mathcal{G}^{\la,\b;s}}\\
&\times\sum_{\rmM^\prime\in \mathbf{D}}\f{1}{\rmM'}\Vert \hat{h}_{ \rmM^\prime}\Vert_{\mathcal{G}^{\la,\b;s}}.
\end{aligned}
\end{equation}
The term $\rmR_{1;\rmN;\rmH\rmH}^{\ep,1;v}$ is easy to deal with, we have by using the fact that $2\b\geq \s$
\begin{align*}
|\rmR_{1;\rmN;\rmH\rmH}^{\ep,1;v}|\lesssim&\,\sum_{\rmM\in \mathbf{D}}\sum_{\rmM\approx \rmM^\prime}\sum_{k,l\neq 0}\int_{\eta,\xi,\xi^\prime}\rmA|\overline{\widehat{\rho}}_{k}(\eta)|e^{\lambda |\xi-\xi^\prime|^s}|\hat{h}(\xi-\xi^\prime)_{\rmM^\prime}|\\
&\times\varphi_{\rmN}(l,\xi)e^{\lambda|l,\xi^\prime|^s}|\hat{\phi}  _l(\xi^\prime)_{\rmM }||k-l,\eta-\xi|e^{c\lambda |k-l,\eta-\xi|^s}|\hat{ \rho}_{k-l}(\eta-\xi)_{<\rmN/8}|\\
&\times\mathbf{1}_{|l|\leq 100|\xi^\prime|}d\eta  d\xi d\xi^\prime\\
\lesssim& \sum_{\rmM\in \mathbf{D}}\sum_{\rmM\approx \rmM^\prime}\f{1}{\rmN}\left\|\rmA \rho_{\sim \rmN}\right\|_2\left\| P_{\neq}\phi_{\rmM}\right\|_{\mathcal{G}^{\la,\b;s}}\Vert h_{\rmM^\prime} \Vert_{\mathcal{G}^{\la,\b;s}}\|\rho\|_{\mathcal{G}^{\la,\b;s}}\\
\lesssim&  \f{1}{\rmN} \left\|\rmA \rho_{\sim \rmN}\right\|_2\left\| P_{\neq}\phi\right\|_{\mathcal{G}^{\la,\b;s}}\Vert h \Vert_{\mathcal{G}^{\la,\b;s}}\|\rho\|_{\mathcal{G}^{\la,\b;s}}. 
\end{align*}

We conclude this subsection by introduce the following lemma. 
\begin{lemma}
Under the bootstrap hypotheses, it holds that,
\begin{align*}
\sum_{\rmN\geq 8}\rmR^{\ep,1}_{1;\rmN}&\lesssim \f{\ep^{4}}{\langle t\rangle^4}+
\ep^2\rmCK_{\la,\rho}+\ep^2\rmCK_{\Theta,\rho}+\ep^2\rmCK_{M,\rho}\\
&\quad+\ep^2\left\|\left\langle\f{\pa_v}{t\pa_z}\right\rangle^{-1}\left(\f{|\na|^{\f{s}{2}}}{\langle t\rangle^{\f{3s}{2}}}\rmA+\sqrt{\f{\pa_t\rmg}{\rmg}}\tilde{\tilde{\rmA}}+\sqrt{\f{\pa_t\Theta}{\Theta}}\tilde{\rmA}\right)\pa_z^{-1}\Delta_{L}^2 P_{\neq}\phi\right\|_2^2.
\end{align*}
\end{lemma}

\subsubsection{The term $\rmR_{1;\rmN}^{2}$}
On the support  of the integrand it holds that 
\begin{equation}
\f{|k|+|\eta-\xi|}{\rmN/8}\leq \f34\qquad \text{and}\qquad \f{\rmN}{2}\leq |\xi|\leq \f{3\rmN}{2} 
\end{equation}
 which implies that 
 \begin{equation}\label{support_k_Inq}
|k|+|\eta-\xi|\leq \f{3}{16}|\xi| 
\end{equation}   
and $|k|\leq |\xi|\approx |\eta|$.

We obtain on the support of the integrand for $0<c<1$
 \begin{equation}\label{Mult_0_xi}
\f{ e^{\la|k,\eta|^s}}{e^{\la|\xi|^s}}\lesssim e^{c\la|k,\eta-\xi|^s},\quad 
\f{\rmJ_k(\eta)}{\rmJ_0(\xi)}\lesssim \langle \xi\rangle e^{C\mu|k,\xi-\eta|^{\f13}},\quad \f{\cM_k(\eta)}{\cM_0(\xi)}\lesssim e^{C\d_{\rmL}^{-1}|k,\xi-\eta|^{\f13}},
\end{equation}
Thus we get that
\beno
|\rmR_{1;\rmN}^{2}|\lesssim \f{1}{\langle t\rangle}\|\rmA \rho_{\sim \rmN}\|_2\langle t\rangle\|\rmA\pa_vg_{\rmN}\|_{2}\|\rho\|_{\cG^{\la,\b;s}}\lesssim \f{\ep}{\langle t\rangle^2}\|\rmA \rho_{\sim \rmN}\|_2^2+\ep\langle t\rangle^2\|\rmA\pa_v^3g_{\rmN}\|_{2}^2.
\eeno

\subsubsection{The term $\rmR_{1;\rmN}^{3}$}
Now, we estimate the term $  \rmR_{1;\rmN}^{3}$. We have by using the fact that $|k-l,\eta-\xi|\lesssim |l,\xi| $ (see \eqref{eq: support})
\begin{equation}
\begin{aligned}  
\rmR_{1;\rmN}^{3}\lesssim&\,\sum_{k,l}\int_{\eta,\xi}\rmA_k(t,\eta)|\overline{\hat{\rho}}_{k}(\eta)||k-l,\eta-\xi||\widehat{{\bf u}}_{l}(\xi)_{N}| \rmA_{k-l}(\eta-\xi)|\hat{\rho}_{k-l}(\eta-\xi)_{<N/8}|d\eta d\xi\\
\lesssim&\, \Vert \rmA \rho_{\sim \rmN}\Vert_{2}\Vert {\bf u}_{\rmN} \Vert_{H^4}\Vert \rmA \rho\Vert_{2}
\lesssim \ep\Vert \rmA \rho_{\sim \rmN}\Vert_{2}(\Vert \pa_v^3g_{\rmN} \Vert_{H^4}+\|v'P_{\neq}\phi_{\rmN}\|_{H^4}).
\end{aligned}
\end{equation}

\subsection{Transport term $\rmT_{1;\rmN}$}\label{Section_Transport}
In this section, we deal with the transport term and prove the following proposition:
\begin{proposition}
Under the bootstrap hypotheses, it holds that,
\begin{align*}
\sum_{\rmN\geq 8}\rmT_{1;\rmN}&\lesssim 
\sqrt{\ep}\rmCK_{\la,\rho}.
\end{align*}
\end{proposition}

Decompose the difference:
\begin{align*}
\rmA_k(\eta)-\rmA_{l}(\xi)&=\rmA_{l}(\xi)[e^{\la|k,\eta|^s-\la|l,\xi|^s}-1]\\
&\quad + \rmA_{l}(\xi)e^{\la|k,\eta|^s-\la|l,\xi|^s}\left[\f{\rmJ_k(\eta)}{\rmJ_l(\xi)}-1\right]\f{\cM_k(\eta)}{\cM_l(\xi)}\f{\langle k,\eta\rangle^{\s}}{\langle l,\xi\rangle^{\s}}\f{\rmB_k(\eta)}{\rmB_l(\xi)}\\
&\quad + \rmA_{l}(\xi)e^{\la|k,\eta|^s-\la|l,\xi|^s}\left[\f{\cM_k(\eta)}{\cM_l(\xi)}-1\right]\f{\langle k,\eta\rangle^{\s}}{\langle l,\xi\rangle^{\s}}\f{\rmB_k(\eta)}{\rmB_l(\xi)}\\
&\quad + \rmA_{l}(\xi)e^{\la|k,\eta|^s-\la|l,\xi|^s}\left[\f{\langle k,\eta\rangle^{\s}}{\langle l,\xi\rangle^{\s}}-1\right]\f{\rmB_k(\eta)}{\rmB_l(\xi)}\\
&\quad +\rmA_{l}(\xi)e^{\la|k,\eta|^s-\la|l,\xi|^s}\left[\f{\rmB_k(\eta)}{\rmB_l(\xi)}-1\right].
\end{align*}
Hence, we write accordingly  
\begin{align*}
\rmT_{1;\rmN}
=&i\sum_{k,l}\int_{\eta,\xi} \rmA_k(\eta)\bar{\hat{\rho}}(k,\eta)\hat{\bf{u}}(k-l,\eta-\xi)_{<\rmN/8}\cdot (l,\xi)(\rmA_{k}(\eta)-\rmA_{l}(\xi))\hat{\rho}_l(\xi)_{\rmN}d\xi d\eta\\
=&\rmT_{1;\rmN}^1+\rmT_{1;\rmN}^2+\rmT_{1;\rmN}^3+\rmT_{1;\rmN}^4+\rmT_{1;\rmN}^5,
\end{align*}
with  
\begin{equation*}
\begin{aligned}
\rmT_{1;\rmN}^1=&\,i\sum_{k,l}\int_{\eta,\xi} \rmA_k(\eta)\bar{\hat{\rho}}(k,\eta)\hat{\bf{u}}(k-l,\eta-\xi)_{<\rmN/8}\cdot (l,\xi)\rmA_{l}(\xi)\hat{\rho}_l(\xi)_{\rmN}\\
&\times[e^{\la|k,\eta|^s-\la|l,\xi|^s}-1]d\xi d\eta\\
\rmT_{1;\rmN}^2=&\,i\sum_{k,l}\int_{\eta,\xi} \rmA_k(\eta)\bar{\hat{\rho}}(k,\eta)\hat{\bf{u}}(k-l,\eta-\xi)_{<\rmN/8}\cdot (l,\xi)\rmA_{l}(\xi)\hat{\rho}_l(\xi)_{\rmN}e^{\la|k,\eta|^s-\la|l,\xi|^s}\\
&\times \left[\f{\rmJ_k(\eta)}{\rmJ_l(\xi)}-1\right]\f{\cM_k(\eta)}{\cM_l(\xi)}\f{\langle k,\eta\rangle^{\s}}{\langle l,\xi\rangle^{\s}} \f{\rmB_k(\eta)}{\rmB_l(\xi)}d\xi d\eta,
\end{aligned}
\end{equation*}
\begin{equation*}
\begin{aligned}
\rmT_{1;\rmN}^3=&\,i\sum_{k,l}\int_{\eta,\xi} \rmA_k(\eta)\bar{\hat{\rho}}(k,\eta)\hat{\bf{u}}(k-l,\eta-\xi)_{<\rmN/8}\cdot (l,\xi)\rmA_{l}(\xi)\hat{\rho}_l(\xi)_{\rmN}e^{\la|k,\eta|^s-\la|l,\xi|^s}\\
&\times \left[\f{\cM_k(\eta)}{\cM_l(\xi)}-1\right]\f{\langle k,\eta\rangle^{\s}}{\langle l,\xi\rangle^{\s}}\f{\rmB_k(\eta)}{\rmB_l(\xi)}d\xi d\eta
\end{aligned}
\end{equation*}
and 
\begin{align*}
\rmT_{1;\rmN}^4=&\,i\sum_{k,l}\int_{\eta,\xi} \rmA_k(\eta)\bar{\hat{\rho}}(k,\eta)\hat{\bf{u}}(k-l,\eta-\xi)_{<\rmN/8}\cdot (l,\xi)\rmA_{l}(\xi)\hat{\rho}_l(\xi)_{\rmN}e^{\la|k,\eta|^s-\la|l,\xi|^s}\\
&\times \left[\f{\langle k,\eta\rangle^{\s}}{\langle l,\xi\rangle^{\s}}-1\right]\f{\rmB_k(\eta)}{\rmB_l(\xi)}d\xi d\eta. \\
\rmT_{1;\rmN}^5=&\,i\sum_{k,l}\int_{\eta,\xi} \rmA_k(\eta)\bar{\hat{\rho}}(k,\eta)\hat{\bf{u}}(k-l,\eta-\xi)_{<\rmN/8}\cdot (l,\xi)\rmA_{l}(\xi)\hat{\rho}_l(\xi)_{\rmN}e^{\la|k,\eta|^s-\la|l,\xi|^s}\\
&\times \left[\f{\rmB_k(\eta)}{\rmB_l(\xi)}-1\right]d\xi d\eta.
\end{align*}     

\subsubsection{Term $\rmT_{1;\rmN}^1$}\label{Section_T_1}
We get that 
\begin{align*}
|\rmT_{1;\rmN}^1|\lesssim&\,\sum_{k,l}\int_{\eta,\xi} |\rmA_k(\eta)\bar{\hat{\rho}}(k,\eta)||\hat{\bf{u}}(k-l,\eta-\xi)_{<\rmN/8}||(l,\xi)||\rmA_{l}(\xi)\hat{\rho}_l(\xi)_{\rmN}|\\
&\times\la\Big||k,\eta|^s-|l,\xi|^s\Big|e^{\la|k,\eta|^s-\la|l,\xi|^s}d\xi d\eta\\ 
\lesssim&\,\lambda \sum_{k,l}\int_{\eta,\xi} |\rmA_k(\eta)\bar{\hat{\rho}}(k,\eta)||\hat{\bf{u}}(k-l,\eta-\xi)_{<\rmN/8}||l,\xi||\rmA_{l}(\xi)\hat{\rho}_l(\xi)_{\rmN}|\\
&\times \f{\Big||k,\eta|-|l,\xi|\Big|}{|k,\eta|^{1-s}+|l,\xi|^{1-s}}e^{\la|k,\eta|^s-\la|l,\xi|^s}d\xi d\eta 
\end{align*}   

On the support of the integrand, we have 
\begin{equation}\label{Support_Ineq}
\begin{aligned}
&\Big||k,\eta|-|l,\xi|\Big|\leq |k-l,\eta-\xi|\leq \f{6}{32}|l,\xi|\\
&\f{26}{32}|l,\xi|\leq |k,\eta|\leq \f{38}{32}|l,\xi|. 
\end{aligned}
\end{equation}
Consequently, we obtain 
\begin{align*}
|\rmT_{1;\rmN}^1|\lesssim&\,\lambda \sum_{k,l}\int_{\eta,\xi} |\rmA_k(\eta)\bar{\hat{\rho}}(k,\eta)||\hat{\bf{u}}(k-l,\eta-\xi)_{<\rmN/8}||k,\eta|^{\f s2}|l,\xi|^{\f s2}|\rmA_{l}(\xi)\hat{\rho}_l(\xi)_{\rmN}|\\
&\times e^{c\la|k-l,\eta-\xi|^s}d\xi d\eta\\
\lesssim&\, \lambda \Vert |\nabla|^{\f s2}\rmA \rho_{\sim\rmN} \Vert_2 \Vert |\nabla|^{\f s2}\rmA \rho_{\rmN} \Vert_2 
\Big(\Vert v'\na P_{\neq}\phi\Vert_{\mathcal{G}^{\la,\b;s}} 
+\|g\|_{\mathcal{G}^{\la,\b;s}_1}\Big)\\
\lesssim& \f{\sqrt{\ep}}{\langle t\rangle^{2-\ep_1/2}}\Vert |\nabla|^{\f s2}\rmA \rho_{\sim\rmN} \Vert_2 \Vert |\nabla|^{\f s2}\rmA \rho_{\rmN} \Vert_2.
\end{align*}    

\subsubsection{Term $\rmT_{1;\rmN}^2$}\label{Section_T_2}
Now, we treat the term $\rmT_{1;\rmN}^2$. Inspired by the estimate in Lemma \ref{Lemma_Cmmutator_J}, we split $\rmT_{1;\rmN}^2$ as  
\begin{equation}
\begin{aligned}
\rmT_{1;\rmN}^2=&\,i\sum_{k,l}\int_{\eta,\xi} \rmA_k(\eta)\bar{\hat{\rho}}(k,\eta)\hat{\bf{u}}(k-l,\eta-\xi)_{<\rmN/8}\cdot(l,\xi)\rmA_{l}(\xi)\hat{\rho}_l(\xi)_{\rmN}e^{\la|k,\eta|^s-\la|l,\xi|^s}\\
&\times \Big[\chi^\rmS+\chi^{\rmL}\Big]\left[\f{\rmJ_k(\eta)}{\rmJ_l(\xi)}-1\right]\f{\cM_k(\eta)}{\cM_l(\xi)}\f{\langle k,\eta\rangle^{\s}}{\langle l,\xi\rangle^{\s}}\f{\rmB_k(\eta)}{\rmB_l(\xi)} d\xi d\eta\\
=&\, \rmT_{1;\rmN}^{2;\rmS}+\rmT_{1;\rmN}^{2;\rmL},
\end{aligned}
\end{equation}
 where $\chi^\rmS=\mathbf{1}_{t\leq \f12\min\{|\xi|^{\f23},|\eta|^{\f23}\}}$ and $\chi^\rmL=1-\chi^\rmS$. 
 
 We now estimate $\rmT_{1;\rmN}^{2;\rmS}$, where Lemma \ref{Lemma_Cmmutator_J} plays a crucial rule in absorbing $\f23$-derivatives. Hence,  by \eqref{Support_Ineq} and Lemma \ref{lem:g/g} together with the fact that  $\langle k,\eta\rangle   \approx\langle l,\xi\rangle$ on the support of the integrand, we obtain 
\begin{align*}
|\rmT_{1;\rmN}^{2;\rmS}|\lesssim&\, \sum_{k,l}\int_{\eta,\xi} \chi^\rmS\rmA_k(\eta)|\bar{\hat{\rho}}(k,\eta)||\hat{\bf{u}}(k-l,\eta-\xi)_{<\rmN/8}||l,\xi|^{\f13}\rmA_{l}(\xi)|\hat{\rho}_l(\xi)_{\rmN}|e^{c\la|k-l,\eta-\xi|^s}\\
&\times \langle k-l,\xi-\eta\rangle e^{C|k-l,\xi-\eta|^{\f13}} d\xi d\eta\\
\lesssim&\, \sum_{k,l}\int_{\eta,\xi} \chi^\rmS\rmA_k(\eta)|\bar{\hat{\rho}}(k,\eta)||\hat{\bf{u}}(k-l,\eta-\xi)_{<\rmN/8}|(1+|l,\xi|^{\f s2}|k,\eta|^{\f s2})\\
&\times\rmA_{l}(\xi)|\hat{\rho}_l(\xi)_{\rmN}|e^{\la|k-l,\eta-\xi|^s}d\xi d\eta\\
\lesssim & \Vert |\nabla|^{\f s2}\rmA \rho_{\sim\rmN} \Vert_2 \Vert |\nabla|^{\f s2}\rmA \rho_{\rmN} \Vert_2 \Big(\Vert v'\na P_{\neq}\phi\Vert_{\mathcal{G}^{\la,\b;s}} 
+\|g\|_{\mathcal{G}^{\la,\b;s}_1}\Big)
\\
\lesssim &\f{\sqrt{\ep}}{\langle t\rangle^{2-\ep_1/2}} \||\nabla|^{\f s2}\rmA \rho_{\sim\rmN} \|_2 \| |\nabla|^{\f s2}\rmA \rho_{\rmN}\|_2.
\end{align*}
 Next, we estimate $\rmT_{1;\rmN}^{2;\rmL}$ which corresponds to $t>\f12\min\{|\xi|^{\f23},|\eta|^{\f23}\}$.  We rewrite $\rmT_{1;\rmN}^{2;\rmL}$ as 
\begin{align*}
\rmT_{1;\rmN}^{2;\rmL}=&\,i\sum_{k,l}\int_{\eta,\xi} \rmA_k(\eta)\bar{\hat{\rho}}(k,\eta)\hat{\bf{u}}(k-l,\eta-\xi)_{<\rmN/8}\cdot(l,\xi)\rmA_{l}(\xi)\hat{\rho}_{l}(\xi)_{\rmN}e^{\la|k,\eta|^s-\la|l,\xi|^s}\\
&\times \chi^{\rmL}(\mathbf{1}_{|l|\leq 100|\xi|}+\mathbf{1}_{|l|\geq 100|\xi|})\left[\f{\rmJ_k(\eta)}{\rmJ_l(\xi)}-1\right]\f{\cM_k(\eta)}{\cM_l(\xi)}\f{\langle k,\eta\rangle^{\s}}{\langle l,\xi\rangle^{\s}} \f{\rmB_k(\eta)}{\rmB_l(\xi)}d\xi d\eta\\
=&\, \rmT_{1;\rmN}^{2;\rmL;v}+\rmT_{1;\rmN}^{2;\rmL;z}. 
\end{align*}
Let us now estimate $\rmT_{1;\rmN}^{2;\rmL,z }$. Let us first see that on the support of the integrand, we  have
\begin{equation}
|\eta|\leq |\xi|+\f{6}{32}(|l|+|\xi|)\leq\f{319}{1600}|l| 
\end{equation}
Hence, we have 
 \begin{equation}\label{J_Estimate_z}
\left|\f{\rmJ_k(\eta)}{\rmJ_l(\xi)}-1\right|\lesssim \f{1}{l^{\f23}}+\f{|k-l|}{|k|^{\f23}+|l|^{\f23}}e^{C\mu|k-l|^{\f13}},
\end{equation}
Since $|l,\xi|\lesssim |l|$ and also from \eqref{Support_Ineq}, we have on the support of the integrand 
$\f{1297}{1600}|l|\leq|k|\leq \f{1903}{1600}|l|$.  
Hence, we obtain 
\begin{align*}
|\rmT_{1;\rmN}^{2;\rmL,z }|\lesssim &\sum_{k,l}\int_{\eta,\xi} \rmA_k(\eta)|\bar{\hat{\rho}}(k,\eta)||\hat{\mathbf{u}}(k-l,\eta-\xi)_{<\rmN/8}||l|^{\f13}\rmA_{l}(\xi)|\hat{\rho}_l(\xi)_{\rmN}|e^{\la|k,\eta|^s-\la|l,\xi|^s}\\
&\times \chi^{\rmL}\mathbf{1}_{|l|\geq 100|\xi|}\langle k-l\rangle e^{C|k-l|^{\f13}} e^{C|k-l,\xi-\eta|^{\f13}} d\xi d\eta\\
\lesssim &\sum_{k,l}\int_{\eta,\xi} \rmA_k(\eta)|k|^{\f s2}|\bar{\hat{\rho}}(k,\eta)||\hat{\mathbf{u}}(k-l,\eta-\xi)_{<\rmN/8}|\chi^{\rmL}\mathbf{1}_{|l|\geq 100|\xi|}\\
&\times |l|^{\f s2}\rmA_{l}(\xi)|\hat{\rho}_l(\xi)_{\rmN}|e^{\la|k-l,\eta-\xi|^s}
 d\xi d\eta\\
 \lesssim & \Vert |\nabla|^{\f s2}\rmA \rho_{\sim\rmN} \Vert_2 \Vert |\nabla|^{\f s2}\rmA \rho_{\rmN} \Vert_2 \Big(\Vert v'\na P_{\neq}\phi\Vert_{\mathcal{G}^{\la,\b;s}} 
+\|g\|_{\mathcal{G}^{\la,\b;s}_1}\Big)\\
 \lesssim & \f{\sqrt{\ep}}{\langle t\rangle^{2-\ep_1/2}}\Vert |\nabla|^{\f s2}\rmA \rho_{\sim\rmN} \Vert_2 \Vert |\nabla|^{\f s2}\rmA \rho_{\rmN} \Vert_2. 
\end{align*}

Now to estimate $\rmT_{1;\rmN}^{2;\rmL,v }$ and in order to use the decay rate of the nonzero mode, we separate the zero mode and the nonzero mode as follows:
\begin{align*}
\rmT_{1;\rmN}^{2;\rmL,v }=&\,i\sum_{k\neq l}\int_{\eta,\xi} \rmA_k(\eta)|\bar{\hat{\rho}}(k,\eta)||\widehat{v^\prime\nabla _{z,v }^{\perp }P_{\neq }\phi}(k-l,\eta-\xi)_{<\rmN/8}||l,\xi|\rmA_{l}(\xi)|\hat{\rho}_{\rmN}(\xi)|e^{\la|k,\eta|^s-\la|l,\xi|^s}\\
&\times \chi^{\rmL}\mathbf{1}_{|l|\leq 100|\xi|}\left[\f{\rmJ_k(\eta)}{\rmJ_l(\xi)}-1\right]\f{\cM_k(\eta)}{\cM_l(\xi)}\f{\langle k,\eta\rangle^{\s}}{\langle l,\xi\rangle^{\s}}\f{\rmB_k(\eta)}{\rmB_l(\xi)} d\xi d\eta\\
&+i\sum_{k}\int_{\eta,\xi} \rmA_k(\eta)|\bar{\hat{\rho}}(k,\eta)||g(\eta-\xi)_{<\rmN/8}||\xi|\rmA_{k}(\xi)|\hat{\rho}_k(\xi)_{\rmN}|e^{\la|k,\eta|^s-\la|l,\xi|^s}\\
&\times \chi^{\rmL}\mathbf{1}_{|k|\leq 100|\xi|}\left[\f{\rmJ_k(\eta)}{\rmJ_k(\xi)}-1\right]\f{\cM_k(\eta)}{\cM_l(\xi)}\f{\langle k,\eta\rangle^{\s}}{\langle k,\xi\rangle^{\s}}\f{\rmB_k(\eta)}{\rmB_k(\xi)} d\xi d\eta\\
=&\,\rmT_{1;\rmN}^{2;\rmL,v,\neq }+\rmT_{1;\rmN}^{2;\rmL,v,0 }.
\end{align*}

Now, we estimate $\rmT_{1;\rmN}^{2;\rmL,v,\neq }$. By the argument in the proof of the reaction term, we have 
\begin{equation}\label{Estimate_J_J}
\f{\rmJ_k(\eta)}{\rmJ_l(\xi)}\lesssim |\xi|e^{C|k-l,\xi-\eta|^{\f13}}.
\end{equation}
We need to pay decay  in time in order to gain regularity.  We have on the support of the integrand  $|l,\xi|\lesssim |\xi|\lesssim t^{\f32}$ which implies, by using  \eqref{Support_Ineq} that 
\begin{equation}\label{Ineq_t_xi}
|l,\xi|^2\lesssim |l,\xi|^{\f s2}|k,\eta|^{\f s2}t^{\f32(2-s)}. 
\end{equation}
Hence, 
we have by making use of  \eqref{Eq_M/M}, \eqref{Support_Ineq} \eqref{Estimate_J_J} and \eqref{Ineq_t_xi},
\begin{equation}    
\begin{aligned}  
|\rmT_{1;\rmN}^{2;\rmL;v,\neq }|\lesssim&\,\sum_{k,l}\int_{\eta,\xi} \rmA_k(\eta)|\bar{\hat{\rho}}(k,\eta)||\widehat{v^\prime\nabla _{z,v }^{\perp }P_{\neq }\phi}(k-l,\eta-\xi)_{<\rmN/8}||l,\xi|\rmA_{l}(\xi)|\hat{\rho}_l(\xi)_{\rmN}|\\
&\times |\xi| \chi^{\rmL}\mathbf{1}_{|l|\leq 100|\xi|}\mathbf{1}_{k\neq l}\langle k-l,\xi-\eta\rangle e^{C\d_{\rmL}^{-1}|k-l,\xi-\eta|^{\f13}}e^{c\la|k-l,\eta-\xi|^s} d\xi d\eta\\
\lesssim&\,
\sum_{k,l}\int_{\eta,\xi}t^{3-\f32 s} |k,\eta|^{\f s2}\rmA_k(\eta)|\bar{\hat{\rho}}(k,\eta)||\widehat{v^\prime\nabla _{z,v }^{\perp }P_{\neq }\phi}(k-l,\eta-\xi)_{<\rmN/8}|\\
&\times|l,\xi|^{\f s2}\rmA_{l}(\xi)|\hat{\rho}_{\rmN}(\xi)|e^{\la|k-l,\eta-\xi|^s} d\xi d\eta\\
\lesssim &\, t^{3-\f32 s}\Vert v^\prime\nabla _{z,v }^{\perp }P_{\neq }\phi\Vert_{\mathcal{G}^{\lambda,\b;s}}\Vert |\nabla|^{\f s2}\rmA \rho_{\sim\rmN} \Vert_2 \Vert |\nabla|^{\f s2}\rmA \rho_{\rmN} \Vert_2\\
\lesssim &\, \f{\ep}{\langle t\rangle^{\f32s+1}}\Vert |\nabla|^{\f s2}\rmA \rho_{\sim\rmN} \Vert_2 \Vert |\nabla|^{\f s2}\rmA \rho_{\rmN} \Vert_2. 
\end{aligned}
\end{equation} 
Our next goal is to estimate $\rmT_{1;\rmN}^{2;\rmL;v,0}$. 
By the argument in the reaction term, we have for $t\notin \rmI_{k,\eta}$ or $t\in \rmI_{k,\eta}\cap \rmI_{k,\xi}$,
\begin{equation}\label{Ineq_Mult_Zero_mode}
\f{\rmJ_k(\eta)}{\rmJ_k(\xi)}\lesssim e^{C\mu|\xi-\eta|^{\f13}},
\end{equation}
and for $t\in \rmI_{k,\eta}$ and $t\notin \rmI_{k,\xi}$, 
\begin{equation}
\f{\rmJ_k(\eta)}{\rmJ_k(\xi)}\lesssim \f{|\eta|}{k^3(1+|t-\f{\eta}{k}|)}e^{C\mu|\xi-\eta|^{\f13}}\lesssim \f{\xi\langle\xi-\eta\rangle}{k^3(1+|t-\f{\xi}{k}|)} e^{C\mu|\xi-\eta|^{\f13}}\lesssim \langle\xi-\eta\rangle e^{C\mu|\xi-\eta|^{\f13}}.
\end{equation}

By making use of  \eqref{Support_Ineq}, \eqref{Eq_M/M}, we obtain by using the fact that on the support of the integrand, we have $|\xi|\lesssim t^{\f32}$ and hence, $|\xi|\leq |\xi|^{\f s2}|\eta|^{\f s2}t^{\f32(1-s)}$ and 
\begin{align*}
|\rmT_{1;\rmN}^{2;\rmL;v,0}|\lesssim&\,\sum_{k}\int_{\eta,\xi} \rmA_k(\eta)|\bar{\hat{\rho}}(k,\eta)||\hat{g}(\eta-\xi)_{<\rmN/8}||\xi|\rmA_{k}(\xi)|\hat{\rho}_{k}(\xi)_{\rmN}|e^{\la|k,\eta|^s-\la|k,\xi|^s}\\
&\times \chi^{\rmL}\mathbf{1}_{|k|\leq 100|\xi|}\left|\f{\rmJ_k(\eta)}{\rmJ_k(\xi)}-1\right|\f{\cM_k(\eta)}{\cM_k(\xi)}\f{\langle k,\eta\rangle^{\s}}{\langle k,\xi\rangle^{\s}} d\xi d\eta\\
\lesssim&\,\sum_{k}\int_{\eta,\xi} t^{\f32(1-s)}\rmA_k(\eta)|k,\eta|^{\f s2}|\bar{\hat{\rho}}(k,\eta)||\hat{g}(\eta-\xi)_{<\rmN/8}||\xi|^{\f s2}\rmA_{k}(\xi)|\hat{\rho}_k(\xi)_{\rmN}|e^{\la|\eta-\xi|^s}\chi^{\rmL} d\xi d\eta\\
 \lesssim &\, t^{\f32-\f32 s}\Vert g\Vert_{\mathcal{G}_1^{\la,\b;s}}\Vert |\nabla|^{\f s2}\rmA \rho_{\sim\rmN} \Vert_2 \Vert |\nabla|^{\f s2}\rmA \rho_{\rmN} \Vert_2
 \lesssim \f{\sqrt{\ep}}{\langle t\rangle^{\f32s+\f12-\f{\ep_1}{2}}}
 \Vert |\nabla|^{\f s2}\rmA \rho_{\sim\rmN} \Vert_2 \Vert |\nabla|^{\f s2}\rmA \rho_{\rmN} \Vert_2. 
\end{align*}

\subsubsection{Term $\rmT_{1;\rmN}^3$} \label{T_3}
 Now, we turn to the term $\rmT_{1;\rmN}^3$. We have 
\begin{equation*}
\begin{aligned}
\rmT_{1;\rmN}^3=&\,i\sum_{k,l}\int_{\eta,\xi} \rmA_k(\eta)\bar{\hat{\rho}}_k(\eta)\hat{\bf{u}}(k-l,\eta-\xi)_{<\rmN/8}\cdot (l,\xi)\rmA_{l}(\xi)\hat{\rho}_{l}(\xi)_{\rmN}e^{\la|k,\eta|^s-\la|l,\xi|^s}\\
&\times (\tilde{\chi}^\rmS+\tilde{\chi}^\rmL)\left[\f{\cM_k(\eta)}{\cM_l(\xi)}-1\right]\f{\langle k,\eta\rangle^{\s}}{\langle l,\xi\rangle^{\s}}\f{\rmB_k(\eta)}{\rmB_l(\xi)}d\xi d\eta\\
=&\, \rmT_{1;\rmN}^{3;\rmS}+\rmT_{1;\rmN}^{3;\rmL},
\end{aligned}
\end{equation*}
where $\tilde{\chi}^\rmS=\mathbf{1}_{t\leq \f12\min\{|\xi|^{\f13},|\eta|^{\f13}\}}$ and $\tilde{\chi}^\rmL=1-\tilde{\chi}^\rmS$.

Following the same steps as in the estimate of the term $\rmT_{1;\rmN}^{2;\rmS}$, we use Lemma \ref{lem: g/g-1-1} to gain $\f23$-derivatives. Indeed, we have by using Lemma \ref{lem: g/g-1-1} and    \eqref{Support_Ineq},  
\begin{align*}
|\rmT_{1;\rmN}^{3;\rmS}|\lesssim&\, \sum_{k,l}\int_{\eta,\xi} \rmA_k(\eta)|\bar{\hat{\rho}}(k,\eta)||\hat{\mathbf{u}}(k-l,\eta-\xi)_{<\rmN/8}| |l,\xi|^{\f13}\rmA_{l}(\xi)|\hat{\rho}_{\rmN}(l,\xi)|\\
&\times e^{\la|k,\eta|^s-\la|l,\xi|^s}\tilde{\chi}^\rmS \langle k-l,\xi-\eta\rangle e^{C\d_{L}^{-1}|k-l,\xi-\eta|^{\f13}}d\xi d\eta\\
\lesssim&\,\sum_{k,l}\int_{\eta,\xi} \rmA_k(\eta)|k,\eta|^{\f s2}|\bar{\hat{\rho}}(k,\eta)||\hat{\mathbf{u}}(k-l,\eta-\xi)_{<\rmN/8}| |l,\xi|^{\f s2}\rmA_{l}(\xi)|\hat{\rho}_{\rmN}(l,\xi)|\\
&\times e^{\la|k-l,\eta-\xi|^s}\tilde{\chi}^\rmS d\xi d\eta\\
\lesssim &\,  \Vert |\nabla|^{\f s2}\rmA \rho_{\sim\rmN} \Vert_2 \Vert |\nabla|^{\f s2}\rmA \rho_{\rmN} \Vert_2 \Big(\Vert v'\na P_{\neq}\phi\Vert_{\mathcal{G}^{\la,\b;s}} 
+\|g\|_{\mathcal{G}^{\la,\b;s}_1}\Big)\\
\lesssim & \f{\sqrt{\ep}}{\langle t\rangle^{2-\ep_1/2}}\Vert |\nabla|^{\f s2}\rmA \rho_{\sim\rmN} \Vert_2 \Vert |\nabla|^{\f s2}\rmA \rho_{\rmN} \Vert_2. 
\end{align*}  

Next, we estimate $\rmT_{1;\rmN}^{3;\rmL}$, we have as in the estimate of  $\rmT_{1;\rmN}^{2;\rmL}$ 
 \begin{equation}
\begin{aligned}
\rmT_{1;\rmN}^{3;\rmL}=&\,i\sum_{k,l}\int_{\eta,\xi} \rmA_k(\eta)\bar{\hat{\rho}}_k(\eta)\hat{\mathbf{u}}(k-l,\eta-\xi)_{<\rmN/8}\cdot (l,\xi)\rmA_{l}(\xi)\hat{\rho}_l(\xi)_{\rmN}\\
&\times e^{\la|k,\eta|^s-\la|l,\xi|^s}\tilde{\chi}^\rmL(\mathbf{1}_{k\neq l}+\mathbf{1}_{k=l})\left[\f{\cM_k(\eta)}{\cM_l(\xi)}-1\right]\f{\langle k,\eta\rangle^{\s}}{\langle l,\xi\rangle^{\s}}\f{\rmB_k(\eta)}{\rmB_l(\xi)}d\xi d\eta\\
=&\rmT_{1;\rmN}^{3;\rmL,\neq}+\rmT_{1;\rmN}^{3;\rmL,0}. 
\end{aligned}
\end{equation}

Let us start by estimating $\rmT_{1;\rmN}^{3;\rmL,\neq}$. We have 
\begin{equation}
\begin{aligned}
\rmT_{1;\rmN}^{3;\rmL,\neq}=&\,i\sum_{k,l}\int_{\eta,\xi} \rmA_k(\eta)\bar{\hat{\rho}}_k(\eta)\widehat{v^\prime\nabla _{z,v }^{\perp }P_{\neq }\phi}(k-l,\eta-\xi)_{<\rmN/8}\cdot (l,\xi)\rmA_{l}(\xi)\hat{\rho}_l(\xi)_{\rmN}\\
&\times e^{\la|k,\eta|^s-\la|l,\xi|^s}\tilde{\chi}^\rmL(\mathbf{1}_{|l|\leq 100|\xi|}+\mathbf{1}_{|l|\geq 100|\xi|})\left[\f{\cM_k(\eta)}{\cM_l(\xi)}-1\right]\f{\langle k,\eta\rangle^{\s}}{\langle l,\xi\rangle^{\s}}\f{\rmB_k(\eta)}{\rmB_l(\xi)}d\xi d\eta\\
=&\rmT_{1;\rmN}^{3;\rmL,\neq,v}+\rmT_{1;\rmN}^{3;\rmL,\neq,z}. 
\end{aligned}
\end{equation}

Apllying \eqref{Eq_M/M}, together with the fact that on the support of the integrand, it holds that $|l,\xi|\lesssim |\xi|\lesssim t^{3}$
which implies 
$|l,\xi|\leq |l,\xi|^{\f s2}|k,\eta|^{\f s2}t^{3(1-s)}$ we get as in the above estimates 
 \begin{equation}
\begin{aligned}
|\rmT_{1;\rmN}^{3;\rmL,\neq,v}|\lesssim &\, t^{3-3 s}\Vert v^\prime\nabla _{z,v }^{\perp }P_{\neq }\phi\Vert_{\mathcal{G}^{\la,\b;s}}\Vert |\nabla|^{\f s2}\rmA \rho_{\sim\rmN} \Vert_2 \Vert |\nabla|^{\f s2}\rmA \rho_{\rmN} \Vert_2\\
\lesssim & \f{\ep}{\langle t\rangle^{3s+1}}\Vert |\nabla|^{\f s2}\rmA \rho_{\sim\rmN} \Vert_2 \Vert |\nabla|^{\f s2}\rmA \rho_{\rmN} \Vert_2.
\end{aligned}
\end{equation}

Next we estimate $\rmT_{1;\rmN}^{3;\rmL,\neq,z}$. We have 
\beno
\left|\f{\cM_k(\eta)}{\cM_l(\xi)}-1\right|\lesssim \f{1}{|l|^{\f23}}+\f{\langle k-l,\eta-\xi\rangle}{|k|^{\f23}} e^{C\d_{\rmL}^{-1}|k-l,\eta-\xi|^{\f13}}. 
\eeno  
Hence, we obtain 
\begin{equation*}
\rmT_{1;\rmN}^{3,\neq;\rmL,z}\lesssim  \Vert |\nabla|^{\f s2}\rmA \rho_{\sim\rmN} \Vert_2 \Vert |\nabla|^{\f s2}\rmA \rho_{\rmN} \Vert_2 \Vert v^\prime\nabla _{z,v }^{\perp }P_{\neq }\phi\Vert_{\mathcal{G}^{\la,\b;s}}\lesssim \f{\ep}{\langle t\rangle^4}\Vert |\nabla|^{\f s2}\rmA \rho_{\sim\rmN} \Vert_2 \Vert |\nabla|^{\f s2}\rmA \rho_{\rmN} \Vert_2. 
\end{equation*}

Let us now turn to the zero mode term $\rmT_{1;\rmN}^{3;\rmL,0}$. We write
We write
\begin{align*}
\rmT_{1;\rmN}^{3;\rmL,0}=&\,i\sum_{k}\int_{\eta,\xi} \rmA_k(\eta)\bar{\hat{\rho}}_k(\eta)\hat{g}(\eta-\xi)_{<\rmN/8} \xi\rmA_{l}(\xi)\hat{\rho}_k(\xi)_{\rmN}e^{\la|k,\eta|^s-\la|k,\xi|^s}\\
&\times (\chi_0^{\rmM}+\chi_0^{\ast})\left[\f{\cM_k(\eta)}{\cM_k(\xi)}-1\right]\f{\langle k,\eta\rangle^{\s}}{\langle k,  \xi\rangle^{\s}} \f{\rmB_k(\eta)}{\rmB_k(\xi)}d\xi d\eta\\
=&\,\rmT_{1;\rmN}^{3;\rmL,0;\rmM}+\rmT_{1;\rmN}^{3,\rmL,0;\ast},
\end{align*}
where 
\begin{equation*}
\chi_0^\rmM=\tilde{\chi}^\rmL \chi^\rmS=\mathbf{1}_{\f12\min\{|\xi|^{\f13},|\eta|^{\f13}\}\leq t\leq \f12\min\{|\xi|^{\f23},|\eta|^{\f23}\}}\quad \text{ and }\quad \chi_0^\ast=\tilde{\chi}^\rmL  \chi^\rmL=\mathbf{1}_{t\geq \f12\min\{|\xi|^{\f23},|\eta|^{\f23}\}}
\end{equation*}

We first estimate the term $\rmT_{1;\rmN}^{3;\rmL,0;\rmM}$. 
  On the support of the integrand, we have  
 \begin{equation}
|\xi|^{\f23}\lesssim  |\xi|^{\f s2}|\eta|^{\f s2} |\xi|^{\f23-s}\lesssim |\xi|^{\f s2}|\eta|^{\f s2}t^{2-3s}.  
\end{equation}
Hence, we have by using Lemma \ref{Lemma_Comm_M_0} and  \eqref{Support_Ineq}, 

 \begin{equation*}
\begin{aligned}
|\rmT_{1;\rmN}^{3;\rmL,0;\rmM}|\lesssim &\,\sum_{k}\int_{\eta,\xi} t^{2-3s} \rmA_k(\eta)|\eta|^{\f s2}|\bar{\hat{\rho}}_k(\eta)||\hat{g}(\eta-\xi)_{<\rmN/8}| |\xi|^{\f s2}\rmA_{l}(\xi)|\hat{\rho}_{k}(\xi)_{\rmN}|e^{\la|\eta-\xi|^{s}}
 \chi_0^{\rmM}d\xi d\eta\\
 \lesssim &\, t^{2-3s}\Vert g\Vert_{\mathcal{G}_1^{\la,s;\b}}\Vert |\nabla|^{\f s2}\rmA \rho_{\sim\rmN} \Vert_2 \Vert |\nabla|^{\f s2}\rmA \rho_{\rmN} \Vert_2
 \lesssim \f{\sqrt{\ep}}{\langle t\rangle^{3s-\f{\ep_1}{2}}}\Vert |\nabla|^{\f s2}\rmA \rho_{\sim\rmN} \Vert_2 \Vert |\nabla|^{\f s2}\rmA \rho_{\rmN} \Vert_2.
\end{aligned}
\end{equation*}
Now, for  the term $\rmT_{1;\rmN}^{3,\rmL,0;\ast}$ on the support of the integrand, we have 
$
|\xi|\lesssim |\xi|^{\f s2}|\eta|^{\f s2} t^{\f32(1-s)}$. Hence, as we did above, we get the following estimate
\beno
\rmT_{1;\rmN}^{3,\rmL,0;\ast}\lesssim t^{\f32-\f32s}\Vert g\Vert_{\mathcal{G}_1^{\la,s;\b}}\Vert |\nabla|^{\f s2}\rmA \rho_{\sim\rmN} \Vert_2 \Vert |\nabla|^{\f s2}\rmA \rho_{\rmN} \Vert_2\lesssim \f{\sqrt{\ep}}{\langle t\rangle^{\f12+\f32s-\f{\ep_1}{2}}}\Vert |\nabla|^{\f s2}\rmA \rho_{\sim\rmN} \Vert_2 \Vert |\nabla|^{\f s2}\rmA \rho_{\rmN} \Vert_2. 
\eeno
Note here we need $\f12+\f32s-\f{\ep_1}{2}\geq \tilde{q}$. 

 \subsubsection{Term $\rmT_{1;\rmN}^4$} \label{Section_T_4}
To estimate $\rmT_{1;\rmN}^4$, we get by \eqref{Support_Ineq} that
\begin{equation*}
\Big|\f{\langle k,\eta\rangle^{\s}}{\langle l,\xi\rangle^{\s}}-1\Big|\lesssim \f{\Big||k,\eta|-|l,\xi|\Big|}{\langle l,\xi \rangle }\lesssim \f{|k-l,\eta-\xi|}{\langle l,\xi \rangle}.
\end{equation*}

 Consequently, using the above estimate, we obtain as above 
\begin{align*}
|\rmT_{1;\rmN}^4|\lesssim&\,\sum_{k,l}\int_{\eta,\xi} |\rmA_k(\eta)||\bar{\hat{\rho}}_k(\eta)||\hat{\bf{u}}(k-l,\eta-\xi)_{<\rmN/8}| |l,\xi||\rmA_{l}(\xi)\hat{\rho}_{l}(\xi)_{\rmN}|\\&\times \f{|k-l,\eta-\xi|}{\langle l,\xi \rangle}e^{\la|k,\eta|^s-\la|l,\xi|^s}d\xi d\eta\\
\lesssim&\,  \Vert \rmA \rho_{\sim\rmN} \Vert_2 \Vert \rmA \rho_{\rmN} \Vert_2 \Vert \na\mathbf{u}\Vert_{\mathcal{G}^{\la,\b;s}}\lesssim \f{\ep}{\langle t\rangle^2}\Vert \rmA \rho_{\sim\rmN} \Vert_2 \Vert \rmA \rho_{\rmN} \Vert_2. 
\end{align*} 
\subsubsection{Trem $\rmT_{1;\rmN}^5$} 
We get that 
\begin{align*}
\rmT_{1;\rmN}^5 \lesssim
&\sum_{k\neq l}\bigg|\int_{\eta,\xi} \rmA_k(\eta)\bar{\hat{\rho}}(k,\eta)\hat{\bf{u}}(k-l,\eta-\xi)_{<\rmN/8}\cdot (l,\xi)\rmA_{l}(\xi)\hat{\rho}_l(\xi)_{\rmN}e^{\la|k,\eta|^s-\la|l,\xi|^s}\\
&\quad \times \left[\f{\rmB_k(\eta)}{\rmB_l(\xi)}-1\right]d\xi d\eta\bigg|\\
&+\sum_{k}\bigg|\int_{\eta,\xi} \rmA_k(\eta)\bar{\hat{\rho}}(k,\eta)\hat{g}(\eta-\xi)_{<\rmN/8}\cdot \xi\rmA_{k}(\xi)\hat{\rho}_k(\xi)_{\rmN}e^{\la|k,\eta|^s-\la|k,\xi|^s}\\
&\quad \times \left[\f{\rmB_k(\eta)}{\rmB_k(\xi)}-1\right]d\xi d\eta\bigg|\\
=&\rmT_{1;\rmN,\neq}^5+\rmT_{1;\rmN,0}^5. 
\end{align*}
Let us first treat the term $\rmT_{1;\rmN,\neq}^5$. 
On the support of the integrand, we have either $|\eta|\geq 3|k|$ or $|\xi|\geq 3|l|$ which implies that $|k,l|\lesssim |\eta,\xi|$, $\min\{|\xi|^{\f13},|\eta|^{\f13}\}\lesssim t$ and $|\eta|\approx |\xi|$. 

Thus we have $|l,\xi|\lesssim |\xi|^{s}|\xi|^{1-s}\lesssim |\xi|^{\f s2}|\eta|^{\f s2}t^{3-3s}$, and then
\begin{align*}
|\rmT_{1;\rmN,\neq}^5|
&\lesssim t^{3-3s}\||\na|^{\f s2}\rmA\rho_{\sim\rmN}\|_2\|v'P_{\neq}\phi\|_{\cG^{\la,0;s}}\||\na|^{\f s2}\rmA\rho_{\rmN}\|_2\\
&\lesssim \f{\ep}{\langle t\rangle^{3s+1}}\||\na|^{\f s2}\rmA\rho_{\sim\rmN}\|_2\||\na|^{\f s2}\rmA\rho_{\rmN}\|_2.
\end{align*}

Next we focus on $\rmT_{1;\rmN,0}^5$, we get that on the support of the integrand, $t\gtrsim |\eta|^{\f13}\approx |\xi|^{\f13}$, $|k|\gtrsim |\eta|^{\f13}$ and 
\begin{align*}
\bigg|\f{\rmB_k(t,\eta)}{\rmB_k(t,\xi)}-1\bigg|
\lesssim \int_0^t\bigg|\f{b(s,{k},\eta)}{1+(s-\f{\eta}{k})^2}-\f{b(s,{k},\xi)}{1+(s-\f{\xi}{k})^2}\bigg|ds\lesssim \f{|\xi-\eta|}{|\eta|^{\f13}}.
\end{align*}
Thus we get that 
\beno
\bigg|\f{\rmB_k(t,\eta)}{\rmB_k(t,\xi)}-1\bigg||\xi|\lesssim 
|\xi|^{\f23}|\eta-\xi|\lesssim |\xi|^{\f s 2}|\eta|^{\f s 2}t^{2-3s}|\eta-\xi|
\eeno
and thus
\beno
|\rmT_{1;\rmN,0}^5|\lesssim t^{2-3s}\|\pa_vg\|_{\cG^{\la,0;s}}\|\na^{\f s 2}\rmA \rho_{\sim \rmN}\|_2\|\na^{\f s 2}\rmA \rho_{\rmN}\|_2
\lesssim \f{\ep}{\langle t\rangle^{3s}}\|\na^{\f s 2}\rmA \rho_{\sim \rmN}\|_2\|\na^{\f s 2}\rmA \rho_{\rmN}\|_2. 
\eeno

\subsection{Remainder}
The remainder is easy. We give the result and omit the proof. 
\ben
\mathcal{R}\lesssim \f{\ep^3}{\langle t\rangle^2}. 
\een

\section{Estimate of $\rmE_d(t)$}
In this section, we deal with the highest energy of the coordinate system $\rmE_d(t)$. We have
\begin{align*}  
\f{1}{2}\f{d}{dt}t\Vert \rmA \pa_{vv}  h \Vert_2^2 =&\,
\f12\Vert \rmA \pa_{vv}  h \Vert_2^2
+t\int \rmA \pa_{vv}\pa_t h \rmA\pa_{vv} hdv\\
&-\rmCK_{\la, h}-\rmCK_{\Theta,h}-\rmCK_{M,h}\\
 =&\,-t\int \rmA \pa_{vv}(g\pa_v h) \rmA\pa_{vv} hdv-\int \rmA \pa_{vv}f_0 \rmA\pa_{vv} hdv\\
 &-\rmCK_{\la,h}-\rmCK_{\Theta, h}-\rmCK_{M,h}
 -\f12\Vert \rmA \pa_{vv}  h \Vert_2^2. 
\end{align*}
Using integration by parts, we have 
\begin{equation}\label{g_h_Comm_1}
\begin{aligned}
\int \rmA\pa_{vv}(g\pa_v h) \rmA_0\pa_{vv} hdv=&\,-\f12 \int \pa_v g |\rmA \pa_{vv} h|^2dv\\
&+\int \rmA\pa_{vv} h \Big[\rmA\pa_{vv} (g\pa_v h)-g\pa_v\rmA\pa_{vv} h\Big]dv
\end{aligned}
\end{equation}
The first term in \eqref{g_h_Comm_1} can be estimated as follows:
\begin{equation}\label{Estimate_infty}
\begin{aligned}
t\left|\int \pa_v g |\rmA \pa_{vv} h|^2dv\right|\lesssim \Vert  \pa_vg\Vert_{H^2}\rmE_d(t). 
\end{aligned}
\end{equation}
 We write by using the paraproduct with respect to $v$:
\begin{equation}
\begin{aligned}
\int \rmA\pa_{vv} h \Big[\rmA\pa_{vv} (g\pa_v h)-g\pa_v\rmA\pa_{vv} h\Big]dv=\f{1}{2\pi}\sum_{\rmM\geq 8}\rmT_{\rmM}^v
+\f{1}{2\pi}\sum_{\rmM\geq 8}\rmR_{\rmM}^v
+\f{1}{2\pi}\cR^v
\end{aligned}
\end{equation}
where
\begin{align*}
\rmT_{\rmM}^v&=2\pi\int \rmA\pa_{vv}h \big[\rmA\pa_{vv}\big(g_{<\rmM/8}\pa_vh_{\rmM}\big)-g_{<\rmM/8}\pa_v\rmA\pa_{vv}h_{\rmM}\big]dv\\
\rmR_{\rmM}^v&=2\pi\int \rmA\pa_{vv}h \big[\rmA\pa_{vv}\big(g_{\rmM}\pa_vh_{<\rmM/8}\big)-g_{\rmM}\pa_v\rmA\pa_{vv}h_{<\rmM/8}\big]dv\quad \\
\cR^v&=2\pi\sum_{\rmM\in \bbD}\sum_{\f18\rmM\leq \rmM'\leq 8\rmM}\int \rmA\pa_{vv}h \big[\rmA\pa_{vv}\big(g_{\rmM}\pa_vh_{\rmM'}\big)-g_{\rmM}\pa_v\rmA\pa_{vv}h_{\rmM'}\big]dv. 
\end{align*}  

One may easily follow the argument in section 5 and get that
\ben
\begin{aligned}
&\left| t\int \rmA\pa_{vv} h \Big[\rmA\pa_{vv} (g\pa_v h)-g\pa_v\rmA\pa_{vv} h\Big]dv\right|\\
&\lesssim \sqrt{\ep} (\rmCK_{\la,h}+\rmCK_{w, h}+\rmCK_{M,h})
+\f{\ep^3}{\langle t\rangle^2}+\ep \langle t\rangle^2\|\rmA \pa_{vvv}g\|_2^2
\end{aligned}
\een

By \eqref{eq:K_0L^2}, \eqref{eq: f_0-1} and \eqref{eq: K_0-1}, we get that
\ben
\begin{aligned}
&\left|\int \rmA \pa_{vv}f_0 \rmA\pa_{vv} hdv\right|
\leq C_1\|\pa_v\rmA K_0\|_2^2+C\f{\ep^2}{\langle t\rangle^2}+\f{1}{16}\|\rmA\pa_{vv}h\|_2^2,
\end{aligned}
\een
with $C_1\geq 1$ independent of $\rmK_{d}$.  

Thus we obtain that 
\ben
\begin{aligned}
\rmE_d(t)+&\f{1}{4}\int_1^t (\rmCK_{\la,h}+\rmCK_{\Theta, h}+\rmCK_{M,h}+\|\rmA\pa_{vv}h\|_2^2)(s)ds \\
&\leq \rmE_d(1)+C\ep^2+C_1\int_1^t\|\pa_v\rmA K_0\|_2^2(s)ds.
\end{aligned}
\een
Note that by choosing $\rmK_{d}\geq 10C_1$, the last term on the right hand side will be absorbed by the dissipation term $\|\na_{L}\rmA K\|_2^2$.   

\section{Estimate of $\rmN\rmL_K^2$}
\label{Section_NL_K_2}  
In this section, we treat $\rmN\rmL_K^2$  and prove Proposition \ref{prop: NL_2}. 

\subsection{Treatment of $\bfK_1=\g^2\int\rmA K\rmA\big(\gamma^2 v^{\prime}\nabla^{\perp}_{z,v}\partial_z P_{\neq }\phi\cdot \nabla_{z,v} \rho\big)dzdv$}
We have
\begin{equation}
\begin{aligned}  
\mathbf{K}_{1}=&\,\gamma^2\int \rmA K \rmA \nabla^{\perp}_{z,v}\partial_z P_{\neq }\phi\cdot \nabla_{z,v} \rho dzdv\\
&\gamma^2\int \rmA K \rmA \big(h\nabla^{\perp}_{z,v}\partial_z P_{\neq }\phi\cdot \nabla_{z,v} \rho\big) dzdv\\
=&\,\mathbf{K}^1_{1}+\mathbf{K}_{1}^\ep.  
\end{aligned}
\end{equation}

As before, we use the paraproduct in $(z,v)$ and write   
\begin{equation*}
\begin{aligned}
\mathbf{K}_{1}^1=&\,\f{1}{2\pi}\sum_{\rmN\geq 8} \mathbf{K}_{1;\rmH\rmL}^{1;\rmN}
+\f{1}{2\pi}\sum_{\rmN\geq 8}\mathbf{K}_{1;\rmL\rmH}^{1;\rmN}
+\f{1}{2\pi}\sum_{\rmN\in \mathbf{D}}\mathbf{K}_{1;\rmH\rmH}^{1;\rmN}
\end{aligned}
\end{equation*}
where
\begin{align*}
\mathbf{K}_{1;\rmH\rmL}^{1;\rmN}=&\,-\sum_{k,l\neq 0}i\int_{\eta,\xi}\rmA \bar{\hat{K}}_{k}(\eta)\rmA_{k}(\eta)l(\eta l-\xi k)\hat{\phi}_{l}(\xi)_{\rmN}\widehat{\rho}_{k-l}(\eta-\xi)_{<\rmN/8}d\eta d\xi\\
\mathbf{K}_{1;\rmL\rmH}^{1;\rmN}=&\,-\sum_{k,l\neq 0}\int_{\eta,\xi}\rmA \bar{\hat{K}}_{k}(\eta)\rmA_{k}(\eta)(k-l)\widehat{\na^{\bot}\phi}_{k-l}(\eta-\xi)_{<\rmN/8}(l,\xi)\widehat{\rho}_{l}(\xi)_{\rmN}d\eta d\xi\\
\mathbf{K}_{1;\rmH\rmH}^{1;\rmN}=&\,
-\sum_{\f18\rmN\leq \rmN'\leq 8\rmN}\sum_{k,l\neq 0}\int_{\eta,\xi}\rmA \bar{\hat{K}}_{k}(\eta)\rmA_{k}(\eta)l\widehat{\na^{\bot}\phi}_{l}(\xi)_{\rmN'}(k-l,\eta-\xi)\widehat{\rho}_{k-l}(\eta-\xi)_{\rmN}d\eta d\xi.
\end{align*}
The high-low interaction similar to the reaction term in section 5. The main difference is there is one more derivative in $z$ direction acting on $\phi$. Luckily, we can take advantage of the dissipation term $\|\na_{L}\rmA K\|_2$. That is 
\beno
|l(\eta l-\xi k)|\leq (|k|+|k-l|)|l,\xi||k-l,\eta-\xi|. 
\eeno
Thus by following the argument of reaction term in section 5, we get that
\begin{align*}
|\mathbf{K}_{1;\rmH\rmL}^{1;\rmN}|
&\lesssim \ep \|\pa_z \rmA K_{\sim \rmN}\|_2
\left\|\left\langle\f{\pa_v}{t\pa_z}\right\rangle^{-1}\left(\f{|\na|^{\f{s}{2}}}{\langle t\rangle^{\f{3s}{2}}}\rmA+\sqrt{\f{\pa_t\rmg}{\rmg}}\tilde{\tilde{\rmA}}+\sqrt{\f{\pa_t\Theta}{\Theta}}\tilde{\rmA}\right)\pa_z^{-1}\Delta_{L}^2 P_{\neq}\phi_{\rmN}\right\|_2\\
&\quad +\ep \|\rmA_0 K_{\sim \rmN}\|_2\left\|\left\langle\f{\pa_v}{t\pa_z}\right\rangle^{-1}\left(\f{|\na|^{\f{s}{2}}}{\langle t\rangle^{\f{3s}{2}}}\rmA+\sqrt{\f{\pa_t\rmg}{\rmg}}\tilde{\tilde{\rmA}}+\sqrt{\f{\pa_t\Theta}{\Theta}}\tilde{\rmA}\right)\pa_z^{-1}\Delta_{L}^2 P_{\neq}\phi_{\rmN}\right\|_2,
\end{align*}
which implies that
\begin{align*}
\left|\sum_{\rmN\geq 8}\mathbf{K}_{1;\rmH\rmL}^{1;\rmN}\right|
&\lesssim \ep \|\na_{L} \rmA K\|_2^2+\f{\ep^3}{\langle t\rangle^{\f92}}\\
&\quad+\ep\left\|\left\langle\f{\pa_v}{t\pa_z}\right\rangle^{-1}\left(\f{|\na|^{\f{s}{2}}}{\langle t\rangle^{\f{3s}{2}}}\rmA+\sqrt{\f{\pa_t\rmg}{\rmg}}\tilde{\tilde{\rmA}}+\sqrt{\f{\pa_t\Theta}{\Theta}}\tilde{\rmA}\right)\pa_z^{-1}\Delta_{L}^2 P_{\neq}\phi\right\|_2^2.
\end{align*}

The low-high inaction is different from the transport term in section 5. Again we will take advantage of the dissipation term $\|\na_{L}\rmA K\|_2$. That is 
\beno
|(k-l)|l,\xi||\lesssim |k-l||k,\eta|\lesssim |k-l||k,\eta-kt|\langle t\rangle
\eeno
Thus we get that 
\begin{align*}
|\mathbf{K}_{1;\rmL\rmH}^{1;\rmN}|
\lesssim \langle t\rangle^2\|\phi_{\neq}\|_{\cG^{\la,0;s}}\|\rmA \rho_{\rmN}\|_2\|\na_{L}\rmA K_{\sim \rmN}\|_2.
\end{align*}
Here we use a rough estimate that 
\beno
\f{\rmJ_k(t,\eta)}{\rmJ_l(t,\xi)}\lesssim \langle t\rangle e^{C|k-l,\eta-\xi|^{\f13}}. 
\eeno
Thus we get that
\beno
\left|\sum_{\rmN\geq 8}\mathbf{K}_{1;\rmL\rmH}^{1;\rmN}\right|\lesssim \ep\|\na_{L} \rmA K_{\sim \rmN}\|_2^2
+\f{\ep^3}{\langle t\rangle^4}. 
\eeno

The high-high interaction is easy to deal with, we get that
\beno
\left|\sum_{\rmN\in \mathbf{D}}\mathbf{K}_{1;\rmH\rmH}^{1;\rmN}\right|\lesssim \f{\ep^3}{\langle t\rangle^4}.
\eeno

The treatment of $\bfK_1^{\ep}$ is similar. We use the paraproduct twice and write   
\begin{equation*}
\begin{aligned}
\mathbf{K}_{1}^\ep=&\,\f{1}{2\pi}\sum_{\rmN\geq 8} \mathbf{K}_{1;\rmH\rmL}^{\ep;\rmN}
+\f{1}{2\pi}\sum_{\rmN\geq 8}\mathbf{K}_{\ep;\rmL\rmH}^{1;\rmN}
+\f{1}{2\pi}\sum_{\rmN\in \mathbf{D}}\mathbf{K}_{1;\rmH\rmH}^{\ep;\rmN}
\end{aligned}
\end{equation*}
where $\mathbf{K}_{1;\rmH\rmL}^{\ep;\rmN}=
\mathbf{K}_{1;\rmH\rmL,\rmH\rmL}^{\ep;\rmN}
+\mathbf{K}_{1;\rmH\rmL,\rmL\rmH}^{\ep;\rmN}
+\mathbf{K}_{1;\rmH\rmL,\rmH\rmH}^{\ep;\rmN}$ and
\begin{align*}
\mathbf{K}_{1;\rmH\rmL,\rmH\rmL}^{\ep;\rmN}=&\,
-\sum_{\rmM\geq 8}\sum_{k,l\neq 0}i\int_{\eta,\xi}\rmA \bar{\hat{K}}_{k}(\eta)\rmA_{k}(\eta)\hat{h}(\xi-\xi')_{\rmM}\hat{\phi}_{l}(\xi')_{<\rmM/8}\\
&\qquad\qquad\qquad\qquad
\times l(\eta l-\xi k)\widehat{\rho}_{k-l}(\eta-\xi)_{<\rmN/8}\varphi_{\rmN}(l,\xi)d\eta d\xi' d\xi\\
\mathbf{K}_{1;\rmH\rmL,\rmL\rmH}^{\ep;\rmN}=&\,
-\sum_{\rmM\geq 8}\sum_{k,l\neq 0}i\int_{\eta,\xi}\rmA \bar{\hat{K}}_{k}(\eta)\rmA_{k}(\eta)\hat{h}(\xi-\xi')_{<\rmM/8}\hat{\phi}_{l}(\xi')_{\rmM}\\
&\qquad\qquad\qquad\qquad
\times l(\eta l-\xi k)\widehat{\rho}_{k-l}(\eta-\xi)_{<\rmN/8}\varphi_{\rmN}(l,\xi)d\eta d\xi' d\xi\\
\mathbf{K}_{1;\rmH\rmL,\rmH\rmH}^{\ep;\rmN}=&\,
-\sum_{\rmM\in \mathbf{D}}\sum_{\f18\rmM\leq \rmM'\leq 8\rmM}\sum_{k,l\neq 0}i\int_{\eta,\xi}\rmA \bar{\hat{K}}_{k}(\eta)\rmA_{k}(\eta)\hat{h}(\xi-\xi')_{\rmM}\hat{\phi}_{l}(\xi')_{\rmM'}\\
&\qquad\qquad\qquad\qquad\qquad
\times l(\eta l-\xi k)\widehat{\rho}_{k-l}(\eta-\xi)_{<\rmN/8}\varphi_{\rmN}(l,\xi)d\eta d\xi' d\xi\\
\mathbf{K}_{1;\rmL\rmH}^{\ep;\rmN}=&\,
-\sum_{k,l\neq 0}\int_{\eta,\xi}\rmA \bar{\hat{K}}_{k}(\eta)\rmA_{k}(\eta)(k-l)\widehat{h\na^{\bot}\phi}_{k-l}(\eta-\xi)_{<\rmN/8}(l,\xi)\widehat{\rho}_{l}(\xi)_{\rmN}d\eta d\xi\\
\mathbf{K}_{1;\rmH\rmH}^{\ep;\rmN}=&\,
-\sum_{\f18\rmN\leq \rmN'\leq 8\rmN}\sum_{k,l\neq 0}\int_{\eta,\xi}\rmA \bar{\hat{K}}_{k}(\eta)\rmA_{k}(\eta)l\widehat{h\na^{\bot}\phi}_{l}(\xi)_{\rmN'}(k-l,\eta-\xi)\widehat{\rho}_{k-l}(\eta-\xi)_{\rmN}d\eta d\xi.
\end{align*}

The treatment of $\mathbf{K}_{1;\rmL\rmH}^{\ep;\rmN}$ and $\mathbf{K}_{1;\rmH\rmH}^{\ep;\rmN}$ are same as $\mathbf{K}_{1;\rmL\rmH}^{1;\rmN}$ and $\mathbf{K}_{1;\rmH\rmH}^{1;\rmN}$ and we get that
\begin{align*}
\left|\sum_{\rmN\geq 8}\mathbf{K}_{1;\rmL\rmH}^{\ep;\rmN}\right|+\left|\sum_{\rmN\in \mathbf{D}}\mathbf{K}_{1;\rmH\rmH}^{\ep;\rmN}\right|\lesssim \ep^2\|\na_{L} \rmA K_{\sim \rmN}\|_2^2
+\f{\ep^4}{\langle t\rangle^4}.
\end{align*}

The treatment of $\mathbf{K}_{1;\rmH\rmL,\rmL\rmH}^{\ep;\rmN}$ is same as $\bfK_{1;\rmH\rmL}^{1,\rmN}$ and we get that
\begin{align*}
\left|\sum_{\rmN\geq 8}\mathbf{K}_{1;\rmH\rmL,\rmL\rmH}^{\ep;\rmN}\right|
&\lesssim \ep^2 \|\na_{L} \rmA K_{\sim \rmN}\|_2^2+\f{\ep^4}{\langle t\rangle^{\f92}}\\
&\quad+\ep^2\left\|\left\langle\f{\pa_v}{t\pa_z}\right\rangle^{-1}\left(\f{|\na|^{\f{s}{2}}}{\langle t\rangle^{\f{3s}{2}}}\rmA+\sqrt{\f{\pa_t\rmg}{\rmg}}\tilde{\tilde{\rmA}}+\sqrt{\f{\pa_t\Theta}{\Theta}}\tilde{\rmA}\right)\pa_z^{-1}\Delta_{L}^2 P_{\neq}\phi\right\|_2^2.
\end{align*}

One can follow the estimate of $\rmR^{\ep,1}_{1;\rmN;\rmH\rmL}$ and $\rmR^{\ep,1}_{1;\rmN;\rmH\rmH}$ in section 5 and get that
\begin{align*}
&\left|\sum_{\rmN\geq 8}\mathbf{K}_{1;\rmH\rmL,\rmH\rmL}^{\ep;\rmN}\right|+\left|\sum_{\rmN\geq 8}\mathbf{K}_{1;\rmH\rmL,\rmH\rmH}^{\ep;\rmN}\right|\\
&\lesssim \ep^2 \|\na_{L} \rmA K\|_2^2+\f{\ep^4}{\langle t\rangle^{4}}
+\ep^2\left\|\left\langle\f{\pa_v}{t\pa_z}\right\rangle^{-1}\f{|\na|^{\f{s}{2}}}{\langle t\rangle^{\f{3s}{2}}}\rmA\pa_z^{-1}\Delta_{L}^2 P_{\neq}\phi\right\|_2^2.
\end{align*}
\subsection{Treatment of $\bfK_2=-2\int\rmA K\rmA\big(h(\pa_v-t\pa_z)\partial_z f\big)dzdv$}
We have
\begin{align*}
\bfK_2=&\,\f{1}{2\pi}\sum_{\rmN\geq 8} \mathbf{K}_{2;\rmH\rmL}^{\rmN}
+\f{1}{2\pi}\sum_{\rmN\geq 8}\mathbf{K}_{2;\rmL\rmH}^{\rmN}
+\f{1}{2\pi}\sum_{\rmN\in \mathbf{D}}\mathbf{K}_{2;\rmH\rmH}^{\rmN},
\end{align*}
with 
\begin{align*}
&\mathbf{K}_{2;\rmH\rmL}^{\rmN}
=-2\sum_{k\neq 0}\int \rmA_k(t,\eta)\overline{\hat{K}}_k(t,\eta)\rmA_k(t,\eta) \hat{h}(\xi)_{\rmN}\widehat{\na_{L}\pa_zf}_{k}(\eta-\xi)_{<\rmN/8}d\xi d\eta,\\
&\mathbf{K}_{2;\rmL\rmH}^{\rmN}
=2\sum_{k\neq 0}\int \rmA_k(t,\eta)\overline{\hat{K}}_k(t,\eta)\rmA_k(t,\eta) \hat{h}(\eta-\xi)_{<\rmN/8}(\xi-kt)k\widehat{f}_k(\xi)_{\rmN}d\xi d\eta,\\
&\mathbf{K}_{2;\rmH\rmH}^{\rmN}
=2\sum_{\f18\rmN\leq \rmN'\leq 8\rmN}\sum_{k\neq 0}\int \rmA_k(t,\eta)\overline{\hat{K}}_k(t,\eta)\rmA_k(t,\eta) \hat{h}(\eta-\xi)_{\rmN'}(\xi-kt)k\widehat{f}_k(\xi)_{\rmN}d\xi d\eta.
\end{align*}
We first treat $\mathbf{K}_{2;\rmH\rmL}^{\rmN}$. We have
\beno
\f{\rmJ_k(t,\eta)}{\rmJ_0(t,\xi)}\lesssim \langle \xi\rangle e^{C|k,\eta-\xi|^{\f13}}
\eeno
Thus we get that
\begin{align*}
\left|\sum_{\rmN\geq 8}\mathbf{K}_{2;\rmH\rmL}^{\rmN}\right|
&\lesssim \sum_{\rmN\geq 8}\|\rmA P_{\neq}\rmK_{\sim\rmN}\|_2\|\langle \pa_v\rangle \rmA h_{\rmN}\|_2\|\na_{L}P_{\neq}f\|_{\cG^{\la,0;s}}\\
&\lesssim \f{\ep^3}{\langle t\rangle^2}+\ep\|\na_{L}\rmA P_{\neq}\rmK\|_2. 
\end{align*}
By using the fact that $\xi-kt\lesssim |\eta-kt|+|\eta-\xi|$, and 
\beno
\f{k^2}{k^2+(\xi-kt)^2}\lesssim \left\langle \f{\xi}{kt}\right\rangle^{-1}, 
\quad \f{\rmA_k(\eta)}{\rmA_k(\xi)}\lesssim e^{C|\eta-\xi|^{\f13}}
\eeno
we have that
\begin{align*}
\left|\sum_{\rmN\geq 8}\mathbf{K}_{2;\rmL\rmH}^{\rmN}\right|
&\lesssim \sum_{\rmN\geq 8}\|\na_{L}\rmA P_{\neq}\rmK_{\sim\rmN}\|_2\|\langle \pa_v\rangle h\|_{\cG^{\la,0;s}}\left\|\left\langle\f{\pa_v}{t\pa_z}\right\rangle^{-1}\rmA\pa_z^{-1}\Delta_{L}P_{\neq}f_{\rmN}\right\|_{2}\\
&\lesssim \ep\|\na_{L}\rmA P_{\neq}\rmK\|_2^2
+\f{\ep}{\langle t\rangle^2}\left\|\left\langle\f{\pa_v}{t\pa_z}\right\rangle^{-1}\rmA\pa_z^{-1}\Delta_{L}P_{\neq}f\right\|_{2}^2. 
\end{align*}

Similarly using the fact that $\xi-kt\lesssim |\eta-kt|+|\eta-\xi|$, we have
\begin{align*}
\left|\sum_{\rmN\in \mathbf{D}}\mathbf{K}_{2;\rmH\rmH}^{\rmN}\right|
\lesssim \ep\|\na_LP_{\neq }\rmA K\|_{2}^2+\f{\ep^3}{\langle t\rangle^2}. 
\end{align*}

\subsection{Treatment of $\bfK_3=2\int\rmA K\rmA\big(f_0v'(\pa_v-t\pa_z)\partial_z f\big)dzdv$}
The treatment of $\bfK_3$ is similar to $\bfK_2$. Here we show the result and omit the proof. We have
\ben
\begin{aligned}
|\bfK_3|\lesssim \ep\|\na_{L}\rmA P_{\neq}\rmK\|_2^2
+\f{\ep}{\langle t\rangle^2}\left\|\left\langle\f{\pa_v}{t\pa_z}\right\rangle^{-1}\rmA\pa_z^{-1}\Delta_{L}P_{\neq}f\right\|_{2}^2+\f{\ep^3}{\langle t\rangle^2}. 
\end{aligned}
\een

\subsection{Treatment of $\bfK_4=-2\int \rmA K\rmA\big((v')^3(\pa_v-t\pa_z)\na_L^{\bot}\phi_{\neq}\cdot\na_{L}(\pa_v-t\pa_z)f\big)dzdv$ and $\bfK_5=-2\int \rmA K\rmA\big(v'\pa_z\na_L^{\bot}\phi_{\neq}\cdot\na_{L}\pa_zf\big)dzdv$}

The treatment of $\bfK_4$ and $\bfK_5$ are similar. We will give the estimate of $\bfK_4$ and only show the result of $\bfK_5$. 
We have  
\begin{align*}
\bfK_4=&\,\f{1}{2\pi}\sum_{\rmN\geq 8} 
\mathbf{K}_{4;\rmH\rmL}^{1,\rmN}
+\f{1}{2\pi}\sum_{\rmN\geq 8} \mathbf{K}_{4;\rmH\rmL}^{\ep,\rmN}
+\f{1}{2\pi}\sum_{\rmN\geq 8}\mathbf{K}_{4;\rmL\rmH}^{\rmN}
+\f{1}{2\pi}\sum_{\rmN\in \mathbf{D}}\mathbf{K}_{4;\rmH\rmH}^{\rmN},
\end{align*}
with $\mathbf{K}_{4;\rmH\rmL}^{\ep,\rmN}
=\mathbf{K}_{4;\rmH\rmL,\rmH\rmL}^{\ep,\rmN}
+\mathbf{K}_{4;\rmH\rmL,\rmL\rmH}^{\ep,\rmN}
+\mathbf{K}_{4;\rmH\rmL,\rmH\rmH}^{\ep,\rmN}$
and
\begin{align*}
\mathbf{K}_{4;\rmH\rmL}^{1,\rmN}
&=i\sum_{k,l\neq 0}\int \rmA_k(\eta)\overline{\hat{K}}_k(\eta)\rmA_k(\eta)(\xi-lt)(l\eta-k\xi)\widehat{\phi}_l(\xi)_{\rmN}\mathcal{F}\big[(\pa_v-t\pa_z)f\big]_{k-l}(\eta-\xi)_{<\rmN/8}d\xi d\eta\\
\mathbf{K}_{4;\rmH\rmL,\rmH\rmL}^{\ep,\rmN}
&=\sum_{\rmM\geq 8}i\sum_{k,l\neq 0}\int \rmA_k(\eta)\overline{\hat{K}}_k(\eta)\rmA_k(\eta)G_3(\xi')_{\rmM}\widehat{\phi}_l(\xi-\xi')_{<\rmM/8}\varphi_{\rmN}(l,\xi)\\
&\qquad\qquad\qquad\times (\xi-\xi'-lt)(l\eta-k(\xi-\xi'))\mathcal{F}\big[(\pa_v-t\pa_z)f\big]_{k-l}(\eta-\xi)_{<\rmN/8}d\xi' d\xi d\eta\\
\mathbf{K}_{4;\rmH\rmL,\rmL\rmH}^{\ep,\rmN}
&=\sum_{\rmM\geq 8}i\sum_{k,l\neq 0}\int \rmA_k(\eta)\overline{\hat{K}}_k(\eta)\rmA_k(\eta)G_3(\xi-\xi')_{<\rmM/8}
\widehat{\phi}_l(\xi')_{\rmM}\varphi_{\rmN}(l,\xi)\\
&\qquad\qquad\qquad\times (\xi'-lt)(l\eta-k\xi')\mathcal{F}\big[(\pa_v-t\pa_z)f\big]_{k-l}(\eta-\xi)_{<\rmN/8}d\xi' d\xi d\eta\\
\mathbf{K}_{4;\rmH\rmL,\rmH\rmH}^{\ep,\rmN}
&=\sum_{\rmM\in \mathbf{D}}\sum_{\f18\rmM\leq \rmM'\leq 8\rmM}i\sum_{k,l\neq 0}\int \rmA_k(\eta)\overline{\hat{K}}_k(\eta)\rmA_k(\eta)G_3(\xi-\xi')_{\rmM'}
\widehat{\phi}_l(\xi')_{\rmM}\varphi_{\rmN}(l,\xi)\\
&\qquad\qquad\qquad\qquad\qquad\times (\xi'-lt)(l\eta-k\xi')\mathcal{F}\big[(\pa_v-t\pa_z)f\big]_{k-l}(\eta-\xi)_{<\rmN/8}d\xi' d\xi d\eta\\
\mathbf{K}_{4;\rmL\rmH}^{\rmN}
&=\sum_{k\neq l}\int \overline{\widehat{\rmA K}}_k(\eta)\rmA_k(\eta)\mathcal{F}\big[G_3(\pa_v-t\pa_z)\na_{L}^{\bot}\phi\big]_{k-l}(\eta-\xi)_{<\rmN/8}\cdot (l,\xi-lt)(\xi-lt)\widehat{f}_{l}(\xi)_{\rmN}d\xi d\eta\\
\mathbf{K}_{4;\rmH\rmH}^{\rmN}
&=\sum_{\f18\rmN\leq\rmN'\leq 8\rmN}\sum_{k\neq l }\int \rmA_k(\eta)\overline{\hat{K}}_k(\eta)\rmA_k(\eta)\mathcal{F}\big[G_3(\pa_v-t\pa_z)\na_{L}^{\bot}\phi\big]_{k-l}(\eta-\xi)_{\rmN'}\\
&\qquad\qquad \qquad\qquad  \cdot (l,\xi-lt)(\xi-lt)\widehat{f}_{l}(\xi)_{\rmN}d\xi d\eta.
\end{align*}
The key idea of estimate of $\mathbf{K}_{4;\rmH\rmL}^{1,\rmN}$ is to move the derivative $\pa_v-t\pa_z$ on $\phi_{\neq}$ to $K$ and using the dissipation $\|\na_{L}\rmA K\|_2$. 

Indeed we have on the support of the integrand, $|\xi-lt|\leq |\eta-kt|+|\eta-\xi-(k-l)t|\lesssim |k,\eta-kt|\langle \eta-\xi-(k-l)t\rangle$ which gives us that 
\begin{align*}
|\mathbf{K}_{4;\rmH\rmL}^{1,\rmN}|
&=\sum_{k,l\neq 0}\int \rmA_k(\eta)
|\overline{\widehat{\na_{L}K}}_k(\eta)|\rmA_k(\eta)|l\eta-k\xi||\widehat{\phi}_l(\xi)_{\rmN}||\mathcal{F}\big[\Delta_{L}f\big]_{k-l}(\eta-\xi)_{<\rmN/8}|d\xi d\eta.
\end{align*}
Thus by following the argument of reaction term in section 5, we get that
\begin{align*}
|\mathbf{K}_{4;\rmH\rmL}^{1;\rmN}|
&\lesssim \ep \|\na_L \rmA K_{\sim \rmN}\|_2
\left\|\left\langle\f{\pa_v}{t\pa_z}\right\rangle^{-1}\left(\f{|\na|^{\f{s}{2}}}{\langle t\rangle^{\f{3s}{2}}}\rmA+\sqrt{\f{\pa_t\rmg}{\rmg}}\tilde{\tilde{\rmA}}+\sqrt{\f{\pa_t\Theta}{\Theta}}\tilde{\rmA}\right)\pa_z^{-1}\Delta_{L}^2 P_{\neq}\phi_{\rmN}\right\|_2,
\end{align*}
where we use the fact that 
$
\|\Delta_L f\|_{\cG^{\la,\s-4;s}}\lesssim \ep. 
$

Similarly, by using the fact that $|\xi'-lt|\leq |\eta-kt|+|\eta-\xi-(k-l)t|+|\xi-\xi'|$, we can still move the derivative $\pa_v-t\pa_z$ on $\phi_{\neq}$ to $K$. Thus we get that 
\begin{align*}
|\mathbf{K}_{4;\rmH\rmL,\rmL\rmH}^{\ep;\rmN}|
&\lesssim \ep^2 \|\na_L \rmA K_{\sim \rmN}\|_2
\left\|\left\langle\f{\pa_v}{t\pa_z}\right\rangle^{-1}\left(\f{|\na|^{\f{s}{2}}}{\langle t\rangle^{\f{3s}{2}}}\rmA+\sqrt{\f{\pa_t\rmg}{\rmg}}\tilde{\tilde{\rmA}}+\sqrt{\f{\pa_t\Theta}{\Theta}}\tilde{\rmA}\right)\pa_z^{-1}\Delta_{L}^2 P_{\neq}\phi_{\rmN}\right\|_2,
\end{align*}

The treatment of $\mathbf{K}_{4;\rmH\rmL,\rmH\rmL}^{\ep;\rmN}$ and $\mathbf{K}_{4;\rmH\rmL,\rmH\rmH}^{\ep;\rmN}$ is similar to the $\rmR_{1;\rmN;\rmH\rmL}^{\ep,1}$ and $\rmR_{1;\rmN;\rmH\rmH}^{\ep,1}$. We get that
\begin{align*}
\sum_{\rmN\geq 8}|\mathbf{K}_{4;\rmH\rmL,\rmH\rmL}^{\ep;\rmN}|+|\mathbf{K}_{4;\rmH\rmL,\rmH\rmH}^{\ep;\rmN}|\lesssim \ep^2 \|\na_{L} \rmA K\|_2^2+\f{\ep^4}{\langle t\rangle^{4}}
+\ep^2\left\|\left\langle\f{\pa_v}{t\pa_z}\right\rangle^{-1}\f{|\na|^{\f{s}{2}}}{\langle t\rangle^{\f{3s}{2}}}\rmA\pa_z^{-1}\Delta_{L}^2 P_{\neq}\phi\right\|_2^2. 
\end{align*}

For the low-high interaction, by the fact that $|\xi-lt|\leq |\eta-kt|+|k-l,\eta-\xi|t$, we have
\begin{align*}
|\mathbf{K}_{4;\rmL\rmH}^{\rmN}|
&\leq \sum_{k\neq l,\, l\neq 0}\int 
|\overline{\widehat{\rmA K}}_k(\eta)|\rmA_k(\eta)\big|\mathcal{F}\big[G_3(\pa_v-t\pa_z)\na_{L}^{\bot}\phi\big]_{k-l}(\eta-\xi)_{<\rmN/8}\big||l||\widehat{\pa_z^{-1}\Delta_{L}f}_{l}(\xi)_{\rmN}|d\xi d\eta\\
&\quad+\sum_{k\neq 0}\int |\overline{\widehat{\rmA K}}_k(\eta)|\rmA_k(\eta)|\mathcal{F}\big[G_3(\pa_v-t\pa_z)\na_{L}^{\bot}\phi\big]_{k}(\eta-\xi)_{<\rmN/8}| 
|\widehat{\pa_{vv}f}_{0}(\xi)_{\rmN}|d\xi d\eta\\
&=\mathbf{K}_{4;\rmL\rmH,1}^{\rmN}
+\mathbf{K}_{4;\rmL\rmH,2}^{\rmN}.
\end{align*}
By using the fact that on the support of integrand, 
\beno
\f{|l|}{|l|+|\xi-lt|}\lesssim \left\langle\f{\xi}{lt}\right\rangle^{-1},
\eeno
and

By using the fact that on the support of integrand, $|l|\leq |k|+|k-l|$ and
\beno
\f{\rmJ_k(\eta)}{\rmJ_l(\xi)}\lesssim \langle t\rangle e^{C|k,\eta-\xi|^{\f13}},
\eeno
we have by \eqref{eq: f_0-1} that, 
\begin{align*}
\sum_{\rmN\geq 8}|\mathbf{K}_{4;\rmL\rmH,1}^{\rmN}|&\lesssim \sum_{\rmN\geq 8}\f{\ep}{\langle t\rangle}\big(\|\pa_z\rmA K_{\sim \rmN}\|_2+\|\rmA P_0K_{\sim\rmN}\|_2)\|\rmA P_{\neq}\pa_z^{-1}\Delta_{L}f_{\rmN}\|_2\\
&\lesssim \sum_{\rmN\geq 8}\f{\ep}{\langle t\rangle}\big(\|\pa_z\rmA K_{\sim \rmN}\|_2+\|\pa_v\rmA P_0K_{\sim\rmN}\|_2+\|P_0K_{\sim \rmN}\|_{L^2})\|\rmA P_{\neq}\pa_z^{-1}\Delta_{L}f_{\rmN}\|_2\\
&\lesssim \ep\|\rmA \na_{L}K\|_2^2+\f{\ep}{\langle t\rangle^2}\|\rmA P_{\neq}\pa_z^{-1}\Delta_{L}f\|_2^2.
\end{align*}

The treatment of $\mathbf{K}_{4;\rmH\rmH}^{\rmN}$ is easy. We have 
\begin{align*}
\sum_{\rmN\in\mathbf{D}}|\mathbf{K}_{4;\rmH\rmH}^{\rmN}|\lesssim \f{\ep^3}{\langle t\rangle^4}+\ep \|\na_{L} \rmA K\|_2^2
\end{align*}

We conclude that
\begin{align*}
|\mathbf{K}_4|+|\mathbf{K}_5|
&\lesssim \f{\ep^3}{\langle t\rangle^4}+\ep \|\na_{L} \rmA K\|_2^2+\f{\ep}{\langle t\rangle^2}\|\rmA P_{\neq}\pa_z^{-1}\Delta_{L}f\|_2^2\\
&\quad+\ep\left\|\left\langle\f{\pa_v}{t\pa_z}\right\rangle^{-1}\left(\f{|\na|^{\f{s}{2}}}{\langle t\rangle^{\f{3s}{2}}}\rmA+\sqrt{\f{\pa_t\rmg}{\rmg}}\tilde{\tilde{\rmA}}+\sqrt{\f{\pa_t\Theta}{\Theta}}\tilde{\rmA}\right)\pa_z^{-1}\Delta_{L}^2 P_{\neq}\phi\right\|_2^2. 
\end{align*}

\subsection{Estimate of $\mathbf{K}_6$}
The treatment of $\mathbf{K}_6$ is similar. Actually we have
\begin{align*}
\mathbf{K}_6
&=\int \pa_z\rmA K \rmA\Big(v''v'(\pa_v-t\pa_z)\phi_{\neq}(\pa_v-t\pa_z)f\Big)dzdv\\
&=\f{1}{2\pi} \sum_{\rmN\geq 8}\mathbf{K}_{6,\rmH\rmL}^{\rmN}+
\f{1}{2\pi} \sum_{\rmN\geq 8}\mathbf{K}_{6,\rmL\rmH}^{\rmN}
+\f{1}{2\pi} \sum_{\rmN\in\mathbf{D}}\mathbf{K}_{6,\rmH\rmH}^{\rmN}
\end{align*}
with $\mathbf{K}_{6,\rmH\rmL}^{\rmN}=\mathbf{K}_{6,\rmH\rmL,\rmL\rmH}^{\rmN}+\mathbf{K}_{6,\rmH\rmL,\rmH\rmL}^{\rmN}+\mathbf{K}_{6,\rmH\rmL,\rmH\rmH}^{\rmN}$, where
\begin{align*}
\mathbf{K}_{6,\rmH\rmL,\rmL\rmH}^{\rmN}
&=-\sum_{\rmM\geq 8}\sum_{l\neq 0,\, k\neq 0 }\int \overline{\widehat{\pa_z\rmA K}}_k(\eta)\rmA_k(\eta)
\widehat{(v''v')}(\xi-\xi')_{<\rmM/8}(\xi'-tl)\widehat{\phi}_l(\xi')_{\rmM}\\
&\quad\qquad\qquad\qquad\times \varphi_{\rmN}(l,\xi) (\eta-\xi-(k-l)t)\hat{f}_{k-l}(\eta-\xi)_{<\rmN/8}d\xi' d\xi d\eta\\
\mathbf{K}_{6,\rmH\rmL,\rmH\rmL}^{\rmN}
&=-\sum_{\rmM\geq 8}\sum_{l\neq 0,\, k\neq 0 }\int \overline{\widehat{\pa_z\rmA K}}_k(\eta)\rmA_k(\eta)
\widehat{(v''v')}(\xi-\xi')_{\rmM}(\xi'-tl)\widehat{\phi}_l(\xi')_{<\rmM/8}\\
&\quad\qquad\qquad\qquad\times \varphi_{\rmN}(l,\xi) (\eta-\xi-(k-l)t)\hat{f}_{k-l}(\eta-\xi)_{<\rmN/8}d\xi' d\xi d\eta\\
\mathbf{K}_{6,\rmH\rmL,\rmH\rmH}^{\rmN}
&=-\sum_{\rmM\in\mathbf{D}}\sum_{\f18\rmM\leq \rmM'\leq 8\rmM}\sum_{l\neq 0,\, k\neq 0 }\int \overline{\widehat{\pa_z\rmA K}}_k(\eta)\rmA_k(\eta)
\widehat{(v''v')}(\xi-\xi')_{\rmM}(\xi'-tl)\widehat{\phi}_l(\xi')_{\rmM'}\\
&\quad\qquad\qquad\qquad\times \varphi_{\rmN}(l,\xi) (\eta-\xi-(k-l)t)\hat{f}_{k-l}(\eta-\xi)_{<\rmN/8}d\xi' d\xi d\eta\\
\mathbf{K}_{6,\rmL\rmH}^{\rmN}
&=i\sum_{l\neq 0,\, k\neq 0 }\int \overline{\widehat{\pa_z\rmA K}}_k(\eta)\rmA_k(\eta)
\mathcal{F}\Big((v''v')(\pa_v-t\pa_z)\phi\Big)_{k-l}(\eta-\xi)_{<\rmN/8} (\xi-lt)\hat{f}_{l}(\xi)_{\rmN}d\xi d\eta\\
\mathbf{K}_{6,\rmH\rmH}^{\rmN}
&=i\sum_{\f18\rmN\leq \rmN'\leq 8\rmN}\sum_{l\neq 0,\, k\neq 0 }\int \overline{\widehat{\pa_z\rmA K}}_k(\eta)\rmA_k(\eta)
\mathcal{F}\Big((v''v')(\pa_v-t\pa_z)\phi\Big)_{k-l}(\eta-\xi)_{<\rmN/8} (\xi-lt)\hat{f}_{l}(\xi)_{\rmN}d\xi d\eta. 
\end{align*}
By using the fact that $\|v''v'\|_{\cG^{\la,\b;s}}\lesssim \f{\ep}{\langle t\rangle}$, 
\beno
\|\na_Lf\|_{\cG^{\la,\b;s}}\lesssim \|\pa_v f_0\|_{L^2}+\|\pa_{vv}f_0\|_{\cG^{\la,\b;s}}+\|\na_{L}P_{\neq}f\|_{\cG^{\la,\b;s}}\lesssim \f{\ep}{\langle t\rangle}, 
\eeno 
and $(\xi'-lt)\lesssim |l,\xi'|\langle t\rangle$, and by following the same argument of $\rmK_{1;\rmH\rmL,\rmL\rmH}^{\ep;\rmN}$, $\rmK_{1;\rmH\rmL,\rmH\rmL}^{\ep;\rmN}$ and $\rmK_{1;\rmH\rmL,\rmH\rmH}^{\ep;\rmN}$, we get that 
\begin{align*}
\left|\sum_{\rmN\geq 8}\mathbf{K}_{6;\rmH\rmL}^{\rmN}\right|
&\lesssim \ep^2 \|\na_{L} \rmA K\|_2^2+\f{\ep^4}{\langle t\rangle^{2}}\\
&\quad+\ep^2\left\|\left\langle\f{\pa_v}{t\pa_z}\right\rangle^{-1}\left(\f{|\na|^{\f{s}{2}}}{\langle t\rangle^{\f{3s}{2}}}\rmA+\sqrt{\f{\pa_t\rmg}{\rmg}}\tilde{\tilde{\rmA}}+\sqrt{\f{\pa_t\Theta}{\Theta}}\tilde{\rmA}\right)\pa_z^{-1}\Delta_{L}^2 P_{\neq}\phi\right\|_2^2.
\end{align*}

The estimates of $\mathbf{K}_{6,\rmL\rmH}^{\rmN}$ and $\mathbf{K}_{6,\rmH\rmH}^{\rmN}$ are obvious. We have
\begin{align*}
\sum_{\rmN\geq 8}|\mathbf{K}_{6,\rmL\rmH}^{\rmN}|
+\sum_{\rmN\in \mathbf{D}}|\mathbf{K}_{6,\rmH\rmH}^{\rmN}|
\lesssim \ep^2 \|\na_{L} \rmA K\|_2^2+\f{\ep^2}{\langle t\rangle^2}\|\pa_z^{-1}\Delta_{L}P_{\neq} f\|_2^2+\f{\ep^4}{\langle t\rangle^2}.
\end{align*}

\section{Linear terms}
\label{Section_Linear_Terms}
In this section, we deal with the linear terms $\Pi_{\rho}$, $\rmE$, $\Pi_K^1$ and $\Pi_K^2$. 
\subsection{Estimate of $\Pi_{\rho}$}
We have 
\begin{align*}
\Pi_{\rho}
&=\f{i}{2\pi}\sum_{k\neq 0}\int \rmA_k(\eta)\overline{\hat{\rho}}_k(\eta) \rmA_k(\eta)k \hat{\phi}_k(\eta)d\eta\\
&=-\f{1}{2\pi}\sum_{k\neq 0}\int \rmA_k(\eta)\overline{\hat{\rho}}_k(\eta) \rmA_k(\eta)\f{k^2}{(k^2+(\eta-kt)^2)^2} \pa_z^{-1}\Delta_{L}^2\hat{\phi}_k(\eta)d\eta.
\end{align*}
For $\eta>\f32kt$, we get that
\beno
\f{k^2}{(k^2+(\eta-kt)^2)^2}\lesssim \f{k^2}{\eta^4}\lesssim \left\langle \f{\eta}{kt}\right\rangle^{-1}\f{1}{t^4}.
\eeno
For $\eta<\f12kt$, we have $\left\langle \f{\eta}{kt}\right\rangle^{-1}\approx 1$ and
\beno
\f{k^2}{(k^2+(\eta-kt)^2)^2}\lesssim \f{1}{k^2t^4}\lesssim \left\langle \f{\eta}{kt}\right\rangle^{-1}\f{1}{t^4}.
\eeno
Thus in both cases we have
\beno
|\Pi_{\rho}|\lesssim \f{1}{\langle t\rangle^4}\|\rmA\rho\|_2\left\|\left\langle\f{\pa_v}{t\pa_z}\right\rangle^{-1}\pa_z^{-1}\Delta_{L}^2\rmA P_{\neq}\phi\right\|_2.
\eeno
Next we focus on the case $\f12kt\leq \eta\leq \f32kt$, in which case, it holds that $\f13\leq \left\langle \f{\eta}{kt}\right\rangle^{-1}\leq 1$. 

If $t\in \bar{\rmI}_{k,\eta}$, then 
\beno
\f{k^2}{(k^2+(\eta-kt)^2)^2}
\leq C\left\langle \f{\eta}{kt}\right\rangle^{-1}\f{1}{k^2(1+|\f{\eta}{k}-t|)^2}\leq C\d_{\rmL}\left\langle \f{\eta}{kt}\right\rangle^{-1}\f{\pa_t\rmg}{\rmg}
\eeno
which together with the fact that $\rmA\leq 2\tilde{\tilde{\rmA}}$ if $|k|\leq |\eta|$, gives us that
\beno
|\Pi_{\rho}|\leq C\d_{\rmL}\left\|\sqrt{\f{\pa_t\rmg}{\rmg}}\rmA\rho\right\|_2\left\|\left\langle\f{\pa_v}{t\pa_z}\right\rangle^{-1}\sqrt{\f{\pa_t\rmg}{\rmg}}\pa_z^{-1}\Delta_{L}^2\tilde{\tilde{\rmA}} P_{\neq}\phi\right\|_2.
\eeno

If $t\notin \bar{\rmI}_{k,\eta}$, then we consider the following two cases:
1. $t\leq t_{\rmE(|\eta|^{\f23}),\eta}$; 2. $t\in [t_{\rmE(|\eta|^{\f23}),\eta},2\eta]$ with $|t-\f{\eta}{k}|\gtrsim \f{\eta}{k^2}$. For the first case, using the fact that $kt\approx \eta$, we get that $k\gtrsim |\eta|^{\f23}\gtrsim t^2$ and then 
\beno
\f{k^2}{(k^2+(\eta-kt)^2)^2}\lesssim \left\langle \f{\eta}{kt}\right\rangle^{-1}\f{1}{t^4}.
\eeno
For the second case, we have $|\eta-kt|\gtrsim \f{\eta}{k}\approx t$ and then
\beno
\f{k^2}{(k^2+(\eta-kt)^2)^2}\lesssim \left\langle \f{\eta}{kt}\right\rangle^{-1}\f{k^2}{(k^2+t^2)^2}\lesssim \left\langle \f{\eta}{kt}\right\rangle^{-1}\f{1}{t^2}.
\eeno
Thus we conclude that
\beq
\begin{aligned}
|\Pi_{\rho}|
&\leq C\d_{\rmL}\rmCK_{M,\rho}
+\d_{\rmL}\left\|\left\langle\f{\pa_v}{t\pa_z}\right\rangle^{-1}\sqrt{\f{\pa_t\rmg}{\rmg}}\pa_z^{-1}\Delta_{L}^2\tilde{\tilde{\rmA}} P_{\neq}\phi\right\|_2^2\\
&\quad+\f{C}{\langle t\rangle^2}\|\rmA\rho\|_2\left\|\left\langle\f{\pa_v}{t\pa_z}\right\rangle^{-1}\pa_z^{-1}\Delta_{L}^2\rmA P_{\neq}\phi\right\|_2.
\end{aligned}
\eeq

\subsection{The diffusion term $\rmE$}
In this section, we study the diffusion term $\rmE$. We have
\begin{align*}
\rmE
&=\int \rmA K \rmA(\Delta_L K)  dzdv
+\int \rmA K \rmA(((v')^2-1)(\pa_v-t\pa_z)^2K)dzdv\\
&\quad+\int \rmA K \rmA(v''(\pa_v-t\pa_z)K)dzdv\\
&=-\|\na_L\rmA K\|_2^2+\rmE_1+\rmE_2.
\end{align*}
Let $\hat{G}_1(t,\eta)=\widehat{((v')^2-1)}(t,\eta)$, then by the fact that $|\xi-kt|\leq |\eta-kt|+|\xi-\eta|$, we have
\begin{align*}
|\rmE_1|
&\lesssim \sum_{k\neq 0}\int_{\xi,\eta}\rmA_k(\eta)|\hat{K}_k(\eta)|\rmA_k(\eta)|\widehat{\pa_vG}_1(t,\eta-\xi)||\xi-kt||\overline{\hat{K}_k(\xi)}|d\xi d\eta\\
&\quad+\sum_{k\neq 0}\int_{\xi,\eta}|\eta-kt|\rmA_k(\eta)|\hat{K}_k(\eta)|\rmA_k(\eta)|\hat{G}_1(t,\eta-\xi)||\xi-kt||\overline{\hat{K}_k(\xi)}|d\xi d\eta\\
&\quad+\int_{\xi,\eta}\rmA_0(\eta)|\hat{K}_0(\eta)|\rmA_0(\eta)|\widehat{\pa_vG}_1(t,\eta-\xi)||\xi||\overline{\hat{K}_0(\xi)}|d\xi d\eta\\
&\quad+\int_{\xi,\eta}|\eta|\rmA_0(\eta)|\hat{K}_0(\eta)|\rmA_0(\eta)|\hat{G}_1(t,\eta-\xi)||\xi||\overline{\hat{K}_k(\xi)}|d\xi d\eta\\
&\lesssim \rmE_{1,1}^{\neq}+\rmE_{1,2}^{\neq}+\rmE_{1,1}^{0}+\rmE_{1,2}^{0},
\end{align*}
and
\begin{align*}
|\rmE_2|
&\lesssim \sum_{k\neq 0}\int_{\xi,\eta}\rmA_k(\eta)|\hat{K}_k(\eta)|\rmA_k(\eta)|\widehat{\pa_vG}_1(t,\eta-\xi)||\xi-kt||\overline{\hat{K}_k(\xi)}|d\xi d\eta\\
&\quad+\int_{\xi,\eta}\rmA_0(\eta)|\hat{K}_0(\eta)|\rmA_0(\eta)|\widehat{\pa_vG}_1(t,\eta-\xi)||\xi||\overline{\hat{K}_0(\xi)}|d\xi d\eta
\lesssim \rmE_{1,1}^{\neq}+\rmE_{1,1}^{0}.
\end{align*}
We also have $v''=\f12\pa_v[(v')^2-1]=\f12\pa_vG_1$, which gives us that $|\rmE_2|\lesssim \rmE_{1,1}$. 
For $k=0$, we use the fact that the norm defined by $\rmA$ is an algebra when restricted to the zero mode and obtain that
\begin{align*}
|\rmE_{1,1}^{0}|+|\rmE_{1,2}^{0}|
&\lesssim \|\rmA\pa_vG_1\|_2\|\rmA\pa_vK_0\|_2\|\rmA K_0\|_2+\|\rmA G_1\|_2\|\rmA\pa_vK_0\|_2^2\\
&\lesssim \|\rmA\langle\pa_v\rangle G_1\|_2\Big(\|\rmA\pa_vK_0\|_2^2+\|K_0\|_{L^2}^2\Big)\lesssim \ep\|\pa_v\rmA K_0\|_2^2+\f{\ep^3}{\langle t\rangle^2}
\end{align*}
By using the fact that
\beno
\rmA_k(\eta)\lesssim \langle \eta-\xi\rangle\rmA_0(\eta-\xi)\rmA_k(\xi),
\eeno
we get that
\beno
|\rmE_{1,1}^{\neq}|+|\rmE_{1,2}^{\neq}|\lesssim \|\na_{L}\rmA P_{\neq}K\|_2\|\rmA\langle\pa_v\rangle^2G_1\|_2\|\na_{L}\rmA P_{\neq}K\|_2\lesssim \ep\|\na_{L}\rmA P_{\neq}K\|_2^2.
\eeno
Thus we get by taking $\ep$ small enough that
\ben
\rmE\leq -\f78\|\na_L\rmA K\|_2^2+\f{C\ep^3}{\langle t\rangle^2}. 
\een

\subsection{Estimate of $\Pi_K^1$}
One can easily follow the argument of $\Pi_{\rho}$ and get that
\begin{align*}
|\Pi_K^1|
&\leq \f{C}{\langle t\rangle^2}\|\pa_z\rmA K\|_{2}
\left\|\left\langle\f{\pa_v}{t\pa_z}\right\rangle^{-1}\pa_z^{-1}\Delta_{L}^2\rmA P_{\neq}\phi\right\|_2\\
&\quad+C\d_{\rmL}\left\|\pa_z\rmA K\right\|_2\left\|\left\langle\f{\pa_v}{t\pa_z}\right\rangle^{-1}\sqrt{\f{\pa_t\rmg}{\rmg}}\pa_z^{-1}\Delta_{L}^2\tilde{\tilde{\rmA}} P_{\neq}\phi\right\|_2\\
&\leq \f{1}{16}\|\pa_z\rmA K\|_{2}^2+\f{C}{\langle t\rangle^4}\left\|\left\langle\f{\pa_v}{t\pa_z}\right\rangle^{-1}\pa_z^{-1}\Delta_{L}^2\rmA P_{\neq}\phi\right\|_2^2\\
&\quad+C\d_{\rmL}\left\|\left\langle\f{\pa_v}{t\pa_z}\right\rangle^{-1}\sqrt{\f{\pa_t\rmg}{\rmg}}\pa_z^{-1}\Delta_{L}^2\tilde{\tilde{\rmA}} P_{\neq}\phi\right\|_2^2.
\end{align*}

\subsection{Estimate of $\Pi_K^2$}\label{section: pi_k^2}
We get that
\begin{align*}
\Pi_K^2
&=2\int \na_{L}\rmA K \rmA\pa_z^2\Delta_{L}^{-1}\pa_z^{-1}\Delta_{L}P_{\neq}fdzdv\\
&=\f{1}{\pi}\sum_{k\neq 0}\int \overline{\widehat{\na_{L}\rmA K}}_k(t,\eta) \rmA_k(t,\eta)\f{k^2}{k^2+(\eta-kt)^2}\widehat{\pa_z^{-1}\Delta_{L}f}_k(\eta)d\eta. 
\end{align*}
We get for $\eta>\f32 kt$, 
\beno
\f{k^2}{k^2+(\eta-kt)^2}\lesssim \f{k^2}{\eta^2}\lesssim \left\langle\f{\eta}{kt}\right\rangle^{-1}\f{1}{t^2},
\eeno 
and for $\eta<\f12 kt$, it holds that $\left\langle\f{\eta}{kt}\right\rangle^{-1}\approx 1$ and
\beno
\f{k^2}{k^2+(\eta-kt)^2}\lesssim \f{k^2}{k^2+k^2t^2}\lesssim \left\langle\f{\eta}{kt}\right\rangle^{-1}\f{1}{t^2}.
\eeno 
Next we focus on the case $\f12kt\leq \eta\leq \f32 kt$. 

Let $M_0\geq 10$ be large enough. For $t\leq M_0$, we have
\beno
\f{k^2}{k^2+(\eta-kt)^2}\lesssim_{M_0} \left\langle\f{\eta}{kt}\right\rangle^{-1}\f{1}{t^2}.
\eeno 

For $t\in [t_{\rmE(|\eta|^{\f13}),\eta},2\eta]$, it holds that
\beno
\f{k^2}{k^2+(\eta-kt)^2}\leq C\sqrt{\d_{\rmL} }\left\langle\f{\eta}{kt}\right\rangle^{-1}\sqrt{\f{\pa_t\rmg}{\rmg}}.
\eeno

For $t\in [M_0,t_{\rmE(|\eta|^{\f23}),\eta}]$, we get that $t\leq 2|\eta|^{\f13}$ and then 
\beno
\f{k^2}{k^2+(\eta-kt)^2}\lesssim 1\lesssim \left\langle\f{\eta}{kt}\right\rangle^{-1}\f{|\eta|^{\f s2}}{t^{\f{3s}{2}}}1_{t\geq M_0}.
\eeno
Note that here we obtain the small parameter from the fact that for $2\tilde{q}<3s$, 
\beno
t^{-2\tilde{q}}\leq M_0^{2\tilde{q}-3s}t^{-3s}. 
\eeno
Another way to obtain the small parameter is by assuming $\la_0,\la'$ large enough while here we hope our results have no restriction on $\la_0,\la'$. 

Now we turn to the case $t\in [t_{\rmE(|\eta|^{\f23}),\eta},t_{\rmE(|\eta|^{\f13}),\eta}]\cap [20,+\infty]$ with $\f12kt\leq \eta\leq \f32 kt$. 

Thus we have 
\beno
\f{k^2}{k^2+(\eta-kt)^2}\lesssim \sqrt{\d_{\rmB}}\left\langle\f{\eta}{kt}\right\rangle^{-1}\f{1}{1+|t-\f{\eta}{k}|}\lesssim  \sqrt{\d_{\rmB}}\sqrt{\f{b(t,k,\eta)}{1+|t-\f{\eta}{k}|^2}}\left\langle\f{\eta}{kt}\right\rangle^{-1}.
\eeno

Thus we conclude that 
\beq
\begin{aligned}
|\Pi_K^2|
&\leq \f{1}{16}\|\na_L\rmA K\|_2^2
+C({M_0})\f{1}{\langle t\rangle^4}\left\|\left\langle \f{\pa_v}{t\pa_z}\right\rangle^{-1}\rmA\pa_z^{-1}\Delta_{L}P_{\neq }f\right\|_2^2\\
&\quad +C\left\|1_{t\geq M_0}\f{|\na|^{\f s 2}}{\langle t\rangle^{\f{3s}{2}}}\left\langle \f{\pa_v}{t\pa_z}\right\rangle^{-1}\rmA\pa_z^{-1}\Delta_{L}P_{\neq }f\right\|_2^2\\
&\quad +C\d_{\rmL}\left\|\sqrt{\f{\pa_t\rmg}{\rmg}}\left\langle \f{\pa_v}{t\pa_z}\right\rangle^{-1}\tilde{\tilde{\rmA}}\pa_z^{-1}\Delta_{L}P_{\neq }f\right\|_2^2\\
&\quad +C\d_{\rmB}\left\|\sqrt{\f{b(t,\na)\pa_{zz}}{\Delta_{L}}}\left\langle \f{\pa_v}{t\pa_z}\right\rangle^{-1}\rmA\pa_z^{-1}\Delta_{L}P_{\neq }f\right\|_2^2. 
\end{aligned}
\eeq

{
\appendix  

\section{Paraproduct tools} \label{F_Tools}            
In this section, we introduce the tools and notations that we should use in the Fourier analysis and paraproduct. We first introduce the dyadic partition of unity that we should use throughout the paper. Let $\varkappa(\xi)$ be a real radial  bump function supported on $\{\xi\in \R\} $  which $\varkappa(\xi)=1$ for $|\xi|\leq 1/2$ and $\varkappa(\xi)=0$ for $|\xi|\geq 3/4$. We define $\varphi(\xi)=\varkappa(\xi/2)-\varkappa(\xi)$ supported on the annulus $\{\f12\leq |\xi|\leq \f32\}$. By construction, we have the following partition of unity: 
\begin{equation*}
1=\varkappa(\xi)+\sum_{k\in\Z}\varphi (\xi/2^k)=\varkappa(\xi)+\sum_{\rmM\in 2^{\mathbb{N}}} \varphi (\rmM^{-1}\xi)
\end{equation*}
and we define the cut-off $\varphi_\rmM=\varphi(\rmM^{-1}\xi)$, each supported in in the annulus  $\f{\rmM}{2}\leq |\xi|\leq \f{3\rmM}{2}$ 

For $\rm{f}\in L^2(\R)$, we define 
\begin{equation}
\begin{aligned}
\rm{f}_\rmM=&\varphi_\rmM(|\partial_v|)\rm{f},\\
\rm{f}_{\f12}=&\varkappa (|\partial_v|)\rm{f},\\
 \rm{f}_{<\rmM}=&\rm{f}_{\f12}+\sum_{\rmK\in 2^\mathbb{N}: \rmK<\rmM}\rm{f}_{<\rmK}
\end{aligned}
\end{equation}
 Hence, we have the decomposition 
 \begin{equation*}
\rm{f}=\rm{f}_{\f12}+\sum_{\rmM\in 2^\mathbb{N}}\rm{f}_{\rmM}. 
\end{equation*}
We also have the almost orthogonality property
\begin{subequations}\label{A_1}
\begin{equation}\label{Norm_Dyadic} 
\Vert \rm{f}\Vert_2^2\approx \sum_{\rmM\in \mathbf{D}}\Vert \rm{f_\rmM}\Vert_2^2
\end{equation}
and the approximate projection property 
\begin{equation}
\Vert \rm{f}_{\rmM}\Vert_2\approx \Vert (\rm{f}_{\rmM})_\rmM\Vert_2. 
\end{equation}
\end{subequations}
More generally if $\rm{f} ={\sum_{k}\rm{D}_{k}}$ for any $\rm{D}_k$ with $\f{1}{\rmC}2^{\rm{k}}\subset \rm{supp } \,\rmD_k\subset \rmC 2^k$ it follows that 
\begin{equation}\label{A_2}
\Vert \rm{f}\Vert_2^2\approx_c \sum_k\Vert \rmD_k\Vert_2^2
\end{equation}

During much of the proof we are also working with Littlewood-Paley decompositions defined in the $(z,v)$ variables, with the notation conventions being analogous. Our convention is to use $\rmN$ to denote Littlewood-Paley projections in $(z,v)$ and $\rmM$ to denote projections only in the $v$ direction.

For any $1\leq p\leq q\leq \infty$, there exists a constant independent of $\rmM$ such that 
\begin{equation}\label{Bren_Inequallity}
{\Vert\pa_x^\alpha \rm f_\rmM\Vert} _q+{\Vert\pa_x^\alpha\rm{f}_{<\rmM/8} \Vert}_q\leq C\rmM^{d(1/p-1/q)+|\alpha|}{\Vert \rm f\Vert}_p. 
\end{equation}

\section{Important inequalies}

If $|x-y|\leq |x|/\rmK$, then it holds that 
\begin{equation}\label{A_7}
|x^s-y^s|\leq \f{s}{(\rmK-1)^{s-1}}|x-y|^s,\qquad 0<s<1.     
\end{equation}  

In many occasions, we use the following inequality, which is a result of \eqref{A_7}:
\begin{equation}\label{A_7_1}
e^{\lambda|x|^s}\leq e^{\lambda |y|^s}e^{c\lambda|x-y|^s}
\end{equation}
for $|x-y|<|x|/\rmK$, with $\rmK>1$,  $s\in (0,1)$ and $c=c(s)\in(0,1)$.

If $|y|\leq |x|\leq \rmK |y|$, then it holds that 
\begin{equation}\label{A_8}
|x+y|^s\leq \Big(\f{\rmK}{\rmK+1}\Big)^{1-s}(x^s+y^s). 
\end{equation}

\begin{lemma}[$L^p-L^q$ decay of the heat kernel]\label{Lemma_Heat_Kernel}

Let $u$ be the solution of the heat equation 
\begin{equation}
u_t-\Delta u=0,\quad u(t=0)=\varphi(x),\quad x\in \R^d,\, t\geq 0. 
\end{equation}
Let $S(t)=e^{t\Delta }$ being the heat operator.   Then it holds that for any $1\leq q\leq p\leq \infty$ 
\begin{equation}
\Big\Vert \pa_t^j\pa_x^\alpha u\Big\Vert_{p}=\Big\Vert \pa_t^j\pa_x^\alpha S(t)\varphi\Big\Vert_{p}\leq C t^{-\f d2(\f{1}{q}-\f{1}{p})+j+\f{|\alpha|}{2}} \Vert \varphi\Vert_{L^q}
\end{equation}
where $j$ is a positive integer and $\alpha=(\alpha_1, \dots,\alpha_d)$, where $\alpha_i, \, 1\leq i\leq d$ is a positive integer and $|\alpha|=\alpha_1+\dots+\alpha_d$ and $\pa_x^\alpha=\pa_{x_1}^{\alpha_1}\pa_{x_2}^{\alpha_2}\dots \pa_{x_d}^{\alpha_d}$.   
\end{lemma}
We define the Gevrey in the physical space as (see e. g. \cite{Levermore_Oliver_1997})
\begin{subequations}\label{Gevrey_Ph}
\begin{equation}\label{L^2_Gevrey_Ph}
\Vert f\Vert_{\cG^{\lambda;s}}\approx \Big[\sum_{n=0}^\infty \Big(\f{\lambda^n}{(n!)^{\f1 s}}\|\rmD^n f\|_2\Big) ^2\Big]^{\f12}
\end{equation}
We can also extend \eqref{L^2_Gevrey_Ph} define the  more general $\ell^qL^p$ based spaces (see \cite{Mouhot_Villani_2011})  as     
 \begin{equation}\label{L_p_Bsed_Space}
\Vert f\Vert_{\ell^q\rmL^p;\la}\approx \Big[\sum_{n=0}^\infty \Big(\f{\lambda^n}{(n!)^{\f1 s}}\|\rmD^n f\|_p\Big) ^q\Big]^{\f1q}
\end{equation}
\end{subequations}
Then it holds that (see \cite[Appendix A.]{Bedrossian_2015}) for $\la>\la'$ and for $p,\,q\in[1,\infty]$, we have 
\begin{equation}\label{Estimate_Generalized_Gevrey}
\begin{aligned}
\Vert f\Vert_{\ell^p\rmL^q;\la'}\leq&\, \Vert f\Vert_{\ell^1\rmL^q;\la}\lesssim_{\la-\la'}\Vert f\Vert_{\ell^p\rmL^q;\la}\\
\Vert f\Vert_{\ell^2\rmL^\infty;\la'}\lesssim&\, \Vert f\Vert_{\ell^2\rmL^2;\la}. 
\end{aligned}
\end{equation}
Next, we show the following lemma, which useful to prove the scattering result in Section \ref{Section_Main_proofs}.  
\begin{lemma}\label{Product_Gevrey_Lemma}
The following inequality  holds true: 
\begin{equation}
\begin{aligned}
\Vert fg\Vert_{\ell^1\rmL^{2};\la}
 \lesssim &\, \Vert f\Vert_{\ell^1\rmL^2;\la}\Vert g\Vert_{\ell^1\rmL^\infty;\la}
\end{aligned}
\end{equation}
\end{lemma}
\begin{proof}
We have by using \eqref{L_p_Bsed_Space} together with Leibniz rule 
\begin{equation}
\begin{aligned}
\Vert fg\Vert_{\ell^1\rmL^{2};\la}=&\,\sum_{n=0}^\infty \f{\la^n}{(n!)^{1/s}}\Vert D^n(fg)\Vert_{\rmL^2} \\
\lesssim &\,\sum_{n=0}^\infty \sum_{k=0}^n \Big(\f{n!}{k!(n-k)!}\Big)^{\f 1s} \f{\la^k \lambda^{n-k}}{(n!)^{1/s}} \Vert D^{k} f\Vert_{\rmL^2}\Vert D^{n-k} g\Vert_{\rmL^\infty}\\
 \lesssim &\, \Vert f\Vert_{\ell^1\rmL^2;\la}\Vert g\Vert_{\ell^1\rmL^\infty;\la},
\end{aligned}
\end{equation}
which gives the lemma.
\end{proof}

\end{document}